\documentclass[12pt]{article}
\usepackage{amsthm,amsmath,amssymb}

\newcommand{\ep}{\varepsilon}
\newcommand{\A}{\mathbb{A}}
\newcommand{\B}{\mathbb B}
\newcommand{\sB}{\mathcal B}
\newcommand{\C}{\mathbb C}
\newcommand{\sC}{\mathcal C}

\newcommand{\D}{\mathbb D}
\newcommand{\fD}{\mathfrak D}
\newcommand{\sD}{\mathcal D}

\newcommand{\sF}{\mathcal F}
\newcommand{\K}{\mathbb K}

\newcommand{\sM}{\mathcal M}
\newcommand{\N}{\mathbb N}
\newcommand{\sN}{\mathcal N}
\newcommand{\sO}{\mathcal O}
\newcommand{\Q}{\mathbb Q}
\newcommand{\sP}{\mathcal P}
\newcommand{\PP}{\mathbb P}
\newcommand{\fP}{\mathfrak P}
\newcommand{\sQ}{\mathcal Q}
\newcommand{\R}{\mathbb R}
\newcommand{\Z}{\mathbb Z}

\newcommand{\fA}{\mathfrak A}
\newcommand{\fR}{\mathfrak R}

\newcommand{\1}{^{-1}}
\newcommand{\frestbar}
{{\scriptscriptstyle\genfrac{.}{.}{0pt}{10}\shortmid\shortmid}}
\newcommand{\brest}[1]{{}_{\frestbar #1}} 
\newcommand{\dv}{_{\rm div}}

\newcommand{\hor}{{}^{\mathrm{h}}}

\newcommand{\pt}{\mathrm{pt.}}
\newcommand{\ld}{a} 
\newcommand{\linsys}[1]{\left|{#1}\right|} 
\newcommand{\logb}{{}^{\mathrm{log}}}
\newcommand{\md}{_{\rm mod}}
\newcommand{\norm}[1]{\left\|{#1}\right\|} 
\newcommand{\rdn}[2]{{#1}_{[{#2}]}} 
\newcommand{\rddown}[1]{\left\lfloor{#1}\right\rfloor} 
\newcommand{\rdup}[1]{\left\lceil{#1}\right\rceil} 
\newcommand{\rest}[1]{{}_{{\textstyle{|}}#1}} 

\newcommand{\simp}[1]{[#1]} 
\newcommand{\spn}[1]{\left\langle{#1}\right\rangle} 
\newcommand{\ver}{{}^{\mathrm{v}}}

\DeclareMathOperator{\act}{{act}}
\DeclareMathOperator{\Amp}{Amp}
\DeclareMathOperator{\ainv}{\text{\"{a}}}
\DeclareMathOperator{\shAmp}{\bold{Amp}}
\DeclareMathOperator{\CDiv}{CDiv}
\DeclareMathOperator{\Card}{\frak {Car}}
\DeclareMathOperator{\cent}{center}
\DeclareMathOperator{\Cl}{Cl}
\DeclareMathOperator{\shCl}{\mathbf {Cl}}

\DeclareMathOperator{\Compd}{\frak C}
\DeclareMathOperator{\Diff}{{Diff}}
\DeclareMathOperator{\Eff}{Eff}
\DeclareMathOperator{\shEff}{\bold {Eff}}
\DeclareMathOperator{\Effd}{\mathfrak E}
\DeclareMathOperator{\EffWDiv}{EffWDiv}
\DeclareMathOperator{\Exc}{Exc}
\DeclareMathOperator{\shExc}{\bold {Exc}}
\DeclareMathOperator{\Excd}{\mathfrak {Exc}}
\DeclareMathOperator{\Fix}{Fix}
\DeclareMathOperator{\Fixd}{\mathfrak F}
\DeclareMathOperator{\shFix}{\bold {Fix}}
\DeclareMathOperator{\glct}{{glct}}
\DeclareMathOperator{\Id}{Id}
\DeclareMathOperator{\LCS}{{LCS}}
\DeclareMathOperator{\lcd}{\mathfrak {lc}}
\DeclareMathOperator{\lct}{{lct}}
\DeclareMathOperator{\Mob}{Mob}
\DeclareMathOperator{\shMob}{\bold {Mob}}
\DeclareMathOperator{\Mobd}{\mathfrak M}
\DeclareMathOperator{\mult}{mult}
\DeclareMathOperator{\NE}{NE}
\DeclareMathOperator{\Nef}{Nef}
\DeclareMathOperator{\shNef}{\bold {Nef}}
\DeclareMathOperator{\Nefd}{\mathfrak {Nef}}
\DeclareMathOperator{\Nc}{N}

\DeclareMathOperator{\Nuld}{\frak N} 
\DeclareMathOperator{\Pd}{\mathfrak P} 
\DeclareMathOperator{\Pic}{Pic}
\DeclareMathOperator{\pr}{pr}
\DeclareMathOperator{\shPic}{\bold {Pic}}
\DeclareMathOperator{\Reg}{R}
\DeclareMathOperator{\reg}{reg}
\DeclareMathOperator{\Rct}{{\mathbb R}-clct}

\DeclareMathOperator{\sAmp}{sAmp}
\DeclareMathOperator{\shsAmp}{\bold {sAmp}}
\DeclareMathOperator{\sAmpd}{\mathfrak {sA}}
\DeclareMathOperator{\Spec}{Spec}
\DeclareMathOperator{\Supp}{Supp}
\DeclareMathOperator{\Tor}{Tor}
\DeclareMathOperator{\shTor}{\bold {Tor}}
\DeclareMathOperator{\WDiv}{WDiv}

\numberwithin{equation}{subsection}

\theoremstyle{definition}
 \newtheorem{defn}{Definition}
 \newtheorem{defn-prop}{Definition-Proposition}
 \newtheorem{conj}{Conjecture}
 \newtheorem{const}{Construction}

\theoremstyle{plain}
 \newtheorem{add}{Addendum}
 \newtheorem{cor-conj}{Corollary-Conjecture}
 \newtheorem{cor}{Corollary}
 \newtheorem{lemma}{Lemma}
 \newtheorem{prop}{Proposition}
 \newtheorem{thm}{Theorem}

\theoremstyle{remark}
 \newtheorem{exa}{Example}
 \newtheorem{rem}{Remark}
 \newtheorem{war}{Warning}

\title{Existence and boundedness of $n$-complements}

\date{December~4, 2020; Baltimore quarantine}

\author{V.V. Shokurov
\thanks{Partially supported
by NSF grant DMS-1400943 and
the MPIM grant.}
}

\begin{document}

\maketitle

\begin{abstract}
Theory of $n$-complements with applications is presented.
\end{abstract}

\section {Introduction} \label{intro}

Recall two key concepts for us.

\begin{defn}\label{r_comp}
Let $(X/Z\ni o,D)$ be a pair
with a [proper] local morphism $X/Z\ni o$ and
with an $\R$-divisor $D$ on $X$.
Another pair $(X/Z\ni o,D^+)$ with the same local morphism
and with an $\R$-divisor $D^+$ on $X$
is called an $\R$-{\em complement\/} of $(X/Z\ni o,D)$ if
\begin{description}
\item[\rm (1)]
$D^+\ge D$;

\item[\rm (2)]
$(X,D^+)$ is lc; and

\item[\rm (3)]
$K+D^+\sim_\R 0/Z\ni o$.

\end{description}
In particular, $(X/Z\ni o,D^+)$ is
a {\em [local relative]\/} $0$-{\em pair}.
[Note that the neighborhood of $o$ for another pair can be
different from the original one.]

The complement is {\em klt} if
$(X,D^+)$ is klt.

\end{defn}

\begin{defn}[{\cite[Definition~5.1]{Sh92}}]\label{n_comp}
Let $n$ be a positive integer, and
$(X/Z\ni o,D)$ be a pair with a local morphism $X/Z\ni o$ and
with an $\R$-divisor $D=\sum d_iD_i$ on $X$.
A pair $(X/Z\ni o,D^+)$, with the same local morphism and
with a $\Q$-divisor $D^+=\sum d_i^+D_i$ on $X$,
is an $n$-{\em complement\/} of $(X/Z\ni o ,D)$ if
\begin{description}
\item[\rm (1)]
for every prime divisor $D_i$ on $X$,
$$
d_i^+\ge
\begin{cases}
1, \text{ if } &d_i=1;\\
\rddown{(n+1)d_i}/n &\text{ otherwise};
\end{cases}
$$

\item[\rm (2)]
$(X,D^+)$ is lc; and

\item[\rm (3)]
$K+D^+\sim_n 0/Z\ni o$.

\end{description}
The number $n$ is called a {\em complementary index\/}.
[Note that the neighborhood of $o$ for another pair
can be different from the original one.]

The $n$-complement is {\em monotonic\/} if
$D^+\ge D$.

\end{defn}

\begin{rem} \label{remark_def_complements}

(1)
By (2) of Definitions~\ref{r_comp} and \ref{n_comp},
$D^+$ is a subboundary.
So, by (1) of both definitions,
$D$ is a subboundary too.
By (1) of both definitions, $d_i^+=1$ if $d_i=1$.

Again by (1) of both definitions $D^+$ is a boundary
if so does $D$.

(2)
Immediate by definition an $n$-complement is monotonic if and only if
it is an $\R$-complement.
In general, an $n$-complement is not monotonic.
However, the latter property can happen
for certain $n$-complements or multiplicities, boundaries,
e.g., for hyperstandard ones \cite[Theorem~1.4 and Lemma~3.5]{PSh08} \cite[Theorem~1.7]{B}.
(Cf. \cite[Theorem~1.6]{HLSh} vs Example~\ref{rddown_(n+1)_m_m}, (4) below.)

It is not surprising that
the existence of an $n$-compliment does not
imply the existence of an $\R$-compliment.
But by (1) of Definition~\ref{n_comp} and
since $\rddown{(n+1)d_i}/n$ is very close to $d_i$
for every $d_i\in [0,1)$ and
sufficiently large $n$, the multiplicity
$d_i^+$ for an $n$-compliment become
$\ge d_i$ or very close to $d_i$.
This observation and the remark (1) will be used in
a proof of the easy part of
Theorem~\ref{R-vs-n-complements} below,
one of our main results.

However, if $d_i\ge 1$ then
$\rddown{(n+1)d_i}/n>1$ and typically
$>d_i$.
On the other hand, if $d_i<0$ then
$\rddown{(n+1)d_i}/n<0$ and typically
$<d_i$.

(3)
By (3) of Definition~\ref{n_comp}, the (Cartier and log canonical) index of $K+D^+$ divides $n$,
in particular, $nD^+$ is integral.
Thus $D^+$ is automatically $\Q$-divisor.

(4) If $X/Z\ni o$ is proper then
$\sim_\R$ in (3) of Definition~\ref{r_comp}
can be replaced by the numerical equivalence $\equiv$
in the following two cases:
\begin{description}

  \item[]
$X/Z\ni o$ has weak Fano type \cite[Corollary~4.5]{ShCh}; or

  \item[]
$D^+$ is a boundary \cite[Theorem~0.1,(1)]{A05} \cite[Theorem~1.2]{G}.

\end{description}

(5) We can define also $\R$- and $n$-complements for pairs $(X/Z,D)$
with not necessarily local morphisms.
But they have more complicated behaviour.
In this situation it is better to use the relative version
$\sim_Z,\sim_{n,Z}$ and $\sim_{\R,Z}$ of linear equivalences
instead of the usual one
$\sim,\sim_n$ and $\sim_\R$ respectively (see~\ref{notation_terminology} below).
Then the local over $Z$ existence of $\R$-complements usually
implies the existence of an $\R$-complement over $Z$
(see Addendum~\ref{R_complement_criterion}, \cite[Corollary~4.5]{ShCh}  and
cf. Proposition~\ref{local_compl}).
However, for local over $Z$ $n$-complements, the index $n$ can depend on
a point $o\in Z$.
To find a universal $n$ is a real challenge (cf. Addendum~\ref{bounded_compon}).

For both cases of (4), $\equiv$ over $Z$ is equal to $\sim_{\R,Z}$.

\end{rem}

\begin{exa} \label{1stexe}

(1)
Every relative log Fano pair $(X/Z,B)$ over a quasiprojective variety $Z$
with a boundary $B$
has an $\R$-complement $(X/Z,B^+)$.
However, $\sim_\R$ should be replaced by its relative
version $\sim_{\R,Z}$ (cf. Remark~\ref{remark_def_complements}, (5)).
In other words, in this situation (3) of Definition~\ref{r_comp}
has the following form:
$$
K+B^+\sim_\R \varphi^*H
\text{ and } K+B^+\equiv 0/Z,
$$
where $\varphi\colon X\to Z$ and $H$ an $\R$-ample divisor on $Z$.
[Recall that] we suppose that $Z$ is quasiprojective.

(2) By definition every pair $(X/Z,0)$ with wFt $X/Z$
has a klt $\R$-complement (see Fano and weak Fano types
in~\ref{notation_terminology} below
and cf.~\cite[Lemma-Definition~2.6, (ii)]{PSh08}).

(3)
Every complete $0$-pair $(X,D)$ has an $\R$-complement and $D^+=D$.
Moreover, if $D^+=D=B$ is a boundary
then $K+D^+=K+B\sim_\R 0\Leftrightarrow \equiv 0$
\cite[Theorem~0.1,(1)]{A05} \cite[Theorem~1.2]{G}.

Every relative proper $0$-pair $(X/Z,D)$ also has an $\R$-complement, e.g., $D^+=D$.
However, $D^+=D$ always if the pair is local and nonklt over $o$
near every connected component of $X_o$, the central fiber.
(Cf. with Maximal lc $0$-pairs in Section~\ref{lc_type_compl}.)

(4)
Let $(X,D)$ be a pair with a toric variety $X$ and
with torus invariant $D$.
Then, for any morphism $X/Z$ and any positive integer $n$,
$(X/Z,D)$ has an $n$-compliment if $D$ is a subboundary.
If $K+D$ is $\R$-Cartier the last assumption is equivalent to
the lc property of $(X,D)$.
For instance, we can take $D^+$ equal to the sum of invariant
divisors.
The complement is torus invariant too.
This is a rear case where we can use $\sim$ instead of
its relative version $\sim_Z$.

For complete toric $X$, $D$ is a subboundary if and only if
$(X,D)$ has an $\R$-compliment.
In this situation, the invariant complement is unique.

(5) Every $n$-complement $(\PP^1,B^+)$ of $(\PP^1,0)$ corresponds
to a polynomial $f\in k[x]$ of the degree $2n$:
$$
B^+=(f)_0/n,
$$
such that every root of $f$ has multiplicity $\le n=(\deg f)/2$
(cf. the semistability of polynomials).
Indeed, for an appropriate affine chart $\A^1=\PP^1\setminus \infty$
with a coordinate $x$, we can assume that
$B^+$ is supported in $\A^1$ and $nB^+$ is given by the zeros of some polynomial
$f$ in $x$.
Notice that $B^+$ is effective by (1) of Definition~\ref{n_comp}.
By (3) of the definition $(\deg f)/n=-\deg K_{\PP^1}=2$ and
$f$ has the required degree.
The multiplicities of roots $\le n$ by (2) of Definition~\ref{n_comp}:
$$
\mult B^+=\frac 1n \mult_a f\le 1, a\in k.
$$

If $\PP^1$ is defined over algebraically nonclosed field $k$,
we can construct a $n$-complement $(\PP^1,B^+)$ over $k$
taking a sufficiently general polynomial $f$ in $k[x]$.
(Its roots belong now to $\overline{k}$.)
Such a polynomial exists: take a polynomial with simple roots.
E.g., for $n=2$, $f$ is a quadratic polynomial with nonzero
discriminant.

\end{exa}

\begin{thm} \label{R-vs-n-complements}
Let $(X/Z\ni o,B)$ be a pair with a wFt morphism $X/Z\ni o$
and a boundary $B$ on $X$.
Then $(X/Z\ni o,B)$ has an $\R$-complement if and only if
$(X/Z\ni o,B)$ has $n$-complements for
infinitely many positive integers $n$.

\end{thm}

The hard part of the theorem about existence of
$n$-complements will be stated more precisely
in Theorem~\ref{exist_n_compl} below and proved in Section~\ref{constant_sheaves}.
The converse statement follows
from the closed property for $\R$-complements in Theorem~\ref{R_compl_polyhedral} below
and proved in Section~\ref{constant_sheaves}.

\paragraph{Restrictions on complementary indices.}
Let $I$ be positive integers,
$\ep$ be a positive real number,
$v$ be a nonrational vector in a finite dimensional $\R$-linear space $\R^l$, and $e$ be a direction in the rational affine span $\spn{v}$ of $v$.
Usually we are looking for $n$-complements with
$n$ satisfying the following properties:
for $n$ there exists a rational vector $v_n\in\spn{v}$ such that
\begin{description}

\item[\em Divisibility\/:]
$I$ divides $n$;

\item[\em Denominators\/:]
$nv_n$ is integral, that is,
$nv_n\in\Z^l$;

\item[\em Approximation\/:]
$\norm{v_n-v}<\ep/n$.

\item[\em Anisotropic approximation\/:]
$$
\norm{\frac{v_n-v}{\norm{v_n-v}}-e}<\ep.
$$

\end{description}
If $\ep\le 1/2$, $v_n$ is unique and
is the best approximation with denominator $n$.
Note that $\norm{\ }$ denotes the {\em maximal absolute value\/} norm \cite[Notation~5.16]{Sh03}.

Most of applications of the theory of complements
are based on these restrictions.
The choice of $n$ in applications depend on
$I,\ep,v,e$ but also on the dimension $d$.

\begin{thm}[Existence of $n$-complements] \label{exist_n_compl}
Let $I,\ep,v,e$ be the data as
in Restrictions on complementary indices above and
$(X/Z\ni o, B)$ be a pair with a wFt morphism $X/Z\ni o$
and a boundary $B$ on $X$.
Suppose also that $(X/Z\ni o,B)$ has an $\R$-complement.
Then $(X/Z\ni o,B)$ has $n$-complements for
infinitely many positive integers $n$
under Restrictions on complementary indices with the given data.
\end{thm}

Originally, it was expected that, in a given dimension $d$,
there exists a finite set of complementary indices $n$
such that
the existence of an $\R$-complement implies the existence
of an $n$-complement.
This is not true if $d\ge 3$
(see Examples~\ref{unbounded_lc_compl}, (1-3)).
We have only the following slightly weaker form
of the expectation.

\begin{thm}[Boundedness of $n$-complements] \label{bndc}
Let $d$ be nonnegative integer,
$\Delta\subseteq (\R^+)^r,\R^+=\{a\in\R\mid a\ge 0\}$, be a compact
subset (e.g., a polyhedron) and
$\Gamma$ be a subset in the unite segment $[0,1]$ such that
$\Gamma\cap\Q$ satisfies the dcc.
Let $I,\ep,v,e$ be the data as
in Restrictions on complementary indices.
Then there exists a finite set
$\sN=\sN(d,I,\ep,v,e,\Gamma,\Delta)$ of positive integers
({\em complementary indices\/})
such that
\begin{description}

\item[\rm Restrictions:\/]
every $n\in\sN$ satisfies
Restrictions on complementary indices with the given data;

\item[\rm Existence of $n$-complement:\/]
if $(X/Z\ni o,B)$ is a pair with wFt $X/Z\ni o$,
$\dim X\le d$,
connected $X_o$ and  with a boundary $B$,
then $(X/Z\ni o,B)$ has an $n$-complement for some $n\in\sN$
under either of the following assumptions:

\item[\rm (1)\/]
$(X/Z\ni o,B)$ has a klt $\R$-complement;
or

\item[\rm (2)\/] $B=\sum_{i=1}^r b_iD_i$ with $(b_1,\dots,b_r)\in \Delta$ and additionally,
for every $\R$-divisor $D=\sum_{i=1}^rd_iD_i$ with $(d_1,\dots,d_r)\in \Delta$,
the pair $(X/Z\ni o,D)$ has an $\R$-complement, where
$D_i$ are
effective Weil divisors (not necessarily prime); or

\item[\rm (3)\/]
$(X/Z\ni o,B)$ has an $\R$-complement and,
additionally, $B\in\Gamma$.

\end{description}

\end{thm}

\begin{add} \label{bounded_compon}
We can relax the connectedness assumption on $X_o$ and
suppose that the number of connected components of $X_o$
is bounded.

Moreover, it is enough to suppose the boundedness
up to local (even formal) isomorphisms over $Z\ni o$
of corresponding neighborhoods.

\end{add}

\begin{add} \label{nonclosed}
We can relax also the assumption that the base field $k$ is algebraically closed
and suppose only that $\mathrm{char}k=0$.

Similarly, the theorem holds for $G$-pairs where $G$ is a finite group.

\end{add}

\begin{rem} \label{N_with_parameters}

(1)
The set $\sN(d,I,\ep,v,e,\Gamma,r)$ is not unique
by Divisibility and Approximation
in Restrictions on complementary indices and
depends on the data $d,I,\ep,v,e,\Gamma,r$.
We use for $\sN$ notation $\sN(d,I,\ep,v,e,\Gamma,r)$
only to show parameters $d,I,\ep,v,e,\Gamma,r$, on which
depend the choice of $\sN$.
Actually, we need only the partial data
$d,I,\ep,v,e$ and $\sN=\sN(d,I,\ep,v,e)$
under the assumption (1) of the theorem.

Below we use also some other sets for
boundary multiplicities instead of $\Gamma$ and
with a different meaning (cf.~\ref{direct_dcc}).

(2)
If $v$ is nonrational then $\spn{v}$ is not a point
and has a direction.
Otherwise there are no directions and
Anisotropic approximation is void.
Nonetheless the other properties hold if we
take $v_n=v$ for every $n\in \sN$ and
suppose that $I$ is sufficiently divisible.
In this situation we get $\sN=\sN(d,I,\ep,v)$.

(3)
Denominators property means that $v_n$
has rational coordinates with
denominators (dividing) $n$.
We already noticed that
by Approximation property, if $\ep\le 1/2$ then $v_n$ is the best
rational approximation of $v$ with those denominators.

(4) Our main results and applications work for any base field $k$
of characteristic $0$ and $G$-pairs.

\end{rem}

For certain boundaries it is enough a single complementary index,
e.g., for boundaries with only hyperstandard multiplicities
\cite[Theorem~1.4]{PSh08} \cite[Theorem~1.7]{B}.
Alas, it is not true in general.

\begin{exa}[Nonsingle index] \label{nonsingle}
Let $d,n$ be two positive integers.
Take the projective space $X=\PP^d$ of dimension $d$ with
a boundary
$$
B=\sum \frac1{n+1} D_i,
$$
where $D_i$ runs through $(d+1)(n+1)$ distinct rather
general hyperplanes.
Then $(X,B)$ is a klt $0$-pair, has a klt $\R$-complement and it is actually
$(X,B)$ itself by Example~\ref{1stexe}, (3).
However, $(X,B)$ doesn't have $n$-complements because
$$
B^+\ge \frac{\rddown{(n+1)B}}n=\frac{n+1}n B>B
$$
by (1) of Definition~\ref{n_comp}.
Hence, for $d\ge 1$,
every set $\sN(d,I,\ep,v,e)$
under (1-2) of Theorem~\ref{bndc} has at least two indices.
Similar examples exist under (3) of the theorem.

Similarly, for any finite set of positive integers $\sN$,
we can find a pair $(X,B)$ of dimension $d$ which has a klt $\R$-complement
but doesn't have $n$-complements for all $n\in\sN$.
However $X$ of such an example should be reducible or not connected!
This is why in Theorem~\ref{bndc} above and
Conjecture~\ref{conjecture_bndc} below
we suppose that $X/Z\ni o$ is
local with connected $X_o$, e.g.,
with a contraction $X/Z\ni o$ (cf. Remark~\ref{bcconj}, (2)).
(Cf. also Examples~\ref{unbounded_lc_compl}, (1-3).)

\end{exa}

\begin{rem} \label{bcconj}

(1)
Taking general hyperplane sections of $Z$, if they are needed,
we can assume in Theorems~\ref{R-vs-n-complements}-\ref{bndc} and
in similar statements that
$o$ is a closed point of $Z$.
However we give a more rigorous treatment below for all points $o$
avoiding taking such sections.

(2)
Under the assumption that $k$ is algebraically closed,
another assumption that $X/Z\ni o$ is a local morphism of [normal] varieties
with {\em connected\/} $X_o$
in Theorem~\ref{bndc} is necessary for the finiteness
of $\sN$, the boundedness of $n$-complements,
even in the klt case (1) (cf. Addendum~\ref{nonclosed}).
In particular, for algebraically closed $k$ and
normal complete $X/k$, the natural morphism $X/o=\Spec k$ is always local and
$X_o=X$ is connected if and only if $X$ is irreducible.
By Example~\ref{nonsingle} the irreducibility is important
for boundedness.

According to Addendum~\ref{bounded_compon}
we can replace the connectedness of $X_o$ by
the {\em boundedness of connected components\/}:
for a natural number $N$, Theorem~\ref{bndc} still holds
if instead of the connectedness of $X_o$
the fiber $X_o$ has at most $N$ connected components.
In particular, for any given pair $(X/Z,B)$ with a wFt morphism
into a quasiprojective variety $Z$ or
a local one $Z\ni o$ (resp. in \'etale sense),
with a boundary $B$ and
over any field $k$, by Theorem~\ref{exist_n_compl} and Proposition~\ref{local_compl}
there exists
an $n$-complements for infinitely many $n$
if some $\R$-complement exists even locally over $Z$.
Actually, the same results expected without
the wFt assumption for $X/Z$ and $X/Z\ni o$
(see Conjecture~\ref{conjecture_bndc}).

Cf. Example~\ref{nonsingle} above.

Complements in a nonnormal case are
a real challenge.

(3) Another challenge is construction of
$\ep$-$n$-complements \cite{Sh04b}
(see Conjecture~\ref{a_n_compl} below).

\end{rem}

Main Theorems~\ref{R-vs-n-complements}-\ref{bndc} have
many important applications.
Some of them are already known but some new.
Certain special applications we meet
in the course of the proof of our main theorems.
Among other applications -- Acc of log canonical thresholds, etc --
see in Section~\ref{applic}.

A few words about the proof of
Theorems~\ref{R-vs-n-complements}-\ref{bndc}.
The key part of the paper is
the existence and boundedness of Theorem~\ref{bndc}
under the existence of klt $\R$-complements, that is,
under (1) of the theorem.
See the proof in Section~\ref{klt_nongeneric}
and also at the end of Section~\ref{lc_type_compl}.
The theorem under (1) implies the theorem under (3) from which
follows the theorem under (2).
See the proof in Section~\ref{lc_type_compl}.
In its turn, Theorem~\ref{exist_n_compl} is
immediate by Theorem~\ref{bndc} under (2) and Addendum~\ref{bounded_compon}
or by \cite[Corollary~1.3]{BSh}.
See the proof in Section~\ref{constant_sheaves}.
Finally, Theorem~1 follows from Theorem~\ref{exist_n_compl}
and the closed property for $\R$-complements.
See the proof also in Section~\ref{constant_sheaves}.

All $n$-complements of Theorem~\ref{bndc} under (1)
can be obtained from $n$-complements for
exceptional pairs by two standard construction:
lifting of $n$-complements from lower dimensional bases
of fibrations and
extension of $n$-complements from lower dimensional lc centers.
These constructions are treated respectively in~\ref{adjunction_div}
and~\ref{adjunction_on_divisor}.

\paragraph{Generalizations and technicalities.}
Both construction are applied to appropriate birational
models of pairs, in particular, we use modifications
which blow up divisors.
Blowdowns  of divisors preserves $n$-complements
but blowups do not so.
To overcome this difficulty we introduce b-$n$-complements
in Definition~\ref{b_n_comp}, a birational version of $n$-complements.
Note that every monotonic $n$-complement is
automatically a b-$n$-complement for log pairs
(cf. Remarks~\ref{remark_def_complements}, (2) and~\ref{rem_def_b_n_compl}, (1-2)).
In particular, this holds for $B$ with hyperstandard
multiplicities and sufficiently divisible $n$.
This is why we do not need such a sophistication
for hyperstandard boundaries.
But usually, $n$-complements
for a boundary $B$ with arbitrary multiplicities
are not b-$n$-complements by Example~\ref{b_n_comp_of_itself}
(cf. Proposition~\ref{b_n_comp-nc}).
Fortunately, it is true for an appropriate
low approximation of $B$ which are introduced in Hyperstandard sets
in Section~\ref{technical}.
This leads to reformulation of our main result
in terms of b-$n$-complements and of
appropriate approximations of boundaries in Theorem~\ref{b_n_comp_klt}
and Addendum~\ref{standard_klt_complements}.
Respectively, existence of a (klt) $\R$-complement is
transform into existence of an $\R$-complement for
approximations of boundaries and hidden in the proof
of main results.
Additionally we generalize our results for
pairs with b-boundaries, including,
the Birkar-Zhang pairs.
Again this is not our caprice or aspiration for
generalization but rather a necessity related
to appearance of moduli part of adjunction divisors
and the adjunction itself that will be recalled in Section~\ref{adj}.
So, we try to use these technical staff only
where they are actually needed and left
to more advanced readers other possible
generalizations.

In most of results we assume that $X/Z\ni o$ has wFt.
Expected generalizations in the nonwFt case see in
Conjecture~\ref{conjecture_bndc}.

\paragraph{What do we use in the proof?}
This is a reasonable question, especially, if
we would like to understand foundations in
the theory of complements and its relation to
the LMMP.
Actually, we use and very essentially two Birkar's
results and the $D$-MMP for Ft $X/Z$:
\begin{description}

  \item[\rm (1)]
boundedness of $n$-complements in the hyperstandard case:
$B\in\Phi=\Phi(\fR)$, where $\fR$ is a finite set of
rational numbers in $[0,1]$ \cite[Theorem~1.7]{B};

  \item[\rm (2)] (BBAB)
boundedness of $\ep$-lc Fano varieties \cite[Theorem~1.1]{B16};
and

 \item[\rm (3)]
the $D$-MMP holds for every $\R$-Cartier
$\R$-divisor $D$ on every Ft $X/Z$ \cite[Section~5]{Sh96}.

\end{description}
Actually, we use~(1) only for exceptional pairs
$(X,B)$ with a boundary $B\in \Phi$ and $\Q$-factorial Ft $X$.
Moreover, arguments as in \cite[Section~6]{PSh08} and \cite{HX} allow
to suppose additionally that $B$ has multiplicities
only in a finite subset of $\Phi$ and $\rho(X)=1$.
In this situation~(1) is equivalent to an effective nonvanishing:
\begin{description}

  \item[\rm (4)]
let $d$ be a nonnegative integer and $\fR$ be
a finite set of rational numbers in $[0,1]$ then
there exists a positive integer $N$ such that
for every exceptional pair $(X,B)$ with $B\in\fR$ and
$\Q$-factorial Ft $X$ with $\dim X=d,\rho(X)=1$,
$$
\linsys{-N(K+B)}\not=\emptyset.
$$

\end{description}
Moreover, we need the boundedness (2) only
in this situation, that is, for $X$ as in~(4).
This boundedness follows from (4) and a birational boundedness in
the log general case \cite[Section ~4]{PSh08}
\cite{HX} \cite{B}.
The $D$-MMP of (3) is well-known and essentially follows
form finiteness of (bi)rational $1$-contractions of $X/Z$  \cite[Corollary~5.5]{ShCh}
and extension results \cite[Theorem~0.1]{S} \cite[Theorem~4.1]{T} \cite[Theorem~5.4]{HM}.
Thus~(4) is exactly what we need.
This key result for us is nontrivial and
a fundamental one.

Boundedness of lc index in Corollary~\ref{bounded_lc_index}
will be established in two steps: first, for hyperstandard
boundary multiplicities and then the general case based
on $n$-complements for arbitrary boundary multiplicities
(including nonrational).
For hyperstandard multiplicities, we use
dimensional induction in the local case and semiexceptional lc type of
Theorem~\ref{semiexcep_compl_with_filtration}
in the global case or again \cite[Theorem~1.7]{B}.

\subsection{Notation and terminology} \label{notation_terminology}

By $X/Z\ni o$ we mean a {\em local morphism over\/} $Z$, where
$o$ is a point in $Z$.
That is, $X/Z$ is a morphism $\varphi\colon X\to Z$
into a neighborhood of the point $o$ in $Z$.
We consider such morphisms as germs, that is, they are {\em equal\/} if
so they do over some (possibly smaller) neighborhood.
Similarly, we understand statements about local morphisms:
properties, conditions, constructions, etc.
For instance, a prime Weil divisor $D$ on $X$
in this situation is such a divisor that
$D$ exists over every neighborhood of $o$,
equivalently, $\varphi(D)$ contains $o$.
Note that $o$ is not necessarily closed.
Another example of this kind the log canonical property (lc)
in (2) of Definitions~\ref{r_comp} and~\ref{n_comp}:
it means lc of $(X,D^+)$ over some neighborhood of $o$ in $Z$,
that is, every nonlc center of $(X,D^+)$ does not
intersect $X_o$, the central fiber.

Usually, we suppose that $Z$ is quasiprojective.

Sometimes we use notation $(X,D)\to Z$ instead of a pair $(X/Z,D)$.
A local pair $(X/Z\ni o,D)$ is {\em global\/} if $X=X_o$.
Usually, we denote such a pair by $(X,D)$.
Similarly for bd-pairs.

We suppose always that $X$ is
a normal irreducible algebraic space or variety
over an algebraically closed field $k$ of characteristic $0$.
Respectively, $Z$ is a quasiprojective algebraic variety or scheme over $k$ and
the neighborhoods are in the Zariski topology.
Usually, we suppose that the central fiber
$X_o=\varphi\1o$ is connected.
E.g., this holds when $X/Z\ni o$ is a contraction.

An $\R$-divisor $D$ on $X$ is an element of
$\WDiv_\R X$, the group of Weil $\R$-divisors on $X$.
Every $\R$-divisor $D$ has a linear presentation
in terms of distinct prime Weil divisors $D_i$,
the standard basis of $\WDiv_\R X$:
$D=\sum d_i D_i$, where
$d_i=\mult_{D_i} D\in \R$ is the multiplicity of $D$
in $D_i$.

For a pair $(X/Z\ni o,D)$
with a local morphism $X/Z\ni o$,
$D=\sum d_i D_i$ is an $\R$-divisor on $X$,
defined locally over $Z\ni o$[, that is,
the prime divisors $D_i$ intersect
the central fiber $X_o$].

An $\R$-divisor $D=\sum d_iD_i$ on $X$ is a {\em subboundary\/} if
all $d_i\le 1$.
Respectively, $D$ is a {\em boundary\/} if all $0\le d_i\le 1$.

For an $\R$-divisor $D=\sum d_iD_i$ on $X$ and
a subset $\Gamma\subseteq \R$,
$D\in\Gamma$ denotes that every $d_i\in\Gamma$.
E.g., $D$ is a subboundary if and only if $D\in(-\infty,1]$.
Respectively, $D$ is a boundary if and only if
$D\in[0,1]$, the unite segment.
Tacite agreement that $0\in \Gamma$ always.

$K=K_X$ denotes a canonical divisor on $X$.
$K_Y$ ditto on any other $Y$.

For a natural number $n$,
two $\R$-divisors $D$ and $D'$ on $X$ are
$n$-{\em linear equivalent} if $nD\sim nD'$
where $\sim$ denotes the linear equivalence.
Respectively, $\sim_n$ will denote $n$-linear
equivalence.

Two $\R$-divisors $D$ and $D'$ are $\R$-linear equivalent if
$D-D'$ is $\R$-principal, that is, an $\R$-linear combination
of principal divisors.
We denote the equivalence by $\sim_\R$.

In general, a linear equivalence $D\sim_Z D'$ or $D\sim D'/Z$ on $X$ over $Z$ is
a local linear equivalences over $Z$, that is, $D-D'\sim 0$ on
some (open) neighborhood of any fiber of $X/Z$.
That is, $D-D'$ is locally principal over $Z$.
This does not imply that $D-D'\sim 0$ or principal on whole $X$.
However, if  $X/Z$ is proper and
$Z$ is quasiprojective then this is true modulo
vertical (base point) free divisors on $X$:
$$
D-D'\sim V-\varphi^*H,
$$
where $V$ is a vertical (base point) free divisor on $X$ and
$H$ is a very ample divisor on $Z$.
The difference is an integral linear combination of
vertical free divisors.
Note that $V$ is automatically vertical, if it is free, because $V$ is numerically trivial
on every fiber of $X/Z$.
Moreover, if $X/Z$ is a contraction then
$V=\varphi^*C$, where $C$ is a free divisor on $Z$ \cite[Proposition~3]{Sh19}.
Thus $V-\varphi^*H=\varphi^*(C-H)$, an inverse image of
a Cartier divisor from $Z$.
The same holds for $\sim_{\Q,Z}$ with vertical $\Q$-free or $\Q$-semiample divisors.
The relative $\R$-linear equivalence $\sim_{\R,Z}$ is more subtle.
We usually, use this relation in the local sense, that is,
$D-D'$ is $\R$-principal locally over $Z$.
However, if $X/Z$ has wFt then
$$
D-D'\sim_\R V-\varphi^*H,
$$
for some vertical $\R$-free divisor $V$ on $X$ and
some ample divisor $H$ on $Z$.
Moreover, if $X/Z$ is a contraction then
$V=\varphi^*C$ for an $\R$-free divisor $C$ on $Z$.

A linear equivalences on $X$ over $Z\ni o$ are local
by definition:
linear equivalences on $X$ over a neighborhood of $o$.
Respectively, a {\em numerical\/} equivalence $\equiv/Z\ni o$
is the numerical equivalence with respect to
a (proper) local morphism $X/Z\ni o$
(or ,more generally, $\equiv$ on every complete curve of $X$;
cf. Nef below).

A proper local pair $(X/Z\ni o,D)$ is
a $0$-{\em pair\/} if $(X,D)$ is lc and
$K+D\sim_\R 0/Z\ni o$.
If $D$ is a boundary we can replace $\sim_\R$ by $\equiv/Z$
(see Remark~\ref{remark_def_complements}, (4) above).

\paragraph{Nef.}
In the paper, if it is not stated otherwise,
we suppose that the (b-)nef property of a divisor means
nonnegative of its intersection with every complete curve \cite[Remark~1]{Sh17}:
an $\R$-divisor $D$ on $X$ is nef if
it is $\R$-Cartier and $(C.D)\ge 0$
for every complete curve $C$ in $X$ (respectively, on a stable model of $X$).
It is the usual nefness if $X$ is complete and
respectively the relative one over a scheme or a space $S$
if $X/S$ is proper.
Notice that if $D$ is nef then $D$ is nef over $S$
for every proper $X\to S$.
That nefness is typical for the b-nef property of $\sD\md$
(cf. Theorem~\ref{b_nef} and Conjecture~\ref{mod_part_b-semiample}).

\paragraph{Fano and weak Fano types.}
Both types can be defined in terms of $\R$-complements.

A {\em weak Fano type\/} (wFt) morphism $X/Z\ni o$
is such a proper local morphism that the pair $(X/Z\ni o,0)$
has a klt $\R$-complement $(X/Z\ni o,B)$ with
big $B$.
Equivalently, there exists a big boundary $B$ on $X$
such that $(X/Z\ni o,B)$ is a klt $0$-pair.
Similarly, we can define a wFt morphism $X/Z$.
Actually, for quasiprojective $Z$,
the latter morphism has wFt if and only if
it has wFt locally over $Z$.

Respectively, a {\em Fano type\/} (Ft) morphism
is such a morphism that the pair $(X/Z\ni o,B)$
is a klt log Fano for some boundary $B$ on $X$.
(In particular, $X/Z\ni o$ is proper.)
Equivalently, $X/Z\ni o$ has Ft if and only if
$X/Z\ni o$ has wFt and projective (cf. Example~\ref{1stexe}, (1)).
The same works for $X/Z$.

For example, every toric morphism has wFt and
Ft if it is projective.

\begin{lemma} \label{wTt_vs_Ft}
Proper $X/Z$ has wFt if and only if $X/Z$ is a small birational modification
of Ft $Y/Z$.
\end{lemma}

\begin{proof}
Almost by definition, Ft $Y/Z$ has a klt $\R$-complement with big $B$.
By Proposition~\ref{small_transform_compl} below
every small birational modification $X/Z$ has a klt $\R$-complement
with big $B$.
Thus $X/Z$ has wFt.
Notice also that the big property is invariant under
small birational modifications.

Conversely, let $(X/Z,B)$ be a klt $0$-pair with big $B$.
A required model $Y/Z$ is an lc model of $(X/Z,(1+\ep)B)$
up to possible resolution of some divisors, where $\ep$
is a sufficiently small positive real number.
We can suppose that $X$ is $\Q$-factorial.
Use a Zariski decomposition $B=M+F$, $M$ is big.
By construction $(Y/Z,B_Y)$ is also a klt $0$-pair where
$B_Y$ is the image of $B$ on $Y$.
Moreover, $X\dashrightarrow Y$ is a $1$-contraction and
contracts only prime Weil divisors $D$ of $X$ trivial with
respect to $K+B$.
They have the log discrepancy $\ld(X,B,D)=\ld(Y,B_Y,D)\in (0,1]$.
By \cite[Theorem~3.1]{Sh96} we can blow up them projectively!
\end{proof}

\begin{cor} \label{semiample}
Let $X/Z$ be of wFt and $D$ be an $\R$-Cartier divisor on $X$.
Then $D$ is semiample over $Z$ if and only if
$D$ is nef over $Z$.
\end{cor}

\begin{proof}
It is sufficient to verify that if $D$ is nef over $Z$
then $D$ is semiample over $Z$.
If $X/Z$ has Ft it well-known.
In general, by Lemma~\ref{wTt_vs_Ft} there exists a birational model $Y/Z$ of $X/Z$
such that $X\dashrightarrow Y$ is a small modification (isomorphism in
codimension $1$) and $Y/Z$ has Ft.
We can suppose also that $Y$ is $\Q$-factorial.
Thus $D$ on $Y$ is also $\R$-Cartier.
However, the small modification may not respect nef and semiample properties.

By definition there exists a big over $Z$ boundary $B$ on $X$ such that
$(X/Z,B)$ is a klt $0$-pair.
This implies that there exists a boundary $B'$ on $X$ and some positive $\ep$
such that $(X,B')$ is klt and
$K+B'\sim_{\R,Z} \ep D$.
The semiampleness conjecture implies the required statement.
The conjecture holds for Ft $X$ and for wFt $X$ too.
Indeed it is enough to verify that $K_Y+B'$ gives
a b-contraction over $Z$ or $(Y/Z,B')$ has a minimal model.
We do not know the pleudoeffective over $Z$ property of $D$
because $X/Z$ is not projective.
Instead we know that $(C.D)\ge 0$ for every curve $C$ on $X$ over $Z$ not
passing through a subset of codimension $\ge 2$.
The same property holds for $K_Y+B'\sim_{\R,Z} \ep D$ on $Y/Z$.
It is sufficient for existence of a minimal model.

\end{proof}

\begin{cor} \label{free}
Let $d$ be a nonnegative integer.
There exists a positive integer $N$ depending only on $d$
such that if
$X/Z$ has wFt with $\dim X=d$ and
$D$ is a nef over $Z$ Cartier divisor on $X$ then
$ND$ is base free over $Z$.

\end{cor}

\begin{proof}
If $D$ is big over $Z$ the corollary follows the relative version of \cite[Thorem~1.1]{K93}.
The numerically trivial over $Z$ case can be reduced to Ft by Lemma~\ref{wTt_vs_Ft}
and holds by ibid.
Otherwise by Corollary~\ref{semiample} there exists a contraction
$f\colon X\to V/Z$ and an $\Q$-ample divisor $A$ on $V$ such that
$f^*A\sim_{\Q,Z} D$.
If we can replace the $\Q$-ample property and respectively
$\sim_{\Q,Z}$ by $m$-ample (that is, $mA$ is ample Cartier) and
$\sim_{m,Z}$, there $m$ is a positive integer depending only on
$d$, then we get a required effective semiampleness for $D$
because $V$ has also Ft.
Below we slightly modify this idea.
It is easy by the (relative) LMMP over $V$ to get a Mori klt fibration
$g\colon T\to W/V$ such that
\begin{description}

\item[]
$X\dashrightarrow T$ is a birational $1$-contraction, composition of
extremal divisorial contractions and flips;
they preserve nef, semiample and freeness property
under the birational transformation of $D$;

\item[]
$D_T$, the birational transform of $D$ on $T$, is nef over $Z$ and
numerically trivial over $W$; it is enough to verify that
$D_T$ is free over $Z$;

\item[]
there exists a nef and big over $Z$
$\Q$-Cartier divisor $D_W$ on $W$  and a positive integer $m$,
depending only on $d$,
such that
$g^*D_W\sim D_T$ and $mD_W$ is Cartier;
here we use the boundedness of torsions for the relative case;
$W/V,V$ have wFt.

\end{description}
Thus the freeness of $N'mD_T$ over $Z$ follows from that of $N'mD_W$
in the big over $Z$ case.

\end{proof}

For a natural number $l$,
$\R^l$ denotes the arithmetic $\R$-linear space
of dimension $l$.
It contains a lattice of {\em integral\/} vectors $\Z^l$,
a $\Q$-linear subspace $\Q^l$ of {\em rational\/} vectors, and is defined
over $\Z$ and over $\Q$:
$$
\R^l=\R\otimes_\Z \Z^l=\R\otimes_\Q \Q^l
\supset \Q^l=\Q\otimes_\Z\Z^l
\supset \Z^l.
$$
In particular, we know whether an affine
subspaces of $\R^l$ is {\em rational\/}, that is,
defined by linear (not necessarily homogeneous)
equations with rational or integral coefficients
(rational hyperplanes).
A vector $v=(v_1,\dots,v_l)\in \R^l$ is {\em nonrational\/} if $v\not\in\Q^l$,
equivalently, one of its coordinates $v_i$ in the standard basis
$[e_1=](1,0,\dots,0), [e_2=](0,1,\dots,0),\dots,
[e_l=](0,0,\dots,1)$ of $\R^l$ is nonrational.

The rational affine span $\spn{v}$ of a vector $v\in \R^l$ is
the smallest rational affine subspace of $\R^l$ containing $v$.
Hence $v$ is rational if and only if $\spn{v}$ is $v$ itself.

A {\em direction\/} of an affine $\R$-space is a unite vector
in its $\R$-linear (tangent) space of translations.

We use a standard norm on $\R^l$, e.g.,
the {\em maximal absolute value\/} norm
$$
\norm{v}=\norm{(v_1,\dots,v_l)}=\max{\norm{v_i}},
$$
where $\norm{v_i}$ denotes the absolute value
of $i$-th coordinate.
The norm on $\spn{v}$ is restricted from $\R^l$.
So, a direction in $\spn{v}$ is a vector $e\in\R^l$ such that
$$
e+\spn{v}=\spn{v} \text{ and }
\norm{e}=1.
$$
A direction in $\spn{v}$ exists exactly when $v$ is nonrational
because then $\dim_\R\spn{v}\ge 1$.

\paragraph{b-Codiscrepancy \cite[p.~88]{Sh03}.}
Let $(X,D)$ be a {\em log\/} pair, that is,
$K+D$ is $\R$-Cartier.
Then $\D=\D(X,D)$ will denote its b-$\R$-divisor of
{\em codiscrepancy\/} or {\em pseudo-boundary\/}:
by definition
$$
\overline{K+D}=\K+\D,
$$
where $\K$ is a canonical b-divisor of $X$.

\section{Elementary}

\subsection{
Two basic inequalities for the Gauss integral part.}
For any two real numbers $a,b$,
\begin{equation}\label{basic_upper_bound}
\rddown{a+b}\le \rddown{a}+\rddown{b}+1.
\end{equation}

\begin{proof}
Any integral shift $a\mapsto a+n,n\in\Z$,
gives an equivalent inequality.
The same holds for $b$.
So, after appropriate shifts, we can suppose that
$\rddown{a}=\rddown{b}=0$.
That is, $0\le a,b<1$.
Hence $a+b<2$ and
$$
\rddown{a+b}\le 1=\rddown{a}+\rddown{b}+1.
$$
\end{proof}

For any two real numbers $a,b$,
\begin{equation}\label{basic_low_bound}
\rddown{a+b}\ge \rddown{a}+\rddown{b}.
\end{equation}

\begin{proof}
As in the proof of the inequality~(\ref{basic_upper_bound})
we can suppose that $\rddown{a}=\rddown{b}=0$.
Then $a+b\ge 0$ and
$$
\rddown{a+b}\ge 0=\rddown{a}+\rddown{b}.
$$
\end{proof}

\begin{exa}[$n$-complement with $\PP^1$; cf. Example~\ref{n_comp_dim_1}]
\label{R_comp_dim 1}
The main case for existence and boundedness of $n$-complements in dimension $1$
concerns $n$-complements for (global) pairs $(\PP^1,B)$,
where $B=\sum b_i P_i$ is a boundary on $\PP^1$:
$b_i\in [0,1]$ and $P_i$ are distinct (closed) points on $\PP^1$.
The assumption~(1) in Theorem~\ref{bndc} in this case
can be strengthen to existence on an $\R$-complement,
equivalently,
\begin{equation}\label{R_comp_inequality}
\sum b_i\le 2=-\deg K_{\PP^1} \text{ or }
\sum{'} b_i\le c=2-l,
\end{equation}
where $0\le l\le 2$ is the number of points $P_i$ with $b_i=1$
(lc singularities) and $\sum'$ runs for $b_i\in [0,1)$.
On the other hand, existence of an $n$-complement
means that
\begin{equation} \label{n_comp_inequality}
\sum{'} \rddown{(n+1)b_i}/n\le c=2-l.
\end{equation}
Indeed, under the inequality~(\ref{n_comp_inequality}), we can take
$$
B^+=\frac{\Delta}n+\sum{'}\frac{\rddown{(n+1)b_i}}n P_i+
\sum{''} P_i,
$$
where $\Delta$ is a reduced divisor supported outside of $\Supp B$
with
$$
2n-ln -\sum{'} \rddown{(n+1)b_i}=2c-\sum{'} \rddown{(n+1)b_i}
$$
points and
$\sum{''}$ runs for $b_i=1$.
To satisfy the inequality~(\ref{n_comp_inequality}) we can joint together small $b_i$
in larger multiplicities because, for any two real numbers $a$ and $b$,
$$
\rddown{(n+1)(a+b)}/n\ge \rddown{(n+1)a}/n+\rddown{(n+1)b}/n
$$
by~(\ref{basic_low_bound}).
The inequality~(\ref{R_comp_inequality}) holds again.
However, we need to keep our assumption that the new multiplicities
also belong to $[0,1)$.
This allows to suppose that all $b_i\ge 1/2$
with at most one exception $0<b_i<1/2$ and we have
at most $4$ nonzero multiplicities $b_i$.
In this case $n$-complements exits and bounded by \cite[Example~1.11]{Sh95}.
Moreover, we can suppose that Restrictions on complementary
indices hold.

Notice that $c$ in inequalities~(\ref{R_comp_inequality}-\ref{n_comp_inequality})
can be replaced by any rational (and even real number if
omit Restrictions on complementary indices).
In particular, for $c=d+1$ this implies existence and boundedness
of $n$-complements for $(\PP^d,B)$, where $B=\sum b_iH_i,b_i\in[0,1]$,
$H_i$ are hypersurfaces in general position and $(\PP^d,B)$
has an $\R$-complement (cf. \cite[Example~1.11]{Sh95}).

\end{exa}

\subsection{Inductive bounds} \label{induc_bound}
For any positive integer $n$  and
any real numbers $d_1,\dots,d_n$,
\begin{equation}\label{induc_upper_bound}
\rddown{\sum_{i=1}^n d_i}\le n-1+
\sum_{i=1}^n\rddown{d_i}.
\end{equation}

\begin{proof}
For $n=1$, the inequality actually become the equality
$$
\rddown{d_1}=1-1+\rddown{d_1}.
$$

For $n\ge 2$, we use the upper bound~(\ref{basic_upper_bound}) and induction:
$$
\rddown{\sum_{i=1}^n d_i}\le 1+\rddown{d_1}+
\rddown{\sum_{i=2}^n d_i}\le
1+n-2+\sum_{i=1}^n\rddown{d_i}=n-1+\sum_{i=1}^n\rddown{d_i}.
$$
\end{proof}

\begin{exa} \label{induc_upper_bound_exa}
$$
1-l+\sum_{i=1}^l \rddown{(n+1)d_i}/n\ge
\rddown{(n+1)(1-l+\sum_{i=1}^l d_i)}/n,
$$
for any positive integers $n,l$ and
real numbers $d_i$.
Indeed, the inequality is equivalent to
$$
l-1+\sum_{i=1}^l \rddown{(n+1)d_i}\ge
\rddown{\sum_{i=1}^l (n+1)d_i},
$$
which follows from~(\ref{induc_upper_bound}).

\end{exa}

Another application of~(\ref{induc_upper_bound}).
For any positive integer $n$ and any real number $d$,
\begin{equation}\label{1st_cor_uuper_bound}
\rddown{nd}\le n-1 +n\rddown{d}.
\end{equation}

\begin{proof}
By the inequality~(\ref{induc_upper_bound}) with $d_1=\dots=d_n=d$.
\end{proof}

Using the low bound~(\ref{basic_low_bound}) and induction we get,
for any nonnegative integer $n$  and
any real numbers $d_1,\dots,d_n$,
\begin{equation} \label{induc_low_bound}
\rddown{\sum_{i=1}^n d_i}\ge
\sum_{i=1}^n\rddown{d_i}
\end{equation}
and
\begin{equation} \label{1st_cor_low_bound}
\rddown{nd}\ge n\rddown{d}.
\end{equation}

\subsection{Special upper bound}
For any real numbers $r,s$ such that $r<1$ and $sr$ is integral,
\begin{equation}\label{special_upper_bound}
\rddown{(s+1)r}\le sr.
\end{equation}

\begin{proof}
Since $sr\in \Z$ and $r<1$,
$$
\rddown{(s+1)r}=\rddown{sr+r}=sr+\rddown{r}\le sr.
$$
\end{proof}

[Even] if $s=n$ is a positive integer then the inequality~(\ref{special_upper_bound}) is
equivalent to
$$
\rddown{(n+1)r}/n\le r
$$
and may not hold unless $nr$ is integral.

\begin{exa} \label{rddown_(n+1)_m_m}
(1)
For $s=1$ and $r=1/2$,
$$
\rddown{(1+1)1/2}=1\not\le 1(1/2)=1/2.
$$

(2) However, if $d\in \Z/n$ and $<1$
then
$$
\rddown{(n+1)d}/n\le d<1
$$
and $=d$ if and only if $d\ge 0$.

(3)
For any real number $d< 1$ and any positive integer $n$,
$$
\rddown{(n+1)d}/n\le 1.
$$
The low bound
$$
\rddown{(n+1)d}/n\ge d
$$
does not work in general (cf. \cite[Lemma~3.5]{PSh08}
and see the next example).

(4)
There exists a closed dcc subset $\Gamma\subset[0,1]\cap\Q$ with
the only accumulation point $1$ and such that,
for every positive integer $n$, there exists
$b\in\Gamma$ with $\rddown{(n+1)b}/n<b$.
For given $n$,
consider a rational number $b(n)<n/(n+1)$ which is very close to $n/(n+1)$.
Then
$$
\frac {n-1}n=\rddown{(n+1)b(n)}/n<b(n).
$$
So, $\Gamma=\{b(n)\mid n\in \N\}$ is a required set.

\end{exa}

\subsection{Approximations and round down}

\begin{lemma} \label{approximation_I}
Let $I,n$ be two positive integers, $l$ be a nonnegative integer
and $b$ be a real number such that $I|n$,
$b,l/I\in[0,1)$ and $\norm{b-l/I}<1/I(n+1)$.
Then
$$
\rddown{(n+1)b}/n=\rddown{(n+1)\frac lI}/n=\frac l I.
$$

\end{lemma}

\begin{proof}
$$
\rddown{(n+1)\frac lI}/n=\rddown{\frac {nl}I+\frac lI}/n=
\frac {nl}I\frac 1n=\frac lI
$$
because $I|n$. Let $\delta=\norm{b-l/I}$. Then $\delta<1/I(n+1)$ and $b=l/I\pm\delta$.
For $b=l/I+\delta$,
$$
\rddown{(n+1)(\frac lI+\delta)}/n=\rddown{\frac {nl}I+\frac lI+(n+1)\delta}/n=
\frac {nl}I\frac 1n=\frac lI.
$$
For $b=l/I-\delta$, $l\ge 1$ and
$$
\rddown{(n+1)(\frac lI-\delta)}/n=\rddown{\frac {nl}I+\frac lI-(n+1)\delta}/n=
\frac {nl}I\frac 1n=\frac lI.
$$

\end{proof}

\begin{lemma} \label{approximation_II}
Let $n$ be a positive integer
and $b$ be a real number such that
$b\in[0,1)$ and $\norm{b-1}<1/(n+1)$.
Then
$$
\rddown{(n+1)b}/n=1.
$$

\end{lemma}

\begin{proof}
Use the proof of Lemma~\ref{approximation_I}
with $I=l=1$ and $b=1-\delta$.

\end{proof}

Similarly we can verify the following estimations.

\begin{lemma}[{\cite[Lemma~4]{Sh06}}] \label{inequal}
Let $n$ be a positive integer,
$b$ be a real number and
$b_n=m/n,m\in \Z$, be a rational number with the denominator $n$
such that
$$
\mid b-b_n\mid=\varepsilon/n.
$$
Then
\begin{description}

\item[\rm (1)\/]
 $b^+=\rddown{(n+1)b}/n\le b_n$ if
 $\varepsilon< 1-b$;

\item[\rm (2)\/]
  $b^+\ge b_n$ if
 $\varepsilon\le b$;
 and
\item[\rm (3)\/]
  $b^+ =b_n$ if
 $\varepsilon< \min\{b,1-b\}$.

\end{description}

\end{lemma}

\section{Technical} \label{technical}

We introduce now a b-version of Definition~\ref{n_comp}.
It will be crucial for our induction.

\begin{defn} \label{b_n_comp}
Let $n$ be a positive integer and
$(X/Z\ni o,D)$ be a {\em log\/} pair with a local morphism $X/Z\ni o$ and
with an $\R$-divisor $D$.
A pair $(X/Z\ni o,D^+)$, with the same local morphism and
with a $\Q$-divisor $D^+$ on $X$,
is a b-$n$-{\em complement\/} of $(X/Z\ni o ,D)$ if
$(X/Z\ni o,D^+)$ is an $n$-complement of $(X/Z\ni o,D)$ and
the same hold for all crepant models of $(X/Z\ni o,D)$ and
of $(X/Z\ni o,D^+)$.
Equivalently, instead of (1) in Definition~\ref{n_comp}
the following b-version holds:
\begin{description}
\item[\rm (1-b)\/]
for every prime b-divisor $P$ of $X$,
$$
d^+\ge
\begin{cases}
1, \text{ if } &d=1;\\
\rddown{(n+1)d}/n &\text{ otherwise,}
\end{cases}
$$
\end{description}
where $d=\mult_P\D$ and
$d^+=\mult_P\D^+$ are multiplicities at $P$ of codiscrepancy b-divisors
$\D$ and $\D^+$ of $(X,D)$ and of $(X,D^+)$ respectively.

The b-$n$-complement is {\em monotonic\/} if
$\D^+\ge \D$.

\end{defn}

\begin{rem} \label{rem_def_b_n_compl}
(1) $K+D,K+D^+$ are $\R$-Cartier because
respectively $(X/Z\ni o,D)$ is a log pair and
by Definition~\ref{n_comp}, (3).
So, b-divisors $\D$ and $\D^+$ are well-defined.
Actually $\D$ is well-defined and the definition
is meaningful {\em only if\/} $(X,D)$ is a log pair
(cf. the remark~(2) below).
Moreover, the definition is {\em birational\/}:
we can replace $(X/Z\ni o,D)$ by a log pair $(X'/Z\ni o,D')$
with a model $X'/Z\ni o$ of $X/Z\ni o$.
Indeed, then $(X/Z\ni o,D^+)$ is a b-$n$-complement of $(X'/Z\ni o,D')$ if
$(X/Z,\D^+)$ is (b-)$n$-complement of $(X'/Z,\D')$,
where $\D'=\D(X',D')$
(cf. (10) in~\ref{adjunction_div}).
In this situation we have induced b-$n$-complement
$(X'/Z\ni o,\D^+{}_{X'})$ of $(X'/Z\ni o,D')$.

If $(X/Z\ni o,D)$ is not necessarily a log pair
then we can also consider b-$n$-complements
taking its small log pair model when such a model
exists.
For instance we can take a $\Q$-factorialization.
In general, it is not expected that
b-$n$-complements are birationally invariant,
even for small birational modifications, if
they are noncrepant
(cf. Proposition~\ref{small_transform_compl} below).
However, if $(X/Z\ni o,D)$ has a b-$n$-complement for a maximal (small) model
then the b-$n$-complements are well-defined
under small birational modifications
(see Construction~\ref{sharp_construction}
and Statement~\ref{monotonicity_I}).

(2) There are no reason to introduce b-$\R$-complements
because usual $\R$-complements already do so.
Indeed, if $(X/Z\ni o,D)$ is a log pair then
(1-3) of Definition~\ref{r_comp} are equivalent
respectively to
\begin{description}
\item[\rm (1)\/]
$\D^+\ge \D$;

\item[\rm (2)\/]
$(X,\D^+)$ is lc; and

\item[\rm (3)\/]
$\K+\D^+\sim_\R 0$.

\end{description}
We can use $\equiv$ over $Z\ni o$ instead of $\sim_\R$ if
$X/Z\ni o$ has wFt (see Remark~\ref{remark_def_complements}, (4)).
Since by Proposition~\ref{small_transform_compl} below small birational modifications preserve
$\R$-complement
we get (3) for every small log model of $(X/Z\ni o,D)$
(when it exists).
These models may have different b-codiscrepancies.
In particular, the inequality (1) also holds for
the largest one $\D^\sharp$ in Construction~\ref{sharp_construction}
(see Statement~\ref{monotonicity_I}).

\end{rem}

\begin{exa} \label{b_n_comp_of_itself}
(1)
Let $(X/Z,D^+)$ be an $n$-complement of $(X/Z,D)$.
Then $(X/Z,D^+)$ is a monotonic b-$n$-complement of itself.
It is enough to verify (1-b) of Definition~\ref{b_n_comp}.
That is, for every $d^+$,
$$
d^+\ge
\begin{cases}
1, \text{ if } &d^+=1;\\
\rddown{(n+1)d^+}/n &\text{ otherwise.}
\end{cases}
$$
It follows from the inequality in Example~\ref{rddown_(n+1)_m_m}, (2).
Indeed, $d^+\in\Z/n$ and $\le 1$
by Definition~\ref{n_comp}, (2-3).

The complement is monotonic: $\D^+\ge \D^+$.

(2)
Let $(X/Z,D^+)$ be a monotonic $n$-complement of
a log pair $(X/Z,D)$.
Then $(X/Z,D^+)$ is a monotonic b-$n$-complement of $(X/Z,D)$.
This follows from Proposition~\ref{D_D'_complement} below.
This is why we do not need to state this property explicitly
for $n$-complements of pairs with hyperstandard boundary
multiplicities \cite{PSh08} \cite{B} even we use it
in the course of proof.

(3) Let
$$
(\A^2\ni P,\frac 13 (L_1+L_2+L_3))
$$
be a local pair on the affine plane $\A^2$ with
three distinct lines $L_1,L_2,L_3$ passing through
a point $P$.
Then $(\A^2,0)$ with $D^+=0$ is an $1$-complement of this pair
but $(A^2,0)$ is not a b-$1$-complement of the pair.
Indeed, let $E$ be the exceptional divisor of
the usual blowup of $P$.
Then $d^+=\mult_P\D^+=-1$ but
$$
d=\mult_P\D(A^2,\frac 13 (L_1+L_2+L_3))=0
$$
and $\rddown{2\cdot 0}/1=0>-1=d^+$.

\end{exa}

The next important for us technic were developed by
Birkar and Zhang \cite{BZ}.

\paragraph{bd-Pairs.}
In general, it is a pair $(X/Z,D+\sP)$
with an $\R$-divisor $D$ and another data $\sP$
\begin{description}
  \item[\rm (Birkar-Zhang)\/]
$\sP$ is a b-$\R$-Cartier divisor of $X$,
defined up to $\sim_\R$;

  \item[\rm (Alexeev)\/]
$\sP=\sum r_i L_i$, where $L_i$ are mobile linear systems on $X$ and
$r_i$ are real numbers (the systems are not necessarily finite dimensional
if $X$ is not complete);

  \item[\rm (b-sheaf)\/]
$\sP=\sum r_i \sF_i$, where $\sF_i$ are invertible sheaves
on a proper birational model $Y/Z$ of $X/Z$ and
$r_i$ are real numbers; sheaves are defined up to isomorphism.

\end{description}
It is easy to covert the b-sheaf data into the Birkar-Zhang data:
to replace every sheaf $\sF_i$ by a b-divisor $\overline{H_i}$,
where $H_i$ is a divisor on $Y$ with $\sO_Y(H_i)\simeq \sF_i$.
Similarly, we can convert the Alexeev data into the Birkar-Zhang data:
to replace every linear system $L_i$ by $\overline{H_i}$ for
its (sufficiently general)
element $H_i$ on a model $Y$, where the linear system is free.
However, a converse does not hold in general.
In the paper usually a bd-pair means a Birkar-Zhang one or
a translation into it of an Alexeev or b-sheaf one (cf. \cite[Corollary~7.18, (ii)]{PSh08}).
In particular, for every bd-pair $(X,D+\sP)$,
a pair $(X,D+\sP_X)$ with the trace $\sP_X$ is
well-defined.
So, $\sP_X$ is an $\R$-divisor defined up to $\sim_\R$.

A pair $(X/Z,D+\sP)$ is a {\em log\/} bd-pair if
$(X/Z,D+\sP_X)$ is a log pair, that is,
$K+D+\sP_X$ is $\R$-Cartier.
So, we can control log singularities of $(X/Z,D+\sP)$
but not in a usual sense because $\sP_X$ is
defined up to $\sim_\R$.
Respectively, the b-divisor of codiscrepancy
$\D=\D(X,D+\sP_X)$ is not unique but
does so up to $\sP$:
$\D\dv=\D-\sP$ is a b-divisor which depend only
on $(X/Z,D+\sP)$ and
is the same for every $\sP$ up to $\sim_\R$.
Note also that $\D\dv{}_{,X}=D$ is the trace.
So, we can defined lc, klt, etc $(X/Z,D+\sP)$ (cf.
subboundaries in \cite{Sh92} and \cite[1.1.2]{Sh96}).
Similarly, numerical properties also can be
imposed on $(X/Z,D+\sP_X)$ and thus on $(X/Z,D+\sP)$.
After that we can introduce a lot of concepts from the LMMP for
those pairs.
E.g., a log bd-pair $(X/Z,D+\sP)$ is (bd-){\em minimal\/} if
\begin{description}

\item[]
$(X,D+\sP)$ is lc, or equivalently,  $\D\dv$ is a b-subboundary, that is,
for every prime b-divisor $P$, $\mult_P \D\dv\le 1$;
and

\item[]
$K+D+\sP_X$ is nef over $Z$, or equivalently,
$\K+\D$ is b-nef over $Z$, e.g.,
$\sP$ has associated b-divisor $\PP$
($=\sP$ for Birkar-Zhang pairs) and
$X/Z$ has a proper birational model $Y/Z$ with
a nef $\R$-divisor $H$ such that
$\K+\D=\K+\D\dv+\PP=\overline{H}$.

\end{description}
In other words, singularities are controlled by
the divisorial part $\D\dv$ and positivity properties
by whole $\D$.
It is not a new phenomenon.
Since $\sP$ has no a fixed support for the Alexeev and b-sheaf data it
does not effect singularities: they only
related to $\D\dv$.
A corresponding theory for pairs with nonnegative
multiplicities $r_i$ is well-known as the MMP for
Alexeev or mobile pairs \cite[Defenition~1.3.1 and Section~1.3]{C}.
The Birkar-Zhang data, defined up to an $\R$-linear even
numerical equivalence, shares the same property of
singularities and a version of the LMMP.

\paragraph{Crepant bd-models.}
(Cf. crepant $0$-contractions in~\ref{adjunction_0_contr}.)
We say that two bd-pairs $(X'/Z',D'+\sP),(X/Z,D+\sP)$
are {\em birationally equivalent\/} or {\em crepant\/},
if both pairs are log bd-pairs and
there exists a commutative diagram
$$
\begin{array}{ccc}
(X',D'+\sP)&\dashrightarrow & (X,D+\sP)\\
\downarrow&&\downarrow\\
Z'&\dashrightarrow & Z
\end{array},
$$
where the horizontal arrow $(X',D'+\sP)\dashrightarrow (X,D+\sP)$
is a crepant proper birational isomorphism,
another horizontal arrow $Z'\dashrightarrow Z$ is
a proper birational isomorphism.
The crepant property for the horizontal arrow is similar
to the case of usual pairs and means
that $D'=D_{X'}=D\dv{}_{,X'}=(\D\dv)_{X'},
\D\dv'=\D(X',D'+\sP_{X'})=\D(X,D+\sP_X)=\D\dv$.
We say also that $(X'/Z',D'+\sP)$ is
a ({\em crepant\/}) {\em model\/} of $(X/Z,D+\sP)$.
Notice that crepant bd-pairs have the same b-$\R$-divisor $\sP$.

\begin{exa}
Let $\varphi\colon X'\to X$ be a birational contraction.
Then the crepant bd-pair $(X',D'+\sP)$ of $(X,D+\sP)$
has $D'$ given uniquely by equation
$$
\varphi^*(K+D+\sP_X)=D'+\sP_{X'}.
$$
Indeed, $\varphi^*(K+D+\sP_X)=\D(X,D+\sP_X)_{X'}=(\D\dv+\sP)_{X'}=
D'+\sP_{X'}$ and $D'=D_{X'}$.
\end{exa}

The typical example of bd-pairs is related to log adjunction,
where $D,\D\dv$ are respectively divisorial and b-divisorial parts of adjunction, and
$\sP$ is the (b-)moduli part of adjunction (see~\ref{adjunction_0_contr}).

However, to have a good theory even for log bd-pairs
we need more restrictions \cite{BZ}.
We follow Birkar and Zhang, and always
suppose that
every bd-pair satisfies
\begin{description}
  \item[\rm (index $m$)\/]
$m\sP$ is a Cartier b-divisor for the Birkar-Zhang data
and $\sP$ is defined up to $\sim_m$, where $m$
is the positive integer;
respectively, for the Alexeev and b-sheaf data
every $mr_i\in\Z$;

 \item[\rm (positivity)\/]
$\sP$ is a b-nef $\R$-divisor of $X$
for the Birkar-Zhang data;
respectively for the Alexeev and b-sheaf data
every $r_i$ is a nonnegative real number,
every $\sF_i$ is a nef invertible sheaf.
In particular, the b-nef property over $Z$ can be applied
to proper $X/Z$ (cf. Nef in Section~\ref{intro}).

\end{description}
Such a bd-pair $(X/Z,D+\sP)$ will be called a {\em bd-pair
of index\/} $m$, where $m$ is a positive integer.
So, for bd-pairs of index $m$: $\sP$ is b-nef, in particular,
$\sP$ is b-nef over $S$ in the usual sense for every proper $X/S$.

The sum $D+\sP$ is formal.
Hence really the bd-pair is a triple $(X/Z,D,\sP)$.

If $(X/Z,D+\sP_X)$ is not a log pair then we can take
its small birational log model $(Y/Z,D+\sP_Y)$ with
the birational transform of $D$.
b-Codiscrepancy $\D=\D(Y,D_Y+\sP)$ and $\D\dv$ depend on
the model but $\sP$ is same.
So, in statements, where we use pairs $(X/Z,D+\sP)$,
in the results which concern birational concepts, e.g.,
b-$n$-complements, we add the assumption that
$(X/Z,D+\sP_X)$ is a log pair
(cf. Addendum~\ref{bd_exceptional_compl} for exceptional $n$-complements
and similar addenda for other type of $n$-complements).

Note for applications that if $\sP$ is a b-sheaf than
$$
m\sP=\otimes\sF_i^{\otimes mr_i}\simeq\sO_Y(H)
$$
is invertible nef and
$\sP_X=\overline{H}_X/m$, where $H$ is a nef Cartier divisor on $Y$.
Thus in this situation the b-sheaf data implies
the Birkar-Zhang one.
Respectively, for the Alexeev data, $H\in m\sP=L$ and
$\sP_X=H$, the birational image of a sufficiently general $H$,
where $L$ is a free linear system on $Y$.
The Alexeev data is more restrictive and
actually is required for general applications
(cf. Conjecture~\ref{conjecture_bndc} and
Corollary-Conjecture~\ref{conj_bounded_lc_index}).
However, for wFt or Ft $X/Z$ the Birkar-Zhang data
works very well (cf. Example~\ref{bd_elliptic} and
Corollary~\ref{Alexeev_index_m}).

For bd-pairs it is better to join $\sP$ with
the canonical divisor and consider the b-divisor $\K+\sP$,
respectively, $\sP_X$ with $K$ in the $\R$-divisor $K+\sP_X$.
The b-divisor $K+\sP$ for bd-pairs of index $m$ is defined up to $\sim_m$ and
behave as $\K$ but not as b-Cartier divisor $\sP$.

If $\sP=0$ then we get usual pairs or log pairs $(X/Z,D)$
with $\D\dv=\D=\D(X,D)$.
The same holds for models $Y/Z$ of $X/Z$ over which
$\sP$ is {\em stabilized\/}: $\sP=\overline{\sP_Y}$.
In this case $(Y,D_Y+\sP)$ is lc (klt, etc) if and only if
$(Y,D_Y)$ is lc (respectively, klt, etc).

Complements can be defined for bd-pairs too.
We consider as Birkar and Zhang \cite[Theorem~1.10 and Definition 2.18, (2)]{B} usually
bd-pairs $(X/Z,D+\sP)$ of index $m$ and their complements
$(X/Z,D^++\sP)$ with other required birational models, e.g.,
$(X^\sharp/Z,D_{X^\sharp}^\sharp+\sP)$, such that the models will have the same
birational part $\sP$.

\begin{defn}
Let $(X/Z\ni o,D+\sP)$ be a bd-pair
with a [proper] local morphism $X/Z\ni o$ and
with an $\R$-divisor $D$ on $X$.
Another bd-pair $(X/Z\ni o,D^++\sP)$ with the same local morphism,
same $\sP$ and with an $\R$-divisor $D^+$ on $X$
is called an $\R$-{\em complement\/} of $(X/Z\ni o,D+\sP)$ if
\begin{description}
\item[\rm (1)\/]
$D^+\ge D$;

\item[\rm (2)\/]
$(X,D^++\sP_X)$ is lc, or equivalently,
$\D\dv^+$ is a b-subboundary; and

\item[\rm (3)\/]
$K+D^++\sP_X\sim_\R 0/Z\ni o$, or equivalently,
$\K+\D^+=\K+\D\dv^++\sP\sim_\R 0/Z$.
\end{description}
In particular, $(X/Z\ni o,D^++\sP)$ is
a {\em [local relative]\/} $0$-{\em bd-pair}.

The complement is {\em klt} if
$(X,D^+)$ is klt.

If $(X/Z\ni o,D+\sP)$ is a log bd-pair then
(1) can be also restated in terms of b-divisors:
\begin{description}
\item[\rm (1)\/]
$\D\dv^+\ge\D\dv$ or $\D^+\ge \D$.
\end{description}

\end{defn}

\begin{defn} \label{bd_n_comp}
Let $n$ be a positive integer, and
$(X/Z\ni o,D+\sP)$ be a bd-pair with a local morphism $X/Z\ni o$ and
with an $\R$-divisor $D=\sum d_iD_i$ on $X$.
A bd-pair $(X/Z\ni o,D^++\sP)$ with the same local morphism,
same $\sP$ and
with an $\R$-divisor $D^+=\sum d_i^+D_i$ on $X$,
is an $n$-{\em complement\/} of $(X/Z\ni o ,D+\sP)$ if
\begin{description}
\item[\rm (1)\/]
for all prime divisors $D_i$ on $X$,
$$
d_i^+\ge
\begin{cases}
1, \text{ if } &d_i=1;\\
\rddown{(n+1)d_i}/n &\text{ otherwise;}
\end{cases}
$$

\item[\rm (2)\/]
$(X,D^++\sP)$ is lc or equivalently,
$\D\dv^+$ is a b-subboundary; and

\item[\rm (3)\/]
$K+D^++\sP_X\sim_n 0/Z\ni o$, or equivalently,
$\K+\D^+=\K+\D\dv^++\sP\sim_n 0/Z$.
\end{description}
The $n$-complement is {\em monotonic\/} if
$D^+\ge D$, or equivalently, $\D^+_X\ge \D_X$ or $\D\dv^+\ge\D\dv$,
when $(X/Z\ni o,D+\sP)$ is a log bd-pair.

A log bd-pair $(X/Z\ni o,D^++\sP)$ is
a b-$n$-complement of $(X/Z\ni o,D+\sP)$ if
instead of (1) in Definition~\ref{n_comp}
the following b-version holds:
\begin{description}
\item[\rm (1-b)\/]
for all prime b-divisors $P$ of $X$,
$$
d^+\ge
\begin{cases}
1, \text{ if } &d=1;\\
\rddown{(n+1)d}/n &\text{ otherwise,}
\end{cases}
$$
\end{description}
where $d=\mult_P\D\dv$ and
$d^+=\mult_P\D\dv^+$ are multiplicities of
(the divisorial part of codiscrepancy) b-divisors
$\D\dv=\D-\sP$ and $\D\dv^+=\D^+-\sP$ of
$(X,D+\sP)$ and of $(X,D^++\sP)$ respectively.

The b-$n$-complement is {\em monotonic\/} if
$\D^+\ge \D$ or $\D\dv^+\ge\D\dv$.

\end{defn}

\begin{rem}[{Cf. \cite[Theorem~1.10 and Definition~2.18, (2)]{B}}]
\label{m_div_n}
Both complements and many other constructions
meaningful for arbitrary pair $(X/Z\ni o,\D)$
with a b-divisor $\D$.
We use and develop these concepts only for
bd-pairs $(X/Z\ni o,D+\sP)$ of index $m$ and actually
for a boundary $D$.
Moreover, usually for $n$-complements we
suppose that $m|n$, e.g., as in Corollary~\ref{B_B_+B_sN_Phi}.
\end{rem}

Now we can state a bd-version of Theorem~\ref{bndc}.

\begin{thm}[Boundedness of $n$-complements for bd-pairs] \label{bd_bndc}
Let $d$ be a nonnegative integer,
$m$ be a positive integer,
$\Delta\subseteq (\R^+)^r,\R^+=\{a\in\R\mid a\ge 0\}$, be a compact
subset (e.g., a polyhedron) and
$\Gamma$ be a subset in the unite segment $[0,1]$ such that
$\Gamma\cap\Q$ satisfies the dcc.
Let $I,\ep,v,e$ be the data as
in Restrictions on complementary indices.
Then there exists a finite set
$\sN=\sN(d,I,\ep,v,e,\Gamma,\Delta,m)$  of positive integers
({\em complementary indices\/})
such that
\begin{description}

\item[\rm Restrictions:\/]
every $n\in\sN$ satisfies
Restrictions on complementary indices with the given data;

\item[\rm Existence of $n$-complement:\/]
if $(X/Z\ni o,B+\sP)$ is a bd-pair of index $m$ with wFt $X/Z\ni o$,
$\dim X\le d$,
connected $X_o$ and  with a boundary $B$,
then $(X/Z\ni o,B+\sP)$ has an $n$-complement for some $n\in\sN$
under either of the following assumptions:

\item[\rm (1-bd)\/]
$(X/Z\ni o,B+\sP)$ has a klt $\R$-complement;
or

\item[\rm (2-bd)\/]
$B=\sum_{i=1}^r b_iD_i$ with $(b_1,\dots,b_r)\in \Delta$ and additionally,
for every $\R$-divisor $D=\sum_{i=1}^rd_iD_i$ with $(d_1,\dots,d_r)\in \Delta$,
the pair $(X/Z\ni o,D+\sP)$ has an $\R$-complement, where
$D_i$ are
effective Weil divisors (not necessarily prime);
or

\item[\rm (3-bd)\/]
$(X/Z\ni o,B+\sP)$ has an $\R$-complement and,
additionally, $B\in\Gamma$.

\end{description}

\end{thm}

\begin{add}
We can relax the connectedness assumption on $X_o$ and
suppose that the number of connected components of $X_o$
is bounded.
\end{add}

This theorem with Theorem~\ref{bndc} will be proven in Section~\ref{lc_type_compl}.
However, we start from some generalities about complements.

\begin{prop}[{Cf. \cite[Lemma~5.3]{Sh92}}] \label{D_D'_complement}
Let $(X/Z,D),(X/Z,D')$ be two (log) pairs
with divisors $D\ge D'$ and
$(X/Z,D^+)$ be a (respectively {\em b-\/})$n$-complement of $(X/Z,D)$.
Then $(X/Z,D^+)$ is a (respectively {\em b-\/})$n$-complement
of $(X/Z,D')$ too.

The same holds for monotonic (respectively {\em b-\/})$n$-complements.

\end{prop}

\begin{add} \label{R_D_D'_complement}
The same holds for $\R$-complements of
pairs and bd-pairs.

\end{add}

\begin{add} \label{bd_D_D'_complement}
The same holds for bd-pairs $(X/Z,D+\sP),(X/Z,D'+\sP)$
with the same b-part $\sP$ and with a ({\em b-\/})$n$-complement
$(X/Z,D^++\sP)$ of $(X/Z,D)$ respectively.

\end{add}

\begin{proof}
Immediate by the monotonicity of $\rddown{\ }$:
if $d'\le d<1$ then
$$
\rddown{(n+1)d'}/n\le \rddown{(n+1)d}/n \le 1
$$
and Example~\ref{rddown_(n+1)_m_m}, (3).

Similarly we treat bd-pairs of Addendum~\ref{bd_D_D'_complement}.

Addendum~\ref{R_D_D'_complement} is immediate by definition.

\end{proof}

\begin{cor} \label{monotonic_n_compl}
Let $(X/Z,D^+)$ be an $n$-complement of $(X/Z,D)$.
Then $(X/Z,D^+)$ is an $\R$- and monotonic $n$-complement of
$(X/Z,\rdn{D}{n})$ where $\rdn{D}{n}$ has the following multiplicities,
for every prime divisor $P$ on $X$,
$$
\mult_P \rdn{D}{n}=
\begin{cases}
1, \text{ if } &d=1;\\
\rddown{(n+1)d}/n &\text{ otherwise}.
\end{cases}
$$

\end{cor}

\begin{proof}
Immediate by Definition~\ref{n_comp}, Proposition~\ref{D_D'_complement}
and Example~\ref{b_n_comp_of_itself}, (1).
\end{proof}

\begin{prop} \label{b_n_comp-nc}
Let $n$ be a positive integer,
$D$ be an $\R$-divisor with normal crossing on
nonsingular $X$ and
$(X/Z,D^+)$ be an $n$-complement of $(X/Z,D)$.
Then $(X/Z,D^+)$ is also a b-$n$-complement of $(X/Z,D)$.

The same holds for a bd-pair $(X/Z\ni o,D+\sP)$
with normal crossings only for $\Supp D$
and stable $\sP$ over $X$.

\end{prop}

\begin{proof}
(Local verification.)
The statement is meaningful because $X$ is nonsingular and
$(X,D)$ is a log pair.
We need to verify only Definition~\ref{b_n_comp}, (1-b).

Step~1. It is enough to verify that
\begin{equation} \label{rddown_nc}
\D(X,\rdn{D}{n})\ge \rdn{(\D(X,D))}{n},
\end{equation}
where, for an $\R$-divisor $D=\sum d_i D_i$
and/or a b-$\R$-divisor $\D=\sum d_i D_i$,
$$
\rdn{D}{n}=\sum \rdn{d_i}{n}D_i
$$
and the rounding $\rdn{x}{n}$ is defined in~\ref{n+1_n_monotonicity}.
Indeed, by Definition~\ref{n_comp}, (1) and~(\ref{rddown_nc}),
$$
D^+\ge\rdn{D}{n} \text{ and }
\D^+=\D(X,D^+)\ge
\D(X,\rdn{D}{n})\ge \rdn{(\D(X,D))}{n}.
$$
Hence Definition~\ref{b_n_comp}, (1-b)
holds for $(X,D^+)$ with respect to $(X,D)$.

Step~2. It is enough to verify that,
for every crepant blowups (monoidal transformations)
$\varphi\colon(Y,D_Y),(Y,\rdn{D}{n}{}_{,Y})\to(X,D),(X,\rdn{D}{n})$
of a point $o\in X$,
\begin{equation}\label{rddown_nc_Y}
\rdn{D}{n}{}_{,Y}\ge \rdn{D_{Y,}}{n}
\end{equation}
holds.
The blowups can be considered birationally, that is,
over a neighborhood of (scheme) point $o$.
Indeed, (\ref{rddown_nc}) for b-$\R$-divisors
follows from the corresponding inequality
for every prime b-divisor $E$.
Since $E$ can be obtained by
a sequence of blowups,
the required inequality in $E$ follows
from~(\ref{rddown_nc_Y}) by induction.

Step~3. It is enough to verify that
\begin{equation}\label{rddown_E}
1-m+\sum_{i=1}^m \rddown{(n+1)d_i}/n\ge
\rddown{(n+1)(1-m+\sum_{i=1}^m d_i)}/n,
\end{equation}
where $m$ is a positive integer and
$d_i$ are real numbers $<1$.
Indeed, let $l\ge 1$ be the codimension of $o$ in $X$ and
$D=\sum_{i=1}^l d_iD_i$ near $o$, where $d_i$ are real numbers $\le 1$,
the prime divisors $D_1,\dots,D_l$ are with normal crossings and
$o$ is the generic point of their intersection $\cap_{i=1}^l D_i$.
It is enough to verify
the inequality~(\ref{rddown_nc_Y}) in the only exceptional
divisor $E$ of $\varphi$.
For this, we can suppose that $d_i=1$ only for $i> m$ and
$d_i<1$ otherwise.

Then
$$
\mult_E D_Y=1-l+\sum_{i=1}^l d_i=
1-l+l-m+\sum_{i=1}^m d_i=1-m+\sum_{i=1}^m d_i.
$$
So, $\mult_E \rdn{D_{Y,}}{n}$ is the right side of~(\ref{rddown_E}),
except for the case, when $\mult_E D_Y=\mult_E \rdn{D_{Y,}}{n}=1$.
In the last case $m=0$ and all $d_i=1$.
Note that in this case the left side is also $1$ and
(\ref{rddown_E}) holds.
Thus we can suppose that $m\ge 1$.

On the other hand,
$$
\rdn{D}{n}=\sum_{i=1}^m \frac{\rddown{(n+1)d_i}}n D_i
+\sum_{i=m+1}^l D_i
$$
near $o$.
Hence
$$
\mult_E \rdn{D}{n}{}_{,Y}=
1-l+l-m+\sum_{i=1}^m \rddown{(n+1)d_i}/n=
1-m+\sum_{i=1}^m \rddown{(n+1)d_i}/n,
$$
the left side of~(\ref{rddown_E}).

Finally, (\ref{rddown_E}) is the inequality of
Example~\ref{induc_upper_bound_exa} in~\ref{induc_bound}.

The bd-pairs can be treated ditto because
$\sP$ is stable over $X$.

\end{proof}

\begin{prop} \label{small_transform_compl}
Every small birational modification preserves $\R$- and
$n$-complements for pairs and bd-pairs.

The same holds for b-$\R$-complements if $X/Z$ has wFt and
the small modification preserves the log pair property.

\end{prop}

See Remark~\ref{rem_def_b_n_compl}, (1)
about b-$n$-complements.

\begin{proof}
Immediate by definition.

The statement about b-$\R$-complements holds
by Addendum~\ref{R_complement_criterion}.
Indeed, if $(X/Z,D)$ has a b-$\R$-complement then
so does $(X^\sharp,D_{X^\sharp}^\sharp)$ by the addendum.
On the other hand, a maximal model exists
under the assumptions of Construction~\ref{sharp_construction}.
Moreover, by the construction, the assumptions and
maximal models are invariant of small birational
modifications of pairs $(X/Z,D)$ or bd-pairs $(X/Z,D+\sP)$.
(Cf. Remarks~\ref{rem_def_b_n_compl}, (2).)

\end{proof}

\paragraph{Hyperstandard sets.}
By definition \cite[3.2]{PSh08} every such set has the form
$$
\Phi=\Phi(\fR)=\{1-\frac rl\mid r\in\fR
\text { and } l \text{ is a positive integer}\},
$$
where $\fR$ is a set of real numbers.
Usually, we take for $\fR$ a finite subset
of rational numbers in $[0,1]$.
Additionally, we suppose that $1\in\fR$.
So, $0=1-1/1\in\Phi$.
(Every standard multiplicity $1-1/l\in\Phi$ too.)
We say that $\Phi$ is {\em associated\/} with $\fR$.

Such a presentation for $\Phi$ is not unique.
The following presentation of hyperstandard sets
is crucial for us (see Proposition~\ref{every_sN_Phi_hyperstandard} below).
It has the form $\frak{G}(\sN,\fR)=\Gamma(\sN,\Phi)$,
the set {\em associated\/} with $\sN$ and $\fR$ or $\Phi$
(see Proposition~\ref{dependence_on_Phi} below),
with the elements
$$
b=1-\frac rl +\frac 1l
(\sum_{n\in\sN} \frac {m_n}{n+1})=
\varphi+\frac 1l
(\sum_{n\in\sN} \frac {m_n}{n+1}),
$$
where $r\in \fR$, $l$ is a positive integer,
$\varphi=1-r/l\in\Phi$, and
$m_n$ are nonnegative integers.
Usually, we suppose that $\sN$ is a finite set of
positive integers.
By construction every $b\ge 0$.
Additionally we suppose that $b\le 1$.
Thus $b\in[0,1]\cap\Q$.

The set $\Gamma(\sN,\Phi)$ is
a dcc set of rational numbers in $[0,1]$
with the only accumulation point $1$.
Actually, the set is a hyperstandard by
Proposition~\ref{every_sN_Phi_hyperstandard} and
by Proposition~\ref{dependence_on_Phi} below depends only on $\Phi$.
Conversely, every hyperstandard set has this form for
$\sN=\emptyset$.

For $\Phi\subseteq \Phi'$ and $\sN\subseteq\sN'$,
$\Gamma(\sN,\Phi)\subseteq \Gamma(\sN',\Phi')$ holds.
By our assumptions, the minimal $\fR=\{1\}$ and
$\sN=\emptyset$.
Thus $\Gamma(\emptyset,\{1\})=\Phi(\{1\})$ is the standard set
without $1$ and
it is contained in every $\Gamma(\sN,\Phi)$.
Put $\Gamma(\sN)=\Gamma(\sN,\{1\})$.
The last set contains $0$ and the fractions $m/(n+1),n\in\sN,
m\in\Z,0\le m\le n+1$ of the great importance for us.
This is why we suppose that $1\in\fR$.

\begin{prop} \label{every_sN_Phi_hyperstandard}
Let $\sN$ be a finite set of positive integers and
$\fR$ be a finite set of rational numbers in $[0,1]$.
Then there exists a finite set of rational numbers $\fR'$ in $[0,1]$
such that
$$
\Gamma(\sN,\Phi)=\Phi(\fR').
$$
More precisely,
$$
\fR'=\{r-\sum_{n\in\sN} \frac {m_n}{n+1}\mid
r\in\fR\text{ \rm and every }m_n\in\Z^{\ge 0}\}\cap [0,+\infty).
$$

\end{prop}

\begin{proof}
Take $b\in\Gamma(\sN,\Phi)$.
Then by definition $b=1-r'/m$, where
$$
r'=r-\sum_{n\in\sN}\frac{m_n}{n+1}.
$$
(Cf. with $\overline{\fR}$  in \cite[p.~160]{PSh08}.)
Since every $m_n\ge 0$ and
$r\le 1$, $r'\le r\le 1$ too.
On the other hand, $b\le 1$.
Hence $r'\ge 0$ and the set $\fR'$ is
finite rational in $[0,1]$.
\end{proof}

\begin{prop} \label{dependence_on_Phi}
Let $\sN$ be a set of positive numbers and
$\fR,\fR'$ be two sets of rational numbers in $[0,1]$
such that $\Phi=\Phi(\fR)=\Phi(\fR')$.
Then $\Gamma(\sN,\Phi)=\frak{G}(\sN,\fR)=\frak{G}(\sN,\fR')$.
\end{prop}

\begin{proof}
Immediate by definition and the property
that every $r'\in\fR'$ has the form $r'=r/m$
for some positive integer $m$ if
$\fR$ is the minimal (actually, smallest)
set such that $\Phi=\Phi(\fR)$.

\end{proof}

\begin{cor}[{cf.~\cite[Lemma~2.7]{Sh95}}] \label{accumulation_1}
Let $\sN$ be a finite set of positive integers and
$\fR$ be a finite set of rational numbers in $[0,1]$.
Then $\Gamma(\sN,\Phi)$ satisfies the dcc
with only one accumulation point $1$.
Equivalently, for every positive real number $\ep$,
the set of rational numbers
$$
\Gamma(\sN,\Phi)\cap [0,1-\ep]
$$
is finite.
\end{cor}

\begin{add} \label{0_in_Gamma_sN_Phi}
If $\sN\not=\emptyset$ then $1\in\Gamma(\sN,\Phi)$.
\end{add}

\begin{proof}
Immediate by definition or
Proposition~\ref{every_sN_Phi_hyperstandard}.

If $n\in\sN$, then
$$
1=1-\frac 11+\frac{n+1}{n+1}\in\Gamma(\sN,\Phi).
$$
\end{proof}

\begin{const}[Low approximations] \label{low_approximation}
Let $\sN$ be a nonempty finite set of positive integers and
$\fR$ be a finite set of rational numbers in $[0,1]$.
Then by Corollary~\ref{accumulation_1} and
Addendum~\ref{0_in_Gamma_sN_Phi}, for every $b\in[0,1]$, there exists
and unique (best low approximation) largest $b'\le b$ in
$\Gamma(\sN,\Phi)$.
We denote that $b'$ by $b_{\sN\_\Phi}$.
Respectively, for a boundary $B$ on $X$, $B_{\sN\_\Phi}$ denotes
the largest boundary on $X$ such that $B_{\sN\_\Phi}\le B$ and
$B_{\sN\_\Phi}\in\Gamma(\sN,\Phi)$.

Respectively, for a boundary $B_Y$ on $Y$,
put $B_{Y,\sN\_\Phi}=(B_Y)_{\sN\_\Phi}$.

We can and will apply also this construction for the divisorial part
$B$ of a bd-pair $(X/Z,B+\sP)$, when $B$ is a boundary.
\end{const}

Notation:

$b_\Phi=b_{\emptyset\_\Phi}$;

$B_\Phi=B_{\emptyset\_\Phi}$ but we suppose that $0,1\in\Phi$;

$b_{n\_\Phi}=b_{\{n\}\_\Phi}$;

$B_{n\_\Phi}=B_{\{n\}\_\Phi}$;

$b_{n\_0}=b_{n\_\{0\}}$;

$B_{n\_0}=B_{n\_\{0\}}$

etc

$B^\sharp$,
$B_{n\_\Phi}{}^\sharp=(\B_{n\_\Phi}{}^\sharp)_X$;
for different birational model $Y$ of $X$:
$B_{n\_\Phi}{}^\sharp{}_Y=(\B_{n\_\Phi}{}^\sharp)_Y$.

\begin{prop} \label{Phi_<Phi'}
$\sN\subseteq\sN',\Phi\subseteq\Phi'\Rightarrow
B_{\sN\_\Phi}\le B_{\sN'\_\Phi'}$.

\end{prop}

\begin{proof}
Immediate by definition.
\end{proof}

\begin{prop} \label{1_of def_2 for B}
Let $B,B^+$ be two boundaries on the same variety and
$n$ be a positive integer such that
$B^+\in\Z/n,n\in\sN$, and
$B^+$ satisfies (1) of Definition~\ref{n_comp}
with respect to $B_{\sN\_\Phi}$ .
Then $B^+$ satisfies {\rm (1)\/} of
Definition~\ref{n_comp} with respect to $B$.

\end{prop}

\begin{proof}
Let $P$ be a prime divisor on the variety.
Put $b^+=\mult_PB^+,b=\mult_P B$.
Then $\mult_PB_{\sN\_\Phi}=b_{\sN\_\Phi}$.
By our assumptions $b^+=m/n\in[0,1],m\in\Z$.

If $b=1$ then $b_{\sN\_\Phi}=1$ and by our assumptions $b^+=1$.
This gives (1) of Definition~\ref{n_comp} for $b$.

If $b<1$ then by definition $b_{\sN\_\Phi}\le b<1$ too.
It is enough to verify that
$$
\rddown{(n+1)b}/n=
\rddown{(n+1)b_{\sN\_\Phi}}/n,
$$

Put $b'=m/(n+1)$, where
$$
m=\max\{m\in\Z\mid\frac m{n+1}\le b\}.
$$
Then by our assumptions and definition $b'\in\Gamma(\sN,\Phi)$.
Hence $b'\le b_{\sN\_\Phi}\le b<1$.
So, by the monotonicity of $\rddown{\ }$ and
by construction
$$
\rddown{(n+1)b'}/n=
\rddown{(n+1)b_{\sN\_\Phi}}/n=
\rddown{(n+1)b}/n,
$$
in particular, the required equality.
Indeed, $b'=m/(n+1)\le b<(m+1)/(n+1)$ and
$$
\rddown{(n+1)b}/n=m/n=\rddown{(n+1)b'}/n.
$$

The threshold $m/(n+1)$ plays an important role
in the paper through $\Gamma(\sN,\Phi)$.

\end{proof}

\begin{cor} \label{B+_B_n_compl}
Let $B,B^+$ be two boundaries on the same variety and
$n$ be a positive integer such that
$B^+\in\Z/n,n\in\sN$ and $B^+\ge B_{\sN\_\Phi}$.
Then $B^+$ satisfies {\rm (1)\/} of
Definition~\ref{n_comp} with respect to $B$.
\end{cor}

\begin{proof}
Immediate by Proposition~\ref{1_of def_2 for B}.
Indeed $B^+\ge B_{\sN\_\Phi}$ implies (1) of Definition~\ref{n_comp}
with respect to $B_{\sN\_\Phi}$ by arguments in Example~\ref{b_n_comp_of_itself}, (1)
and the proof of Proposition~\ref{D_D'_complement}.

\end{proof}

\begin{cor} \label{B_B_+B_sN_Phi}
Let $(X/Z,B)$ be a pair with a boundary $B$ such that
$(X/Z,B_{\sN\_\Phi})$ has an $n$-complement $(X/Z,B^+)$ with $n\in\sN$.
Then $(X/Z,B^+)$ is an $n$-complement of $(X/Z,B)$ too.

The same holds for bd-pairs of index $m$ $(X/Z,B+\sP),(X/Z,B_{\sN\_\Phi}+\sP),(Z/Z,B^++\sP)$
under assumption $m|n$.
\end{cor}

\begin{proof}
Immediate by Proposition~\ref{1_of def_2 for B}
because $B^+\in\Z/n$ by (3) of Definition~\ref{n_comp}.

For bd-pairs we use (3) of Definition~\ref{bd_n_comp} and
our assumption $m|n$ (cf. Remark~\ref{m_div_n}).
[Indeed, $m|n$ implies that $B^+\in\Z/n$ in this case too.]
\end{proof}

\begin{const}[Maximal model or
adding fixed components] \label{sharp_construction}
Let $(X/Z,D)$ be a pair such that
\begin{description}

\item[]
$X/Z$ has wFt;
and

\item[]
$-(K+D)$ is (pseudo)effective modulo $\sim_{\R,Z}$, or
equivalently,

\item[]
there exists $D^+\ge D$ such that $K+D^+\sim_{\R,Z}0$
(an nonlc $\R$-complement).

\end{description}

Warning: $(X/Z,D^+)$ is not necessarily lc and
an $\R$-complement of $(X/Z,D)$.

Consider a Zariski decomposition $-(K+D)=M+F$,
where $M,F$ are respectively the $\R$-mobile and
fixed parts of $-(K+B)$ over $Z$.
Such a decomposition exists by Lemma~\ref{wTt_vs_Ft} and
\cite[Corollary~4.5]{ShCh}.
In addition, there exists a commutative triangle
$$
\begin{array}{ccccc}
(X,D)&\stackrel{\psi}{\dashrightarrow}&(X^\sharp,D_{X^\sharp}^\sharp)\\
\searrow&&\swarrow\\
& Z
\end{array},
$$
where $\psi$ is a birational morphism making $M$ $\R$-free
[-semiample] over $Z$ as follows.
First, we make a small birational modification $X\dashrightarrow Y/Z$
such that $M$ is $\R$-Cartier on $Y$, e.g., we can use
a $\Q$-factorialization.
By Lemma~\ref{wTt_vs_Ft} we can suppose also that $X/Z$ has Ft.
Second, we can apply the $M$-MMP to $Y/Z$.
This gives $X^\sharp/Z$.
Put $D_{X^\sharp}^\sharp=D+F$, where $D,F$ are birational transforms
of $D,F$ on $X^\sharp$.
Constructed transformation $X\dashrightarrow X^\sharp/Z$
is small and
$$
-(K_{X^\sharp}+D_{X^\sharp}^\sharp)=-(K_{X^\sharp}+D)-F=M
$$
is the birational transform of $M$.
So, $-(K_{X^\sharp}+D_{X^\sharp}^\sharp)$ is $\R$-free [-semiample] over $Z$.
All relative varieties in the construction and, in particular,
$X^\sharp/Z$ have wFt.
Finally, we can take any crepant model of
$(X^\sharp/Z,D_{X^\sharp}^\sharp)$ and usually denote
in the same way.
Every such a model is assumed to be a log pair and
its morphisms to previously constructed $X^\sharp$ is rational but not
necessarily regular.
For those pairs $X^\sharp/Z$ has not necessarily wFt.
But all pairs $(X^\sharp/Z,D_{X^\sharp}^\sharp)$ are log pairs.
They have the same b-$\R$-divisor $\D^\sharp=\D(X^\sharp,D_{X^\sharp}^\sharp)$.
By definition $D_{X^\sharp}^\sharp=\D^\sharp{}_{X^\sharp}$,
where the last divisor is the trace of $\D^\sharp$ on $X^\sharp$.
We can take the trace on any birational model of $X/Z$, e.g.,
for $X$ itself: $D^\sharp=\D^\sharp{}_X=D+F$.
However, $\D^\sharp$ is {\em stable\/} (BP) only
over such birational models $X^\sharp/Z$ that
$\D^\sharp=\D(X^\sharp,\D_{X^\sharp}^\sharp)$
is the b-codiscrepancy.
Equivalently, $\sM^\sharp=\overline{-(K_{X^\sharp}+D_{X^\sharp}^\sharp)}$
is stable over those models but the stabilization is different (Cartier):
$\sM^\sharp=\overline{\sM^\sharp_{X^\sharp}}$.

Note also, that if we apply $-(K+D)$-MMP to $X/Z$ we
get (anticanonical model) $(X^\sharp/Z,D_{X^\sharp}^\sharp)$
and possibly contract some divisors, e.g., components of $F$.
To construct the model we use antiflips.
So, the model is {\em maximal\/} in
contrast to a minimal one.
It is also maximal and even {\em largest\/} with respect to $\D^\sharp$.
For any log pair [model] $(Y/Z,D_Y^\sharp)$ of
the [bd-]pair $(X/Z,\D^\sharp)$,
$$
\D(Y,D_Y^\sharp)\le \D^\sharp
$$
and $=$ holds exactly when $(Y/Z,D_Y^\sharp)$ is a maximal model of $(X/Z,D)$.

Equivalently,
$$
\overline{M_{X^\sharp}}\le
\overline{M_Y}=\sM^\sharp
$$
and $=$ holds exactly when $(Y/Z,D_Y^\sharp)$ is maximal,
where $M_{X^\sharp},M_Y$ are birational transform of $M$
on $X^\sharp,Y$ respectively.
However, to establish this in such a more general situation
it is better to use the negativity \cite[2.15]{Sh95}
(see Lemma~\ref{b_negativity} below).

The same construction works for a bd-pair $(X/Z,D+\sP)$
such that
\begin{description}

\item[]
$X/Z$ has wFt;
and

\item[]
$-(K+D+\sP_X)$ is (pseudo)effective modulo $\sim_{\R,Z}$, or
equivalently,

\item[]
there exists $D^+\ge D$ such that $K+D^++\sP_X\sim_{\R,Z}0$.

\end{description}
A maximal model in this case will be
a log bd-pair $(X/Z,D_{X^\sharp}^\sharp+\sP)$
with same $\sP$.
That is, $\D^\sharp=\D\dv^\sharp+\sP$ or $\sP^\sharp=\sP$.
We add the fixed part $F$ only to the divisorial part:
$(D+\sP_X)^\sharp=D^\sharp+\sP_X$.
However, the $\R$-free  over $Z$ property holds for
$-(K_{X^\sharp}+D_{X^\sharp}^\sharp+\sP_{X^\sharp})$ but
usually not for $-(K_{X^\sharp}+D_{X^\sharp}^\sharp)$.
The construction actually works for any b-divisor $\sP$
under our assumptions.

By construction, the assumptions and
maximal models are invariant of small birational
modifications of pairs $(X/Z,D)$ or bd-pairs $(X/Z,D+\sP)$.

\end{const}

\begin{prop} [Monotonicity I] \label{monotonicity_I}
Let $(X/Z,D)$ be a pair under
the assumptions of Construction~\ref{sharp_construction}.
Then for every $D^+\ge D$ such that $K+D^+\sim_{\R,Z} 0$,
$\D^+\ge\D^\sharp$ and $D^\sharp\ge D$.
Moreover, $\D^\sharp\ge \D$ if $(X/Z,D)$ is a log pair.

\end{prop}

\begin{add}
$D\ge 0$ implies $D^\sharp\ge 0$.

\end{add}

\begin{add} \label{R_complement_criterion}
$(X/Z,D)$ has an $\R$-complement if and only if
$(X^\sharp,D_{X^\sharp}^\sharp)$ is lc, equivalently,
$(X,\D^\sharp)$ is lc or $\D^\sharp$ is a b-subboundary.

\end{add}

\begin{add}
The same holds for bd-pairs.

\end{add}

\begin{proof}
Immediate by definition and Construction~\ref{sharp_construction}.
Notice for this that $\D^+=\D(X,D^+)$ is a birational invariant of
$(X/Z,D^+)$, that is, every model $(Y/Z,D_Y^+),D_Y^+=\D_Y^+$, of
$(X/Z,D^+)$ is crepant.
In Addendum~\ref{R_complement_criterion},
an $\R$-complement of $(X^\sharp/Z,D_{X^\sharp}^\sharp)$
exists if it is lc (cf. Examples~\ref{1stexe}, (1-2)).

Similarly we can treat bd-pairs.

\end{proof}

\begin{cor} [Monotonicity II] \label{monotonicity_II}
Let $(X/Z,D)$ be a pair under
the assumptions of Construction~\ref{sharp_construction} and
$D'$ be an $\R$-divisor on $X$ such that $D'\le D$.
Then $(X/Z,D')$ also satisfies
the assumptions of Construction~\ref{sharp_construction} and
$\D'^\sharp\le\D^\sharp$.

\end{cor}

\begin{lemma}[b-Negativity] \label{b_negativity}
Let $\sD,\sD'$ be two b-$\R$-divisor of
a variety or space $X$ with a proper morphism $X\to S$
to a scheme $S$ such that
\begin{description}

\item[\rm (1)\/]
$\sD'$ is b-pseudoantinef;

\item[\rm (2)\/]
$\sD_X'\ge \sD_X$;
and

\item[\rm (3)\/]
$\sD$ is {\em stable\/} over $X$, that is,
$\sD=\overline{\sD_X}$
{\em (cf. \cite[Discent of divisors~5.1]{Sh03})\/}.

\end{description}
Then
$\sD'\ge \sD$.
\end{lemma}

\begin{proof}
The b-psedoantinef means that $\sD'=\lim_{n\to \infty}\sD_n$,
where the limit is weak (multiplicity wise) and
every $\sD_n$ is a b-antinef.
So, we can suppose that $\sD'$ is itself b-antinef.
Take any birational proper model $Y/X$ of $X$ over which
$\sD'$ is stable.
Then $\sD'_Y\ge \sD_Y$ by \cite[Negativity~1.1]{Sh92}.
Apply to $D=\sD'_Y-\sD_Y$ and to
the birational contraction $Y\to X$.
This implies $\sD'\ge \sD$ because we can take
an arbitrary high model $Y$.

\end{proof}

\begin{cor}[Monotonicity III] \label{monotonicity_III}
Let $(X/S,D)$ be a log pair with proper $X\to S$ and
$\D'$ be a (pseudoBP) b-$\R$-divisor of $X$ such that
\begin{description}

\item[\rm (1)\/]
$\K+\D'$ is b-pseudoantinef;
and

\item[\rm (2)\/]
$\D_X'\ge D$.

\end{description}
Then $\D'\ge\D$, where $\D=\D(X,D)$.

The same holds for a log bd-pair $(X,D+\sP)$ with
\begin{description}

\item[\rm (1-bd)\/]
$\K+\D'+\sP$ is b-pseudoantinef;
and

\end{description}
and $\D=\D(X,D+\sP)-\sP$.
\end{cor}

\begin{proof}
Immediate by Lemma~\ref{b_negativity}
with $\sD=\K+\D$ and $\sD'=\K+\D'$.
Assumptions~(1-2) of the corollary correspond
respectively to (1-2) of the statement.
Since $(X,D)$ is a log pair,
$\sD$ is defined.
So, $D=\D_X$ and $\sD$ is stable over $X$ as a b-$\R$-Cartier divisor
by definition.

Respectively, for the bd-pair,
$\sD=\K+\D+\sP$ and $\sD'=\K+\D'+\sP$, that is,
the b-divisorial part is the same.

Note that (1) implies the pseudoBP property of $\D'$,
the limit of BP b-$\R$-divisors.
\end{proof}

\begin{proof}[Proof of Corollary~\ref{monotonicity_II}]
Immediate by Corollary~\ref{monotonicity_III}.
Apply the corollary to a maximal model
$(X^\sharp,D_{X^\sharp}'^\sharp)$ with
a birational $1$-contraction
$\psi\colon X\dashrightarrow X^\sharp$ such that
$D_{X^\sharp}'^\sharp=\psi(D')$.
\end{proof}

(Non)existence of $\R$-complements has the local nature.
The same holds for $n$-complements.

\begin{prop} \label{local_compl}
Let $(X/Z\ni o,D)$ be a local pair with wFt $X/Z\ni o$
and $n$ be a positive integer.
Then the existence of an $\R$-complement (respectively
of an $n$-complement) is a local property
with respect to connected components of $X_o$, that is,
in the \'etal topology.

The same holds for bd-pairs $(X/Z\ni o,D+\sP)$.

\end{prop}

\begin{proof} (Cf. the proof of Step~8 of Theorem~\ref{invers_stability_R_complements}.)

Step~1. ({\em Reduction to a semilocal case.\/})
We can replace $(X/Z\ni o,D)$ by $(X/Y\ni o_1,\dots,o_l,D)$
with wFt $X/Y\ni o_1,\dots,o_l$ and connected fibers
$X_{o_1},\dots,X_{o_l}$.
We can suppose that $X_{o_i}$ are connected components
of $X_{o}$.
Indeed, take Stein factorization $X\twoheadrightarrow Y\to Z$.
Let $o_1,\dots,o_l$ be the points on $Y$ over $o$.
Notice that semilocal $X/Y\ni o_1,\dots,o_l$ has also wFt, equivalently,
every local $X/Y\ni o_i$ has wFt.
We need to verify that $(X/Y\ni o_1,\dots,o_l,D)$ has
an $\R$-complement (respectively an $n$-complement) if and only
if every $(X/Y\ni o_i,D)$ has and $\R$-complement (respectively
an $n$-complement).
Actually, we need to verify only the if statement.

Step~2. ({\em $\R$-Complements\/}.)
Use Addendum~\ref{R_complement_criterion}.
Indeed, we can suppose that every $(X/Y\ni o_i,D)$ has
and $\R$-complement, in particular,
every $-(K+D)$ is effective modulo $\sim_{\R}$ over $Y\ni o_i$.
Thus Construction~\ref{sharp_construction} gives
$(X^\sharp/Y\ni o_1,\dots,o_l;D^\sharp_{X^\sharp})$ for
$(X/Y\ni o_1,\dots,o_l,D)$.
So, by Addendum~\ref{R_complement_criterion} $(X^\sharp,D^\sharp_{X^\sharp})$
is lc over $Y\ni o_1,\dots,o_l$.
Thus by the same addendum $(X/Y\ni o_1,\dots,o_l,D)$ has
an $\R$-complement.

Step~3. ({\em $\R$-Complements\/}.)
Immediate by the criterion of existence for
$n$-complements in terms of linear systems \cite[after Definition~5.1]{Sh92}.

Similarly we can treat bd-pairs.

\end{proof}

\section{Constant sheaves} \label{constant_sheaves}

\paragraph{Componentwise constant sheaves.}
We consider constant sheaves $\sF$ on a variety $T$
of Abelian groups, vector spaces, monoids, convex cones,
union of monoids, polyhedrons with a polyhedral decomposition
and of sets.
E.g., if $\sF$ is a constant sheaf of Abelian groups $A$ then,
for every point $t\in T$, $\sF_t=A$ and, for every
connected open subset $S\subseteq T$,
$\sF_S=\Gamma(S,\sF)=A$, where $=$ means the canonical
isomorphism given by the restriction.
Usually in description of such a sheaf $\sF$ we give
its sections $\sF_t$ for closed points.
We also give its global sections $\sF_T$.
For connected $T$, $\sF_T=\sF_t=A$ under the restriction.

Below by a {\em constant\/} sheaf we mean also a sheaf
which is constant on every connected component of $T$.
So, it is actually constant if $T$ is connected.
Strictly speaking these sheaves are
{\em componentwise constant\/}.

We need the following constant sheaves.
However, they are really constant only for families
of very special varieties, only over connected components of $T$ and under
an appropriate parametrization (cf. Addendum~\ref{all_together} below).
These sheaves will be sheaves of different nature:
sheaves of sets,
of Abelian groups and monoids,
of $\R$-,$\Q$-linear spaces, of finite union of
monoids, of cones and of polyhedrons.

{\em Marked divisors\/}
$D_{1,t},\dots,D_{r,t}$
on a family  $X/T$ of varieties $X_t$,
are its sections over $t$.
Actually, they form an ordered set.
For connected $T$, sections form an ordered set
of global divisors $D_{1,T},\dots,D_{r,T}$ on $X$.
The identification of sections is given by
the restriction $D_{i,t}=D_{i,T}\rest{X_t}$.
So,
divisors should be good and the family $X/T$ too.
Usually, we consider distinct prime Weil divisors
$D_{i,t}$.

{\em Abelian group generated by marked divisors\/}
$\fD_t=\fD_t(D_{1,t},\dots,D_{r,t})$.
We consider this group as a subgroup of $\WDiv X_t$.
The group $\fD_t$ is free Abelian of rank $r$
with the standard basis $D_{1,t},\dots,D_{r,t}$ if
the divisors are distinct prime.
The corresponding constant sheaf we denote by $\fD=\fD(D_1,\dots,D_r)$.
Similarly, we can define real and rational (sub)spaces $\fD_\R,\fD_\Q$ with
$\fD_{\R,t}\subseteq\WDiv_\R X_t,
\fD_{\Q,t}\subseteq\WDiv_\Q X_t$ generated by divisors $D_{i,t}$;
$\fD\subseteq\fD_\Q\subseteq\fD_\R$.
If divisors $D_{i,t}$ are linearly free (in particular, distinct) in $\WDiv X_t$
then $\fD_{\R,t}=\fD_t\otimes\R,\fD_{\Q,t}\otimes\Q$.

We will use also constant divisors
$$
K_{X_t}\in\fD,\sP_{X,t}\in\fD_\Q.
$$
\begin{align*}
  K_{T}=K_{X/T}\in &\fD_T=\Z D_{1,T}\oplus\dots\oplus\Z D_{r,T}; \\
  &\fD_{\R,T}=\R D_{1,T}\oplus\dots\oplus\R D_{r,T};\\
  \sP_X=\sP_{X,T}\in &\fD_{\Q,T}=\Q D_{1,T}\oplus\dots\oplus\Q D_{r,T}.
\end{align*}
Warning: usually we do not suppose that $\sP$ exits globally but
$\sP_t$ exists for every closed $t\in T$, $\sP_{X,t}$ exists for every $t\in T$ and
$\sP_X=\sP_{X,T}$ exists too.
However, if $\sP$ exists globally over $T$ then
$\sP_T=\sP$ is meaningful and $\sP_X$ is its trace on $X$.

{\em Abelian monoid of effective divisors generated by marked divisors\/}
$\fD_t^+=\fD_t^+(D_{1,t},\dots,D_{r,t})\subseteq \WDiv X_t$ with
elements $D_t\in\fD_t,\Z^{\ge 0}$.
The monoid $\fD_t^+$ is free Abelian of rank $r$
with the standard basis $D_{1,t},\dots,D_{r,t}$ if
the divisors are distinct prime.
The corresponding constant sheaf we denote by $\fD^+=\fD^+(D_1,\dots,D_r)$.
Similarly, we can define closed convex rational polyhedral (sub)cones
$\fD_{\R,t}^+\subseteq\fD_{\R,t},
\fD_{\Q,t}^+\subseteq\fD_{\Q,t}$ generated by divisors $D_{i,t}$.

\begin{align*}
D_{1,T},\dots,D_{r,T}\in  &\fD_T^+=\Z^{\ge 0} D_{1,T}\oplus\dots\oplus\Z^{\ge 0} D_{r,T}; \\
  &\fD_{\R,T}^+=\R^+ D_{1,T}\oplus\dots\oplus\R^+ D_{r,T},\ \R^+=[0,+\infty);\\
&\fD_{\Q,T}^+=\Q^+ D_{1,T}\oplus\dots\oplus\Q^+ D_{r,T},\ \Q^+=\R^+\cap\Q.
\end{align*}

{\em Monoid of linearly effective divisors\/}
$\Effd_t\subseteq \WDiv X_t$ with
elements $D_t\in\fD_t$, which are effective modulo $\sim$.
In general the monoid is not finitely generated.
The corresponding constant sheaf we denote by $\Effd=\Effd(D_1,\dots,D_r)$.
Similarly, we can define (sub)cones
$\Effd_{\R,t}\subseteq\fD_{\R,t},
\Effd_{\Q,t}\subseteq\fD_{\Q,t}$
with
elements $D_t\in\R,\Q$,
which are effective modulo $\sim_\R,\sim_\Q$ respectively.
In general those cones are not closed rational polyhedral.
However, for wFt $X/T$, the monoid finitely generated by Corollary~\ref{eff_linear_rep_D} and
cones are closed convex rational polyhedral \cite[Corollary~4.5]{ShCh}.
Moreover, the monoid and cones have finite decompositions
into respectively finitely generated monoids,
convex rational polyhedral cones (possibly not closed) such that
the components correspond to the same rational $1$-contraction
\cite[Corollary~5.3]{ShCh}.

\begin{align*}
&\fD_T^+\subseteq \Effd_T=\Effd X=
\{D\in\fD_T\mid\ D\text{ is effective modulo\/}\sim\}; \\
&\fD_{\R,T}^+\subseteq \Effd_{\R,T}=\Effd_\R X=
\{D\in\fD_{\R,T}\mid\ D\text{ is effective modulo\/}\sim_\R\};\\
&\fD_{\Q,T}^+\subseteq\Effd_{\Q,T}=\Effd_\Q X=
\{D\in\fD_{\Q,T}\mid\ D\text{ is effective modulo\/}\sim_\Q\}.
\end{align*}

Important submonoids and subcones of semiample, nef, mobile
divisors $D_t$
$$
\sAmpd_t,\Nefd_t,\Mobd_t\subseteq\Effd_t,\
\sAmpd_{\R,t},\Nefd_{\R,t},\Mobd_{\R,t}\subseteq\Effd_{\R,t},\
\sAmpd_{\Q,t},\Nefd_{\Q,t},\Mobd_{\Q,t}\subseteq\Effd_{\Q,t}
$$
have constant sheaves
$\sAmpd=\sAmpd(D_1,\dots,D_r),\sAmpd_\R=\sAmpd_\R(D_1,\dots,D_r),
\sAmpd_\Q=\sAmpd_\Q(D_1,\dots,D_r)$ etc for $\Nefd,\Mobd,\Effd$.
Global sections are respectively
\begin{align*}
&\sAmpd_T=\{D\in\fD_T\mid\ D\text{ is semiample\/}\};\\
&\sAmpd_{\R,T}=\{D\in\fD_{\R,T}\mid\ D\text{ is semiample\/}\};\\
&\sAmpd_{\Q,T}=\{D\in\fD_{\Q,T}\mid\ D\text{ is semiample\/}\}.
\end{align*}
Etc for $\Nefd,\Mobd$.

\begin{war}[cf. \cite{ShCh}] \label{sim_vs_equiv}
Here we use $\sAmpd,\Nefd,\Mobd$ and $\Effd$
for the absolute case with multiplicities in $\Z$.
For $\Effd,\Compd$  we use $\sim,$ etc instead of $\equiv$.
We also use absolute $\sim$, etc but not relative $\sim_T$, etc.
Of course the last substitution needs appropriate parametrization
(cf. Step~3 in the proof of Proposition~\ref{bounded_rank} below).
\end{war}

{\em lc Divisors\/} $\lcd_t\subseteq \WDiv X$ with
elements $D_t\in\fD_t$,
such that $(X_t,D_t)$ is lc.
The corresponding constant sheaf we denote by $\lcd=\lcd(D_1,\dots,D_r)$.
Similarly, we can define closed convex rational polyhedrons
(possibly noncompact, e.g., $d_i\le 1$ correspond to
subboundaries) \cite[(1.3.2)]{Sh92}
$\lcd_{\R,t}\subseteq\fD_{\R,t},
\lcd_{\Q,t}\subseteq\fD_{\Q,t}$
with
elements $D_t\in\R,\Q$ respectively,
such that $(X_t,D_t)$ is lc.

Similarly, for bd-pairs, $\lcd_t\sP_t\subseteq \WDiv X$ with
elements $D_t\in\fD_t$,
such that $(X_t,D_t+\sP_t)$ is lc as a bd-pair.
The corresponding constant sheaf we denote by
$\lcd \sP=\lcd(D_1,\dots,D_r,\sP)$.
Similarly, we can define closed convex rational polyhedrons
\cite[(1.3.2)]{Sh92}
$\lcd_{\R,t}\sP_t\subseteq\fD_{\R,t},
\lcd_{\Q,t}\sP_t\subseteq\fD_{\Q,t}$
with
elements $D_t\in\R,\Q$ respectively,
such that $(X_t,D_t+\sP_t)$ is lc.
For the constant sheaf property, it is better to suppose that
the b-divisor $\sP$ exists globally over $T$ but
this is not applicable in our paper (cf. the proof of Addendum~\ref{bd_exceptional_compl}
in Step~8 of the proof of Theorem~\ref{excep_comp}).

{\em Divisors with $\R$-complement\/} $\Compd_t\subseteq \WDiv X$ with
elements $D_t\in\fD_t$,
such that $(X_t,D_t)$ has an $\R$-complement.
The corresponding constant sheaf we denote by $\Compd=\Compd(D_1,\dots,D_r)$.
Similarly, we can define convex sets
$\Compd_{\R,t}\subseteq\fD_{\R,t},
\Compd_{\Q,t}\subseteq\fD_{\Q,t}$
with
elements $D_t\in\R,\Q$ respectively,
such that $(X_t,D_t)$ has an $\R$-complement (cf. Theorem~\ref{R_compl_polyhedral}).

Similarly, for bd-pairs, $\Compd_t\sP_t\subseteq \WDiv X$ with
elements $D_t\in\fD_t$,
such that $(X_t,D_t+\sP_t)$ has an $\R$-complement as a bd-pair.
The corresponding constant sheaf we denote
by $\Compd \sP=\Compd(D_1,\dots,D_r,\sP)$.
Similarly, we can define convex sets
$\Compd_{\R,t}\sP_t\subseteq\fD_{\R,t},
\Compd_{\Q,t}\sP_t\subseteq\fD_{\Q,t}$
with
elements $D_t\in\R,\Q$ respectively,
such that $(X_t,D_t+\sP_t)$ has an $\R$-complement (cf. Addendum~\ref{bd_R_compl_polyhedral}).

By definition and our assumptions, for $K\in\fD,K+\sP_X\in\fD_\Q$,
respectively
$$
\Compd\subseteq -K-\Effd_\R, \ \
\Compd \sP\subseteq -K-\sP_X-\Effd_\R
$$
(see Addenda~\ref{comp_in_lc_eff} and~\ref{bd_R_compl_polyhedral}).
Respectively, if every $X_t$ is $\Q$-factorial, then
$$
\Compd\subseteq \lcd\cap (-K-\Effd_\R), \ \
\Compd \sP\subseteq\lcd \sP\cap (-K-\sP_X-\Effd_\R)
$$
(see again Addenda~\ref{comp_in_lc_eff} and~\ref{bd_R_compl_polyhedral})
Etc over $\R,\Q$.
In general $=$ does not hold even over $\R$
(however, cf. the exceptional case in Step~3
of the proof of Theorem~\ref{excep_comp}).
$$
\lcd_T=\{D\in\fD_T\mid (X,D)\text{ is lc}\},
\lcd_{\R,T}=\{D\in\fD_{\R,T}\mid (X,D)\text{ is lc}\},
\lcd_{\Q,T}=\{D\in\fD_{\Q,T}\mid (X,D)\text{ is lc}\}.
$$
$$
\Compd_T=\{D\in\fD_T\mid (X/T,D)\text{ has an $\R$-complement}\}.
$$
Etc over $\R,\Q$.

Notice also that $\lcd_\R,\Compd_\R$ are usually not
compact.
However, for effective divisors $D$ both sets
$\lcd,\Compd$ are compact because in this case
$D$ is a boundary, a compact condition.

{\em Sheaf $\shCl=\shCl X/T$ of Abelian class groups
of Weil divisors modulo $\sim$\/} with
$\shCl_t X/T=\Cl X_t$ and $\shCl_T X/T=\Cl X/T$.
Respectively, $\shCl_\R X/T,\shCl_\Q X/T$ with
$\shCl_{\R,t}X/T=\Cl_\R X_t,\shCl_{\Q,t} X/T=\Cl_\Q X_t$ and
$\shCl_{\R,T}X/T=\Cl_\R X/T,\shCl_{\Q,T}X/T=\Cl_\Q X/T$.
The relative class group $\Cl X/T$ is defined modulo
relative $\sim_T$.
Respectively, $\Cl_\R X/T,\Cl_\Q X/T$ modulo $\sim_{\R,T},\sim_{\Q,T}$.

Other important subsheaves in $\shCl$ include invertible sheaves,
classes of semiample, nef, mobile, effective and torsion divisors
$$
\shPic=\shPic X/T,\shsAmp=\shsAmp X/T,\shNef=\shNef X/T,\shMob=\shMob X/T,\shEff=\shEff X/T,
\shTor X/T,
$$
etc for $\R,\Q$ with
$$
\shPic_t=\Pic X_t,\shsAmp_t=\sAmp X_t,
\shNef_t=\Nef X_t,\shMob_t=\Mob X_t,\shEff_t=\Eff X_t,
\shTor_t=\Tor X_t
$$
and
$$
\shPic_T=\Pic X/T,\shsAmp_T=\sAmp X/T,
\shNef_T=\Nef X/T,\shMob_T=\Mob X/T,\shEff_T=\Eff X/T,
\shTor_T=\Tor X/T
$$
etc for $\R,\Q$.
Notice that $\shTor_\R=\shTor_\Q=0$ is trivial.
Cf. Warning~\ref{sim_vs_equiv}.

Sheaves $\shPic,\shTor$ are sheaves of Abelian groups.
Sheaves $\shPic_\R,\shPic_\Q$ are sheaves of respectively
$\R$- and $\Q$-linear spaces.
Sheaves $\shsAmp,\shNef,\shMob,\shEff$ are sheaves of Abelian monoids.
Their $\R,\Q$ versions are sheaves of convex cones.

Remark:
Usually $\sAmp_\R X_t,\Nef_\R X_t,\Mob_\R X_t,\Eff_\R X_t$
are considered in $\rm N^1 X_t$, the space of $\R$-divisors of $X_t$
modulo the numerical equivalence $\equiv$ \cite[Section 4]{ShCh}.
However, in the paper we usually consider wFt $X_t$ and
$\equiv$ is $\sim_\R$ in this situation \cite[Corollary~4.5]{ShCh}.

\begin{prop} \label{bounded_rank}
Let $X_t,t\in T$, be a variety or an algebraic space in
a bounded family of rationally connected spaces.
Then $\Cl X_t$ has bounded rank $r$.
More precisely, for appropriate parametrization $X/T$,
every $X_t$ has
marked distinct prime divisors $D_{1,t},\dots,D_{r,t}$
such that, for every $D\in\WDiv X_t$,
$$
D\sim d_1D_{1,t}+\dots+d_rD_{r,t}
$$
for some $d_1,\dots,d_r\in \Z$.

\end{prop}

\begin{add} \label{const_lc}
$\lcd,\lcd_\R,\lcd_\Q$
are also constant.

\end{add}

\begin{add} \label{const_bn_lc}
$\lcd\sP,\lcd_\R\sP,\lcd_\Q\sP$ are constant if
$\sP$ defined over $T$.

\end{add}

\begin{add} \label{const_Cl}
Sheaves $\shCl X/T,\shTor X/T,\fD,\Pd,\Nuld$ are constant,
$$
\shCl X/T=\fD/\sim=\fD/\Pd,
\shCl_T X/T=\fD_T/\sim=\fD_T/\Pd_T
\text{ and }\Cl X_t=\fD_t/\sim=\fD_t/\Pd_t,
$$
$$
\shTor X/T=\Tor (\shCl X/T)=\Nuld/\sim=\Nuld/\Pd,
$$
$$
\shTor_T X/T=\Tor(\shCl_T X/T)=\Nuld_T/\sim=\Nuld_T/\Pd_T
\text{ and }
\Tor X_t=\Tor(\Cl X_t)=\Nuld_t/\sim=\Nuld_t/\Pd_t
$$
for every (closed) $t\in T$,
where $\Pd,\Nuld$ are respectively subsheaves of principal and
of numerically trivial over $T$ divisors in $\fD$.
Canonical isomorphisms are given by homomorphisms
$\fD_t\to\Cl X_t, D_t\mapsto D_t/\sim=D_t/\Pd_t$.
The following sheaves
$$
\shCl_\R X/T,\shCl_\Q X/T,\Pd_\R,\Pd_\Q,
\Nuld_\R,\Nuld_\Q
$$
are also constant and
$$
\shCl_\R X/T=(\shCl X/T)\otimes \R,
\shCl_\Q X/T=(\shCl X/T)\otimes \Q,
$$
$$
\Nuld_\R=\Pd_\R=\Nuld\otimes\R=\Pd\otimes \R,
\Nuld_\Q=\Pd_\Q=\Nuld\otimes\Q=\Pd\otimes \Q,
$$
$$
\shTor_\R X/T=\shTor_\Q X/T=0.
$$

\end{add}

\begin{add} \label{const_Pic_etc}
If additionally every $X_t$ has only rational singularities then
the following sheaves
$$
\shPic X/T,\shPic_\R X,\shPic_\Q X/T,
\Card,\Card_\R,\Card_\Q
$$
are also constant and
$$
\shPic X/T=\Card/\sim=\Card/\Pd,
$$
$$
\shPic_\R X/T=(\shPic X/T)\otimes \R=\Card_\R/\sim_\R=\Card_\R/\Pd_\R,
\shPic_\Q X/T=(\shPic X/T)\otimes \Q=\Card_\Q/\sim_\Q=\Card_\Q/\Pd_\Q,
$$
$$
\Card_\R=\Card\otimes \R,\Card_\Q=\Card\otimes\Q,
$$
where $\Card$ is the subsheaf of Cartier or locally
principal divisors in $\fD$.
$$
\shPic_T X/T=\Card_T/\sim=\Card_T/\Pd_T
\text{ and }
\Pic X_t=\Card_t/\sim=\Card_t/\Pd_t,
$$
for every (closed) $t\in T$.
Canonical isomorphisms are induced by
the homomorphism $\fD\to \shCl X/T$.
Equivalences $\sim_\Q,\sim_\R$ can be replaced by $\equiv/T$.

\end{add}

\begin{proof}

Step~1. We can suppose that every $X_t$ is
nonsingular, projective, the family $X/T$
is smooth, $X,T$ are nonsingular, irreducible.
Use a resolution of singularities,
reparametrization  and
Noetherian induction.
For every closed $t\in T$, $X_t$  is again rationally connected.

Indeed, let $Y\to X/T$ be a resolution and reparametrization
such that $Y\to X$ and $Y_t\to X_t,t\in T$ are birational.
We can suppose that $X,T$ also satisfies required properties.
By definition there are canonical surjections
$$
\Cl Y/T\twoheadrightarrow\Cl X/T,
\Cl Y_t\twoheadrightarrow\Cl X_t,
$$
and the first one commutes with restrictions.
Indeed, the group of Weil divisors on $Y$ modulo $\sim$
is the group of Cartier divisors on $Y$ modulo $\sim$ and
goes canonically to the group of Weil divisors on $X$
modulo $\sim$.
The same works for $Y_t\to X_t$.
Note also that the restriction of divisor commutes with
its image.
Since we consider relative $\Cl Y/T$ we need to
take the quotient of $\Cl Y$ modulo the vertical over $T$ divisors or
to take sufficiently small $T$.
(We can suppose that they are pullbacks of Cartier
divisors from $T$.)
Their images are also vertical over $T$.

The nonexceptional over $X_t$
marked distinct prime divisors on $Y_t$ goes
to marked distinct prime divisors on $X_t$.
Recall, that by definition of marked divisors:
$D_{1,t}=D_1\rest{Y_t},\dots,D_{s,t}=D_s\rest{Y_t}$
on $Y_t$, where $D_1,\dots,D_s$ are prime divisors on $Y$.
Suppose that $D_1,\dots,D_s$ are required generators.
Let $D_{1,t},\dots,D_{1,r}, r\le s$, be the only nonexceptional
among them, equivalently, $D_1,\dots, D_r$ be
the only nonexceptional over $X$ among $D_1,\dots,D_s$.
Then
\begin{align*}
\fD(D_1,\dots,D_s)/\sim=&\fD(D_1,\dots,D_r)/\sim\bigoplus
\fD(D_{r+1},\dots,D_s)=\Cl Y/T=\Cl Y_t=\\
&\fD(D_{1,t},\dots,D_{r,t})/\sim\bigoplus
\fD(D_{r+1,t},\dots,D_{s,t}.
\end{align*}
Note that for distinct prime divisors
$D_{r+1},\dots,D_s$ with exceptional $D_{r+1}+\dots+D_s$,
$\fD(D_{r+1},\dots,D_s)/\sim=\fD(D_{r+1},\dots,D_s)$
(e.g., by \cite[1.1]{Sh92}).
Thus the required statements for $X_t$ follows
from that of $Y_t$.
Note for this that
$\fD(D_1,\dots,D_r),\fD(D_{1,t},\dots,D_{r,t})$ are
birational invariants with respect to blowdowns on
$X,X_t$ respectively.

For sufficiently small $T$ we can use $\sim$ instead of $\sim_T$
(see Step~3).
We can replace quotients by $\sim$ by
that of $\Pd$, the principal divisor sheaf.
(For $\sim_T$ we can use locally over $T$ principal divisors.)

In the following we suppose that $X/T,X_t$
satisfy required properties.

Step~2. $\shCl X/T=\shPic X/T$
is constant sheaf of $\Pic X_t$.
The equation holds because $X$ is nonsingular and
all Weil divisors are Cartier.
For closed points $t\in T$, $X_t$ is rationally connected.
Hence $\Pic X_t$ is a finitely generated Abelian group.
(This follows from the finite generation of the Neron-Severi group
and from the vanishing $H^0(X_t,\Omega_{X_t}^1)=0$; e.g.,
see \cite[Corollary~3.8, Chapter IV]{Kol91}.)
This implies also that $\Pic X/T$ is locally constant
in \'etal topology.
Moreover, the monodromy is finite because
it transforms ample divisors into ample ones of
the same degree.
By rational connectedness again there are only finitely many of those
divisors up to $\sim$.
After a base change we can suppose that
the monodromy is trivial and $\shPic X/T$ is constant.

Note that if we are working over $k=\C$ then
$\shPic X/T$ is the constant sheaf of $\Pic X_t=H^2(X_t,\Z)$.
The same holds in characteristic $0$
by the Lefshchetz principal.

Step~3. Search for generators.
Take a closed point $t\in T$ and
generators $D_{1,t},\dots,D_{r,t}$ of $\Cl X_t=\Pic X_t$.
More precisely, their classes module $\sim$
are those generators.
By construction and Step~2 there exists horizontal divisors $D_1,\dots,D_r$
such that $D_{1,t}\sim D_1\rest{X_t},\dots,D_r\sim D_r\rest{X_t}$.
Replacing every $D_{i,t}$ by the restriction $D_i\rest{X_t}$ we get
equation instead of $\sim$.
Since $X/T$ is projective, sufficiently general $D_i$
are reduced with reduced restriction $D_{i,t}$ for
every $t\in T$ and with pairwise disjoint support.
When it is needed, we can take sufficiently general $t$, that is,
a nonempty open subset in $T$.
The prime components of $D_i$ corresponds to the prime components of
$D_{i,t}$.
Take prime components of all $D_{i,t}$ we get
required marked distinct prime divisors $D_{1,t},\dots,D_{r,t}$ for every $t\in T$.
(Possible with a different $r$.)
Here we need again a finite covering of $T$ related to
the monodromy of components.
By construction $D_{i,t}$ generate $\WDiv X_t=\CDiv X_t$
modulo $\sim$ for every $t$.
This concludes the proof of the proposition
and Addendum~\ref{const_Cl}.
By construction, for every closed $t\in T$ the natural homomorphism
$$
\fD\to\fD_t/\sim=\Cl X_t=\Cl X/T
$$
is surjective.
Its kernel consists of divisors $D\in\fD$ such that
$D\sim 0$ is principal modulo vertical divisors.
In general, we can't remove vertical divisors.
But for every given linear equivalence we can make this
if we consider smaller $T$.
Since we have a finitely generated Abelian group $\fD_T$,
$\Pd_T$ is finitely generated too and
it is enough finitely many relations and
we can get a required sufficiently small $T$.

Assume that $Y/T$ is a log resolution of $X/T$:
the exceptional divisors
$E_{i,t}$ with the birational transform of  $D_{1,t},\dots,D_{r,t}$
are normal crossing.
In this situation we can use the arguments of
the proof of \cite[(1.3.2)]{Sh92} for
Addendum~\ref{const_lc}.

Addendum~\ref{const_bn_lc} can be verified similarly,
assuming that the log resolution
is sufficiently high: $\sP$ is stable over $X$.
The latter means that $\sP$ is defined over $T$.

Addendum~\ref{const_Cl} for $\shTor$ and $\Nuld$
follows from the fact that topologically or numerically
trivial divisors in the rationally connected variety $X_t$
correspond to torsions of $\Tor(H^2(X_t,\Z))=\Tor(\Pic X_t)$.
That is, they are torsions modulo $\sim$.

Addendum~\ref{const_Pic_etc} is immediate
by Addendum~\ref{const_Cl} under the rationality of singularities.
Indeed, the divisors of $\Card_t$ correspond to
the divisors on the resolution $Y_t$ which are
linearly trivial over $X_t$, equivalently,
vertical Cartier on $Y_t$ over $X_t$
\cite[Propositions~2-3]{Sh19}.
By the assumption, $\sim$ over $Y_t$ is $\equiv/T$
up to torsions.
This implies the constant property for
sheaf of $\Q$-Cartier divisors $\Card_\Q\cap\fD$.
The Cartier divisors of the last sheaf are $\Q$-Cartier Weil divisors
of Cartier index $1$.
This is an open condition with respect to $T$.
A Noetherian induction concludes the proof.
We need to use here also the boundedness
of the Cartier index that follows from the boundedness
of torsions.

\end{proof}

\begin{thm} \label{eff_linear_rep_Cl}
Let $X/Z$ be a wFt morphism.
Then the classes of effective Weil divisors in $\Cl X/Z$
form a finitely generated Abelian monoid $\Eff X/Z$.
The same holds for classes effective Cartier divisors in $\Cl X/Z$.

\end{thm}

\begin{add} \label{b_free_C}
More precisely, there exists a finite set of
classes $C_1,\dots,C_r\in\Cl X/Z$ of effective divisors
with b-free mobile part over $Z$, in particular,
their Zariski decomposition over $Z$ is defined over $\Z$,
and a finite set of classes $W_1,\dots,W_s\in\Cl X/Z$ of
effective Weil divisors such that every class $E\in\Cl X/Z$
of an effective Weil divisor has the form
\begin{equation}\label{generation}
E=W_j+n_1C_1+\dots n_rC_r,
\text{ every } n_i\in \Z^{\ge 0}.
\end{equation}

\end{add}

\begin{add} \label{effective_Cartier}
If $X$ is $\Q$-factorial,
then we can suppose additionally that every $C_i$ is
the class of an effective Cartier divisor.

\end{add}

For us a different presentation of effective classes
in terms of prime divisors will be more important,
especially, on the level of divisors
(cf. Step~6 in the proof of Theorem~\ref{excep_comp}).

\paragraph{Finite linear presentation.}
Let $X/Z$ be a morphism and $D_1,\dots,D_r$ be
a finite collection of distinct prime divisors on $X$.
They give a monoid of
effective Weil divisors supported on $D_1+\dots+D_r$
$$
\fD^+=\fD^+(D_1,\dots,D_r)=
\{d_1D_1+\dots+d_rD_r\mid d_1,\dots,d_r\in\Z^{\ge 0}\}
\subseteq\fD=\fD(D_1,\dots,D_r).
$$
The monoid is Abelian with $0$ and free, finitely generated.
We say that $\fD^+$ is a {\em linear representative\/} if
every effective divisor $E$ on $X$ is in $\fD^+$ modulo $\sim$:
$$
E\sim_Z d_1D_1+\dots+d_rD_r\in\fD^+.
$$
For local $X/Z\ni o$, it is enough $\sim$ instead of $\sim_Z$.

The same can be defined for Cartier divisors with
effective indecomposable Cartier generators, not necessarily prime.
An effective Cartier divisor $D$ is {\em indecomposable\/},
if $D>0$ and $D=C_1+C_2$, where $C_1,C_2$ are effective Cartier
then $D=C_1$ for $C_1>0$, otherwise $D=C_2$.
However, Cartier generators may be not free in $\CDiv X$.

\begin{cor} \label{eff_linear_rep_D}
Let $X/Z$ be a wFt morphism.
Then there exists a linear representative
finitely generated by prime Weil divisors monoid $\fD^+$.
The same hold for Cartier divisors
with effective indecomposable Cartier generators.
\end{cor}

\begin{add} \label{eff_linear_rep_DR}
Every (\/$\R$-linearly) effective $\R$-divisor $E$
is $\R$-linearly equivalent to an element in $\fD_\R^+$.
Respectively, every ($\Q$-linearly) effective $\Q$-divisor $E$
is $\Q$-linearly equivalent to an element in $\fD_\Q^+$.
\end{add}

\begin{proof}
Immediate by Theorem~\ref{eff_linear_rep_Cl}.
Take as generators of $\fD^+$ distinct prime components $D_i$ of
effective Weil divisors $E_j$ such that classes of
$E_j$ generate $\Eff(X/Z)$.

Effective Cartier generators can be replace
by indecomposable ones.

The addendum is immediate.
If Weil divisors $W_1,\dots,W_n$ are
linearly equivalent to effective divisors
$E_1,\dots,E_n\in\fD^+$ respectively, then
$$
D=r_1W_1+\dots+r_nW_n\sim_\R r_1E_1+\dots+r_nE_n\in\fD_\R^+
$$
for every real numbers $r_1,\dots,r_n\ge 0$.
The same works for over $\Q$.

\end{proof}

\begin{cor}
Let $X/Z$ be a wFt morphism. Then the classes of
effective exceptional divisors of $X/Z$ in $\Cl X/Z$ is a finite union
of free finitely generated Abelian saturated monoids of
classes of effective exceptional divisors.
Respectively, the classes of effective linearly fixed over $Z$ divisors in $\Cl X/Z$
belong to a finite union of orbits with action of
free saturated monoids of classes of exceptional divisors.

In particular, exceptional and effective linearly fixed divisors supported
a finite reduced divisor.
\end{cor}

\begin{add} \label{fix_exc}
$\Fix X/Z=\Fixd$ has finite support and
is a finite union of orbits with action of
exceptional submonoids of $\Excd^+$.
Every saturated submonoid in $\Excd^+$ is exceptional.
\end{add}

Exceptional divisors here and everywhere include birationally
exceptional divisors, that is, divisors contractible by
a rational $1$-contraction.

\begin{proof}
A sum of (effective) exceptional divisors can be nonexceptional.
For a wFt morphism or, equivalently, for
a Ft morphisms $X/Z$ the exceptional divisors belong to a union of exceptional
divisors corresponding to birational $1$-contractions
$X\dashrightarrow Y/Z$.
The exceptional divisors of such a contraction form
a free Abelian group generated by the prime exceptional divisors of
this contraction.
Respectively, the effective exceptional divisors of this contraction form
a free Abelian saturated monoid generated by the prime exceptional divisors
of this contraction.
The required results follows from \cite[Corollary~4.5 and Proposition~5.14]{ShCh}.

An effective {\em linearly fixed\/} divisor is
the only divisor in its linear system.
A sum of those divisors is effective but not necessarily linearly fixed.
By Corollary~\ref{eff_linear_rep_D} every
effective linearly fixed divisor $D$ is supported in $D_1+\dots+D_r$, that is,
$D=n_1D_1+\dots+n_rD_r,n_i\in\Z^{\ge 0}$.
On the other hand,
nonnegative integral multiplicities $n_i$ can be
arbitrary large only for exceptional $D_i$.
This concludes the statement about
effective linearly fixed divisors.

In Addendum~\ref{fix_exc}, $=$ holds because every effective fixed divisor is
unique in its class modulo $\sim$.
An {\em exceptional\/} submonoid of $\Excd^+$ in the addendum is
the monoid of effective exceptional divisors
for a birational $1$-contraction $X\dashrightarrow Y/Z$.
Respectively, {\em saturation\/} means that if $n_1E_1+\dots+n_mE_m$
belongs to the submonoid, where every $n_i\in \N$ and
every $E_i$ is an effective Weil divisor, then
every $E_i$ belongs to the submonoid.

\end{proof}

The proof of Theorem~\ref{eff_linear_rep_Cl}
will be reduced to the following [general fact
from algebra].

\begin{lemma} \label{fg_monoid}
Let $C$ be a (free) Abelian monoid generated
by a finite set of generators $C_1,\dots,C_r$, and
$S$ be a set with a transitive action of $C$ and
$T$ be a subset in $S$ closed under the action of $C$.
Then $T$ is finitely generated over $C$, that is, can be covered
by finitely many orbits of $C$ in $T$.
\end{lemma}

\begin{proof}
Induction on $r$.
The case $r=0$ has empty set of generators and
$S$ is has at most one element.

Suppose now that $r\ge 1$ and  $S\not=\emptyset$.
Then by transitivity there exists an element $0\in S$ such
that every element $s\in S$ has the form
$$
s=0+n_1 C_1+\dots+n_rC_r,
\text{ every } n_i\in \Z^{\ge 0},
$$
where $+$ denotes the action.
We can suppose also that $T\not=\emptyset$.
Then $T$ has an element
$$
t=0+m_1C_1+\dots+m_r C_r,
\text{ every }m_i\in \Z^{\ge 0}.
$$
This element gives a big orbit $t+C$ in $T$.
This means that for all other element
$$
t'=0+m_1'C_1+\dots+m_r' C_r,
\text{ every }m_i'\in \Z^{\ge 0}
$$
in $T$ some $m_i'<m_i$.
The elements $t'$ with fixed $m_i'$ belong to
the orbit
$$
S'=0+m_i'C_i+ C',
$$
where $C'$ is a free Abelian submonoid in $C$ generated by
$C_j$ with $j\not=i$.
By induction on $r$ the intersection $S'\cap T$ can be covered
by finitely many orbits of $C'$.
Since we have finitely many subsets as $S'$ in $S$,
$T$ can be covered by finitely many orbits of $C$.

\end{proof}

\begin{proof}[Proof of Theorem~\ref{eff_linear_rep_Cl}]

Step~1. (Monoid.)
If $E,E'\in\Eff X/Z$ are classes of effective
Weil divisors $D,D'$ then $E+E'$ is the class of
effective Weil divisor $D+D'$.
The zero class of $\Eff X/Z$ is the class of $0\ge 0$.
So, $\Eff X/Z$ is a Abelian monoid.

Step~2. Taking a $\Q$-factorialization $\varphi\colon Y\dashrightarrow X$
with a small birational modification of Lemma~\ref{wTt_vs_Ft},
we can suppose that $X$ is $\Q$-factorial and $X/Z$ has Ft.
Since $\varphi$ is a small birational
transformation, it preserves required generators
according to the commutative diagram
$$
\begin{array}{ccc}
\Eff Y/Z&\subset &\Cl Y/Z\\
\downarrow&&\downarrow\varphi_*\\
\Eff X/Z&\subset&\Cl X/Z
\end{array}
$$
with vertical isomorphisms.

Step~3. Generators $C_i$.
We find generators modulo $\sim_\Q$.
The classes of effective Weil divisors map naturally
(on generators over $\R^+$) into the cone of effective divisors:
$$
c\colon\Eff X/Z\to
\Eff_\R X/Z\subset \Cl_\R X/Z=(\Cl X/Z)\otimes_\Z \R,
$$
a class modulo $\sim_Z$ goes to the class modulo
$\sim_{\R,Z}$ or $\equiv$ over $Z$.
(The image generate the cone over $\R^+$.)
The cone $\Eff_\R X/Z$ is closed convex rational polyhedral
\cite[Corollary~4.5]{ShCh}.
It has a rational simplicial cone decomposition.
So, it is enough to establish the required finite generatedness
of classes of effective Weil divisors over any
rational simplicial cone $Q\subseteq \Eff_\R X/Z$.
Moreover, we can suppose that $Q$ is standard
with respect to the lattice
$$
(\Cl X/Z)/\Tor\subset\Cl_\R X/Z,
$$
where $\Tor$ denotes the subgroup of torsions in
$\Cl X/Z$.
In other words, the primitive vectors $e_1,\dots,e_r$ of edges
of $Q$ generate the monoid $Q\cap((\Cl X/Z)/\Tor)$.
By definition and construction,
for every $e_i$,
some positive multiple $C_i=m_ie_i$ is linear equivalent
to an effective Weil divisor.
By \cite[Proposition~5.14]{ShCh} we can take $C_i$
which satisfy Addenda~\ref{b_free_C} and \ref{effective_Cartier}.
The b-{\em free\/} assumption means that, after a small birational modification
of $X/Z$, the mobile part of $C_i$ become the class of
a linearly free divisor over $Z$.
In this situation, the Zariski decomposition of every $C_i$
is defined in $\Cl X/Z$, in particular, over $\Z$.
We can suppose also that the positive integers $m_i$ are
sufficiently divisible, e.g., kill torsions:
for every $m_i$, $m_i\Tor=0$.
So, a lifting of $C_i$ to $\Cl X/Z$ is well-defined:
we identify $C_i\in Q$ with the lifting $C_i=m_i e_i'\in\Cl X/Z$,
where $e_i'\in\Cl X/Z$ goes to $e_i\in Q$.
Every class $C_i$ in $\Cl X/Z$ is not a torsion
and belongs to $\Eff X/Z$.
The free Abelian monoid $C$ generated by $C_1,\dots,C_r$
will be considered as a submonoid in $\Eff X/Z$ and
in $Q$.

Step~4. Generators $W_j$.
The submonoid $C$ acts naturally on $\Eff X/Z,Q$ and on $c\1 Q$:
for every element $n_1C_1+\dots+n_rC_r\in C$,
a class $E\in\Eff X/Z$ goes to
$$
E+n_1C_1+\dots+ n_r C_r\in\Eff X/Z
$$
under the action of $n_1C_1+\dots+n_r C_r$.
Indeed, if $E\in c\1 Q$, then
$$
c(E+n_1C_1+\dots+n_r C_r)=c(E)+n_1C_1 +\dots+n_rC_r\in Q
$$
and
$$
E+n_1C_1+\dots+n_r C_r\in c\1 Q.
$$
So, if $E\in c\1 Q\cap (\Eff X/Z)$,
then for every $n_1C_1+\dots+ n_r C_r\in C$,
$$
E+n_1C_1\dots+ n_r C_r\in c\1 Q\cap(\Eff X/Z)
$$
and $C$ acts on $c\1 Q\cap(\Eff X/Z)$.
The required finite generatedness over $Q$ means that
there exists finitely many classes $W_1,\dots,W_s\in c\1 Q\cap (\Eff X/Z)$
such that every class $E\in c\1 Q\cap (\Eff X/Z)$ has
the form~(\ref{generation}).

Step~5. Reduction to Lemma~\ref{fg_monoid}.
Actually, it is enough to establish this for every
orbit $E+C$ of $c\1 Q$.
The orbits are not disjoint but finitely many of them
cover $c\1 Q$: for finitely many $E_1,\dots,E_l\in c\1 Q$,
$$
c\1 Q=\cup_{i=1}^l (E_i+ C).
$$
Indeed, the same holds for $Q$ with
$$
E_j=n_1e_1+\dots+n_r e_r, \text{ every }0\le n_i<m_i.
$$
The required classes for $c\1 Q$ are
the liftings of these $E_j$, that is, the elements
of $c\1 E_j$.
Every last set is finite and has as many elements as
$\Tor$.

The required finite generatedness of $(E+C)\cap  (\Eff X/Z)$ for every given class
$E+C$, that is, the intersection can be covered by finitely many
orbits under the action of $C$, follows from Lemma~\ref{fg_monoid}.

The version with Cartier divisors can be established similarly.

\end{proof}

\begin{prop} \label{finite_lin_repres}
Let $X_t,t\in T$, be a wFt variety or an algebraic space in
a bounded family.
Then a linear representative monoid $\fD^+_t$
has bounded rank $r$.
More precisely, for appropriate parametrization $X/T$,
every $X_t$ has
marked distinct prime divisors $D_{1,t},\dots,D_{r,t}$
such that for every effective $D\in\WDiv X_t$,
$$
D\sim d_1D_{1,t}+\dots+d_rD_{r,t}
$$
for some $d_1,\dots,d_r\in \Z^{\ge 0}$.

\end{prop}

Notice that these marked divisors also satisfy
Proposition~\ref{bounded_rank} (cf. Addendum~\ref{all_together} below).

\begin{add}
The sheaf of linear representative monoids
$\fD^+=\fD^+(D_1,\dots,D_r)$ is constant.
Every $D\in\Effd$ is
linear equivalent to an effective divisor in $\fD_T^+$.

\end{add}

Remark: $\fD^+\subseteq\Effd$ but not $=$ in general.

\begin{add} \label{const_Mob}
Sheaves
$
\shEff X/T,\shMob X/T,\shFix X/T,\shExc X/T,\shExc^+ X/T,
\Effd,\Mobd, \Fixd,\Excd,\Excd^+
$
are constant,
$$
\shEff X/T=\Effd/\sim=\Effd/\Pd,
\shEff_T X/T=\Effd_T/\sim=\Effd_T/\Pd_T
\text{ and }
\Eff X_t=\Effd_t/\sim=\Effd_t/\Pd_t,
$$
$$
\shMob X/T=\Mobd/\sim=\Mobd/\Pd,
\shMob_T X/T=\Mobd_T/\sim=\Mobd_T/\Pd_T
\text{ and }
\Mob X_t=\Mobd_t/\sim=\Mobd_t/\Pd_t,
$$
$$
\shFix X/T=\Fixd,
\shFix_T X/T=\Fix X/T=\Fixd_T
\text{ and }
\Fix X_t=\Fixd_t,
$$
$$
\shExc X/T=\Excd,
\shExc_T X/T=\Exc X/T=\Excd_T
\text{ and }
\Exc X_t=\Excd_t,
$$
$$
\shExc^+ X/T=\Excd^+,
\shExc_T^+ X/T=\Exc^+ X/T=\Excd_T^+
\text{ and }
\Exc^+ X_t=\Excd_t^+,
$$
for every (closed) $t\in T$,
where $\Fixd,\Excd,\Excd^+$ are respectively subsheaves of
(linearly) fixed, exceptional, effective exceptional divisors
over $T$ in $\fD$ and
$\shFix X/T,\shExc X/T,\shExc^+ X/T$ are corresponding
subsheaves in $\shCl X/T$.

The following sheaves
$$
\shEff_\R X/T,\shEff_\Q,\shMob_\R X/T,\shMob_\Q X/T,
\shExc_\R X/T,\shExc_\Q X/T,\shExc_\R^+ X/T,
\shExc_\Q^+ X/T
$$
$$
\Effd_\R,\Effd_\Q,\Mobd_\R\Mobd_\Q,
\Excd_\R,\Excd_\Q,\Excd_\R^+,\Excd_\Q^+
$$
are also constant,
generated respectively by
$\shEff X/T,\shMob X/T,\shExc X/T,\shExc^+ X/T,
\Effd,\Mobd,\Excd,\Excd^+$ over $\R^+,\Q^+$ and
$$
\shEff_{\R,T}X/T=\Effd_{\R,T}/\sim_\R=
\Effd_{\R,T}/\Pd_{\R,T},
\shEff_{\Q,T}=\Effd_{\Q,T}/\sim_\Q=\Effd_{\Q,T}/\Pd_{\Q,T},
$$
$$
\shMob_{\R,T}X/T=\Mobd_{\R,T}/\sim_\R=
\Mobd_{\R,T}/\Pd_{\R,T},
\shMob_{\Q,T}=\Mobd_{\Q,T}/\sim_\Q=\Mobd_{\Q,T}/\Pd_{\Q,T},
$$
$$
\shExc_{\R,T}X/T=\Exc_\R X/T=\Excd_{\R,T},
\shExc_{\Q,T}=\Exc_\Q X/T=\Excd_{\Q,T},
$$
$$
\shExc_{\R,T}^+X/T=\Exc_\R^+ X/T=\Excd_{\R,T}^+,
\shExc_{\Q,T}^+=\Exc_\Q X/T=\Excd_{\Q,T}^+.
$$
Canonical isomorphisms are given by homomorphisms
$\fD_t\to\Cl X_t, D_t\mapsto D_t/\sim=D_t/\Pd_t$.
Equivalences $\sim_\Q,\sim_\R$ can be replaced by $\equiv/T$.

\end{add}

Remark: We do not consider $\shFix_\R X/T,\shFix_\Q X/T$
and respectively $\Fixd_\R,\Fixd_\Q$
because to be fixed even under wFt is preserved for multiplicities
of a divisor if and only if it is exceptional
(Zariski decomposition).

For $X/Z$, monoids $\Effd$ and $\Mobd$ are defined
respectively as effective and mobile divisors in $\fD$ modulo $\sim_Z$.
The corresponding classes in $\Cl X/Z$ are $\Eff X/Z$ and $\Mob X/Z$ respectively.
But $\Fixd$ defined as divisors
without $\sim_Z$ in $\fD$ and
$\Fixd\supseteq \Excd^+$.
The corresponding classes belong to $\Fix X/Z$ and
$\Fix X/Z=\Fixd$ under the canonical identification.

\begin{add}
If additionally $X/T$ is projective, that is,
has Ft, then sheaves of monoids
$$
\shsAmp X/T=\shNef X/T,\sAmpd=\Nefd
$$
are constant and
$$
\shsAmp X/T=\sAmpd/\sim=\sAmpd/\Pd,
\shsAmp_T X/T=\sAmpd_T/\sim=\sAmpd_T/\Pd_T,
$$
$$
\shNef X/T=\Nefd/\sim=\Nefd/\Pd,
\shNef_T X/T=\Nefd_T/\sim=\Nefd_T/\Pd_T.
$$

The following sheaves of cones
$$
\shsAmp_\R X/T=\shNef_\R X/T,
\shsAmp_\Q X/T=\shNef_\Q X/T,
\sAmpd_\R=\Nefd_\R,\sAmpd_\Q=\Nefd_\Q
$$
are also constant, generated respectively by
$$
\shsAmp X/T,\sAmpd
$$
over $\R^+,\Q^+$ and
$$
\shsAmp_{\R,T}X/T=\sAmpd_{\R,T}/\sim_\R=
\sAmpd_{\R,T}/\Pd_{\R,T},
\shsAmp_{\Q,T}=\sAmpd_{\Q,T}/\sim_\Q=\sAmpd_{\Q,T}/\Pd_{\Q,T},
$$
$$
\shNef_{\R,T}X/T=\Nefd_{\R,T}/\Nuld_{\R,T},
\shNef_{\Q,T}=\Nefd_{\Q,T}/\Nuld_{\Q,T}.
$$
Canonical isomorphisms are given by homomorphisms
$\fD_t\to\Cl X_t, D_t\mapsto D_t/\sim=D_t/\Pd_t$
Equivalences $\sim_\Q,\sim_\R$ can be replaced by $\equiv/T$.

\end{add}

\begin{proof}
For good properties of the family $X_t,t\in T$,
we need to change it or to take an appropriate
parametrization.
Usually we alternate the base $T$ and cut out
its closed subfamilies.
We use for this a Noetherian induction.
However, the key point in all proofs is a boundedness that
typically means finiteness and boundedness of
generators.

Step~1. We can suppose that every $X_t$ is
$\Q$-factorial, has Ft and
there exists a polarization $H_t$ on $X_t$
compatible with a complement.
The latter means that there exists an $\R$-complement
$(X_t,B_t)$ of $(X_t,0)$ and effective $E_t\le B_t$
such that $E_t\equiv hH_t$, where $h$ is a positive rational number,
independent of $t$,
and $H_t$ is an ample Cartier divisor on $X_t$.
For this we can convert our family $X/T$ into
a wFt and $\Q$-factorial Ft variety
with $\Q$-factorial Ft fibers $X_t$
by Lemma~\ref{wTt_vs_Ft} and after a $\Q$-factorialization.
Then we can construct an $\R$-complement and
actually a klt $n$-complement $(X/T,B)$ of $(X/T,0)$,
a relative polarization $H$ on $X$ over $T$
and an effective $\Q$-divisor $E\le B$ on $X$
such that $E\equiv hH/T$, where $h$ is a required
positive rational number.
Moreover, by Proposition~\ref{bounded_rank}
we can suppose that $K,B,E,H$ are constant, that is,
belong to $\fD_\Q$.
Put $B_t=B\rest{X_t},E_t=E\rest{X_t},H_t=H\rest{X_t}$.

In particular, boundedness of divisors or other cycles on $X_t$
can be measured by $H$ or, more precisely, by $H_t$.

Note also that sheaves $\shEff,\shMob,\shFix,\shExc,\shExc^+$,
their divisorial versions $\Effd,\Mobd,\Fixd,\Excd,\Excd^+$ and
their $\R,\Q$ versions are invariants of small birational modification.
But sheaves $\shsAmp=\shNef,\sAmpd=\Nefd$ and
their $\R,\Q$ versions depend on a model of
$X/T$; $=$ holds by \cite[Corollary~4.5]{ShCh} because every $X_t$ has Ft.

Our main objective is to verify the constant sheaf property for
$\Effd$ and $\shEff X/T$.
We will be sketchy for other sheaves.
By Proposition~\ref{bounded_rank} we can
suppose that the sheaf of $\R$-linear spaces $\shCl_\R X/T$
is constant.

Step~2. Its subsheaf of
closed convex rational polyhedral cones $\shEff_\R X/T$
is also constant.
Indeed, $\Eff_\R X/T$ is covered by
finitely many convex rational polyhedral cones (countries)
$\fP_\varphi$ (relative geography).
The cone $\fP_\varphi$ corresponds to
a models $\varphi\colon X\dashrightarrow Y/T$,
where $\varphi$ is a rational $1$-contraction.
The class of an $\R$-divisor $D$ belongs to $\fP_\varphi$
if its $1$-contraction $\varphi_D$ ($D$-model) is $\varphi$ and
the birational $1$-contractions, on which $D$ is nef,
are the same (minimal $D$-models).
Cones $\fP_\varphi$ are not
necessarily closed.
(They depend also on the minimal models.)
The finiteness and existence of such
a covering  for $\Eff_\R X/T$ and for $\Eff_\R X_t$
see in \cite[Theorem~3.4 and Corollary~5.3]{ShCh}.

To verify that $\shEff_\R X/T$ is constant it is
easier to do this with the constant sheaf property
of cones $\fP_\varphi$.
For this we need to verify that the restriction $\sf_\varphi\rest{X_t}$
of cone $\fP_\varphi$ for $\Eff_\R X/T$ is
$\fP_{\varphi_t}$, where
$$
\varphi_t=\varphi\rest{X_t}\colon
X_t\dashrightarrow Y_t
$$
and the corresponding models with nef $D_t$
are restrictions too.

We can do this inductively with respect to models $\varphi$.
We start with $\varphi=\Id_X$, the identical model.
In this case $\fP_\varphi=\Amp_\R X/T$ is
the cone of ample $\R$-divisors on $X/T$.
The sheaf of such cones $\shAmp_\R X/T$ is also constant
over appropriate nonempty open subset in $T$.
Moreover, so does its closure $\shNef_\R X/T=\shsAmp_\R X/T$.
The equality $=$ holds because $X/T$ has Ft.
The closure is covered by cones $\fP_\varphi$ with
contractions $\varphi\colon X\to Y/T$ given
by nef over $T$ divisors $D$.
Contractions $X_t\to Y_t$ are restricted from
those contractions of $X$ over $T$ after taking
an appropriate nonempty open subset in $T$.
Both $X/T$ and $X_t$ have finitely many contractions.
Moreover, they are bounded.
This allows to remove a proper closed subset in $T$
of contraction of $X_t$ which are not restricted.
The boundedness uses the $n$-complement structure.
If $\varphi$ is a fibration then
generic fibers $X_y=\varphi\1 y,y\in Y$, are bounded:
$(X_y,B_y-E_y)$ is a klt log Fano of bounded (local) lc index,
where $B_y=B\rest{X_y},E_y=E\rest{X_y}$.
If $\varphi$ is birational then its exceptional locus
is covered by bounded curves $C$ over $Y$ and bounded itself.
To find bounded curves $C$ we can take ones with
$-(C.K+B-E)\le 2\dim X$ \cite{MM}.
Since $(C.K+B)=0$
then $(C.H)\le 2\dim X/h$ is bounded.

Fibers and curves are treated as effective cycles.
Perhaps, we need an alteration of $T$ to have
an appropriate (constant) family of those cycles.

If $\Eff_\R X/T=\Nef_\R X/T$ we are done.
Otherwise we take a cone $\fP_\varphi$ of the maximal dimension
as $\Nef_\R X/T$ with the closure intersecting
$\Nef_\R X/T$.
In this case $X$ is birational to $Y$ under $\varphi$ and
$\fP_\varphi=\Amp_\R Y/T$.
(Actually, we can suppose that $\varphi$ is
an extremal divisorial contraction or small birational modification,
an extremal flop of $(X/T,B)$.)
So, we get a bounded family of Ft varieties $Y_t$.
We can repeat above arguments and extend
our constant sheaf $\shNef_\R X/T$ by
also a constant sheaf $\shNef_\R Y/T$.
Since we have finitely many models $\varphi$ and
closures of their cones $\fP_\varphi$ are
connected by the convex property of
$\shEff_\R X/T$, we get the constant property
of $\shEff_\R X/T$, possible taking
a nonempty open subset in $T$.
For the rest in $T$ we use Noetherian induction.

Since all cones $\fP_\varphi$ are rational
we proved also that $\shEff_\Q X/T$ is constant too.
Similarly we can prove that $\shMob_\R X/T,\shMob_\Q X/T$
are constant.
Sheaves $\shExc X/T=\Excd,\shExc^+X/T=\Excd^+$ and
their $\R,\Q$ versions
are also constant.
They are union of corresponding monoids of exceptional
locus for birational models $\varphi$.
For $=$ we need a marked divisors $D_{i,t}$ as
in Proposition~\ref{bounded_rank}.
Note that $\Excd_t,\Excd_t^+$ are defined as divisors
without $\sim$ in $\fD_t$.
The same marked divisors work for
the constant sheaf property with other properties
of $\Effd_\R,\Mobd_\R$ and
their $\Q$ versions.

By the way we proved that sheaves $\shNef_\R X/T=\shsAmp_\R X/T$
and their $\Q$ version are constant too.
The marked divisors as above can be used to prove
that $\Nefd_\R=\sAmpd_\R$ and their $\Q$ versions
are constant and satisfy required properties.

The other sheaves of integral divisors need
more sophisticated technique of the proof
of Theorem~\ref{eff_linear_rep_Cl}.
A Zariski decomposition is not useful here since
it is not integral usually.
On the other hand, the decomposition into
a fixed and a mobile parts is not stable (Zariski) under multiplication.
Thus we can't reduce the constant property for
$\shEff X/T$ to that of $\shFix X/T$ and of $\shMob X/T$.
Cf. also Step~4.

Step~3. $\shEff X/T$ is constant.
We consider $\shEff X/T$ as
a sheaf of submonoids in $\shCl X/T$.
We use the proof of Theorem~\ref{eff_linear_rep_Cl}
with minor [modifications] improvements (cf. Lemma~\ref{fg_monoid}).

By Proposition~\ref{bounded_rank}, $\shCl X/T$
is constant and there exists a natural constant
sheaf homomorphism
$$
c\colon \shCl X/T \to \shCl_\R X/T
$$
induced by $\otimes\R$.
We consider a constant rational simplicial cone
$Q\subseteq\shEff_\R X/T\subseteq\shCl_\R X/T$ as in Step~3 of
the proof of Theorem~\ref{eff_linear_rep_Cl}.
We can suppose also that $Q$ lies in
the closure of a cone $\fP_\varphi$ in $\shEff_\R X/T$.
We suppose also that $C_i\in Q$ are free on
a nef model for $C_i$ (cf. Addendum~\ref{b_free_C}).
We take also constant classes $W_j$ corresponding
to $W_j$ of Step~4 in the proof of Theorem~\ref{eff_linear_rep_Cl}.
Notice for this that $\shCl_{\R,T}X/T=\Cl_\R X/T$
by Proposition~\ref{bounded_rank}.
We state that restrictions $C_{i,t}=C_i\rest{X_t},
W_{j,t}=W_j\rest{X_t}$ generates $c\1 Q\cap(\Eff X_t)$,
possibly after taking a nonempty open subset in $T$.
This implies that $c\1 Q\cap(\shEff X/T)$ is constant and
$\shEff X/T$ is constant too.

For this we consider orbits $W+C$,
where $c(W)$ belongs to integral points of $Q$.
Since the kernel of $c$ is finite and
$Q$ is covered by finitely many orbits $c(W)+C$, the inverse image
$c\1 Q$ is covered by finitely many orbits $W+C$ too.
We verify that, for every $t$ in some nonempty open subset
in $T$, every orbit $W_t+C_t$ has
only restricted classes of $\Eff X/T$ in
the intersection with $\Eff X_t$, where
$W_t=W\rest{X_t},C_t=C\rest{X_t}$.
This implies required statement about generators.
If it is not true, for some $W$,
there exist {\em infinitely\/} many elements
$W_t+n_1C_{1,t}+\dots+n_rC_{r,t}\in\Eff X_t$
with nonnegative integers $n_i$ which are not
restricted from elements of $\Eff X/T$.
That is, there exist infinitely many nonnegative integral
vectors $(n_1,\dots,n_r)$ and
some $t\in T$ for every such vector that
$W_t+n_1C_{1,t}+\dots+n_rC_{r,t}\in\Eff X_t$
is not
restricted from elements of $\Eff X/T$.
Otherwise, we can remove finitely many closed proper subsets in $T$ corresponding
to vectors $(n_1,\dots,n_r)$ of nonrestricted elements.
The effective divisors with the linear equivalence class
$W_t+n_1C_{1,t}+\dots+n_rC_{r,t}$ are bounded.
(By definition $\Eff X_t$ consists of
the classes of effective Weil divisors modulo $\sim$.)
Replacing $W$ by $W+m_1C_1+\dots+m_rC_r$ with
a fixed nonnegative integral vector $(m_1,\dots,m_r)$ and interchanging
classes $C_i$ we can suppose that
nonnegative integral vectors $(n_1,\dots,n_r)$ has
the form $(n_1,\dots,n_q,0,\dots,0),q\le r$,
where $n_i,i=1,\dots,q$, are arbitrary large:
for every nonnegative integral vector $(m_1,\dots,m_q)$
there exists $(n_1,\dots,n_q)$ with every $n_i\ge m_i$.
This is impossible after removing
a proper closed subset in $T$.

Indeed, consider the contraction $X'\to Y/T$
corresponding to the nef model of $\varphi$.
By construction if the class of linear equivalence
$W_t+n_1C_{1,t}+\dots+n_qC_{q,t}\in\Eff X_t$
then it has an effective divisor $D_t$ on $X_t$.
So, a sufficiently general fiber $X_{t,y}',y\in Y$,
of $X'/Y$
has an effective divisor $D_{t,y}=D_t\rest{X_{t,y}'}$
in the class of linear equivalence $W_{t,y}=W_t\rest{X_{t,y}'}$.
In other words, the class $W_t$ has a representative modulo $\sim$
with the effective restriction $D_{t,y}$,
that is, the horizontal part of the representative is effective.
Those representative are bounded as cycles on $X'/T$.
If the corresponding closed subset in $T$ is proper then
we can remove this set with all classes
$W+n_1C_1\dots+n_qC_q$ and there are no
nonrestricted classes.
Thus we can assume that, for every $t\in T$,
$W_t$ has a representative modulo $\sim$ such that
its horizontal part over $Y$ is effective.
Adding $n_1C_{1,t}+\dots+n_qC_{q,t}$ with sufficiently large
$n_1,\dots,n_q$ we get
$W_t+n_1C_{1,t}+\dots+n_qC_{q,t}$ in $\Eff X_t$.
Indeed, $n_1C_{1,t}+\dots+ n_q C_{q,t}$ is
the pull back of a very ample divisor over $T$ from $Y$.
Thus $W_t+n_1C_{1,t}+\dots+n_qC_{q,t}$ is
the restriction of $W+c_1C_1\dots+c_qC_q\in\Eff X/T$,
a contradiction.

Theorem~\ref{eff_linear_rep_Cl} implies that
$\shEff X/T$ is finitely generated.
Thus we can find marked divisors
$D_i$ and construct $\Effd\subset\fD$ with
surjection $\Effd\twoheadrightarrow\shEff X/T$.
The sheaf $\Effd$ is also constant.

Step~4. $\shMob X/T$ is constant by the similar arguments
as for $\shEff X/T$ in Step~3.
This allows to construct the constant sheaf $\Mobd$
with required properties.

$\shFix X/T=\Fixd$ is constant too by Addendum~\ref{fix_exc}.

Finally, we can replace $\sim_T$ by $\sim$ for smaller $T$
in relations such as $\shEff X/T=\Effd/\sim$ because
$\sP$ is finitely generated.
However, for the sheaf $\shEff X/T$ it is better to use
$\sim_T$. The same holds for other relations and over $\R,\Q$.

\end{proof}

\begin{prop} \label{const_comp}
Let $X_t,t\in T$, be a wFt variety or an algebraic space in
a bounded family.
Then for appropriate parametrization $X/T$ sheaves
$\Compd,\Compd_\R, \Compd_\Q$
are constant.

\end{prop}

\begin{add} \label{const_bd_comp}
$\Compd\sP,\Compd_\R\sP, \Compd_\Q\sP$ are also constant
if $\sP$ is defined over $T$ (globally).

\end{add}

\begin{add} \label{all_together}
We can suppose that all our standard sheaves are constant simultaneously.
Moreover, they are bounded or, equivalently, finitely generated.

\end{add}

\begin{exa} [cf. Step~8 Addendum~\ref{bd_exceptional_compl}
in Proof of Theorem~\ref{excep_comp}]
Consider a family $(X_t,D_t'+\sP_t), t\in T$,
of exceptional wFt pairs, not assuming that $\sP$ is
defined over $T$, but assuming that the pairs $(X_t,D_t'+\sP_{X_t})$
form a family with $\sP_{X_t}=(\sP_t)_{X_t}$.
Then sheaves $\Compd_t\sP_t\cap\{D_t\in\fD_t\mid D_t\ge D_t'\},
\Compd_{\R,t}\sP_t\cap \{D_t\in\fD_{\R,t}\mid D_t\ge D_t'\},
\Compd_{\Q,t}\sP_t\cap\{D_t\in\fD_{\Q,t}\mid D_t\ge D_t'\}$
are well-defined and constant;
additionally $\Compd_{\R,t}\sP_t\cap \{D_t\in\fD_{\R,t}\mid D_t\ge D_t'\}$ is
closed convex rational polyhedral.
Indeed, $(X_t,D_t+\sP_t)$ is also exceptional if and only if
$-(K_{X_t}+D_t+\sP_{X_t})\in\Effd_{\R,t}$,
equivalently, has an $\R$-complement.
However, the inclusion assumption is
constant and polyhedral by Addendum~\ref{const_Mob} and \cite[Corollary~4.5]{ShCh}.
We assume here that $K_{X_t},\sP_{X_t}\in\fD_t$ are also constant.

Notice that $\Compd_t\sP_t,\Compd_{\R,t}\sP_t,\Compd_{\Q,t}\sP_t$
are possibly nonconstant and $\Compd_\R\sP$ is not well-defined if
$\sP$ is not defined over $T$.

\end{exa}

\begin{proof}[Proof of Proposition~\ref{const_comp}]
It is enough to be constant for $\Compd_\R$.
To establish this we can use the proof of
Theorem~\ref{R_compl_polyhedral} below.
In particular, we can suppose that
the $1$-contractions $\varphi_t\colon X_t\dashrightarrow Y_t$
and cones $\fP_{\varphi_t}$ are constant.
In this situation, Addendum~\ref{const_lc} implies
that $\Compd_\R$ is constant.

Respectively, Addendum~\ref{const_bn_lc} implies
Addendum~\ref{const_bd_comp}.

Addendum~\ref{all_together}
is immediate by proofs of Propositions~\ref{bounded_rank},
\ref{finite_lin_repres} and their Addenda.
First, we begin from the constant property of $\shCl X/T$ and
its $\R,\Q$ versions.
Second, we add its constant
subsheaves $\shEff X/T,\shMob X/T$ etc and their $\R,\Q$ versions.
Third, we pick up marked divisors $D_1,\dots,D_r$ in
classes of $\shCl X/T$ as generators of $\shEff X/D$.
This gives the constant sheaf $\fD$,
adds its constant subsheaves $\Effd,\Mobd$ etc and their $\R,\Q$ versions.
We can suppose that marked divisors are prime and
generate $\shCl X/T$.
Forth, we can include into generators prime exceptional and
prime components of fixed divisors.
Here we use the finiteness of Addendum~\ref{fix_exc}.
Finally, for projective $X_t$, we can add the constant
subsheaves $\shsAmp X/T,\shNef X/T,\sAmpd,\Nefd$ and
their $\R,\Q$ versions.
We add constant divisors $K,\sP_{X_t}$ etc,
finitely many other constant divisors or
their classes.

The same works for families of bd-pairs $(X_t,D_t+\sP_t)$, e.g.,
for sheaves as $\lcd_\R\sP,\Compd_\R \sP$.

\end{proof}

\begin{lemma} \label{effecive_lin_Qeff}
Let $\fP$ be a compact rational polyhedron in $\Effd_\R$.
Then there exists a positive integer $J$ such that,
for any rational divisor $D\in \fP$ with $qD\in\Z$,
where $q$ is a positive integer,
$JqD$ is effective modulo $\sim$.
\end{lemma}

\begin{add} \label{effecive_lin_Qeff_T}
If $\Effd$ is also constant then,
for a constant compact rational polyhedron
$\fP_t\subseteq\Effd_{\R,t}$,
there exists a positive integer $J$ such that
the lemma hold for same $J$ for every $D_t\in\fP_t$
(possibly after reparametrization).
\end{add}

\begin{proof}
Step~1. Reduction to the case when
$\fP=[D_1,\dots,D_r]$ is a simplex.
Decompose the polyhedron into finitely
many simplices $\fP_1,\dots,\fP_l$.
Take $J=J_1\dots J_r$ where $J_i$ is
a required integer for $\fP_i$.

Step~2. Since every vertex $D_i$ is
a rational in $\Effd_\R$, there exists
a positive integer $J_1$ such that
$J_iD_i\sim E_i$ and $E_i$ is an effective Weil divisor.
In particular, every $J_iD_i\in\Z$.

Step~3. We contend that $J=J_1^2\dots J_r^2$ is
a required integer.
Indeed, take $D\in[D_1,\dots,D_r]$ with $qD\in \Z$
for a positive integer $q$.
Then $D=w_1D_1\dots+w_rD_r$ with
$0\le w_1,\dots,w_r\in\Q,w_1+\dots+w_r=1$.
Then by Step~2
every $J_1\dots J_rqw_i$ is a nonnegative integer and
$$
JqD=Jqw_1D_1+\dots+Jqw_rD_r\sim
(Jqw_1/J_1)E_1+\dots+(Jqw_r/J_r)E_r\ge 0
\text{ and }\in\Z.
$$

Step~4. The addendum holds for constant vertices
$D_{1,t},\dots,D_{r,t}$.
Indeed, every $J_iD_{i,t}$
is constant too and $J_iD_{i,t}\sim E_{i,t}$ for
some $E_{i,t}\in\Effd_t$.
(We do not assume that $E_{i,t}$ is constant.
However, it is true for an appropriate parametrization
by Addendum~\ref{all_together} if $X/T$ has wFt and
the collection of marked divisors is sufficiently large.)

\end{proof}

\begin{thm} \label{R_compl_polyhedral}
Let $X/Z$ be a wFt morphism and
$D_1,\dots,D_r$ be a finite collection of distinct prime divisors on $X$.
Then the set
$$
\Compd_\R=
\{D\in\fD_\R\mid (X/Z,D) \text{\rm \ has an }
\R\text{\rm-complement}\}
$$
is a closed convex rational polyhedron in $\fD$.

\end{thm}

\begin{add} \label{comp_in_lc_eff}
If $K\in\fD$ then
$$
\Compd_\R\subseteq(-K-\Effd_\R).
$$
Moreover, if $X$ is $\Q$-factorial then
$$
\Compd_\R\subseteq\lcd_\R\cap(-K-\Effd_\R).
$$

\end{add}

In general, $=$ does not hold but it is true
for divisors $D$ such that $-(K+D)$ is nef over $Z$
(cf. the proof below).

\begin{add} \label{bd_R_compl_polyhedral}
For any b-divisor $\sP$ of $X$,
the same holds for the set
$$
\Compd_\R \sP=
\{D\in\fD_\R\mid (X/Z,D+\sP) \text{\rm\ has an }
\R\text{\rm-complement}\}.
$$
In Addendum~\ref{comp_in_lc_eff}, $K$ should be replace by $K+\sP_X$.

\end{add}

Actually, we need the last statement especially for bd-pairs $(X/Z,D+\sP)$.

\begin{proof}
Convexity of $\Compd_\R$ is by definition:
[if $(X/Z,D_1),(X/Z,D_2)$ have respectively
$\R$-complements $(X/Z,D_1^+),(X/Z,D_2^+)$ and
$a_1,a_2\in[0,1],a_1+a_2=1$, then
$(X/Z,a_1D_1^++a_2D_2^+)$ is an $\R$-complement of
$(X/Z,a_1D_1+a_2D_2)$].

It is enough to verify the statement
for a sufficiently large collection of
divisors $D_1,\dots,D_r$.
In particular, we can suppose that some $K\in\fD$.

Step~1. Reduction to the nef and lc properties.
The cone decomposition $\fP_\varphi$ of $\Eff_\R X/Z$
\cite[Theorem~3.4 and Corollary~5.3]{ShCh}
(see also Step~2 in the proof of Proposition~\ref{finite_lin_repres})
induces a convex rational polyhedral cone decomposition
of $\Effd_\R$.
We use the same notation for cones $\fP_\varphi$ of $\Effd_\R$ as for $\Eff_\R X/Z$.
It is enough to establish the polyhedral property of $\Compd_\R$
for its intersection
$$
\Compd_\R\cap (-K-\overline{\fP_\varphi})
$$
with the closure of every cone $\fP_\varphi$ of the decomposition.
Indeed, $\fP_\varphi$ corresponds to a rational $1$-contraction
$\varphi\colon X\dashrightarrow Y/Z$
(and its birational nef models).
By definition, for every $D\in \Compd_\R$,
$-(K+D)\in \Effd_\R$ holds, or equivalently,
$D\in -K-\Effd_\R$.
Moreover, a birational $1$-contraction $\psi\colon X\dashrightarrow X^\sharp/Z$,
on which $D$ is nef, is one of maximal models of
Construction~\ref{sharp_construction}.
By Addendum~\ref{R_complement_criterion}, if
$D\in(-K-\overline{\fP_\varphi})$ then,
for the birational transform $D_{X^\sharp}^\sharp$
of divisor $D$ on $X^\sharp$,
$(X/Z,D)$ has an $\R$-complement if and only
if $(X^\sharp/Z,D_{X^\sharp}^\sharp)$ is lc.
By construction, $-(K_{X^\sharp}+D_{X^\sharp}^\sharp)$ is nef over $Z$.
On the other hand, mapping of
$$
-K-\overline{\fP_\varphi}\to
-K_{X^\sharp}-\Nefd_\R X^\sharp/Z, D\to D_{X^\sharp}^\sharp,
$$
is affine over $\Q$ with a polyhedral image,
where $\fD_\R X^\sharp$ is generated by
the birational images $D_{i,X^\sharp}$ of
nonexceptional divisors $D_i$ on $X^\sharp$.
Thus $\Compd_\R\cap (-K-\Nefd_\R)$
is the preimage of
$\lcd_\R X^\sharp\cap(-K_{X^\sharp}-\Nefd_\R X^\sharp/Z)$ and
we need to verify that the latter is closed rational polyhedral.
Notice also that $K$ goes birationally to
$K_{X^\sharp}\in \fD X^\sharp$.

Step~2.
We can consider only the case $\overline{\fP_\varphi}=\Nefd_\R$
with $\varphi=\Id_X$.
In this situation
$$
\Compd_\R\cap (-K-\Nefd_\R)=\lcd_\R\cap(-K-\Nefd_\R).
$$
In other words, for nef $-(K+D)$ over $Z$,
$(X/Z,D)$ has an $\R$-complement if and only if
$(X,D)$ is lc (cf. Example~\ref{1stexe}, (1)).
(In this case $(X,D)$ is a log pair.)
But $\lcd_\R$ is a closed convex rational polyhedron
\cite[(1.3.2)]{Sh92}.
Thus the above intersection is
closed convex rational polyhedron too.

Step~3.
Addendum~\ref{comp_in_lc_eff} follows from the effective property
of $-(K+D)$ noticed in Step~1.
The rest of the addendum follows from the lc property
of $\R$-complements
[if $(X/Z,D)$ has an $\R$-complement and $(X,D)$ is a log pair
then $(X,D)$ is lc].

The pairs with a b-divisor $\sP$
can be treated similarly.

\end{proof}

\begin{proof}[Proof of Theorem~\ref{exist_n_compl}]
Immediate by \cite[Corollary~1.3]{BSh} and
the rational polyhedral property Theorem~\ref{R_compl_polyhedral}.
(Cf. Step~5 in the proof of Theorem~\ref{excep_comp} below.)
For simplicity, we can use here
a $\Q$-factorialization of $X$ and
Proposition~\ref{small_transform_compl}.
\end{proof}

\begin{lemma} \label{approximation}
Let $c$ be a real number.
Then every real number $b$ has
a sufficiently close approximation
by $\rddown{(n+c)b}/n$ for every sufficiently large (positive) real number $n$:
$$
\lim_{n\to+\infty} \rddown{(n+c)b}/n=b.
$$
\end{lemma}

\begin{proof}
Indeed,
$$
|(n+c)b-\rddown{(n+c)b}|<1.
$$
Hence
$$
\lim_{n\to+\infty} \rddown{(n+c)b}/n=
\lim_{n\to+\infty} (n+c)b/n=b.
$$

\end{proof}

\begin{proof}[Proof of Theorem~\ref{R-vs-n-complements}]
The easy part of the theorem is about the existence of
$\R$-complements when $n$-complements $(X/Z\ni o,B^+)$ exist.
This part is immediate by Theorem~\ref{R_compl_polyhedral}
and approximation of $b$ by $\rddown{(n+1)b}/n$
in Lemma~\ref{approximation} with $c=1,n\in\Z$.
Notice also that by Corollary~\ref{monotonic_n_compl}
$(X/Z\ni o,B^+)$ is an $\R$- and monotonic $n$-complement of
$(X/Z,\rdn{B}{n})$ (cf. $\rdn{x}{n}$ in~\ref{n+1_n_monotonicity}).

Conversely, the existence of
$n$-complements, when an $\R$-complement exists,
is immediate by Theorem~\ref{exist_n_compl}.

\end{proof}

\section{Exceptional complements:
explicit construction} \label{excep_n_comp}

\paragraph{Exceptional pairs.}
They can be defined in terms of $\R$-complements.
A pair $(X,D)$ is called {\em exceptional\/}
if it has an $\R$-complement and
every its $\R$-complement is klt.

The same definition works
[for pairs $(X,\D)$ with an arbitrary b-divisor $\D$, in particular,]
for bd-pairs $(X,D+\sP)$ (of index $m$).

Note that exceptional pairs can be defined
for local pairs $(X/Z\ni o,D)$ but,
under the klt assumption,
the only exceptional pairs are
global.
(For local exceptional pairs we should allow
{\em exceptional\/} nonklt singularities; see \cite[\S 7]{Sh92}.)

Exceptionality is related to boundedness.

\begin{cor} \label{bound_ex_pairs}
Let $d$ be a nonnegative integer and
$\Phi=\Phi(\fR)$ be a hyperstandard set associated with
a finite set of rational numbers $\fR$ in $[0,1]$.
The exceptional pairs $(X,B)$
with Ft $X$ of dimension $d$ and
with $B\in\Phi$ are bounded.

The boundedness of pairs includes
the boundedness of $B$
with multiplicities.
\end{cor}

\begin{add} \label{bound_ex_bd_pairs}
The same holds for bd-pairs $(X,B+\sP)$ of [fixed] index $m$.
In this case the boundedness includes additionally
the boundedness of $\sP_X$  modulo $\sim_m$.
\end{add}

Q. Is the b-divisor $\sP$ itself bounded modulo $\sim_m$
in higher dimensions?
It is true for surfaces because $C'^2\le C^2-1$ for
the birational transform $C'$ of a curve $C$ on a surface
under the monoidal transformation in a point of $C$
\cite[6.1, (8)]{IShf}.

\begin{proof}
Follows from two fundamental results of Birkar:
$n$-complements with hyperstandard multiplicities and
BBAB \cite[Theorem~1.8]{B} \cite[Theorem~1.1]{B16}.

Can be reduced to the crucial case: exceptional klt $(X,0)$  \cite[Theorem~1.3]{B}.
In this case, it is enough to establish the nonvanishing:
there exists a positive integer $n$, depending only on
the dimension $d=\dim X$ such that $h^0(X,-nK)\not= 0$.
After that we can use the boundedness of Hacon-Xu \cite[Theorem~1.3]{HX}.

The boundedness of boundary $B$ means that of
$\Supp B$ and the finiteness of multiplicities of $B$.
This holds because $B$ is effective and its multiplicities are
hyperstandard.
Indeed, $B=b_1D_1+\dots+b_rD_r$ and every $b_i\ge 0$.
Moreover, $b_i=0$ or $b_i\ge c$, where
$$
c=\min\{b\in\Phi\setminus 0\}
$$
is a positive rational number because
$\Phi$ is a dcc set.
Since $-K-B$ is effective modulo $\sim_\R$,
$\deg(-K-B)\ge 0$ and $\deg B\le-\deg K$,
where the degree $\deg$ is taken with respect to a bounded polarization of $X$.
Hence $\Supp B$ is bounded because $\deg K$
is bounded.
Since $(X,B)$ is exceptional, the set of possible $b_i$ is finite by Corollary~\ref{accumulation_1}.
Otherwise an $n$-complement $(X,B)$, with some $b_i$ close to $1$,
gives a nonklt $\R$-complement of $(X,B)$
\cite[Theorem~1.8]{B}, a contradiction.

Similarly for bd-pairs of index $m$ with bounded $X$,
$\sP_X$ is also bounded:
$$
\sP\sim_m E_1-E_2,
$$
where $E_1,E_2$ are two bounded effective $\Q$-divisors.
By construction $E_1,E_2\in\Z/m$.
The class of $\sP_X$ modulo $\sim_\R$
is a bounded point in the closed convex rational polyhedral cone
$\Eff_\R(X)$.
(Moreover, the class belongs to the mobile cone $\Mob_\R(X)$.)
Indeed, by definition $(X,B+\sP_X)$ has an $\R$-complement.
Thus  $-(K+B+\sP_X)$ is effective modulo $\sim_\R$ and
$\deg (B+\sP_X)\le -\deg K$.
Since $\sP$ is b-nef, $\sP_X$ is effective (even mobile) modulo
$\sim_\R$ \cite[Corollary~4.5]{ShCh}.
This implies that the divisorial part $B$ is bounded as above.
This gives also the boundedness of the class of $\sP_X$
with respect to the polarization.
The class of $m\sP_X$ is integral.
Thus $m\sP$ modulo $\sim_\R$ is $E_1'-E_2'$, where
$E_1',E_2'$ are two bounded Weil divisors on $X$.
Here we use Addenda~\ref{const_Cl}, \ref{const_Mob} and \ref{all_together}.
In particular, for bounded family of (w)Ft varieties,
sheaves $\shCl,\shEff,\shMob,\shTor$ are constant
and modulo $\shTor$ are subsheaves
of $\shCl_\R$.
The classes of $K,B,\sP_X,E_1',E_2'$ and of the polarization in $\Cl_\R X$ are restrictions on $X$
of corresponding sections of $\shCl_\R$.
Moreover, $\sim_\R$ can be replaced by $\sim$ and
torsions, that is, Weil divisors $D$ such with $nD\sim 0$.
The torsions are bounded on bounded wFt varieties by Addendum~\ref{const_Cl}
that gives required presentation with $E_1=E_1'/m,E_2=E_2'/m$
modulo torsions.

\end{proof}

\begin{lemma} \label{lifting_direction}
Let $\lambda\colon V\to W$ be a linear or just continuous map
of finite dimensional $\R$-linear spaces with norms and
$e$ be a direction in $V$ such that $\lambda(e)\not=0$.
Then, for every real number $\ep>0$, there exists
a real number $\delta>0$ such that if $e'$ is
another direction in $V$ with $\norm{e'-e}<\delta$ then
\begin{description}

\item[\rm (1)]
$\lambda(e')\not=0$ too;
and

\item[\rm (2)]
$$
\norm{\frac{\lambda(e')}{\norm{\lambda(e')}}-
\frac{\lambda(e)}{\norm{\lambda(e)}}}<\ep.
$$

\end{description}

\end{lemma}
In other words, close directions go to close ones.

\begin{proof}
The map $\lambda$ is continuous.
Hence $\lambda(e')\not=0$ for every $e'\in W$
with $\norm{e'-e}<\delta'$.

The map
$$
\frac{\lambda(e')}{\norm{\lambda(e')}}-
\frac{\lambda(e)}{\norm{\lambda(e)}}
$$
is also continuous for $e'\in W$
with $\norm{e'-e}<\delta'$.
Moreover,
$$
\lim_{e'\to e}\frac{\lambda(e')}{\norm{\lambda(e')}}-
\frac{\lambda(e)}{\norm{\lambda(e)}}=0.
$$
This gives required $0<\delta\le\delta'$.

\end{proof}

\begin{thm}[Exceptional $n$-complements] \label{excep_comp}
Let $d$ be a nonnegative integer,
$I,\ep,v,e$ be the data as
in Restrictions on complementary indices, and
$\Phi=\Phi(\fR)$ be a hyperstandard set associated with
a finite set of rational numbers $\fR$ in $[0,1]$.
Then there exists a finite set
$\sN=\sN(d,I,\ep,v,e,\Phi)$
of positive integers such that
\begin{description}

\item[\rm Restrictions:\/]
every $n\in\sN$ satisfies
Restrictions on complementary indices with the given data.

\item[\rm Existence of $n$-complement:\/]
if $(X,B)$ is a pair with wFt $X$,
$\dim X=d$, with a boundary $B$,
with an $\R$-complement and
with exceptional $(X,B_\Phi)$
then $(X,B)$ has an $n$-complement $(X,B^+)$ for some $n\in\sN$.

\end{description}

\end{thm}

\begin{add} \label{B_B_n_Phi_sharp_B_n_Phi}
$B^+\ge B_{n\_\Phi}{}^\sharp\ge B_{n\_\Phi}$.

\end{add}

\begin{add} \label{adden_Cartier_index_B+}
$nB^+$ is Cartier, if it is $\Q$-Cartier and $X$ has Ft.

\end{add}

Q. Does the same hold for wFt varieties
with bounded Ft models?

\begin{add} \label{standard_excep_compl}
$(X,B^+)$ is a monotonic $n$-complement of itself and
of $(X,B_{n\_\Phi}), (X,B_{n\_\Phi}{}^\sharp)$,
and is a monotonic b-$n$-complement of itself and
of $(X,B_{n\_\Phi}),(X,B_{n\_\Phi}{}^\sharp),(X^\sharp,B_{n\_\Phi}{}^\sharp{}_{X^\sharp})$,
if $(X,B_{n\_\Phi}),(X,B_{n\_\Phi}{}^\sharp)$ are log pairs respectively.

\end{add}

\begin{add} \label{bd_exceptional_compl}
The same holds for bd-pairs $(X,B+\sP)$ of index $m$
with $\sN=\sN(d,I,\ep,v,e,\Phi,m)$.
That is,
\begin{description}

\item[\rm Restrictions:\/]
every $n\in\sN$ satisfies
Restrictions on complementary indices with the given data and
$m|n$.

\item[\rm Existence of $n$-complement:\/]
if $(X,B+\sP)$ is a bd-pair of index $m$ with wFt $X$,
$\dim X=d$, with a boundary $B$,
with an $\R$-complement and
with exceptional $(X,B_\Phi+\sP)$
then $(X,B+\sP)$ has an $n$-complement $(X,B^++\sP)$ for some $n\in\sN$.

\end{description}
Addenda~\ref{B_B_n_Phi_sharp_B_n_Phi}-\ref{adden_Cartier_index_B+}
hold literally.
In Addenda~\ref{standard_excep_compl}
$(X,B^++\sP)$ is
a monotonic $n$-complement of itself and
of $(X,B_{n\_\Phi}+\sP), (X,B_{n\_\Phi}{}^\sharp+\sP)$,
and is a monotonic b-$n$-complement of itself and
of $(X,B_{n\_\Phi}+\sP),(X,B_{n\_\Phi}{}^\sharp+\sP),
(X^\sharp,B_{n\_\Phi}{}^\sharp{}_{X^\sharp}+\sP)$,
if $(X,B_{n\_\Phi}+\sP_X),(X,B_{n\_\Phi}{}^\sharp+\sP_X)$ are log bd-pairs
respectively.

\end{add}

\begin{proof}

Step~1. Construction of an appropriate family of pairs.
We construct a bounded family of pairs with marked divisors such that
every pair of the theorem will be in the family and
the theorem holds for every (typical) pair $(X,B)$ in the family
if and only if it holds in any other pair of
the connected component of the family with $(X,B)$.
The boundedness implies the finiteness of typical pairs
which can be chosen.
This reduces the theorem to the case where
$X$ is a fixed $\Q$-factorial Ft variety and $B$
is a boundary such that
\begin{description}

\item[\rm (1)]
$B\ge B'$, where $B'\in\Phi$ is also fixed, $(X,B')$ is exceptional; and

\item[\rm (2)]
$(X,B)$ has an $\R$-compliment.

\end{description}

Taking a $\Q$-factorialization and by
Lemma~\ref{wTt_vs_Ft} we can suppose that every $X$ has Ft
and is $\Q$-factorial.
Indeed, we use for this a small birational modification and
it preserves all required complements
by Proposition~\ref{small_transform_compl}.
However, this works only for
a projective factorialization for Addendum~\ref{adden_Cartier_index_B+}.
So, by Corollary~\ref{bound_ex_pairs}
the pairs $(X,B_\Phi)$ associated with pairs $(X,B)$ of the theorem
are bounded.
It is enough to consider one connected family of
pairs $(X_t,B_{\Phi,t}),t\in T$.
Hence we can suppose that $B_{\Phi,t}=B_t'\in\Phi$ is fixed
or constant in the following sense.
There exists a positive integer $r$ and
rational numbers $b_1',\dots,b_h'\in \Phi$ such that
every pair $(X_t,B_{\Phi,t})$ in the family has
{\em marked\/} distinct prime divisors
$D_{1,t},\dots,D_{h,t}$ and
$$
B_{\Phi,t}=B_t'=b_1'D_{1,t}+\dots+b_h'D_{h,t}\in\fD_{\R,t}.
$$
(1) holds by our assumptions for
every pair $(X,B')$ in our family.
We contend that
construction of $n$-complements
is [uniform][constant for an appropriate parametrization [family]
(see Section~\ref{constant_sheaves} and cf. Proposition~\ref{const_comp}), that is,
if an $n$-complement with required properties
exists for one pair $(X_t,B_t)$ then
the same holds with same $n$ for
every other pair of the family with the same boundary.
Possibly we need an appropriate
reparametrization of Addendum~\ref{all_together}.

By Proposition~\ref{bounded_rank}, Addendum~\ref{const_Cl} and
Proposition~\ref{finite_lin_repres} we can suppose that
the collection [ordered set]
of marked divisors $D_1,\dots,D_h$ is sufficiently large and
the parametrization is appropriate.
This means that, for every pair $(X_t,B_t')$ in our family,
\begin{description}

\item[]
$$
K_{X_t}\sim d_1D_{1,t}+\dots+d_hD_{h,t}
$$
for some $d_1,\dots,d_h\in \Z$ independent on $t$;
so, we can suppose that
a canonical divisor is constant:
\item[]
$$
K_{X_t}=d_1D_{1,t}+\dots+d_hD_{h,t}\in\fD_t[=\fD(D_{1,t},\dots,D_{h,t})];
$$
and

\item[]
every effective Weil divisor on $X_t$ is
linearly equivalent to an element of
a constant monoid of effective Weil divisors
$\fD_t^+=\fD^+(D_{1,t},\dots,D_{h,t})$,
the monoid of effective Weil divisors
generated by $D_1,\dots,D_h$.

\end{description}
Actually, we need slightly more: Proposition~\ref{const_comp} and
Addendum~\ref{const_Mob}.
We can assume this and that some other (standard) sheaves
are also constant by Addendum~\ref{all_together}.

The reduction to one (typical) pair
will be done in Step~3
in the special case when $B$ is supported
on $D_1+\dots+D_h$, that is, $B\in\fD^+$.
The general case will be treated in Step~7,
that is, an $n$-complement of $(X,B)$ under (1-2)
with $n\in \sN$ of Step~2 below can be constructed in terms
of an $n$-complement of the special case.
But before we reformulate the existence of
$n$-complements in terms of linear systems.
Note also the dependence of $\sN$ on $X$ instead of $d=\dim X$
in Step~2.
Since we have finitely many typical $(X,B')$ of dimension $d$,
the total finite set $\sN=\sN(d,I,\ep,v,e,\Phi)$ is
a finite union of $\sN(X,B',I,\ep,v,e,\Phi)$.

Step~2. (Reduction to linear complements.)
Let $(X,B')$ be pair as in Step~1, that is,
under assumption (1).
Suppose that there exists a finite set of positive integers
$\sN=\sN(X,B',I,\ep,v,e,\Phi)$ such that,
every $n\in \sN$ satisfies
Restrictions on complementary indices and,
for every boundary $B$ on $X$ under (1-2),
there exists $n\in\sN$ such that

\begin{description}

\item[\rm (3)]
$$
\rddown{(n+1)B}/n\ge B_{n\_\Phi};
$$
and

\item[\rm (4)]
$$
-nK-\rddown{(n+1)B}
$$
is linear equivalent to an effective Weil divisor $E$
(cf. $\overline{D}$ after \cite[Definition~5.1]{Sh92}).

\end{description}
Then $(X,B)$ has an $n$-complement.
The same holds for the bd-pairs $(X,B)$ of index $m|n$
with (possibly different) $\sN=\sN(X,B'+\sP_X,I,\ep,v,e,\Phi,m)$.

Indeed,
$(X,B)$ has an $n$-complement $(X,B^+)$ with
required properties, including
\begin{equation}
B^+\ge B_{n\_\Phi}.
\end{equation}
Take
$$
B^+=\rddown{(n+1)B}/n+E/n.
$$
We verify only that $(X,B^+)$ is an $n$-compliment of $(X,B)$.
The addenda will be explained later in Step~8.
By construction and (4)
$$
n(K+B^+)=nK+\rddown{(n+1)B}+E\sim 0.
$$
This implies~(3) of Definition~\ref{n_comp}.
By~(3) and Proposition~\ref{Phi_<Phi'},
since $B^+\ge B_{n\_\Phi}\ge B_\Phi\ge B'$
and $(X,B')$ is exceptional by (1),
$(X,B^+)$ is lc, moreover, klt.
This implies (2) and (1) of Definition~\ref{n_comp}
by Corollary~\ref{B+_B_n_compl}.

Step~3. Reduction to one pair in the special case.
Suppose that
$(X,B)=(X_t,B_t)$ for some $t\in T$.
The special case means that $B\in\fD^+$ and $B_t\in\fD_{\R,t}^+$.

The pair $(X_t,B_t')$ satisfies (1).
By construction every other pair $(X_s,B_s')$
also satisfies (1).
We need to verify the following.
Let $B_s\in\fD_{\R,s}^+$ be a boundary under assumptions (1-2),
then $(X_s,B_s)$ satisfies (3-4) for some $n\in\sN$ of Step~2.
In particular, by Step~2 $(X_s,B_s)$ has also a required $n$-complement.

By our assumptions and construction
$B_s=b_1D_{1,s}+\dots+b_hD_{h,s}$ satisfies (1-2).
Hence $B=B_t=b_1D_{1,t}+\dots+b_hD_{h,t}$ also
satisfies (1-2).
Indeed, we can also treat divisors $B,B',K$ and $D_i$ as
constant ones $B_T,B_T',K_{X/T},D_{i,T}$ in $\fD_{\R,T}$ as in
Addendum~\ref{const_Cl}.
Then (1) means that $B_s\ge B_s'$ and implies
$$
B_T\ge B_T'\text{ and }
B_t=b_1D_{1,t}+\dots+b_hD_{h,t}\ge
B_t'=b_1'D_{1,t}+\dots+b_h'D_{h,t}.
$$
However, (2) is more advanced and requires
the constant cone $\Compd_\R$,
the cone of divisors $D\in \fD_\R$ with
$\R$-complements (see Proposition~\ref{const_comp}).
Since it is constant, (2) for $B_s$ implies
that of for $B_t$.
Actually, in the exceptional case under (1) it is
enough the effective property:
$\Compd_{\R,t}\cap\{B_t\ge B_t'\}=
(-K_{X_t}-\Effd_{\R,t})\cap\{B_t\ge B_t'\}$ means (2) under (1).

Now according to our assumptions, for $(X_t,B_t)$,
there exists $n\in\sN$ such that $B_t$ satisfies (3-4).
Hence $B_s$ also satisfies (3-4).
Indeed, since $B_t,B_{n\_\Phi,t}\in\fD_{\R,t}$, (3) is independent on $t$.
Similarly, (4) independent on $t$ by Addendum~\ref{const_Mob}
because $K_{X_t},\rddown{(n+1)B_t}\in\fD_t$ and
$-nK_{X_t}-\rddown{(n+1)B_t}\in\Effd_t$.
Actually, we can choose $E_t\in\fD_t^+$ and $B_t^+\in\fD_{\Q,t}^+$.

Note that the constructed complement under (3-4)
in Step~2 satisfies $B_s^+\ge B_{s,n\_\Phi}$.

Step~4. Preparatory for the special case.
Below we consider fixed $(X,B')=(X_t,B_t')$.

By Lemma~\ref{effecive_lin_Qeff}
there exists a positive integer $J$ such that,
\begin{description}

\item[\rm (5)]
for every positive integer $q$,
if $B_q\in\fD_\R$ under (1-2) and such that $qB_q\in \Z$,
then $$
-JqK-JqB_q
$$
is linear equivalent to an effective Weil divisor $E$.

\end{description}
Actually, $B_q$ is a boundary because $(X,B')$ is exceptional.
By Step~1, we can suppose that $K\in\fD$, integral.
So, (1-2) mean that
$$
B_q\in\Compd_\R\cap
\{B_q\ge B'\},
$$
a compact rational polyhedron in $\Compd_\R$ \cite[Corollary~4.5]{ShCh}.
Recall that
$\Compd_\R\cap\{B\ge B'\}=
(-K-\Effd_\R)\cap\{B\ge B'\}$.
We apply Lemma~\ref{effecive_lin_Qeff} to
$\Effd_\R\cap\{D=-K-B_q\le-K-B'\}$.
Since by Addendum~\ref{const_Mob} the polyhedral is constant over $T$,
there exists $J$ independent on $t\in T$ by
Addendum~\ref{effecive_lin_Qeff_T}.

Consider divisors $B\in \Compd_\R\cap\{B\ge B'\}$.
The multiplicities
$$
b_{n\_\Phi}=\mult_P B_{n\_\Phi},\ \
$$
where $P$ is a prime divisor on $X$ and
$n$ is a positive integer,
are not accumulated to $1$ for all $i$.
Indeed, since $(X,B')$ is exceptional,
$B$ is a (klt) boundary.
On the other hand, by definition
$$
b_{n\_\Phi}=1-\frac r l+\frac m {l(n+1)},
$$
where $r\in\fR$, $m$ is a nonnegative integer and
$l,n$ are positive integers.
We can always find a boundary $B''\in\Phi$ between $B_{n\_\Phi}$ and $B'$:
$B\ge B_{n\_\Phi}\ge B''\ge B'$.
For every prime divisor $P$ on $X$,
$$
b''=\begin{cases}
1-\frac r {l'}, \text{ if } &b_{n\_\Phi}\ge 1-\frac 1 {l'}\ge b'=\mult_P B'
\text{ for some positive integer } l',\\
b' &\text{ otherwise},
\end{cases}
$$
where $b''=\mult_Pb''$.
By the klt property of $B$, $r>0$.
By construction $(X,B'')$ is exceptional.
Hence $l'$ and $l$ are bounded by Corollary~\ref{bound_ex_pairs}
and $b_{n\_\Phi}$ are not accumulated to $1$.
So, the set $\fR'$ of rational numbers
$$
r'=1-\frac r l
$$
for all multiplicities $b_{n\_\Phi}$
is finite and independent of $n$.
However, $n$ and $m$ can be large.

Step~5. (Linear complements: the special case.)
To find a required set $\sN=\sN(X,B',I,\ep,v,e,\Phi)$ of Step~2
we use Diophantine approximations in the $\R$-linear space
$\fD_\R=\fD_\R(D_1,\dots,D_h)$.
The special case means that we consider only boundaries $B\in\fD_\R$.
For those boundaries $B_\Phi,B_{n\_\Phi}\in\fD_\R$ too.
For every such $B$ under (1-2), we find $n\in\sN$ such that
(3-4) holds.

Fix for this a sufficiently divisible positive integer $N$:
$I|N,N\fR'\subset \Z$ and $J|N$, where $J$ of Step~4.
We also suppose that $N\ge 2$.
Take a positive real number $\delta$ such that
$\delta<1/4N$,
$\delta<\ep/N$, the $\delta$-neighborhood of $B$ in
$\spn{B}$ lies in $\Compd_\R\cap\{B\ge B'\}$ and
$\delta$ satisfies Lemma~\ref{lifting_direction} for the projection
$\pr\colon\fD_\R\times \R^l\to \R^l$ and a direction $e'$ in $\spn{(B,v)}$
which goes to
$$
e=\frac{\pr(e')}{\norm{\pr(e')}}.
$$
We assume also that
$$
\delta\le\frac{\min\{b_i,1-b_j\mid i,j=1,\dots,l,
\text{ and } b_i>0\}}{3N}.
$$
[Warning: the minimum for all $1-b_j$ but not for all $b_i$.]
The minimum is positive because every $b_j<1$ by
the klt property of $\R$-complements.

Such $\delta$ exists because $\Compd_\R\cap\{B\ge B'\}$
is a compact rational polyhedron and the minimum is positive.
Note also that $e'$ exists because the restriction of projection
$$
\spn{(B,v)}\to\spn{v}
$$
is surjective.
Then by \cite[Corollary~1.3]{BSh} there exists a positive integer $q$ and
a pair of vectors $B_q\in \spn{B},v_q\in \spn{v}$ such that
\begin{description}

\item[\rm (6)]
$qB_q\in\Z,qv_q\in\Z^l$;

\item[\rm (7)]
$\norm{B_q-B},\norm{v_q-v}<\delta/q$;
and

\item[\rm (8)]
$$
\norm{\frac{(B_q,v_q)-(B,v)}{\norm{(B_q,v_q)-(B,v)}}-e'}<\delta.
$$

\end{description}
The approximation $(B_q,v_q)$ applies to the vector $(B,v)$ in
$\spn{(B,v)}\subseteq\fD_\R\times \R^l$ with the direction $e'$ in $\spn{(B,v)}$.

Put $n=Nq$ and $v_n=v_q$.
Then $v_n$ satisfies required properties of
Restrictions on complementary indices:
Divisibility by the divisibility of $N$,
Denominators by (6), Approximation by (7) and Anisotropic approximation
by (8) and Lemma~\ref{lifting_direction} respectively.
(The lemma is applied to the direction $e'$ and another direction
$$
\frac{(B_q,v_q)-(B,v)}{\norm{(B_q,v_q)-(B,v)}}
$$
in $\fD_\R\times\R^l$.)
To avoid different $v_n$ for the same $n$,
it would be better to suppose that $\ep<1/2$.
Instead, we assumed above that $\delta<1/4N\le 1/4$.
Then $v_n$ independent of $B$ and $(X,B')$ too.

Now we verify that, for every boundary $B''\in\fD_\R$,
in the $\delta/q$-neighbourhood of $B$ in
$\Compd_\R\cap\{B\ge B'\}$,
$n$ satisfies the properties~(3-4).
In other words, the properties holds for $B''\in\fD_\R$
if $\norm{B''-B}<\delta/q$,
$B''\ge B'$ and $(X,B'')$ has an $\R$-complement.
Since $\Compd_\R\cap\{B\ge B'\}$ is compact
it has a finite cover
by those neighbourhoods and $\sN$ is
the finite set of numbers $n$ for the neighbourhoods of such a covering.

However  we start from the property
\begin{description}

\item[\rm (9)]
$$
\rddown{(n+1)B''}/n=B_q.
$$

\end{description}
Equivalently, for every $i=1,\dots,h$,
$$
\rddown{(n+1)b_i''}/n=b_{i,q},\
b_i''=\mult_{D_i}B''\text { and }
b_{i,q}=\mult_{D_i}B_q.
$$
By construction and (6), $b_{i,q}\in \Z/q\subseteq \Z/n$.
By construction and (7)
\begin{equation}\label{2delta_ineq}
\norm{b_i''-b_{i,q}}\le
\norm{b_i''-b_i}+ \norm{b_i-b_{i,q}}<
\frac{2\delta}q, b_i=\mult_{D_i}B.
\end{equation}

Case~$b_i=0$.
By construction and~(7),
$$
\norm{b_{i,q}}=\norm{b_{i,q}-b_i}
<\frac \delta q=\frac{N\delta}n<
\frac 1{4n}<\frac 1q.
$$
Hence by (6) $b_{i,q}=0$.
By construction $b_i''\ge 0$.
The inequality~(\ref{2delta_ineq}) implies
$$
\norm{b_i''}=\norm{b_i''-b_{i,q}}
<\frac{2\delta}q<\frac {2N\delta}n
<\frac 1{2n}\le\frac 1{n+1}.
$$
Thus
$$
\rddown{(n+1)b_i''}/n=0=b_{i,q}.
$$

Case~$b_i>0$.
By Lemma~\ref{inequal} it is enough to
verify the inequality
$$
\norm{b_i''-b_{i,q}}<\frac{\min\{b_i'',1-b_i''\}}n.
$$

By construction, our assumptions and
the positivity of $b_i$
$$
b_i''>b_i-\frac \delta q
\ge\min\{b_i,1-b_j\mid i,j=1,\dots,l,
\text{ and } b_i>0\}-\frac \delta q
\ge 3N\delta -\frac \delta q\ge
2N\delta.
$$
Similarly,
$$
1-b_i''>1-b_i-\frac \delta q\ge
3N\delta -\frac \delta q\ge
2N\delta.
$$
Recall, that $1-b_i>0$ because $(X,B)$
is klt.
Both inequalities gives
$$
\min\{b_i'',1-b_i''\}>2N\delta.
$$
Hence by the inequality~(\ref{2delta_ineq})
$$
\norm{b_i''-b_{i,q}}<\frac{2\delta}q=
\frac{2N\delta}n<
\frac{\min\{b_i'',1-b_i''\}}n.
$$

Property~(3): By~(9), (3) means
$$
B_q\ge B_{n\_\Phi}''
$$
or, equivalently,
for every $i=1,\dots,l$,
$$
b_{i,q}\ge b_{i,n\_\Phi}''.
$$
By definition
$$
b_i''\ge b_{i,n\_\Phi}''.
$$
So, it is enough to disprove the case
$$
b_i''\ge b_{i,n\_\Phi}''>b_{i,q}.
$$
By the inequality~(\ref{2delta_ineq}) this
implies the inequalities
$$
0<b_{i,n\_\Phi}''-b_{i,q}<\frac{2\delta}q.
$$
The low estimation~(\ref{low_estimation}) of the difference below gives
a contradiction.
For this we use the following
form of numbers under consideration:
$$
b_{i,n\_\Phi}''=r+\frac {m_i}{l_i(n+1)}
\text{ and }
b_{i,q}=r+\frac mn,
$$
where $r=1-r_i/l_i\in\fR'$
[(not necessarily positive for general $\fR$ \cite[3.2]{PSh08})],
$l_i$ is a positive integer,
$m_i$ is a nonnegative integer and $m$ is an integer.
Such an integer $m$ exists because by construction
$r\in\fR'\subset \Z/n$
and $b_{i,q}\in\Z/n$ (see the choice of $N$ and~(9) above).
The inequality $b_{i,n\_\Phi}''>b_{i,q}$ implies that
$$
\frac {m_i}{l_i(n+1)}>\frac mn.
$$
Hence $l_im<m_i$.
Since $l_im$ is an integer,
$$
l_im\le m_i-1 \text{ and }
b_{i,q}=r+\frac mn=
r+\frac {l_im}{l_i n}\le
r+\frac {m_i-1}{l_i n}.
$$
Note also that the assumption $\fR\subset [0,1]$
implies that $r_i\le 1$ and
$$
r=1-\frac {r_i}{l_i}\ge 1-\frac 1{l_i}.
$$
On the other hand, $b_{i,n\_\Phi}''\le b_i''<1$.
So,
$$
\frac {m_i}{l_i (n+1)}=b_{i,n\_\Phi}''-r<1-1+\frac 1{l_i}=
\frac 1{l_i},\ \
\frac{m_i}{n+1}<1
$$
and $m_i\le n$ because $m_i$ is integer.
Hence
$$
\frac{m_i-1}n< \frac{m_i}{n+1}
\text{ and }
b_{i,q}\le r+\frac {m_i-1}{l_i n}<
r+\frac {m_i}{l_i(n+1)}=b_{i,n\_\Phi}''.
$$
Thus by the inequality~(\ref{2delta_ineq})
$$
\frac {m_i}{l_i(n+1)}-\frac {m_i-1}{l_i n}=
r+\frac {m_i}{l_i(n+1)}-r-\frac {m_i-1}{l_i n}<
\frac{2\delta}q.
$$
That is,
$$
\frac 1n(\frac 1{l_i}-\frac{m_i}{l_i(n+1)})=
\frac {m_i}{l_i(n+1)}-\frac {m_i-1}{l_in}<
\frac{2\delta}q=\frac{2N\delta}n.
$$
Hence
\begin{equation} \label{upper_estimation}
\frac 1{l_i}-\frac{m_i}{l_i(n+1)}<
2N\delta.
\end{equation}

On the other hand, by construction and the inequality~(\ref{2delta_ineq})
$$
1-\frac{r_i}{l_i}+\frac {m_i}{l_i(n+1)}=
r+\frac {m_i}{l_i(n+1)}=
b_{i,n\_\Phi}''\le b_i''<b_i+\frac{2\delta}q.
$$
Since $r_i\le1$ and $1-b_i>0$, $1-b_i\ge 3N\delta$,
the last inequality implies
$$
\frac 1{l_i}-\frac{m_i}{l_i(n+1)}=
1-\frac{r_i}{l_i}+\frac 1{l_i}
-1+\frac{r_i}{l_i}-\frac{m_i}{l_i(n+1)}=
1-\frac{r_i}{l_i}+\frac 1{l_i}-b_{i,n\_\Phi}''
$$
\begin{equation} \label{low_estimation}
>1-\frac{r_i}{l_i}+\frac 1{l_i}-b_i-\frac{2\delta}q\ge
1-b_i-\frac{2\delta}q>3N\delta -\frac{2\delta}q.
\end{equation}
This gives by the inequality~(\ref{upper_estimation})
$$
3N\delta -\frac{2\delta}q<2N\delta.
$$
Equivalently,
$$
Nq<2
$$
a contradiction because $N\ge 2,q\ge 1$.

Property~(4) for $B''$ is immediate by (5).
Indeed, $B_q$ satisfies (1-2) by (7) and our assumptions,
$qB_q\in\Z$ by (6), and by (9) and assumptions
$$
-nK-\rddown{(n+1)B''}=
-nK -nB_q=
\frac N J (-JqK-JqB_q)
$$
is linear equivalent to an effective Weil divisor.

Finally, in this step it is not necessarily to
suppose that $X$ is irreducible.
It is enough a bound on the number of irreducible
components of $X$.
In particular, we can take finitely many $B$ in
$\Compd_\R\cap\{B\ge B'\}$ if the number of boundaries
is bounded.
Indeed, constructions of $\sN$ and of $n$-complements
are combinatorial: we work with the lattice $\fD$ and
approximations.

Step~6. Preparatory for the general case.
There exists a homomorphism
$$
\mu\colon \WDiv X\to \fD
$$
of the free Abelian group $\WDiv X$ of
Weil divisors to the lattice $\fD$.
For every prime Weil divisor $P$ on $X$,
put $\mu(P)=E$, where $E\in \fD$ such that
$P\sim E $.
Since $P$ is effective we can suppose that
$E\in \fD^+$, that is, also effective.
Moreover, for the generators $D_i$,
we take $\mu(D_i)=D_i$,
that, is $\mu$ is a projection on $\fD$.
That is, for every divisor $D$ in $\fD$,
$\mu(D)=D$.
In particular, $\mu(K)=K$.
By definition and construction, for every $D\in \WDiv X$,
$\mu(D)\sim D$.

Since $\mu$ is a homomorphisms of groups, for every two divisors $D,D'$ on $X$ and
every integer $n$,
$\mu(D+D')=\mu(D)+\mu(D')$ and $\mu(nD)=n\mu(D)$.

Below
we always assume that, for every prime Weil divisor
$P$ on $X$, $\mu(P)$ is effective.
So, we obtain a homomorphism
$$
\mu\colon \EffWDiv X\to \fD^+
$$
of the free Abelian monoids
[$\EffWDiv X$ effective Weil divisors on $X$].

For an $\R$-divisor $D=\sum d_iQ_i$, where $Q_i$ are
distinct prime Weil divisors,
put
$$
\mu(D)=\sum d_i \mu(Q_i).
$$
This gives a natural $\R$-linear extension
$$
\mu\colon \WDiv_\R X\to \fD_\R=\fD\otimes\R
$$
of the above homomorphism of free groups.
In particular,
for every $r\in\R$ and every $D\in\WDiv_\R X$,
$\mu(rD)=r\mu(D)$ holds.
For every $\R$-divisor $D$ in $\fD_\R$,
$\mu(D)=D$.
In particular, $\mu(B')=B'$.
Note also that, in the important for us situation
of Step~7,
$B_\Phi=B'$ and $\mu(B_\Phi)=B_\Phi$ but
$\mu(B)_\Phi\not=\mu(B_\Phi)$ is possible if $B$
is not supported in $D_1+\dots+D_h$.

Taking nonnegative real numbers $d_i$
we obtain a homomorphism of monoids
$$
\mu\colon \EffWDiv_\R X\to \fD_\R^+,
$$
[where $\EffWDiv_\R X$ is the cone of effective
$\R$-divisors and $\fD_\R^+$ is the cone of
effective $\R$-divisors supported on $D_1+\dots+D_h$].
It is an $\R^+$-homomorphism, that is,
for every $r\in\R^+=[0,+\infty)$ and every $D\in\EffWDiv_\R X$,
$rD\in\EffWDiv_\R X$ and $\mu(rD)=r\mu(D)$ holds.
Note also the following monotonicity
\begin{description}

\item[\rm (10)\/]
if $D,D'\in \WDiv_\R X$ and $D\ge D'$,
then $\mu(D)\ge \mu(D')$.

\end{description}

However, $\mu$ is not unique and not canonical.

We use in Step~7 the following monotonicity:
for any $\R$-divisor $D$
\begin{equation}\label{mu_n+1_monotonicity}
\rddown{(n+1)\mu(D)}\ge
\mu(\rddown{(n+1)D}).
\end{equation}
Both sides are supported in $D_1+\dots+D_h$.
So, it is enough to verify the inequality
for the multiplicities in every $D_i$.
Let $D=\sum d_j Q_j$,
where $Q_j$ are distinct prime Weil divisors and
$$
\mu(Q_j)=\sum_{i=1}^h n_{i,j}D_i,
$$
where $n_{i,j}$ are nonnegative integers.
Then
$$
\mu(D)=\sum d_j \mu(Q_j)=
\sum_{i=1}^h (\sum n_{i,j}d_j)D_i
\text{ and }
\mult_{D_i}\mu(D)=\sum n_{i,j}d_j.
$$
Respectively,
$$
\mult_{Q_j}\rddown{(n+1)D}=
\rddown{(n+1)\mult_{Q_j}D}=
\rddown{(n+1)d_j},
$$
$$
\mu(\rddown{(n+1)D})=\sum_{i=1}^h(\sum n_{i,j}\rddown{(n+1)d_j})D_i
$$
and
$$
\mult_{D_i}
\mu(\rddown{(n+1)D}=\sum n_{i,j}\rddown{(n+1)d_j}.
$$
Hence by inequalities~(\ref{induc_low_bound}-\ref{1st_cor_low_bound}),
for nonnegative integers $n_i$ and real $d_i$,
$$
\rddown{(n+1)(\sum n_id_i)}\ge
\sum n_i\rddown{(n+1)d_i}
$$
and
$$
\mult_{D_i}\rddown{(n+1)\mu(D)}=
\rddown{(n+1)\mult_{D_i}\mu(D)}=
\rddown{(n+1)(\sum n_{i,j}d_j)}\ge
$$
$$
\sum n_{i,j}\rddown{(n+1)d_j}=
\mult_{D_i}\mu(\rddown{(n+1)D}).
$$

By definition and construction, for every $D\in \WDiv_\R X$,
$\mu(D)\sim_\R D$.
In particular, if $(X,D^+)$ is an $\R$-complement
of $(X,D)$ then
$$
\mu(D^+)\sim_\R D^+\sim_\R-K.
$$
Moreover, $\mu(D),\mu(D^+)\in\fD_\R$.
By the monotonicity~(10), $\mu(D^+)\ge \mu(D)$.
Hence $(X,\mu(D^+))$ is an $\R$-complement of $(X,\mu(D))$ if
$(X,\mu(D^+))$ is lc.
Indeed,
$$
K+\mu(D^+)= \mu(K)+\mu(D^+)=\mu(K+D^+)\sim_\R K+D^+
\sim_\R 0.
$$
If additionally $D=B$ is a boundary then $B^+=D^+$ is a boundary and,
under the lc property of $(X,\mu(B)),(X,\mu(B^+))$,
$\mu(B),\mu(B^+)$ are boundaries too.

In conclusion to this step, if
$(X,B)$ has an $\R$-complement $(X,B^+)$ and $B\ge B'$,
then $(X,\mu(B^+))$ is lc and
is an $\R$-complement of $(X,\mu(B))$.
Indeed, $B^+\ge B'$ and
$\mu(B^+)\ge \mu(B')=B'$
because $B'\in\fD_\R$.
By the exceptional property of $(X,B')$, $(X,\mu(B^+))$ is klt.

Step~7. (Linear complements: the general case.)
We consider $(X,B)$ under (1-2). Moreover,
we can suppose in this case that $B_\Phi=B'$.
Put
$$
B''=\sum_{i=1}^h b_iD_i,
\text{ every }b_i=\mult_{D_i}B.
$$
Note that $D_1,\dots,D_h$ include all
prime divisors $P$ on $X$ with
$$
b_{P,\Phi}=\mult_P (B_\Phi)=(\mult_P b)_\Phi>0.
$$
Indeed, $B_\Phi\in\fD_\R$.

By Step~6 and construction both divisors $\mu(B),B''\in\fD_\R$
and satisfy (1-2).
Thus by Step~5 we have $n\in \sN$ which satisfies
required properties (3-4) simultaneously for $\mu(B)$ and $B''$.
We verify now that (3-4) holds also for $B$ and the same $n$.

Property (3): (Here we need $n$ only for $B''$.)
By definition $b_{P,n\_\Phi}$ is the largest number
$$
1-\frac rl+\frac m{l(n+1)}\le b_P=\mult_P B,
$$
where $P$ is a prime Weil divisor of $X$,
$r\in\fR$, $l$ is a positive integer and
$m$ is a nonnegative integer.

Case $1-r/l>0$.
Then, for some $i=1,\dots,l$, $P=D_i$ and
$b_P=b_i=\mult_{D_i}B''$.
Hence (3) holds for this prime Weil divisor $P$.

Case $1-r/l=0$.
Then $l=1,r=1$ and $b_{P,n\_\Phi}=m/(n+1)$.
By definition
$$
b_P\ge \frac m{n+1}.
$$
So,
$$
\rddown{(n+1)b_P}/n\ge
\rddown{(n+1)\frac m{n+1}}/n=\frac mn
>\frac m{n+1}=b_{P,n\_\Phi}.
$$

Property~(4): (Here we need $n$ only for $\mu(B)$.)
Let $E$ be an effective Weil divisor such that
$$
E\sim -nK-\rddown{(n+1)\mu(B)}.
$$
On the other hand, by~(\ref{mu_n+1_monotonicity})
$$
E'=\rddown{(n+1)\mu(B)}-
\mu(\rddown{(n+1)B})
$$
is an effective Weil divisor.
Thus by Step~6
$$
-nK-\rddown{(n+1)B}\sim
-nK-\mu(\rddown{(n+1)B})=
$$
$$
-nK-\rddown{(n+1)\mu(B)}+
\rddown{(n+1)\mu(B)}-
\mu(\rddown{(n+1)B})=
$$
$$
-nK-\rddown{(n+1)\mu(B)}+E'
\sim E+E'.
$$
By construction $E+E'$ is also effective.

Step~8. (Addenda.)
By Step~2, $B^+\ge B_{n\_\Phi}$.
So, Proposition~\ref{monotonicity_I} implies Addendum~\ref{B_B_n_Phi_sharp_B_n_Phi}.

Addendum~\ref{adden_Cartier_index_B+} follows
from Definition~\ref{n_comp}, (3) and
from the boundedness of {\em canonical index\/},
that is, the index of canonical divisors $K_{X_t}$
for a bounded family $X/T$ of Ft varieties.
Note that $X$ is $\Q$-Gorenstein if
$nB^+$ is $\Q$-Cartier.
The index of $K$ on $X$ gives the bound on
the index of all $K_{X_t},t\in T$.
If $N$ of Step~5 will be divisible by the last index
then $n$ will be sufficiently divisible and satisfies
the addendum.

Addendum~\ref{standard_excep_compl} is standard.
For $n$-complements, it follows from Addendum~\ref{B_B_n_Phi_sharp_B_n_Phi}
and from Definition~\ref{n_comp} for $(X,B_{n\_\Phi}),(X,B_{n\_\Phi}{}^\sharp)$.
For b-$n$-complements,
it follows from Example~\ref{b_n_comp_of_itself}, (1) and
from Proposition~\ref{D_D'_complement}
for $(X,B_{n\_\Phi}),(X,B_{n\_\Phi}{}^\sharp)$,
if respectively they are log pairs.
Note also that $(X/Z\ni o,B^+)$ is a log pair by definition.

For $(X^\sharp,B_{n\_\Phi}{}^\sharp{}_{X^\sharp})$,
we use Addendum~\ref{B_B_n_Phi_sharp_B_n_Phi} and
Proposition~\ref{monotonicity_I}.
The pair $(X^\sharp,B_{n\_\Phi}{}^\sharp{}_{X^\sharp})$ is a log one
by definition.
The statement is meaningful because the $n$-complement
$(X,B^+)$ induces an $n$-complement for a special
maximal model $(X^\sharp,B_{n\_\Phi}{}^\sharp{}_{X^\sharp})$ with
a small transformation $X\dashrightarrow X^\sharp$ (cf. Proposition~\ref{small_transform_compl})
and then for any other its crepant model by Remark~\ref{rem_def_b_n_compl}, (1).

In this step b-$n$-complements of $(X^\sharp,B_{n\_\Phi}{}^\sharp{}_{X^\sharp})$  is
our final destination.
They give all other complements in
Addendum~\ref{standard_excep_compl} and
an $n$-complements of $(X,B)$ as well.
However, in many other situation we first construct
a b-$n$-complement for $(X^\sharp,B_{n\_\Phi}{}^\sharp{}_{X^\sharp})$.
The former is actually induced from this root (exceptional) complement.

Addendum~\ref{bd_exceptional_compl} about bd-pairs $(X,B+\sP)$ can
be treated as the case of usual pairs with minor improvement
as follows.
First, (1-2) should be replaced by
\begin{description}

\item[\rm (1-bd)\/]
$B\ge B'$, where $B'\in\Phi$ is also fixed and
$(X,B'+\sP)$ is exceptional with fixed $\sP_X$; and

\item[\rm (2-bd)\/]
$(X,B+\sP)$ has an $\R$-compliment.

\end{description}
Note for this that by Corollary~\ref{bound_ex_pairs} and
Addendum~\ref{bound_ex_bd_pairs} we
can suppose that $\sP_X$ is fixed too and
corresponds to a constant divisor $\sP_{X,t},t\in T$.
(We do not suppose the existence of entire
constant b-divisor $\sP_t$.)

We can construct $n$-complements using also
linear complements with same (3-bd)=(3) and modified
\begin{description}

\item[\rm (4-bd)\/] $m|n$ and
$$
-nK-n\sP_X-\rddown{(n+1)B}
$$
is linear equivalent to an effective Weil divisor $E$.

\end{description}

In this situation we modify the effective property:
$\Compd_{\R,t}\sP_t\cap\{B_t\ge B_t'\}=
(-K_{X_t}-\sP_{X,t}-\Effd_{\R,t})\cap\{B_t\ge B_t'\}$
means (2-bd) under (1-bd).
This is a constant compact rational polyhedron and we
can use Lemma~\ref{effecive_lin_Qeff} again.
There exists a positive integer $J$ such that,
$m|J$ and
\begin{description}

\item[\rm (5-bd)\/]
for every positive integer $q$,
if $B_q\in\fD_\R$ under (1-2-bd) and such that $qB_q\in \Z$,
then $$
-JqK-Jq\sP_X-JqB_q
$$
is linear equivalent to an effective Weil divisor $E$.

\end{description}
Since $m|J$, $Jq\sP_X$ is integral and well-defined modulo $\sim 0$.
The rest of proof for usual pairs is related to boundary
computations and works literally.

Addendum~\ref{adden_Cartier_index_B+} for bd-pairs follows
from Definition~\ref{bd_n_comp}, (3) and
from the boundedness of {\em bd-semicanonical index\/},
that is, the index of bd-semicanonical divisors $K_{X_t}+\sP_{X,t}$
for a bounded family $X/T$ of Ft varieties.

\end{proof}

\section{Adjunction  correspondence} \label{adjunction_cor_multiplicities}

Let $r$ be a real number and $l$ be a positive integer.
The {\em adjunction correspondence\/} with
{\em parameters\/} $r,l$ is the following
transformations of two real numbers $b,d$
\begin{equation}\label{direct_cor}
d=1-\frac rl+\frac bl\ \ \text{ (direct)\/}
\end{equation}
and
\begin{equation}\label{invers_cor}
b=r-l+ld\ \ \text{ (inverse)}.
\end{equation}
The correspondence is $1$-to-$1$.
How the correspondence is related to
the divisorial part of adjunction for morphism
see Section~\ref{adj}.
In this section we establish the basic properties
of this correspondence (almost) without any relation adjunction.

\subsection{Linearity}
Both transformations~(\ref{direct_cor}) and (\ref{invers_cor}) are
linear respectively for variables $b$ and $d$.

\begin{proof}
Immediate by (\ref{direct_cor}) and (\ref{invers_cor}).

\end{proof}

\subsection{Monotonicity} \label{adjunction_monotonicity}
Both transformations~(\ref{direct_cor}) and (\ref{invers_cor}) are
monotonically strictly increasing: for corresponding $b_1,d_1$ and
$b_2,d_2$,
$$
b_1>b_2\Leftrightarrow d_1>d_2.
$$

\begin{proof}
Immediate by~~(\ref{direct_cor}) and (\ref{invers_cor}) and positivity of $l$.

\end{proof}

\subsection{Rationality}
Both transformations~(\ref{direct_cor}) and (\ref{invers_cor}) are
rational linear functions if $r$ is rational.
So, for rational $r$ and corresponding $b,d$,
$$
b\in\Q\Leftrightarrow d\in\Q.
$$

\subsection{Direct log canonicity (klt)
(cf. \cite[Lemma~7.4, (iii)]{PSh08})} \label{direct_lc}
If $b\le r$ ($<r$) then the corresponding $d\le 1$ (respectively $<1$).
Note that, for adjunction,
$b\le r$ means lc of $b$, that is, $b\le 1$ (see~\ref{adjunction_mult}).
So, if $d$ is a multiplicity of a divisor $D$
in a prime divisor $P$ of $X$ then $(X,D)$ is lc near $P$.
Note also that if additionally $r\le 1$ then $b\le 1$ too.

\begin{proof}
By Monotonicity
$$
d=1-\frac rl + \frac bl\le 1-\frac rl+\frac rl=1
\text{ (respectively }<1\text{)}.
$$
\end{proof}

\subsection{Inverse log canonicity (klt)
(cf. \cite[Lemma~7.4, (iii)]{PSh08})}
\label{invers_lc}
If $d\le 1$ ($d<1$) then the corresponding $b\le r$ (respectively $<r$).
So, if additionally $r\le 1$ then $b\le 1$ (respectively $<1$).
Note that $r\le 1$ holds for adjunction of morphisms
by construction (see~\ref{adjunction_mult}).

\begin{proof}
By Monotonicity
$$
b=r-l + ld\le r-l+l=r
\text{ (respectively }<r\text{)}.
$$
\end{proof}

\subsection{Direct positivity (cf. \cite[Lemma~7.4, (i)]{PSh08})}
\label{direct_positivity}
If $b\ge 0$ and $r\le l$ then the corresponding $d\ge 0$.
The second assumption always holds for $r\le 1$
because $l\ge 1$.

\begin{proof} By Monotonicity and positivity of $l$
$$
d=1-\frac rl + \frac bl\ge 1-\frac r l\ge 0
$$
\end{proof}

However, the inverse positivity does not hold in general.
E.g., if $0\le d\ll 1$ then $b\approx r-l$ and $<0$
for $l> r$ which is typical for adjunction.

\subsection{Direct boundary property (cf. \cite[Lemma~7.4, (iv)]{PSh08})}
\label{direct_boundary}
If $0\le b\le r\le 1$ then the corresponding $d\in[0,1]$, that is,
$0\le d\le 1$.

\begin{proof}
Immediate by Direct log canonicity and positivity.
\end{proof}

\subsection{Direct hyperstandard property (cf. \cite[Proposition~9.3 (i)]{PSh08})}
\label{direct_hyperst}
Let $\fR,\fR''$ be two finite subsets of rational numbers in $[0,1]$ such that
$1\in\fR,\fR''$.
Then
\begin{equation}\label{const_adj}
\fR'=\{r'-l(1-r)\mid r\in\fR'',r'\in\fR,l\in \Z, l>0
\text{ and } r'-l(1-r)\ge 0\}.
\end{equation}
is also a finite subset of rational numbers in $[0,1]$ and $1\in\fR'$.

For every parameters $r,l$ such that $r\in\fR''$ and
$b\in\Phi=\Phi(\fR)$ such that $b\le r$,
$$
d\in\Phi'=\Phi(\fR'),
$$
where $d$ corresponds to $b$.

If additionally $\sN$ is a (finite) set of positive integers
then, for every $b\in\Gamma(\sN,\Phi)$ such that $b\le r$,
$$
d\in\Gamma(\sN,\Phi'),
$$
where $d$ corresponds to $b$.

\begin{proof}
Rationality of elements of $\fR'$ is
immediate by definition.
Since $r'\ge 0$ and $r\le 1$,
the set $\fR'$ is a subset of $[0,1]$.
Notice for this also that elements of $\fR'$ are
nonnegative by definition.

If $r=1$ then $r'-l(1-r)=r'\in \fR$.
Thus the set of such numbers is finite and $1$
belongs to it.

If $r<1$ then $1-r>0$.
Thus the set of numbers $r'-l(1-r)\ge0$ with $r'\in\fR$ is
finite.
Recall that $l$ is a positive integer.

By definition and our assumptions
$$
d=1-\frac rl+\frac bl,
$$
where $r\in\fR''$ and $l$ is a positive integer.
On the other hand, if $b\in\Gamma(\sN,\Phi)$ then
$$
b=1-\frac{r'}{l'}+\frac 1{l'}(\sum_{n\in\sN}\frac {m_n}{n+1}),
$$
where $r'\in \fR$, $l'$ is a positive integer and
$m_n,n\in\sN$, are nonnegative integers
(if $\sN$ is infinite then almost all $m_n=0$).
Hence
$$
d=1-\frac rl+\frac {1-\frac{r'}{l'}+\frac 1{l'}(\sum_{n\in\sN}\frac {m_n}{n+1})}l=
1-\frac{r'-l'(1-r)}{ll'}+\frac 1{ll'}(\sum_{n\in\sN}\frac {m_n}{n+1}).
$$
By our assumptions $b\le r$.
Thus by Direct log canonicity~\ref{direct_lc},
$d\le 1$.
Thus $r'-l'(1-r)\ge 0,\in\fR'$ by~(\ref{const_adj}) and
$d\in\Gamma(\sN,\Phi')$.
If particular, if $\sN=\emptyset$ then $d\in\Phi'$.

\end{proof}

\subsection{$n\_\Phi$ inequality} \label{n_Phi_inequality}

Under assumptions and notation of~\ref{direct_hyperst},
let $b_1,d'$ be real numbers such that
$0\le b_1\le r$ and
\begin{equation}\label{n_Phi'}
d'\ge d_{1,n\_\Phi'},
\end{equation}
where $d_1$ corresponds to $b_1$ by~(\ref{direct_cor}).
Then
\begin{equation}\label{n_Phi}
b_{1,n\_\Phi}\le b',
\end{equation}
where $b'$ corresponds to $d'$ by~(\ref{invers_cor}).

\begin{proof}
By~\ref{direct_boundary}, $d_1\in [0,1]$
because $r\le 1$ by \ref{direct_hyperst}.
Thus the assumption~(\ref{n_Phi'}) and
the conclusion~(\ref{n_Phi}) are meaningful.
By definition
$$
b_{1,n\_\Phi}\le b_1
\text{ and }
b_{1,n\_\Phi}\in\Gamma(n,\Phi).
$$
Hence by~\ref{adjunction_monotonicity}
$$
d_1'\le d_1,
$$
where $d_1'$ corresponds to $b_{1,n\_\Phi}$
by~(\ref{direct_cor}).
On the other hand, $d_1'\in\Gamma(n,\Phi')$
by~\ref{direct_hyperst}.
Thus again by definition
$$
d_1'\le d_{1,n\_\Phi'}.
$$
This implies that
$$
d_1'\le d'
$$
by the assumption~(\ref{n_Phi'}).
Again \ref{adjunction_monotonicity} implies
the required inequality~(\ref{n_Phi}).

\end{proof}

\subsection{Direct dcc} \label{direct_dcc}
Let $\Gamma\subset [0,1]$ be a dcc set and
$\fR''$ be a finite subset in $[0,1]$.
The the corresponding set
$$
\Gamma'=\{1-\frac rl+\frac bl\mid
r\in\fR'',l\in \Z, l>0,b\in\Gamma
\text{ and } r\ge b\}
$$
also satisfies dcc.
Note that the last assumption equivalent to
$1-\frac rl+\frac bl\le 1$
(cf. Direct log canonicity \ref{direct_lc}).

\begin{proof}
Since $r,b$ are bounded and $r-b\ge 0$,
any strictly decreasing sequence in $\Gamma'$ has
finite possibilities for $l$.
Hence the finiteness of $\fR''$ and the dcc for $\Gamma$
imply the dcc for $\Gamma'$.

\end{proof}

\subsection{Direct hyperstandard property for
adjunction on divisor (cf. \cite[Lemma~4.2]{Sh92})} \label{direct_hyperst_on_div}

For any subset $\Gamma\subseteq [0,1]$,
put
$$
\widetilde{\Gamma}=\{1-\frac 1l+
\sum_i\frac{l_i} l b_i
\le 1\mid l,l_i \text{ are positive integers and }
b_i\in \Gamma \}\cup\{1\}.
$$
Then $0,1\in\widetilde{\Gamma},\Gamma\subseteq\widetilde{\Gamma}$ and
$\widetilde{\widetilde{\Gamma}}=\widetilde{\Gamma}$.
Transition form $\Gamma$ to $\widetilde{\Gamma}$
corresponds to lc adjunction on a divisor
\cite[Corollary~3.10]{Sh92}.

Let $\Phi=\Phi(\fR)$ be a hyperstandard set associated with
a (finite) set of (rational) numbers $\fR$ in $[0,1],1\in\fR$, and
$\sN$ be a (finite) set of positive integers.
Then
$$
\widetilde{\Gamma(\sN,\Phi)}=
\widetilde{\frak{G}(\sN,\fR)}=
\Gamma(\sN,\widetilde{\Phi})=
\frak{G}(\sN,\overline{\fR}\cup\{0\}),
$$
in particular,
$\widetilde{\Phi}=\Phi(\overline{\fR}\cup\{0\})$,
where
$$
\overline{\fR}=\{r_0-\sum_i (1-r_i)\mid
r_i\in \fR \}\cap [0,1].
$$
The set $\overline{\fR}$
is (rational) finite, if $\fR$ is (rational) finite, and
same as $\overline{\fR}$ in \cite[p.~160]{PSh08}
and $\overline{\overline{\fR}}=\overline{\fR}$.
(If $1\not\in\fR$, we need to replace
$\overline{\fR}$ by $\overline{\fR\cup\{1\}}$.)
Note also that $\widetilde{\Phi}$ depends only on $\Phi$
and above $\overline{\fR}$ is one of possible (finite) sets in
$[0,1]$ to determine $\widetilde{\Phi}$
(see Example~\ref{standard_set} below and
cf. Proposition~\ref{dependence_on_Phi}).

\begin{proof}
It is enough to verify  that
$\widetilde{\frak{G}(\sN,\fR)}=
\frak{G}(\sN,\overline{\fR}\cup\{0\})$.
Then by definition, for $\sN=\emptyset$,
$\Phi=\frak{G}(\emptyset,\fR)=\Phi(\fR)$ and
$\widetilde{\Phi}=\widetilde{\frak{G}(\emptyset,\fR)}=
\frak{G}(\emptyset,\overline{\fR}\cup\{0\})=
\Phi(\overline{\fR}\cup\{0\})$.
Additionally $\widetilde{\frak{G}(\sN,\fR)}=
\frak{G}(\sN,\overline{\fR}\cup\{0\})=\Gamma(\sN,\widetilde{\Phi})$
by definition and Proposition~\ref{dependence_on_Phi}.

Take $b\in\widetilde{\frak{G}(\sN,\fR)}$.
First, we verify that $b\in\frak{G}(\sN,\overline{\fR}\cup\{0\})$.

Step~1. We can suppose that $b<1$.
Then by definition
$$
b=1-\frac 1l+
\sum_{i=0}^s\frac{l_i} l(1-\frac {r_i}{m_i}+
\frac 1{m_i}\sum_{n\in\sN} \frac {l_{i,n}}{n+1}),
$$
where $l,l_i,m_i$ are positive integers,
$l_{i,n}$ are nonnegative integers and $r_i\in\fR$.
Otherwise, the only possible case $b=1$.
But then $1=1-0/1\in\frak{G}(\emptyset,{0})\subseteq\
\frak{G}(\sN,\overline{\fR}\cup\{0\})$.

More precisely, we verify that $b<1$
belongs to $\frak{G}(\sN,\overline{\fR})\subseteq
\frak{G}(\sN,\overline{\fR}\cup\{0\})$,
We can suppose that $m_0\ge m_1\ge\dots\ge m_s$.

Step~2. $m_i=1$ for all $i\ge 1$.
Otherwise, $m_1\ge 2$.
Then $m_0,m_1\ge 2$ and
$$
b\ge 1-\frac 1l+\frac 1l(1-\frac 12)+\frac 1l(1-\frac 12)=1,
$$
a contradiction.
Here we use the inequality $r_i\le 1$ for every $i$.

Step~3. If $m_0=m\ge 2$ then $l_0=1$.
Otherwise, $l_0\ge 2$ and
$$
b\ge 1-\frac 1l+\frac 2l(1-\frac 12)=1.
$$

Hence we have the following two cases.

Case~1. $m_i=1$ for all $i\ge 1,m_0=m\ge 2$
and $l_0=1$.
So,
\begin{align*}
b=&1-\frac 1l+\frac 1l(1-\frac {r_0}m+
\frac 1m\sum_{n\in\sN} \frac {l_{0,n}}{n+1})+
\sum_{i=1}^s\frac{l_i} l(1-r_i+
\sum_{n\in\sN} \frac {l_{i,n}}{n+1})=\\
&1-\frac{r_0-\sum_{i=1}^s ml_i(1-r_i)}{lm}+
\frac 1{lm}({\sum_{n\in\sN}\frac{l_{0,n}+\sum_{i=1}^sml_il_{i,n}}{n+1}})
\end{align*}
belongs to $\frak{G}(\sN,\overline{\fR})$.
Indeed, $0<r_0-\sum_{i=1}^s ml_i(1-r_i)\le r_0\le 1$ and
belongs to $\overline{\fR}$
because respectively $1>b$ and every $r_i\le 1$.

Case~2. Every $m_i=1$.
Then

\begin{align*}
b=1-\frac 1l+
\sum_{i=0}^s\frac{l_i} l(1-r_i+
\sum_{n\in\sN} \frac {l_{i,n}}{n+1})=\\
1-\frac{1-\sum_{i=0}^s l_i(1-r_i)}l+
\frac 1l\sum_{n\in\sN}\frac{\sum_{i=0}^s l_il_{i,n}}{n+1}
\end{align*}
also belongs to $\frak{G}(\sN,\overline{\fR})$
because $1\in\fR$ and
$1-\sum_{i=0}^s l_i(1-r_i)\in\overline{\fR}$.

Step~4.
Conversely, take $b\in\frak{G}(\sN,\overline{\fR}\cup\{0\})$.
Now, we verify that $b\in\widetilde{\frak{G}(\sN,\fR)}$.
By definition $b\le 1$.

Case~1. $b=1$. Then $b=1=1-1/l+1/l\in \widetilde{\frak{G}(\emptyset,\{1\})}
\subset \widetilde{\frak{G}(\sN,\fR)}$ because $1\in\fR$.

Case~2. $b<1$. Then by definition
$$
b=1-\frac rl+
\frac 1l\sum_{n\in\sN}\frac{l_n}{n+1},
$$
where $r\in\overline{\fR}\cup\{0\}$ and $r>0$.
Hence $r\in\overline{\fR}$ and
$$
b=1-\frac{r_0-\sum_{i=1}^s(1-r_i)}l+
\frac 1l\sum_{n\in\sN}\frac{l_n}{n+1}=
1-\frac 1l+\frac 1l\sum_{i=0}^s(1-r_i)+
\frac 1l \sum_{n\in\sN}\frac{l_n}{n+1}.
$$
Note now that every
$$
1-r_i=1-\frac {r_i} 1 \in
\Phi=\Gamma(\emptyset,\Phi)=\frak{G}(\emptyset,\fR)\subseteq
\Gamma(\sN,\Phi)=\frak{G}(\sN,\fR)
$$
and
$$
\frac 1{n+1}=1-\frac 11+\frac 1{n+1}
\in \Gamma(\sN)=\frak{G}(\sN,\{1\})\subseteq\Gamma(\sN,\Phi)=
\frak{G}(\sN,\fR).
$$
So, $b\in \widetilde{\frak{G}(\sN,\fR)}$.

\end{proof}

\begin{exa} \label{standard_set}
As an exception, take $\fR=\emptyset$.
Then $\overline{\fR}=\Phi=\Phi(\fR)=\emptyset$ too and
$$
\widetilde{\Phi}=\{1-\frac 1l\mid l\text{ is a positive integer}\}
\cup\{1\},
=\Phi(\{0,1\}),
$$
the standard set.
However, $\widetilde{\Phi}=\Phi(\overline{\fR'})$
for $\fR'\subseteq\{1/l\mid l \text{ is a positive integer}\}$
if and only if $\fR'=\{1,1/2\}$.

\end{exa}

\subsection{Main inequality} \label{main_inequality_statement}
For any positive integers $n,l$,
any real number $d$, and
any rational number $r$ such that
$nr\in\Z$ and $r\le 1$,
\begin{equation}\label{main_inequality}
r-l+l\rddown{(n+1)d}/n\ge
\rddown{(n+1)(r-l+ld}/n.
\end{equation}

\begin{proof}
The inequality~(\ref{main_inequality}) is equivalent to
$$
nr-nl+l\rddown{(n+1)d}\ge
\rddown{(n+1)(r-l+ld}=\rddown{(n+1)(r+ld)}-(n+1)l
$$
or to
\begin{equation}\label{main_inequality_v2}
nr+l+l\rddown{(n+1)d}\ge\rddown{(n+1)(r+ld)}.
\end{equation}

For $r=1$, the inequality~(\ref{main_inequality_v2}) has
the form
$$
n+l+l\rddown{(n+1)d}\ge\rddown{(n+1)(1+ld)}=
n+1+\rddown{l(n+1)d},
$$
that is,
$$
l-1+l\rddown{(n+1)d}\ge\rddown{l(n+1)d}.
$$
It follows from the inequality~(\ref{1st_cor_uuper_bound}).

For $r<1$, by upper bounds~(\ref{basic_upper_bound}),
(\ref{special_upper_bound}) and (\ref{1st_cor_uuper_bound})
\begin{align*}
\rddown{(n+1)(r+ld)}\le &
1+\rddown{(n+1)r}+\rddown{l(n+1)d}\le\\
&1+nr+l-1+l\rddown{(n+1)d}=
nr+l+l\rddown{(n+1)d},
\end{align*}
that completes the proof of~(\ref{main_inequality_v2}).

\end{proof}

\subsection{Inverse $\rddown{(n+1)-}/n$-monotonicity}
\label{n+1_n_monotonicity}
For a real number $x$, put
$$
\rdn{x}{n}=\begin{cases}
1, \text{ if } &x=1;\\
\rddown{(n+1)x}/n &\text{ otherwise};
\end{cases}
$$
Under the assumptions of~\ref{main_inequality_statement}
suppose additionally that $d\le 1$. Then
$$
\rdn{b}{n} \le b^{[n]}
$$
where real numbers $b,b^{[n]}$ correspond respectively to
$d,\rdn{d}{n}$ according to~(\ref{invers_cor}).

\begin{proof}
\

Case~1. $d <1$. Immediate by~(\ref{main_inequality}).
Indeed, $b<1$ by \ref{invers_lc} and our assumptions.

Case~2. $d=1,r<1$. Then $\rdn{d}{n}=1=d$ and
$b=b^{[n]}=r$ by~(\ref{invers_cor}).
The required inequality follows from our assumptions
by Example~\ref{rddown_(n+1)_m_m}, (2).

Case~3. $d=r=1$. As in Case~2
$b=b^{[n]}=r$. Since $r=1$,
$\rdn{b}{n}=1=b^{[n]}$.

\end{proof}

\subsection{Inverse inequality~(1) of Definition~\ref{n_comp}}
\label{invers_1_n_comp}
Under the assumptions and notation
of~\ref{main_inequality_statement}-\ref{n+1_n_monotonicity},
for corresponding $b^+,d^+$ according to~(\ref{invers_cor}),
\begin{align*}
\rdn{b}{n} \le &b^+\le r\le 1\\
&\Uparrow \\
\rdn{d}{n} \le &d^+\le 1.
\end{align*}

\begin{proof}
The second inequality of the top row holds by~\ref{invers_lc}
and the third one by our assumptions.

Since $\rdn{d}{n}\le 1$, $d\le 1$ holds (cf. Remark~\ref{remark_def_complements}, (2)).
Hence by~\ref{n+1_n_monotonicity} and~\ref{adjunction_monotonicity}
$$
\rdn{b}{n} \le b^{[n]}\le b^+.
$$

\end{proof}

\section{Adjunction} \label{adj}

We recall some basic facts about (log) adjunction or
subadjunction in the Kawamata terminology \cite{K98}.

\subsection{Adjunction for $0$-contractions \cite[Section~7]{PSh08}}\label{adjunction_0_contr}
Let $f\colon X\to Z$ be a contraction of normal
algebraic varieties or spaces and
$D$ be an $\R$-divisor on $X$ such that
\begin{description}

\item[\rm (1)\/]
$(X,D)$ is lc generically over $Z$;

\item[\rm (2)\/]
$D$ is a boundary generically over $Z$
or, equivalently, $D\hor$ is a boundary,
where $D\hor$ denotes the horizontal part of $D$
with respect to $f$;
and

\item[\rm (3)\/]
$K+D\sim_{\R,Z} 0$ (cf. \ref{adjunction_index_of}, (3)
and \cite[Construction~7.5]{PSh08}), in particular,
$K+D$ is $\R$-Cartier and $(X,D)$ is a log pair.

\end{description} Then there exist two $\R$-divisors on $Z$:
\begin{description}

\item[]
the {\em divisorial part of adjunction\/} $D\dv$ of $(X,D)\to Z$
\cite[Construction~7.2]{PSh08};
and

\item[]
the {\em moduli part of adjunction\/} $D\md$  of $(X,D)\to Z$
\cite[Construction~7.5]{PSh08}

\end{description}
such that the following generalization of Kodaira formula
\begin{equation} \label{adjunction_f}
K+D\sim_\R f^*(K_Z+D\dv+D\md)
\end{equation}
holds.
It is also called the ({\em log\/}) {\em adjunction formular\/}
for $(X,D)\to Z$.
The pair $(Z,D\dv+D\md)$ is a log pair and
$f^*(K_Z+D\dv+D\md)$ is well-defined.
The pair $(Z,D\dv+D\md)$ is called
the {\em adjoint pair\/} of $(X,D)\to Z$.
The divisor $D\dv$ is unique but $D\md$ is
defined up to $\sim_\R$.

The adjunction has birational nature.
We say that two contractions $f\colon (X,D)\to Z,f'\colon (X',D')\to Z'$
are {\em birationally equivalent\/} or {\em crepant\/} if
there exists a commutative diagram
$$
\begin{array}{ccc}
(X',D')&\dashrightarrow & (X,D)\\
f'\downarrow&&\downarrow f\\
Z'&\dashrightarrow & Z
\end{array},
$$
where the horizontal arrow $(X',D')\dashrightarrow (X,D)$
is a crepant proper birational isomorphism,
another horizontal arrow $Z'\dashrightarrow Z$ is
a proper birational isomorphism and $f'$ satisfies the same
assumptions as $f$.
In particular,
$(X,D),(X',D')$ are birationally equivalent or crepant
if $(X/X,D),(X'/X',D')$ does so.
If $X$ is complete then this also means that
$(X/\pt,D),(X'/\pt,D')$ does so.
Notice that $D'=D_{X'}=\D(X,D)_{X'}$ and
$(X',D')$ is lc generically over $Z'$ automatically.
(The property to be a boundary for $D'\hor$ or, equivalently,
to be effective for $D'\hor$ can be omitted.)
We say also that $f'\colon (X',D')\to Z'$ is
a ({\em crepant\/}) {\em model\/} of $f\colon (X,D)\to Z$.
In this situation the adjoint pair $(Z',D'\dv+D'\md)$ is defined and
$(Z',D'\dv+D'\md)\dashrightarrow(Z,D\dv+D\md)$ is
also a crepant proper birational isomorphism.
This allows to define two b-$\R$-divisors $\D\dv$ and $\sD\md$ of $Z$
\cite[Remark~7.7]{PSh08} such that
$$
D\dv=(\D\dv)_Z,D\md=(\sD\md)_Z
\text{ and }
D'\dv=(\D\dv)_{Z'},D'\md=(\sD\md)_{Z'}.
$$
Of course we use that same $\sim_\R$ for $f'$ as
in~(\ref{adjunction_f}) and
suppose that $K_{X'}=(\K)_{X'}$ and $K_{Z'}=(\K_Z)_{Z'}$,
where $\K_Z$ is a canonical b-divisor of $Z$.
Otherwise $D'\md\sim_\R(\sD\md)_{Z'}$.
Indeed, for every prime b-divisor $Q$ of $Z$
there exist a model $f'$ of $f$ and a prime
divisor $P$ on $X'$ such that $f'(P)=Q$ is
a divisor on $Z'$.
(For details see \cite[Section~7]{PSh08}.)
Respectively, for b-$\R$-divisors, (\ref{adjunction_f}) become
\begin{equation}\label{adjunction_f_b}
\K+\D\sim_\R f^*(\K_Z+\D\dv+\sD\md)
\end{equation}
and $(Z,D\dv+\sD\md)$ become
the {\em adjoint\/} log {\em bd-pair\/} of $(X,D)\to Z$,
that is, $\D\dv=\D(Z,D\dv+D\md)-\sD\md$.
Actually, in our applications,
$\sD\md$ will be a b-nef b-$\Q$-divisor and
$(Z,D\dv+\sD\md)$ will be a bd-pair of
some positive integral index
(see Theorem~\ref{b_nef} and cf. Conjecture~\ref{mod_part_b-semiample}).

For a bd-pair $(X,D+\sP)$ and
its contraction $f\colon (X,D+\sP)\to Z$ under
assumptions
\begin{description}

\item[\rm (1-bd)\/]
$(X,D+\sP)$ is lc generically over $Z$;

\item[\rm (2-bd)\/]
$D$ is a boundary
and $\sP$ is b-nef generically over $Z$;
and

\item[\rm (3-bd)\/]
$K+D+\sP_X\sim_{\R,Z} 0$, in particular,
$K+D+\sP_X$ is $\R$-Cartier and $(X,D+\sP)$ is a log bd-pair.

\end{description}
Then there exist the following $\R$-divisors on $Z$ and b-$\R$-divisors of $Z$
\cite[Section~4]{F18}:
\begin{description}

\item[]
the $\R$-divisor,
the {\em divisorial part of adjunction\/} $(D+\sP)\dv{}_{,Z}$ of $(X,D+\sP)\to Z$;

\item[]
respectively,
b-$\R$-divisor $(D+\sP)\dv$, e.g., $((D+\sP)\dv)_Z=(D+\sP)\dv{}_{,Z}$;

\item[]
the $\R$-divisor,
{\em moduli part of adjunction\/} $(D+\sP)\md{}_{,Z}$  of $(X,D+\sP)\to Z$;
and,

\item[]
respectively,
b-$\R$-divisor $(D+\sP)\md$, e.g., $((D+\sP)\md)_Z=(D+\sP)\md{}_{,Z}$;

\end{description}
such that the following generalizations of Kodaira formula
$$
K+D+\sP_X\sim_\R f^*(K_Z+(D+\sP)\dv{}_{,Z}+(D+\sP)\md{}_{,Z})
$$
and
$$
\K+\D(X,D+\sP_X)\sim_\R f^*(\K_Z+(D+\sP)\dv+(D+\sP)\md)
$$
hold.
It is also called the ({\em log\/}) {\em adjunction formular\/}
for $(X,D+\sP)\to Z$.
The bd-pair $(Z,(D+\sP)\dv{}_{,Z}+(D+\sP)\md)$ is a log bd-pair and
$f^*(K_Z+(D+\sP)\dv{}_{,Z}+(D+\sP)\md{}_{,Z})$ is well-defined.
The pair $(Z,(D+\sP)\dv{}_{,Z}+(D+\sP)\md)$ is called
the {\em adjoint bd-pair\/} of $(X,D+\sP)\to Z$.
Indeed, $(D+\sP)\dv=\D(Z,(D+\sP)\dv{}_{,Z}+(D+\sP)\md{}_{,Z})-(D+\sP)\md$.
The b-$\R$-divisor $(D+\sP)\dv$ is unique but b-$\R$-divisor $(D+\sP)\md$ is
defined up to $\sim_\R$.
Actually, in our applications,
$(D+\sP)\md$ will be a b-nef b-$\Q$-divisor and
$(Z,(D+\sP)\dv{}_{,Z}+(D+\sP)\md)$ will be a bd-pair of
some positive integral index
(see Addendum~\ref{b_nef_bd} and cf. Conjecture~\ref{mod_part_b-semiample}).

Adjunction for $(X,D+\sP)\to Z$ also has birational nature.
In this situation, we say that two contractions
$f\colon (X,D+\sP)\to Z,f'\colon (X',D'+\sP)\to Z'$
are {\em birationally equivalent\/} or {\em crepant\/} if
there exists a commutative diagram
$$
\begin{array}{ccc}
(X',D'+\sP)&\dashrightarrow & (X,D+\sP)\\
f'\downarrow&&\downarrow f\\
Z'&\dashrightarrow & Z
\end{array},
$$
where the horizontal arrow $(X',D'+\sP)\dashrightarrow (X,D+\sP)$
is a crepant proper birational isomorphism,
another horizontal arrow $Z'\dashrightarrow Z$ is
a proper birational isomorphism and $f'$ satisfies the same
assumptions as $f$.
Notice that $D'=D_{X'}=\D(X,D+\sP_X)\dv{}_{,X'}-\sP_{X'},
\D(X',D'+\sP_{X'})=\D(X,D+\sP_X)$ and
$(X',D'+\sP)$ is lc generically over $Z'$ automatically.
(The property to be a boundary for $D'\hor$ or, equivalently,
to be effective for $D'\hor$ can be omitted.)
We say also that $f'\colon (X',D'+\sP)\to Z'$ is
a ({\em crepant\/}) {\em model\/} of $f\colon (X,D+\sP)\to Z$.
In this situation the adjoint bd-pair
$(Z',(D+\sP)\dv{}_{,Z'}+(D+\sP)\md)$ is defined,
$(Z',(D+\sP)\dv{}_{,Z'}+(D+\sP)\md{})\dashrightarrow
(Z,(D+\sP)\dv{}_{,Z}+(D+\sP)\md)$ is
a crepant proper birational isomorphism and
$$
(D+\sP)\dv{}_{,Z'}=((D+\sP)\dv)_{Z'},
(D+\sP)\md{}_{,Z'}=((D+\sP)\md)_{Z'}.
$$

We can modify above concepts,
e.g., birational equivalence or to be crepant, to
the relative situation.
However, adjunction formulae~(\ref{adjunction_f}-\ref{adjunction_f_b})
and muduli part of adjunction
have absolute nature and defines up to $\sim_\R$.
For bd-pairs we assume the same for $\sP$.
(Usually we omit $S$ if $S=\pt$)

The following result about the moduli part of adjunction
holds in much more general situation \cite{Sh13}.
Actually we expect more \cite[Conjecture~7.13]{PSh08}
(see also Conjecture~\ref{mod_part_b-semiample}).
But we need the result only under stated assumptions.
On the other hand, the nef property is assumed
here in the weak sense (see Nef in Section~\ref{intro}).
In particular, the property applies to the situation,
where $X$ is complete or
it has a proper morphism $X\to S$ to a scheme or a space $S$
compatible with the contraction $X\to Z$, equivalently,
$X\to Z$ is a morphisms over $S$ and $Z/S$ is proper.

\begin{thm}[\cite{A04}]\label{b_nef}
Under notation and assumptions of~\ref{adjunction_0_contr}
suppose additionally that
$D$ is an effective $\Q$-{\em divisor\/} generically over $Z$.
Then $\sD\md$ is a b-$\Q$-divisor of $Z$,
defined up to a $\Q$-linear equivalence.
Moreover, $\K_Z+\D\dv$ is a b-$\R$-Cartier divisor of $Z$ and $\sD\md$ is b-nef.
\end{thm}

\begin{add}[{\cite[Theorem~4.1]{F18}}]\label{b_nef_bd}
The same holds for every contraction $(X,D+\sP)\to Z$ as
in~\ref{adjunction_0_contr} under the additional assumptions of
the theorem and the assumption that $(X,D+\sP)$ is a bd-pair of
some positive integral index.
\end{add}

[
Remark: the theorem holds either without the generic $\Q$-assumption over $Z$,
but with a b-$\R$-divisor $\sD\md$ of $Z$, or
without the generic effective property over $Z$.
However, the effective property is very important when $D$ has
horizontal nonrational multiplicities \cite{Sh13} \cite{Sh08}.

About notation: For b-$\R$-Cartier divisors we use mathcal, e.g.,
$\sD\md,\sP$.
For b-$\R$-divisors with BP we use mathbb: e.g., $\D$.
However, $\D+\sP,-\K$ have BP and $\K+\D$ is b-$\R$-Cartier.
]

\begin{proof}
See proofs in \cite{A04} and \cite{F18}.
Notice only that the necessary assumption
for the existence of $\sD\md$ that
$K+D\sim_\R f^*L$ for some $\R$-Cartier divisor $L$ on $Z$ \cite[Construction~7.5]{PSh08}
holds by~(3) of~\ref{adjunction_0_contr}.
Similarly, in the addendum we use (3-bd) of~\ref{adjunction_0_contr}.

\end{proof}

\subsection{Adjunction correspondence of multiplicities} \label{adjunction_mult}

Let $f\colon X\to Z$ be a proper surjective morphism, e.g.,
a contraction as in~\ref{adjunction_0_contr}.
Then, for every {\em vertical over\/} $Z$
prime b-divisor $P$ of $X$, its image $Q=f(P)$ as
a {\em prime\/} b-divisor is well-defined.
For this consider a model $f'\colon X'\to Z'$ as in~\ref{adjunction_0_contr}
of $f$ such that $P$ is a divisor on $X'$ and
$f'(P)=f(P)=Q$ is a divisor on $Z'$.
Since $f$ is proper surjective, it is also {\em surjective on prime\/}
b-{\em divisors\/}, that is, for every prime b-divisor $Q$ of $Z$
there exists a vertical prime b-divisor $P$ of $X$ such that $f(P)=Q$.
Indeed, for a birational model $f'$ as above,
$P$ is any divisorial irreducible component of ${f'}\1 Q$,
equivalently, $P,Q$ are prime divisors on $X',Z'$ respectively and
$f'(P)=Q$.
The prime b-divisors $P,Q$ such that $f(P)=Q$ will be called {\em corresponding\/}
with respect to $f$.

Let $f\colon (X,D)\to Z$ be a morphism as in~\ref{adjunction_0_contr}.
Consider prime b-divisors $P,Q$ of $X,Z$ respectively such that
$f(P)=Q$.
Put
$$
r=\mult_P (D'+c_Q f'^*Q) \text{ and }
l=\mult_P f'^*Q,
$$
where $f'$ is the above model of $f$,
$D'=D_{X'}=(\D(X,D))_{X'}$,
$c_Q$ is the log canonical threshold as in \cite[Construction~7.2]{PSh08} and
$f'^*Q$ is well-defined generically over $Q$.
Then $r$ is a real number $\le 1$, because $(X',D_{X'}+c_Q f'^*Q)$
generically lc over $Q$ and near $P$, and $l$ is a positive integer
such that
\begin{equation}\label{direct_cor_d}
d_Q=1-\frac rl+\frac {d_P}l
\end{equation}
and
\begin{equation}\label{invers_cor_d}
d_P=r-l+ld_Q.
\end{equation}
This is exactly the adjunction correspondence~(\ref{direct_cor}-\ref{invers_cor})
between $b=d_P=\mult_PD'=\mult_P\D$ and $d=d_Q=\mult_Q\D\dv$.
The parameters $r=r_P,l=l_P$ depend only on $P$ but not on
a model $f'$ (cf. Proposition~\ref{adjunction_same_mod_etc}, (4)).
They will be called {\em adjunction constants\/} of $(X,D)\to Z$ at $P$.
Similar constants $r,l$ can be defined for a morphism $(X,D+\sP)\to Z$
as in~\ref{adjunction_0_contr}.

If $(X,D)$ is lc (klt) then $d_P\le r$ ($< r$ respectively).
Conversely, $(X,D)$ is lc (klt over $Z$) if
$d_P\le r$ ($< r$ respectively) for every vertical prime b-divisor $P$ of $X$.
The relative klt over $Z$ means
$\mult_P\D<1$ but only for vertical $P$.

The same holds for adjunction of $(X,D+\sP)\to Z$ as in~\ref{adjunction_0_contr}.

\begin{proof}
By definition and construction
$$
r=\mult_P(D'+c_Q f'^*Q)=
\mult_P D'+c_Q \mult_P f'^*Q=d_P+c_Q l
$$
and $c_Q=(r-d_P)/l$.
On the other hand, by definition
$$
d_Q=1-c_Q=1-\frac{r-d_P}l=1-\frac rl+\frac{d_P}l.
$$
This proves~(\ref{direct_cor_d}).
The latter implies~(\ref{invers_cor_d}).

It is enough to establish the independence of $r,l$ of $f'$ in
the case when $Z$ is a curve \cite[Remark~7.3]{PSh08};
if $Z$ is a point then every prime b-divisor $P$ of $X$
is horizontal over $Z$.
But for a curve $Z$, the divisor $f'^*Q$ can be replaces
by its Cartier b-divisor.
Then $r,l$ are independent of a proper birational model $X'$ of $X$ over $Z$.

If $(X,D)$ is lc (klt) then by definition
$c_Q\ge 0$ ($>0$ respectively).
Hence $r\ge d_P$ (respectively $>d_P$)
because $l>0$. The converse can be established similarly.

The same works for adjunction of $(X,D+\sP)\to Z$.

\end{proof}

[
Remark: $r_P$ is not a multiplicity of a (b-)divisor at $P$,
that is, $\sum r_P P$ is not a (b-)divisor!
However it is a (b-)divisor for $P/Q$
(cf.~\ref{adjunction_div}).
]

\subsection{Index of adjunction} \label{adjunction_index_of}

Let $f\colon (X,D)\to Z$ be a $0$-contraction as in~\ref{adjunction_0_contr} and
$I$ be a positive integer.
We say that the $0$-contraction has an {\em adjunction index\/} $I$ if
\begin{description}

\item[\rm (1)\/]
$\K+\D\sim_I f^*(\K_Z+\D\dv+\sD\md)$, in particular,
$K+D\sim_I f^*(K_Z+D\dv+D\md)$;

\item[\rm (2)\/]
$I$ is an lc index of $(X/Z\ni\eta,D)$ and
$I\D$ is b-Cartier generically over $\eta$,
where $\eta\in Z$ is the generic point of $Z$;
actually $\sD\md$ is defined modulo $I_\eta$,
the generic index;

\item[\rm (3)\/]
$I\sD\md$ is b-Cartier and
$\sD\md$ is defined modulo $\sim_I$,
in particular, $\K+\D\sim_{I,Z} f^*(\K_Z+\D\dv)$
(cf.~\ref{adjunction_0_contr}, (3)
and Corollary-Conjecture~\ref{conj_bounded_lc_index});
and

\item[\rm (4)\/]
$Ir_P\in \Z$ for every adjunction constant $r_P$ of
$(X,D)\to Z$.

\end{description}
Actually, (1) is equivalent to (2) and,
by Proposition~\ref{adjunction_same_mod_etc}
and the reduction in Step~1 of the proof of Theorem~\ref{adjunction_index} below,
(3) implies~(1-2) and~(4).

Respectively, a $0$-contraction $f\colon (X,D+\sP)\to Z$ of
a bd-pair $(X,D+\sP)$ as in~\ref{adjunction_0_contr}
has an {\em adjunction index\/} $I$ if
\begin{description}

\item[\rm (0-bd)\/]
$(X,D+\sP)$ is a bd-pair of index $I$, in particular,
$X,Z$ are complete or proper over some scheme $S$, e.g.,
$(X,D+\sP)$ is a bd-pair of index $m$ and $m|I$;

\item[\rm (1-bd)\/]
$\K+\D(X,D+\sP_X)\sim_I f^*(\K_Z+(D+\sP)\dv+(D+\sP)\md)$, in particular,
$K+D+\sP_X\sim_I f^*(K_Z+(D+\sP)\dv{}_{,Z}+(D+\sP)\md{}_{,Z})$;

\item[\rm (2-bd)\/]
$I(D+\sP)\md$ is b-Cartier and
$(D+\sP)\md$ is defined modulo $\sim_I$;

\item[\rm (3-bd)\/]
$I$ is an lc index of $(X/Z\ni\eta,D+\sP_X)$ and
$I\D(X,D+\sP_X)$ is b-Cartier generically over $\eta$,
where $\eta\in Z$ is the generic point of $Z$;
and

\item[\rm (4-bd)\/]
$Ir_P\in \Z$ for every adjunction constant $r_P$ of
$(X,D+\sP)\to Z$.

\end{description}

Corollary~\ref{adjunction_Q_horizontal} below
implies that under assumptions in~\ref{adjunction_0_contr},
if $D\hor$ is a $\Q$-divisor, then $(X,D)\to Z$ has some
adjunction index (cf. Corollary~\ref{bounded_lc_index}).
We need a similar result for certain families of
$0$-contractions.

\begin{thm} \label{adjunction_index}
Let $d$ be a nonnegative integer and
$\Gamma$ be
a dcc set of rational numbers in $[0,1]$.
Then there exists a positive integer $I=I(d,\Gamma)$ such that
every $0$-contraction $f\colon(X,D)\to Z$ as in~\ref{adjunction_0_contr}
with wFt $X/Z,\dim X=d$ and $D\hor\in\Gamma$ has
the adjunction index $I$.

\end{thm}

\begin{add} \label{adjunction_index_adjoint}
$(Z,D\dv+\sD\md)$ is a log bd-pair of index $I$.

\end{add}

\begin{add} \label{adjunction_index_hor}
There are two finite sets of rational numbers
$\Gamma\hor(d)\subseteq\Gamma$ and
$\fR''=\fR''(d,\Gamma)\subset [0,1]$ such that,
for every $0$-contraction $(X,D)\to Z$ in as the theorem,
\begin{description}

\item[]
$D\hor\in\Gamma\hor(d)$;
and

\item[]
every nonnegative adjunction constant $r_P$ of $(X,D)\to Z$
belongs $\fR''$.

\end{description}

\end{add}

\begin{add} \label{adjunction_index_ver}
Let $\Gamma''$ be another dcc in $[0,1]$, e.g., $\Gamma''=\Gamma$.
Then there exists a dcc set $\Gamma'\subset[0,1]$ such that,
for every $0$-contraction $(X,D)\to Z$ in as the theorem and
with lc $(X,D)$,
$$
D\ver\in\Gamma''\Rightarrow D\dv\in\Gamma'.
$$
$\Gamma'$ depends on $d,\Gamma,\Gamma''$ and
is rational if $\Gamma''$ is rational.

\end{add}

\begin{add} \label{adjunction_index_div}
Let $\fR$ be a finite set of rational numbers in $[0,1]$.
Then there exists a finite set of rational numbers $\fR'\subset [0,1]$
such that, for every (finite) set of integers $\sN$ and
every $0$-contraction $(X,D)\to Z$ in as the theorem
with lc $(X,D)$,
$$
D\ver\in\Gamma(\sN,\Phi)\Rightarrow D\dv\in\Gamma(\sN,\Phi'),
$$
where $\Phi=\Phi(\fR),\Phi'=\Phi(\fR')$ are
hyperstandard sets associated with $\fR,\fR'$
respectively.
The set $\fR'$ depends on $d,\Gamma$ and $\fR$.
\end{add}

\begin{add} \label{adjunction_index_bd}
The same holds for every contraction $(X,D+\sP)\to Z$ as
in~\ref{adjunction_0_contr} under the additional assumptions of
the theorem and the assumption that $(X,D+\sP)$ is a bd-pair
of index $m$.
In this situation $I=I(d,\Gamma,m),
\fR''=\fR''(d,\Gamma,m),\Gamma(d,m)$ depend also on $m$ and
$\Gamma',\fR'$ depend
respectively on $d,\Gamma,\Gamma'',m$ and $d,\Gamma,\fR,m$.
The adjoint bd-pair in Addendum~\ref{adjunction_index_adjoint}
$(Z,(D+\sP)\dv{}_{,Z}+(D+\sP)\md)$ is a log bd-pair of index $I$.
Additionally $m|I$.

\end{add}

\begin{proof} (Hyperstandard case.)
For the general case see Section~\ref{applic}.

Step~1. (Reduction to relative tlc singularities.)
There exists a model $f'\colon(X',D')\to Z'$ of $(X,D)\to Z$
and a boundary $B$ on $X'$
such that
\begin{description}

\item[\rm (1)\/]
$\sD\md$ is stable over $Z'$: $\sD\md=\overline{\sD\md{}_{,Z'}}$,
the Cartier closure over $Z'$ \cite[Example~1.1.1]{Sh96};

and

\item[\rm (2)\/]
$B\hor=D\hor$ and
$(X'/Z',B)$ is a $0$-pair with tlc (toroidal log canonical)
singularities.

\end{description}
The latter means that $(X'/Z',B)$ is toroidal
near lc but not klt points, in particular, $(X',B')$ is lc,
and with vertical reduced boundary: $B\ver\in\{0,1\}$.
Additionally, it means that $Z\setminus \Supp B\dv$ is
nonsingular and, for every nonsingular hypersurface
$H$ in $Z$, $(X',B'+f'^*H)$ is lc over $Z\setminus \Supp B\dv$ too
(equisingularity).
(Actually, it is enough that $(X'/Z',B)$ has
such a toroidal model.)
For the property (2), it is enough weak semistable reduction
and relative LMMP (see also \cite[Theorem~1.1]{B12}).
For (1), $Z'$ should be sufficiently high over $Z$
by Theorem~\ref{b_nef}.

By \cite[Remark~7.5.1]{PSh08} (cf. Proposition~\ref{adjunction_same_mod_etc}),
$(X'/Z',B)$ has the same adjunction index.
Thus for simplicity we can suppose that
$(X'/Z',B)=(X/Z,D)$.

Step~2. (Adjunction index.)
By~(2), $D\dv\in\{0,1\}$ (reduced) and
$(Z,D\dv)$ has lc index $1$, that is,
$K_Z+D\dv$ is Cartier.
Moreover, by~(2) and Corollary~\ref{bounded_lc_index}
there exists a positive integer $I=I(d,\Gamma)$ and
a finite set of rational numbers $\Gamma\hor(d)$
depending on $d$ and $\Gamma$, such that
\begin{description}

\item[\rm (3)\/]
$K+D\sim_{I,Z} 0$;
and

\item[\rm (4)\/]
$D\hor\in\Gamma\hor(d)$.

\end{description}
Note that in this section we can prove Corollary~\ref{bounded_lc_index}
and the theorem
only for hyperstandard sets, that is, assuming $\Gamma=\Phi(\fR)$,
where $\fR$ is a finite set of rational numbers in $[0,1]$, possibly,
different from $\fR$ of Addendum~\ref{adjunction_index_div}.
In this situation $\Gamma\hor(d)=\Gamma\hor(d,\fR)$ depends only
on $d$ and $\fR$.
Indeed, the proof of Corollary~\ref{bounded_lc_index}
uses boundedness of $n$-complements.
In the case $\dim Z\ge 1$ we can use
dimensional induction.
In the case $\dim Z=0$, again we can use
dimensional induction as in \cite[4.13]{PSh08}
if $(X,D)$ is lc but not klt.
Finally, if $\dim Z=0$ and $(X,D)$ is klt then
$(X,D)$ is bounded by Corollary~\ref{bound_ex_pairs}  because $D=D\hor\in\Phi$.
A different approach can use \cite[Theorem~1.8]{B}
\cite[Theorem~1.3]{HX} \cite[Theorem~1.6]{HLSh}.

We can apply Corollary~\ref{bounded_lc_index}
because the construction in Step~1 does not change
the following properties of $(X,D)\to Z$:
it is still $0$-contraction as in~\ref{adjunction_0_contr}
with wFt $X/Z,\dim X=d$ and $D\hor\in\Gamma$.
By construction the $0$-pair $(X/Z,D)$ is maximal lc over
$\Supp D\dv$.
Thus by Corollary~\ref{bounded_lc_index} or \cite[Theorem~1.8]{B}
(3) holds over a neighborhood of $\Supp D\dv$.
By~(4), $D\hor\in\Gamma\hor(d)$.
Actually, by~(3), $I\Gamma\hor(d)\in\Z$ and
$\Gamma\hor(d)$ is a finite set of rational numbers in
$[0,1]\cap (\Z/I)$.
Again by~(2), for a nonsingular hyperplane $H$
through any point $z$ of $Z\setminus \Supp D\dv$,
$(X/Z,D+f^*H)$ is lc over $Z\setminus D\dv$.
(It is sufficiently to take effective Cartier $H$
such that $(Z,H)$ is lc near $z$.)
By the same reason as above,
(3) holds locally over $z$ because
$f^*H$ is Cartier.
Hence (3) holds over $Z$.

[Remark:]
By~(3), there exists a Cartier divisor $L$ on $Z$
such that $I(K+D)\sim f^*L$.
(An additional property of the adjunction index $I$ which
follows from (1) of~\ref{adjunction_index_of} for appropriate $D$ and
on an appropriate model of $(X,D)\to Z$.)
Thus (1-3) of \ref{adjunction_index_of} hold
by (1) and because $K_Z+D$ is Cartier (cf. \cite[Construction~7.5]{PSh08}).
Since the $0$-pair $(X/Z,D)$ is maximal lc (locally) over
$\Supp D\dv$ and $K+D$ has Cartier index $I$,
$Ir_P\in \Z$ for every prime b-divisor $P$ over $\Supp D\dv$,
that is, (4) of~\ref{adjunction_index_of} over $\Supp D\dv$.
Adding $f^*H$ as above we can establish (4) of~\ref{adjunction_index_of}
everywhere over $Z$.
Actually in this situation,
just~(1) of~\ref{adjunction_index_of}
implies (2-4) of~\ref{adjunction_index_of} because $K_Z+D\dv$ is Cartier and
$K+D$ has Cartier index $I$.
(In general, (3) implies~(1-2) and~(4).)

In the proof of the step, a choice of Cartier $L$
is crucial.
The divisor $L$ is defined up to $\sim$ on $Z$.
Canonical divisors $K,K_Z$ are also defined up to $\sim$
on $X,Z$ respectively.
Moreover, the linear equivalence $K\sim K'$ for $K$, that is, $K-K'$
should be vertical principal.
Thus in $K+D\sim_I f^*(L/I)$ we can suppose that
$\sim_I$, that is, $K+D-f^*(L/I)$, is fixed.
This gives $\Q$-divisors $L/I$ and $D\md$
with Cartier index $I$ which are defined up to $\sim_I$.
Thus the adjunction~(\ref{adjunction_f}) become
(1) in~\ref{adjunction_index_of} with $\sim_I$ instead
of general $\sim_\R$ (cf. \cite[Conjecture~7.13 and (7.13.4)]{PSh08}).

Step~3. (Addenda.)
Addendum~\ref{adjunction_index_adjoint} follows
from~\ref{adjunction_0_contr}, \ref{adjunction_index_of}, (3) and
Theorem~\ref{b_nef}.

In Addendum~\ref{adjunction_index_hor},
we can take
$$
\Gamma\hor(d)=\Gamma\cap \frac \Z I\subseteq\fR''=[0,1]\cap \frac \Z I
$$
by~(2) and~(4) of~\ref{adjunction_index_of} (cf. (4) above).

Addendum~\ref{adjunction_index_ver} follows
from~(\ref{direct_cor_d}), Addendum~\ref{adjunction_index_hor}
and~\ref{direct_dcc}.

Similarly, Addendum~\ref{adjunction_index_div} follows
from~(\ref{direct_cor_d})
and~\ref{direct_hyperst} with $\fR'$ defined by~(\ref{const_adj}).
Notice that $\fR'$ depends on $d,\Gamma$ and $\fR$
because $\fR''$ depends on $d,\Gamma$.

Finally, all the same works for $0$-contractions of
bd-pairs of Addendum~\ref{adjunction_index_bd}
with the additional new parameter $m$ and
the new assumption $m|I$.

\end{proof}

\subsection{Generically crepant adjunctions}

Let $f\colon (X,D)\to Z,f'\colon (X',D')\to Z'$ be two $0$-contractions as in~\ref{adjunction_0_contr}.
We say that they are {\em birationally equivalent\/} or  {\em crepant generically over\/} $Z$ or $Z'$ if there exist nonempty open subsets
$U,U'$ in $X,X'$ respectively such that
$f\rest{U}\colon (X_U,D_U)\to U,f'\rest{U'}\colon (X_{U'}',D_{U'}')\to U'$
are crepant, where
$X_U=f\1U,D_U=D\rest{X_U}$ and
$X_{U'}'=f'{}\1U',D_{U'}=D\rest{X_{U'}}$ respectively.
Equivalently, $X'\to Z'$ is a model of $X\to Z$ and
$\D\hor=\D'\hor$ under the birational equivalence,
where $\D'=\D(X',D')$.
We say also that $(X',D')\to Z'$ is a ({\em crepant\/})
{\em model\/} of $(X,D)\to Z$.
In this situation, $X'\to Z'$ is a model of $X\to Z$
with given proper birational isomorphisms
$X'\dashrightarrow X, g'\colon Z'\dashrightarrow Z$.
Thus we can compare certain (birational) invariants of adjunction, e.g.,
the moduli part of adjunction $\sD\md$ on $Z$ with that of $\sD\md'$ on $Z'$,
where $\sD\md'$ is the moduli part of adjunction of $(X',D')\to Z'$.
So, $\sD\md=\sD\md'$ means that $\sD\md'=g'^*\sD\md$.
(In our applications $g'$ is identical on
a nonempty open subset $U'$ in $X'$.)
Notice that we can omit (2) in~\ref{adjunction_0_contr},
the effective property of $D\hor,D'\hor$.
In some applications we omit also the assumption
that $X\to Z,X'\to Z'$ are regular.
But we still
assume that they are proper rational contractions and
$f\colon (X,D)\dashrightarrow Z,f'\colon (X',D')\dashrightarrow Z'$
are proper rational $0$-contractions.
Moreover, the latter ones are crepant respectively over $Z,Z'$
to $0$-contraction as in~\ref{adjunction_0_contr}
(cf. Lemma~\ref{B_D_r_f}).
Thus we apply the definition and results to
the latter contractions.

The same applies to bd-pairs.

Some of important invariants of adjunction in~\ref{adjunction_0_contr}
depend on $(X,D)\to Z$ only generically over $Z$.
The same applies to bd-pairs.

\begin{prop} \label{adjunction_same_mod_etc}
Let $f\colon (X,D)\to Z,f'\colon (X',D')\to Z'$ be two
birationally equivalent generically over $Z$
$0$-contractions as in~\ref{adjunction_0_contr}.
Then they have the same following invariants
under the birational equivalence:
\begin{description}

\item[\rm (1)\/]
the moduli part of adjunction:
$$
\sD\md=\sD\md'\text{ and }
D\md'=(\sD\md)_{Z'};
$$

\item[\rm (2)\/]
if
$$
K+D\sim_\Q f^*(K_Z+D\dv+D\md)\text{ and }
\K+\D\sim_\Q f^*(\K_Z+\D\dv+\sD\md)
$$
then respectively
$$
K_{X'}+D'\sim_\Q f^*(K_{Z'}+D\dv'+D\md')\text{ and }
\K_{X'}+\D'\sim_\Q f^*(\K_{Z'}+\D\dv'+\sD\md');
$$
the same holds for $\sim_I$ instead of $\sim_\Q$,
where $I$ is a positive integer;

\item[\rm (3)\/]
the adjunction correspondence of (b-)$\R$-divisors;

\item[\rm (4)\/]
the adjunction constants $l_P,r_P$ at
every vertical over $Z$ prime b-divisor $P$ of $X$;
actually, $l_P$ depend only on $f\colon X\to Z$
generically over $Z$;

\item[\rm (5)\/]
the adjunction index $I$ if such one exists for either of models;
and

\item[\rm (6)\/]
the horizontal b-codiscrepancies: $\D\hor=\D'\hor$.

\end{description}

The same holds for bd-pairs with same $\sP$.
In (2) and (5) the bd-pairs have index $m|I$.
\end{prop}

\begin{proof}
Up to a birational equivalence as in~\ref{adjunction_0_contr},
we can suppose that $(X/Z,D\hor)$ is equal to $(X'/Z',D'\hor)$
(cf. (1) in~\ref{adjunction_div}).
In this situation we can use \cite[Remark~7.5.1]{PSh08}
for (1-2).
Of course, we suppose that $\K_{X'}=\K$ and
$\sim_\R$ (respectively $\sim_\Q,\sim_I$)
is the same under the birational equivalence.

(3) holds by definition (see \ref{adjunction_div} below).

For (4) we can use reduction to
a $1$-dimensional base $Z$ \cite[Remark~7.3, (i)]{PSh08}.
In this case $l_P$ depends only on $X\to Z$ and
$r$ does so additionally on $D\hor$.
Indeed, $c_Q$ depend on $D$ but $D+c_Qf^*Q$ depend only
on $D\hor$.

(5-6)
hold by definition and already established facts about
other invariants.

The same works for bd-pairs.
\end{proof}

\begin{cor} \label{adjunction_Q_horizontal}
Let $(X,D)\to Z$ be a $0$-contraction as in \ref{adjunction_0_contr} and
additionally $D\hor$ be a $\Q$-divisor.
Then~(\ref{adjunction_f}) holds with $\sim_\Q$ instead of $\sim_\R$,
$\sD\md$ is a b-$\Q$-divisor and
every adjunction constant $r_P$ is rational.

The same holds for bd-pairs with a b-$\Q$-divisor $\sP$.
\end{cor}

\begin{proof}
We can suppose that $D'$ is a $\Q$-divisor.
Thus~(\ref{adjunction_f}) holds with $\sim_\Q$,
$\sD\md$ is a b-$\Q$-divisor and
all $c_Q,r_P$ are rational by construction.

Similarly we can treat bd-pairs.

\end{proof}

\subsection{Adjunction correspondence of divisors} \label{adjunction_div}

Let $f\colon (X,D)\to Z$ be a $0$-contraction as
in~\ref{adjunction_0_contr}.
We say now that b-$\R$-divisor $\D\dv$ {\em corresponds under adjunction\/} to
the b-$\R$-divisor $\D=\D(X,D)$, the b-codiscrepancy of $(X,D)$.
We can formally extend this correspondence to all
b-$\R$-divisors $\D_Z$ of $Z$: the {\em adjunction correspondent\/}
to $\D_Z$ on $X$ is the b-$\R$-divisor $\D'$ of $X$
such that
\begin{description}

\item[\rm (1)\/]
$\D'\hor=\D\hor$;
and

\item[\rm (2)\/]
$\D'\ver$ can be determined by
the adjunction correspondence of multiplicities
in~\ref{adjunction_mult}: for every vertical over $Z$
prime b-divisor $P$,
$$
\mult_P \D'=d_{P}'=r_P-l_P+l_Pd_Q \ \text{( cf.~(\ref{invers_cor_d}))}
$$
where $r_P,l_P$ are the adjunction constants of $(X,D)\to Z$
at $P$, $Q=f(P)$ and $d_Q=\mult_Q\D_Z$.

\end{description}
$\D'$ is a b-$\R$-divisor but it is not necessarily
satisfies BP as the b-discrepancy $\D$
(cf.~(3) below).
If we consider the correspondence for all
b-$\R$-divisors $\D_Z$ of $Z$ then their image consists of
b-$\R$-divisors $\D'$ of $X$ which satisfy (1)
of the definition and
$\D'-\D=f^*\D_Z'$ for some b-$\R$-divisor $\D_Z'$
of $Z$.
The pull back $f^*$ here is birational, that is,
for every model $f'\colon X'\to Z'$ of $X\to Z$,
\begin{equation}\label{adjunction_cor_div}
\D'_{X'}-\D_{X'}=f^*(\D_{Z,Z'}')
\end{equation}
generically over divisorial points of $Z$.
Such a divisor $\D_Z'$ is unique and
$=\D_Z-\D\dv$ by~(\ref{direct_cor_d}-\ref{invers_cor_d}).

Equivalently, we can define the inductive limit
for the correspondence between $\R$-divisors
in the opposite directions.
For every model $X'\to X$ of $X\to Z$, which
is isomorphic to $X\to Z$ generically over $Z$, and
every b-$\R$-divisor $\D'$ under above assumptions
(in the image), $(X',D')\to Z'$ is
a $0$-contraction as in~\ref{adjunction_0_contr}
generically over divisorial points of $Z'$,
where $D'=\D_{X'}'$.
Thus the $\R$-divisor $D\dv'$ is well-defined and
{\em adjunction corresponds\/} to the $\R$-divisor $D'$.
Since every prime b-divisor $Q$ is a divisor
on $Z'$ for an appropriate model $X'\to Z'$,
$\D\dv'$ is well-defined and adjunction corresponds to $\D'$:
$\D\dv'=\D_Z$.
We usually denote by $\D\dv'$ the adjunction correspondent
b-$\R$-divisor to $\D'$, that is, the {\em divisorial
part of adjunction\/} for $f\colon (X,\D')\to Z$.
According to the above (cf.~(\ref{adjunction_cor_div}))
\begin{equation} \label{adjunction_cor_b_div}
\D'=\D+f^*(\D\dv'-\D\dv).
\end{equation}
Equivalently, the adjunction correspondence can be defined
and determined by the {\em actual\/} adjunction formula
(cf.~(\ref{adjunction_f_b})):
\begin{equation}\label{adjunction_f_b_'}
\K+\D'\sim_\R f^*(\K_Z+\D\dv'+\sD\md),
\end{equation}
where $\K_Z$ denotes a canonical b-divisor of $Z$.
The moduli part of adjunction independent of $\D'$:
$\sD\md'=\sD\md$.
By definition, Corollary~\ref{adjunction_Q_horizontal} and
\ref{adjunction_index_of}, (1),
in the adjunction formula $\sim_\R$ can be replaced by
$\sim_\Q$, if $D\hor\in\Q$, and even by $\sim_I$,
if $(X,D)\to Z$ has an adjunction index $I$
(cf. Proposition~\ref{adjunction_same_mod_etc}, (2)).

A similar construction works a $0$-contraction
$(X,D+\sP)\to Z$ with
a bd-pair $(X,D+\sP)$ as in~\ref{adjunction_0_contr}.
We denote by $(\D'+\sP)\dv$ the b-$\R$-divisor
adjoint correspondent to $\D'$.
However, in this situation $\D'$ should satisfy
the following assumption: for every model
$X'\to Z'$ of $X\to Z$, which
is isomorphic to $X\to Z$ generically over $Z$,
$(X',D'+\sP)\to Z'$ is
a $0$-contraction as in~\ref{adjunction_0_contr}
generically over divisorial points of $Z'$,
where $D'=\D_{X'}'$
In particular, for $\D'=\D(X,D+\sP_X)-\sP$,
$(\D'+\sP)\dv=(\D(X,D+\sP_X))\dv=(D+\sP)\dv$ holds, that is,
the b-divisorial part of $(X,D+\sP)$ adjunction corresponds
to the b-divisorial part of adjunction of $(X,D+\sP)\to Z$.
In general
$$
\K+\D'+\sP\sim_\R
f^*(\K_Z+(\D'+\sP)\dv+(D+\sP)\md)
$$
with the moduli part $(\D'+\sP)\md=(\D(X,D+\sP_X))\md=(D+\sP)\md$
independent of $\D'$.

The correspondence for $\R$-divisors is considered
for any fixed model $X'\to Z'$ of $X\to Z$ and
generically over divisorial points.
In this situation $\D'$ satisfies BP (generically over divisorial points), e.g.,
by~(3) below because every $\R$-divisor in divisorial points
satisfies BP and additionally is $\R$-Cartier.

The correspondence has the following properties.

(1) {\em Injectivity.\/}
The adjunction correspondence of b-$\R$-divisors and
$\R$-divisors of $Z'$ is $1$-to-$1$ on its image in
b-$\R$-divisors and $\R$-divisors of $X'$ respectively.

\begin{proof}
Immediate by~(\ref{direct_cor_d}-\ref{invers_cor_d}).
\end{proof}

(2) {\em Linearity.} The correspondence is a affine $\R$-linear isomorphisms
of affine $\R$-spaces.

\begin{proof}
Immediate by~(\ref{adjunction_cor_b_div}).
\end{proof}

(3) {\em BP.\/}
$$
\D' \text{ satisfies BP }\Leftrightarrow
\text{ so does } \D\dv'.
$$
This gives a $1$-to-$1$ affine $\R$-linear isomorphism of
affine $\R$-spaces of b-$\R$-divisors under BP.
Recall that BP (boundary property) of a b-$\R$-divisor $\D$
of $X$ means that
there exists a model $X'$ of $X$ such that
$\D=\D(X',\D_{X'})$ \cite[p.~125]{Sh03}.
Equivalently, $\K+\D$ is b-$\R$-Cartier.
In particular, if $\K+\D\sim_\R 0$, b-$\R$-principal,
or $\sim_{\R,Z} 0$, relative b-$\R$-principal,
then $\D$ satisfies BP.

If $\D'$ satisfies BP than there exists a $0$-contraction
$(X',D')\to Z'$ as in~\ref{adjunction_0_contr}, with $D'=\D_{X'}'$
and $\D'=\D(X',D')$, which
generically over $Z'$ is isomorphic to $(X,D)\to Z$ and
hence is generically crepant to the latter.
The new contraction $(X',D')\to Z'$ gives the same
adjunction correspondence (see Proposition~\ref{adjunction_same_mod_etc}, (3)).
In particular, all facts that are stated for $\D$
can be applied to $\D'$ under BP with $(X',D')$
as in this paragraph.

\begin{proof}
On $X$ this is the affine $\R$-space of b-$\R$-divisors
$\D+\sC$, where $\sC$ is vertical b-$\R$-Cartier over $Z$,
that is, $\sC=f^*\sC_Z$
for some b-$\R$-Cartier $\sC_Z$ of $Z$.
So, on $Z$ this is the affine $\R$-space  of b-$\R$-divisors
$\D\dv+\sC_Z$, where
$\sC_Z$ is b-$\R$-Cartier.
In both cases BP follows from BP for $(X,D)$ and
$(Z,\D\dv)$ respectively.

The last statement about the correspondence means
the change of the origin in both affine $R$-spaces:
$\D',\D\dv'$ instead of $\D,\D\dv$ respectively
(cf.~(\ref{adjunction_cor_b_div})):
$$
\D'-\D=f^*(\D\dv'-\D\dv),
$$
where both differences are b-$\R$-Cartier.

\end{proof}

(4) {\em Monotonicity.}
$$
\D''\ge \D'\Leftrightarrow \D\dv''\ge\D\dv';
$$
and
$$
\D_{X'}''=D_{X'}''\ge D'\Leftrightarrow D\dv''{}_{,Z'}\ge D\dv'{}_{,Z'}
$$

\begin{proof}
Immediate by~(\ref{invers_cor_d}) and
the equation~(1) in the definition of the correspondence.
\end{proof}

(5) {\em Rationality.} If $D\hor$ is a $\Q$-divisor then
$$
\D'\text{ is a } \Q\text{-divisor }\Leftrightarrow
\D\dv' \text{ is a } \Q\text{-divisor};
$$
and
$$
D' \text{ is a } \Q\text{-divisor }\Leftrightarrow
D\dv'{}_{,Z'} \text{ is a } \Q\text{-divisor}.
$$
That is the correspondence is $\Q$-linear.

\begin{proof}
Immediate by~(\ref{direct_cor_d}).
Indeed, by definition
every constant $r=r_P\in\Q$ in this case
(cf. Addendum~~\ref{adjunction_index_hor},
Corollary~\ref{adjunction_Q_horizontal} and
\cite[Lemma~7.4, (iv)]{PSh08}).
\end{proof}

(6) {\em Klt, lc and nonlc.\/} (Cf. the strict $\delta$-lc
property in Definition~\ref{strict_delta_lc}.)
$$
(X,\D') \text{ is lc (respectively nonlc) }
\Leftrightarrow
(Z,\D\dv') \text{ is lc (respectively nonlc)};
$$
$$
(X',D') \text{ is lc (respectively nonlc) }
\Leftrightarrow
(Z',D\dv'{}_{,Z'}) \text{ is lc (respectively nonlc)};
$$
$$
(X,\D') \text{ is klt over } Z
\Leftrightarrow
(Z,\D\dv') \text{ is klt};
$$
$$
(X',D') \text{ is klt over } Z'
\Leftrightarrow
(Z',D\dv'{}_{,Z'}) \text{ is klt}.
$$
If $(X,\D')$ is klt generically over $Z$,
that is, $\D\hor$ is klt (horizontally), then
$$
(X,\D') \text{ is klt }
\Leftrightarrow
(Z,\D\dv') \text{ is klt};
$$
$$
(X',D') \text{ is klt }
\Leftrightarrow
(Z',D\dv'{}_{,Z'}) \text{ is klt}.
$$

The lc (klt) property here is formal:
for all prime b-divisors $P$ of $X$,
$\mult_P\D'\le1$ (respectively $<1$)
(cf. \cite[Example~1.1.2]{Sh96}).

\begin{proof}
Immediate by~\ref{direct_lc}-\ref{invers_lc}
and~(\ref{direct_cor_d}-\ref{invers_cor_d}).
By the equation (1) and \ref{adjunction_0_contr}, (1),
it is enough to consider only
vertical prime b-divisors $P$ of $X$.
In this case, by~\ref{adjunction_mult},
$b=d_P\le r=r_P$ ($<r=r_P$) means
that $(X',D_{X'}')$ is lc (respectively klt over $Z'$).

\end{proof}

(7) $\sim_{\R,S} 0$. For a proper morphism $Z\to S$ to a scheme $S$,
$$
\K+\D\sim_{\R,S} 0\Leftrightarrow \K_Z+\D\dv+\sD\md\sim_{\R,S} 0;
$$
or equivalently,
$$
K+D\sim_{\R,S} 0\Leftrightarrow K_Z+D\dv+\sD\md\sim_{\R,S} 0;
$$
where we can replace $X\to Z$ by any its model $X'\to Z'/S$
and $\sim_{\R,S}$ by $\equiv$ over $S$ ($\sim_{\Q,S},\sim_{I,S}$ if
respectively $D\hor\in\Q$,
the $0$-contraction $(X,D)\to Z$ has the adjunction index $I$).

We can replace $\sim_{\R,S} 0$ by the $\R$-free or -antifree property over $S$.
Respectively, we can replace the $\R$-free or -antifree property over $S$
by the nef or antinef property over $S$,
by $\Q$($I$)-free or -antifree property over $S$ if $D\hor\in\Q$
(respectively, $(X,D)\to Z$ has the adjunction index $I$).

\begin{proof}
Immediate by the adjunction formula~(\ref{adjunction_f_b_'})
and a relative over $S$ $\R$-version of~\cite[Proposition~3]{Sh19}.
Notice also that $(X',D_{X'})\to Z'$ with $D_{X'}=\D_{X'}$
is a $0$-contraction as in~\ref{adjunction_0_contr},
except for~(2) in \ref{adjunction_0_contr} but
it is crepant to $(X,D)\to Z$ and satisfies
the same adjunction properties by Proposition~\ref{adjunction_same_mod_etc}.
\end{proof}

(8) $\R$-{\em complements\/}.
For a proper morphism $Z\to S$ to a scheme $S$,
\begin{align*}
(X/S,\D^+)\text{ is a (b-)}&\R\text{-complement of }(X/S,\D) \\
&\Updownarrow\\
(Z/S,\D\dv^++\sD\md)\text{ is a (b-)}&\R\text{-complement of }(Z/S,\D\dv+\sD\md).
\end{align*}
So,
$$
(X/S,\D)\text{ has a (b-)}\R\text{-complement}
\Leftrightarrow
(Z/S,\D\dv+\sD\md)\text{ has a (b-)}\R\text{-complement}.
$$
Equivalently, for $\R$-divisors,
\begin{align*}
(X/S,D^+)\text{ is an }&\R\text{-complement of }(X/S,D) \\
&\Updownarrow\\
(Z/S,D\dv^++\sD\md)\text{ is an }&\R\text{-complement of }(Z/S,D\dv+\sD\md);
\end{align*}
$$
(X/S,D)\text{ has an }\R\text{-complement}
\Leftrightarrow
(Z/S,D\dv+\sD\md)\text{ has an }\R\text{-complement}.
$$

\begin{proof}
Immediate by definition (see also Remark~\ref{rem_def_b_n_compl}, (2)),
(4) and (6-7).
Notice that $(X,D)$ is a log pair by~\ref{adjunction_0_contr}, (3)
and $(Z,D\dv+\sD\md)$ is a log bd-pair by construction.

\end{proof}

For $n$-complements, we have only inverse results~(10) below
for b-$n$-complements.

(9) {\em Inverse inequality\/}~(1-b) {\em of Definition\/}~\ref{b_n_comp}.
Suppose additionally that $(X,D)\to Z$
has an adjunction index $I$.
Let $n$ be a positive integer and $I|n$.
Then
\begin{align*}
\D^+\text{ is lc and satisfies (1-b) of} &\text{ Definition~\ref{b_n_comp}
with respect to } \D\\
&\Uparrow\\
\D\dv^+\text{ satisfies (1-b-2) of} &\text{ Definition~\ref{bd_n_comp}
with respect to } \D\dv.
\end{align*}
Equivalently,
\begin{align*}
(X,D^+)\text{ satisfies (1-b-2) of} &\text{ Definition~\ref{b_n_comp}
with respect to } (X,D)\\
&\Uparrow\\
(Z,D\dv^++\sD\md)\text{ satisfies (1-b-2) of} &\text{ Definition~\ref{bd_n_comp}
with respect to } (Z,D\dv+\sD\md).
\end{align*}
Additionally,
\begin{align*}
\D^+ &\ge \D\\
&\Updownarrow\\
\D\dv^+&\ge \D\dv.
\end{align*}

\begin{proof}
Immediate by~\ref{invers_1_n_comp} and~\ref{adjunction_mult}.
Indeed, for every vertical prime b-divisor $P$,
$Ir=Ir_P\in \Z$ by~(4) in~\ref{adjunction_index_of} and
$r\le 1$ by \ref{adjunction_mult}.
For corresponding $Q=f(P)$,
$d^+=d_Q^+=\mult_Q\D\dv^+\le 1$ by (2) of Definition~\ref{bd_n_comp}.
This and~\ref{adjunction_mult} prove (1-b) of Definition~\ref{b_n_comp}
and that $\D^+$ is lc for vertical $P$.
For a horizontal prime b-divisor $P$,
we can use the equation (1): $\D^+\hor=\D\hor$.
Then $\D^+\hor$ is lc by (1) in~\ref{adjunction_0_contr} and
satisfies (1-b) of Definition~\ref{b_n_comp}
with respect to $\D\hor=\D^+\hor$
by  the equation (1), Example~\ref{b_n_comp_of_itself}, (1)
and (2) in~\ref{adjunction_index_of}.
The additional equivalence of
the monotonicity of b-$n$-complements follows from (4).

Notice also that $(X,D)$ is a log pair by~\ref{adjunction_0_contr}, (3)
and $(Z,D\dv+\sD\md)$ is a log bd-pair by construction.

\end{proof}

(10) {\em Inverse\/} b-$n$-{\em complements\/}.
Under the assumptions of~(9),
for a proper morphism $Z\to S$ to a scheme $S$,
\begin{align*}
(X/S,\D^+)\text{ is a (b-)}&n\text{-complement of }(X/S,\D) \\
&\Uparrow\\
(Z/S,\D\dv^++\sD\md)\text{ is a (b-)}&n\text{-complement of }(Z/S,\D\dv+\sD\md).
\end{align*}
So,
$$
(X/S,\D)\text{ has a (b-)}n\text{-complement}
\Leftarrow
(Z/S,\D\dv+\sD\md)\text{ has a (b-)}n\text{-complement}.
$$
Equivalently,
\begin{align*}
(X/S,D^+)\text{ is a b-}&n\text{-complement of }(X/S,D) \\
&\Uparrow\\
(Z/S,D\dv^++\sD\md)\text{ is a b-}&n\text{-complement of }(Z/S,D\dv+\sD\md);
\end{align*}
$$
(X/S,D)\text{ has a b-}n\text{-complement}
\Leftarrow
(Z/S,D\dv+\sD\md)\text{ has a b-}n\text{-complement}.
$$
Moreover, if the b-$n$-complement
$(Z/S,D\dv^++\sD\md)$ is monotonic, equivalently,
$(Z/S,\D\dv^++\sD\md)$ is monotonic, then so does
$(X/S,D^+)$, equivalently, $(X/S,\D^+)$.

\begin{proof}
Immediate by~(4),(6) and~(9).
As a definition in (8-10) for b-$\R$-divisors
we use Definitions~\ref{b_n_comp} and \ref{bd_n_comp}
with $\D,\D^+,\D\dv+\sD\md,\D\dv^++\sD\md$
instead of $D,D^+,D\dv+\sD\md,D\dv^++\sD\md$.
\end{proof}

However, in applications we use
a more technical result,~Theorem~\ref{invers_b_n_comp} below and
its addenda.

(11) {\em Direct effectivity.\/}
The vertical effective (b-)support of $\D'$ goes
to the effective (b-)support of $\D\dv'$, that is,
for every vertical prime b-divisor $P$ of $X$,
$$
\mult_P\D'\ge 0\Rightarrow
\mult_Q\D\dv'\ge 0,
$$
where $Q=f(P)$.
In particular,
$$
D_{X'}'\ge 0\Rightarrow
D\dv'{}_{,Z'}\ge 0.
$$

\begin{proof}
Immediate by~\ref{direct_positivity} and~(\ref{direct_cor_d}).
Indeed, every $r=r_P\le 1$ by~\ref{adjunction_mult}.
\end{proof}

(12) {\em Direct boundary property.\/}
Suppose additionally that $(X,\D')$ is lc or has an $\R$-complement
over some scheme $S$.
Then the vertical boundary (b-)support of $\D'$ goes
to the boundary (b-)support of $\D\dv'$, that is,
for every vertical prime b-divisor $P$ of $X$,
$$
\mult_P\D'\in [0,1]\Rightarrow
\mult_Q\D\dv'\in [0,1],
$$
where $Q=f(P)$.
In particular,
$$
D_{X'}'\in [0,1]\Rightarrow
D\dv'{}_{,Z'}\in [0,1].
$$

\begin{proof}
Immediate by~(6) and~(11) (cf.~\ref{direct_boundary}).
Note also that the existence of an $\R$-complement
implies that $(X,\D')$ is lc (cf. Remark~\ref{remark_def_complements}, (1)).
\end{proof}

Converses of the directed properties do not hold in general.

(13) {\em bd-Pairs.\/}
All the above properties of the adjunction correspondence
hold for a $0$-contraction with a bd-pair $(X,D+\sP)$ as in~\ref{adjunction_0_contr}.
Notice that the correspondence is only between
(b-)-divisorial parts, $\sP$ and $\sD\md$ are fixed
(modulo $\sim_\R$ or $\sim_\Q,\sim_I$ if
$\sP$ is b-$\Q$-Cartier,$I\sP$ is b-Cartier respectively).

For a bd-pair $(X,\D+\sP)$ in~(3), (bd-)BP means that
there exists a model $(X',\D_{X'}+\sP)$ which is a log bd-pair
and $\D=\D(X',\D_{X'}+\sP_{X'})-\sP$, the divisorial part
of the latter bd-pair.

In (5) to add that $\sP$ is b-$\Q$-Cartier.

In (6-7) $\D+\sP,D+\sP,\D'+\sP,D'+\sP$
should be instead of $\D,D,\D',D'$ respectively.
Similarly, in other properties: sometimes
$+\sP$ is needed, sometimes not.

\begin{thm} \label{invers_b_n_comp}
Let $I,n$ be two positive integers such that $I|n$ and
$\Phi,\Phi'$ be two hyperstandard sets as in~\ref{direct_hyperst}
with $\fR''=[0,1]\cap (\Z/I)$.
Let $(X,B)\to Z/S$ be a $0$-contraction as in~\ref{adjunction_0_contr}
and additionally of the adjunction index $I$, with a boundary $B$ and
a proper morphism $Z\to S$ to a scheme $S$.
Suppose also that $(X,B)$ is lc or has an $\R$-complement over $S$, and
$Z/S$ has a model $Z'/S$
and an $\R$-divisor $D\dv'$ on $Z'$ such that
\begin{description}

\item[\rm (1)\/]
for every vertical over $Z$ prime divisor $P$ on $X$,
the image $Q=f(P)$ of $P$ on $Z'$ is a divisor and
$$
d\dv'{}_{,Q}\ge b_{Q,n\_\Phi'}
$$
where
$d\dv'{}_{,Q}=\mult_Q D\dv'$ and
$b_Q=\mult_Q \B\dv$;

\item[\rm (2)\/]
$K_{Z'}+D\dv'+B\md{}_{,Z'}$ is antinef over $S$; and

\item[\rm (3)\/]
$(Z'/S,D\dv'+\sB\md)$ has
a b-$n$-complement $(Z'/S,D\dv^++\sB\md)$.

\end{description}
Then
$(X/S,D^+)$ is respectively a b-$n$-complement of
$(X^\sharp/S,B_{n\_\Phi}{}^\sharp{}_{X^\sharp})$
and an $n$-complement of $(X/S,B)$,
where $D^+=\D_X^+$ and $\D^+$ adjunction corresponds to $\D\dv^+$.

\end{thm}

\begin{add} \label{invers_b_n_comp_monotonic}
If additionally
\begin{description}

\item[\rm (4)\/]
$$
d\dv^+{}_{,Q}\ge b_{Q,n\_\Phi'}
$$
for every $Q$ as in~(1) of the theorem,
where $d\dv^+{}_{,Q}=\mult_Q D\dv^+$.

\end{description}
Then we can take $D\dv'=D\dv^+$ and
the monotonicity
\begin{description}

\item[\rm (5)\/]
$D^+\ge B_{n\_\Phi}{}^\sharp\ge B_{n\_\Phi}$

\end{description}
holds.
For this addendum, it is enough to be an $n$-complement in~(3).
\end{add}

\begin{add}
The same holds for every contraction
$(X,B+\sP)\to Z/S$ as in~\ref{adjunction_0_contr}
and additionally of the adjunction index $I$ with
a bd-pair $(X/S,B+\sP)$ of index $m|I$,
with a boundary $B$ and proper $Z\to S$.

\end{add}

\begin{proof}

Step~1. $b_Q\in [0,1]$.
Immediate by~(12).
Thus the assumption~(1) of the theorem is meaningful.

Step~2. $(X/S,\D^+)$ is
a b-$n$-complement of $(X/S,\D')$,
where $\D'$ adjunction corresponds to
$\D\dv'=\D(Z',D\dv'+B\md{}_{,Z'})-\sB\md$.
Immediate by (3) and the property~(10),
Inverse b-$n$-complements.
However to apply~(10) note that
by definition, the BP property~(3) of
the adjunction correspondence and the assumption (2) of the theorem
$\D',\D\dv'$ satisfy BP and
there exists generically crepant model $(X',D')\to Z'/S$
of $(X,B)\to Z/S$ with $D'=\D'_{X'},\D'=\D(X',D')$ and
of the same adjunction index $I$ by~Proposition~\ref{adjunction_same_mod_etc}, (5).
But $D'$ and $D\dv'$ are not necessarily boundaries.

Step~3. $\B_{n\_\Phi}{}^\sharp\le\D'$ by (1-2) and
Corollary~\ref{monotonicity_III}.
Indeed, the assumption (1) of the corollary
follows from the assumption~(2) of the theorem by
the property~(7) of the adjunction correspondence
with b-antinef $\K_{Z'}+\D\dv'+B\md{}_{,Z'}$ over $S$.
The BP b-$\R$-divisor $\D'$ corresponds to $\D\dv'$.

We take a maximal model $(X^\sharp/S,B_{n\_\Phi}{}^\sharp{}_{X^\sharp})$
of Construction~\ref{sharp_construction}
as a log pair $(X/S,D)$ of Corollary~\ref{monotonicity_III}.
We suppose also that $X\dashrightarrow X^\sharp$
is a birational $1$-contraction, that is,
does not blow up divisors, and
$B_{n\_\Phi}{}^\sharp{}_{X^\sharp}=B_{n\_\Phi,X^\sharp}$
has the same divisorial multiplicities as $B_{n\_\Phi}$.
By construction $(X^\sharp/S,B_{n\_\Phi}{}^\sharp{}_{X^\sharp})$
is a log pair and $\B_{n\_\Phi}{}^\sharp=
\D(X^\sharp,B_{n\_\Phi}{}^\sharp{}_{X^\sharp})$.
The inequality (2) of the corollary
\begin{equation}\label{n_Phi_sharp}
B_{n\_\Phi}{}^\sharp{}_{X^\sharp}\le \D_{X^\sharp}'
\end{equation}
follows from the equation~(1) in the definition
of the adjunction correspondence and
the assumption~(1) of the theorem by~\ref{n_Phi_inequality}.
Indeed,
by construction and definition $\D_{X^\sharp}'\hor=B_{X^\sharp}\hor$,
where $B_{X^\sharp}=\B_{X^\sharp}$.
Thus, for the horizontal part of~(\ref{n_Phi_sharp}) over $Z$,
$$
B_{n\_\Phi}{}^\sharp{}_{X^\sharp}{}\hor=
B_{n\_\Phi,X^\sharp}{}\hor\le
B_{X^\sharp}{}\hor=\D'_{X^\sharp}{}\hor.
$$

The vertical part of~(\ref{n_Phi_sharp}) means that
for every vertical over $Z$ prime divisor $P$ of $X^\sharp$,
\begin{equation}\label{n_Phi_sharp_b}
b_{P,n\_\Phi}=b_{P,n\_\Phi}{}^\sharp\le d_P',
\end{equation}
where $b_P=\mult_P B=\mult_P B_{X^\sharp},
b_{P,n\_\Phi}{}^\sharp=\mult_P B_{n\_\Phi}{}^\sharp{}_{X^\sharp}$
and $d_P'=\mult_P\D'$.
Note for this that by construction $P$ is also a divisor on $X$ and
$B_{n\_\Phi}{}^\sharp{}_{X^\sharp}=B_{n\_\Phi,X^\sharp}$.
Since $B$ is a boundary and
$(X,B)$ is lc or has an $\R$-complement over $S$,
$0\le b_P\le r_p\le 1$ by~\ref{adjunction_mult}.
Thus~(\ref{n_Phi_sharp_b}) is meaningful.
To verify~(\ref{n_Phi_sharp_b}) we apply~\ref{n_Phi_inequality}
to $b_1=b_P,r=r_P,l=l_P$ and $d'=\mult_Q D\dv'$,
where $Q=f(P)$.
The image $Q$ is a prime divisor on $Z'$
by~(1) of the theorem.
In this situation~(\ref{n_Phi'}) of~\ref{n_Phi_inequality}
means the inequality in~(1) of the theorem.
Indeed, by construction $d_1$ corresponds to $b_1=b_P$
by~(\ref{direct_cor}) with $r=r_P,l=l_P$ and
$r\in\fR''$ by \ref{adjunction_index_of}, (4)
and assumptions of the theorem.
Thus by~\ref{adjunction_mult} and construction
$d_1=b_Q$.
The other assumptions and notation of~\ref{n_Phi_inequality}
hold by assumptions and notation of the theorem.
By construction $b'$ corresponds to $d'=\mult_Q D\dv'$
by~(\ref{invers_cor}) with $r=r_P,l=l_P$.
Thus $d'=d_P'$ again by~\ref{adjunction_mult}.
Now since $b_1=b_P$ and $b'=d_P'$,
(\ref{n_Phi}) means exactly~(\ref{n_Phi_sharp_b}).

Step~4. Steps~2-3 and Proposition~\ref{D_D'_complement} imply that
$(X/S,\D^+)$ is
a (b-)$n$-complement of $(X/S,\B_{n\_\Phi}{}^\sharp)$.
Or, equivalently,
$(X/S,D^+)$ is
a b-$n$-complement of $(X^\sharp/S,B_{n\_\Phi}{}^\sharp{}_{X^\sharp})$,
actually,
of any maximal model of $(X/S,B_{n\_\Phi})$ in Construction~\ref{sharp_construction}.
(In particular, we can replace $X'/S$ by $X^\sharp/S$
for those b-$n$-complements.)
Additionally, $(X/S,D^+)$ is an $n$-complement of
$(X/S,B)$ by Example~\ref{b_n_comp_of_itself}, (1) and
Proposition~\ref{1_of def_2 for B}.

Step~5. (Addenda.)
In Addendum~\ref{invers_b_n_comp_monotonic}
we can replace $D\dv'$ by $D\dv^+$.
Indeed, (1) holds by~(4).
In~(2) $K_{Z'}+D\dv^++B\md{}_{,Z'}\sim_n 0/S$ and is antinef over $S$
by Definition~\ref{bd_n_comp}, (3).
And, finally, (3) holds by~(3) and~Example~\ref{b_n_comp_of_itself}, (2).
In this situation it is enough to be an $n$-complement in (3).

The monotonicity~(5) follows from Proposition~\ref{monotonicity_I} and
Step~3.
Indeed, in the addenda $\D'=\D^+$ adjunction corresponds
to $\D\dv'=\D^+$.

The proof of the addendum for bd-pairs,
is similar to the above proof for usual pairs.

\end{proof}

\subsection{Adjunction on divisor} \label{adjunction_on_divisor}

Under the divisorial adjunction $\Gamma(\sN,\Phi)$ goes to
$\Gamma(\sN,\widetilde{\Phi})$:
for every (finite) set of positive integers $\sN$ and
hyperstandard sets $\Phi,\widetilde{\Phi}$
as in~\ref{direct_hyperst_on_div},
\begin{equation}\label{Phi-to-wtPhi}
B\in\Gamma(\sN,\Phi)\Rightarrow
B_{S\nu}\in\Gamma(\sN,\widetilde{\Phi}),
\end{equation}
where $B$ is a boundary of a pair $(X,B+S)$,
lc in codimension $\le 2$,
$S$ is a reduced prime divisor of $B+S$,
$S^\nu$ is a normalization of $S$ and
$(S^\nu,B_{S^\nu})$ is the adjoint pair of $(X,B)$,
that is, $B_{S^\nu}=\Diff B$,
(see \cite[3.1]{Sh92}).
In particular, $S^\nu=S$ when $S$ is normal.
The property~(\ref{Phi-to-wtPhi}) is immediate
by~\cite[Corollary~3.10]{Sh92} and~\ref{direct_hyperst_on_div}.

\begin{thm} \label{extension_n_complement}
Let $n$ be a positive integer and $(X/Z\ni o,B)$ be
a pair with a boundary $B$ such that
\begin{description}

\item[\rm (1)\/]
$X/Z\ni o$ is proper with connected $X_o$;

\item[\rm (2)\/]
$(X,B)$ is plt with a reduced divisor $S$ of $B$
over $Z\ni o$;

\item[\rm (3)\/]
$-(K+B)$ is nef and big over $Z\ni o$;
and

\item[\rm (4)\/]
the adjoint pair $(S/Z\ni o,B_S)$ has
a b-$n$-complement $(S/Z\ni o,B_S^+)$.

\end{description}
Then there exists a b-$n$-complement $(X/Z\ni o,B^+)$ of
$(X/Z\ni o,B)$ with an adjunction extension $B^+$ of $B_S^+$, that is,
$\mult_S B^+=1$ and
\begin{description}

  \item[\rm (5)\/]
$$
\Diff(B^+-S)=B_S^+.
$$

\end{description}

The same holds for a bd-pair $(X/Z\ni o,B+\sP)$
of index $m|n$.

\end{thm}

Note that by (3) $S$ is actually unique and irreducible over $Z\ni o$
by its connectedness \cite[Connectedness Lemma~5.7]{Sh92} \cite[Theorem~17.4]{K}
\cite[Theorem~6.3]{A14} etc.
Moreover, the central fiber $S_o$ is also connected.

Remark: 1. $B$ is assumed to be a boundary, in particular, effective.

2. The b-$n$-complement on $X/Z\ni o$ is not
necessarily monotonic even if the b-$n$-complement in (4)
is monotonic.
However, for special boundary multiplicities we can
get the monotonic property a posteriori
\cite[Lemma~3.3]{PSh08} \cite[Theorem~1.7]{B}
\cite[Theorem~1.6]{HLSh}.

3. The theorem is the first main theorem (from global to local), that is,
similar to \cite[Proposition~6.2]{PSh01} \cite[Proposition~4.4]{P}
where the condition 4),(iv) respectively imply our (4).
However the proof is similar to the proof of \cite[Theorem~5.12]{Sh92}.

\begin{proof}

Step~1.
For construction of an $n$-complement we replace $(X,B)$ by a crepant pair $(Y,D)$.
We can suppose that $\Supp D$ has only normal crossings.
For this we can take a log resolution of $(X/Z\ni o,D)$.
Then one can verify all required assumptions (1-4),
except for, the boundary property of $D$.
In particular, (4) holds by definition (cf. Remark~\ref{rem_def_b_n_compl}, (1)).
This implies the existence of a required complement for
$(Y/Z\ni o,D)$ and also $(X/Z\ni o,B)$
with~(5) (cf. \cite[Lemma~5.4]{Sh92}).
For this we can use the adjunction formula \cite[3.1]{Sh92}
because birational modifications of pairs and their complements
are crepant and b-$n$-complements agree with
those modifications by Remark~\ref{rem_def_b_n_compl}, (1).

Note now that every $n$-complement in the normal crossing case is
actually a b-$n$-complement by the same arguments with sufficiently high crepant
modifications
or by Proposition~\ref{b_n_comp-nc}.

Step~2. Construction of $D^+$.
The complete linear system $\linsys{-nK_Y-\rddown{(n+1)D}+S}=
\linsys{-n(K_Y+S)-\rddown{(n+1)(D-S)}}$ on $Y$ cuts out
the complete linear system  $\linsys{-nK_S-\rddown{(n+1)D_S}}$ on $S$
over a suitable neighborhood of $o$ in $Z$.
(Tacit we use the same notation $S$ for its
birational transform on $Y$.)
Indeed,
$$
\rddown{(n+1)D_S}=\rddown{(n+1)(D-S)\rest{S}}=
\rddown{(n+1)(D-S)}\rest{S}
$$
by normal crossings.
By \cite[Corollary~3.10]{Sh92}
$D_S=(D-S)\rest{S}$ because $S$ is nonsingular.
The required surjectivity follows from the Kawamata-Viehweg vanishing theorem
$$
R^1\varphi_*\sO_Y(-nK_Y-\rddown{(n+1)D})=
R^1\varphi_*\sO_Y(K_Y+\rdup{-(n+1)(K+D)})=0
$$
by (3), where $\varphi\colon Y\to Z\ni o$.
So, for every effective divisor
$E_S\in \linsys{-nK_S-\rddown{(n+1)D_S}}$,
there exists an effective divisor $E\in \linsys{-nK_Y-\rddown{(n+1)D}+S}$ with
$E\rest{S}=E_S$.

By (2) the pair $(S/Z\ni o,D_S)$ is klt and by (4)
it has an $n$-complement $(S/Z\ni o,D_S^+)$.
So, by (1) and (3) of Definition~\ref{n_comp} (cf. also \cite[Definition~5.1]{Sh92}
$$
D_S^+=\frac {E_S}n+\frac{\rddown{(n+1)D_S}}n,
$$
where $E_S\in \linsys{-nK_S-\rddown{(n+1)D_S}}$.
Put
$$
D^+=\frac {E}n+\frac{\rddown{(n+1)D}}n-\frac Sn,
$$
where $E\in \linsys{-nK_Y-\rddown{(n+1)D}+S}$ with
$E\rest{S}=E_S$.

Actually, $(Y/Z\ni o,D^+)$ is an $n$-complement of $(Y/Z\ni o,D)$ and
induces the $n$-complement $(X/Z\ni o,B^+)$ of $(X/Z\ni o,B)$
with crepant $B^+=\psi(D^+)$, where $\psi\colon Y\to X$.
The $n$-complement properties (2-3) of Definition~\ref{n_comp}
are equivalent for both pairs (see Step~1)
and one of them we verify on $X$ (see Step~6 below).

Step~3. Adjunction:
$\mult_SD^+=1$, and by normal crossings and construction
\begin{align*}
\Diff(D^+-S)=&
(\frac {E}n+\frac{\rddown{(n+1)(D-S)}}n)\rest{S}=
\frac {E_S}n+\frac{\rddown{(n+1)(D-S)\rest{S}}}n=\\
&\frac {E_S}n+\frac{\rddown{(n+1)D_S}}n=D_S^+.
\end{align*}

Step~4. Since $E\ge 0$,
$$
D^+\ge
\frac{\rddown{(n+1)D}}n-\frac Sn.
$$
This implies (1) of Definition~\ref{n_comp}.

Step~5.
By construction in Step~2,
$$
nK_Y+nD^+=nK_Y+E+\rddown{(n+1)D}-S\sim 0.
$$
This implies (3) of Definition~\ref{n_comp}.

Step~6.
It is enough to verify (2) of Definition~\ref{n_comp} on $X$.
Notice that (5) holds by construction and Step~3 (see also Step~1).
Thus by inversion of adjunction \cite[3.3-4]{Sh92} \cite[Theorem~17.7]{K}
\cite[Theorem~6.3]{A14}
$(X,B^+)$ is lc near $S$.
The connectedness of the lc locus (if $(X,B^+)$ is not lc) and its
lc \cite[Theorem~6.3]{A14}
imply lc of $(X,B^+)$ over $Z\ni o$, that is,
(2) of Definition~\ref{n_comp}.
Note that in the local case for the lc connectedness we need
the connectedness of $\varphi\1o$ that holds by (1).

Step~7.
For a bd-pair $(X/Z\ni o,B+\sP)$, we can take in Step~1
a log resolution over which $\sP$ is stable and
has normal crossings together with $D$.
(Actually, according to
the Kawamata-Viehweg vanishing theorem,
we do not need the last normal crossings
because $n\sP$ is integral: $n\sP\in \Z$.)
We can assume also that
$\sP_Y$ is in general position (modulo $\sim_n$)
with respect to $S$ on $Y$ and
$\sP_Y\rest{S}$ is well-defined.
Then $\sP_S=\sP\brest{S}=\overline{\sP_X\rest{S}}$ is stable over $S$
(a definition $\brest{}$ see in \cite[Mixed restriction~7.3]{Sh03}) and
is nef Cartier on $S$ over $Z\ni o$
because $m|n$.
Then we can use above arguments with $K_Y,K_S$ replaced by
$K_Y+\sP_Y,K_S+(\sP_S)_S$ respectively (cf. general ideology in Crepant bd-models
in Section~\ref{technical}).
To extend a complement as in Step~2 we need only
the aggregated bd-version of (3): $-(K_Y+B+\sP_Y)$ is nef and big over $Z\ni o$
but the b-nef property of $\sP$ is not needed here.
However, to verify the lc property of Step~6
we need the lc connectedness (if $(X,B^++\sP)$ is not lc) which uses
the pseudoeffective modulo $\sim_\R$ property of $\sP_X$ over $Z\ni o$ that follows
from the b-nef property of $\sP$ over $Z\ni o$
(cf. \cite[Theorem~1.2]{FS}).

\end{proof}

\begin{cor} \label{plt_b_n_complement}
Let $n$ be a positive integer and
$\Phi=\Phi(\fR)$ be a hyperstandard set associated with
a (finite) set of (rational) numbers $\fR$ in $[0,1]$.
Let $(X/Z\ni o,B)$ be a pair such that
\begin{description}

  \item[\rm (1)\/]
$X/Z\ni o$ is proper with connected $X_o$;

  \item[\rm (2)\/]
$(X,B)$ is plt with a boundary $B$ and
a reduced divisor $S$ of $B$ over $Z\ni o$;

  \item[\rm (3)\/]
$-(K+B)$ is nef and big over $Z\ni o$;
and

  \item[\rm (4)\/]
$(S'^\sharp/Z\ni o,B_{S',n\_\widetilde{\Phi}}{}^\sharp)$
has a b-$n$-complement
$(S'^\sharp/Z\ni o,B_{S'^\sharp}^+)$,
where $(S'/Z\ni o,B_{S'})$ is a highest crepant model
with a boundary of the adjoint pair $(S/Z\ni o,B_S)$
(cf. Construction~\ref{adjoint_pair} and \cite[(3.2.3)]{Sh92}) and
$\widetilde{\Phi}$ is defined in~\ref{direct_hyperst_on_div}.

\end{description}
Then there exists a b-$n$-complement $(X/Z\ni o,B^+)$ of
$(X^\sharp/Z\ni o,{B_{n\_\Phi}}^\sharp)$
with an adjunction extension $B^+$ of $B_S^+$, that is,
$\mult_SB^+=1$ and
\begin{description}

  \item[\rm (5)\/]
$$
\Diff(B^+-S)=B_S^+=\B_{S'^\sharp,S}^+.
$$

\end{description}

The same holds for a bd-pair $(X/Z\ni o,B+\sP)$
of index $m|n$.

\end{cor}

\begin{proof}

Step~1. {\em Construction of $(S'/Z\ni o,B_{S'}),
(S'/Z\ni o,B_{S',n\_\widetilde{\Phi}})$.\/}
By~(2) the adjoint pair $(S/Z\ni o,B_S)$ is
a klt pair with normal and irreducible $S$ over $Z\ni o$.
It has finitely many prime b-divisors $P$ over $Z\ni o$
with $\mult_P\B_S\ge 0$.
A {\em highest\/} crepant model $(S',B_{S'})$ of $(S,B_S)$ with
a {\em boundary\/} $B_{S'}$ (equivalently, $B_{S'}\ge 0$)
blows up those exceptional $P$ \cite[Theorem~3.1]{Sh96} \cite[Theoem~1.1]{B12}.
Usually, the blowup is not unique (and
is not necessarily $\Q$-factorial).
It is defined up to a small flop over $Z\ni o$.

After that we change the boundary $B_{S'}$ and
get the pair $(S'/Z\ni o,B_{S',n\_\widetilde{\Phi}})$
with a boundary $B_{S',n\_\widetilde{\Phi}}$.
(By construction $(S',B_{S'})$ is a log pair.
But $(S'/Z\ni o,B_{S',n\_\widetilde{\Phi}})$ is
not necessarily a log pair when $S'$ is not $\Q$-factorial.)

In our applications we construct a maximal model
$(S'^\sharp/Z\ni o,B_{S',n\_\widetilde{\Phi}}{}^\sharp)$
using Construction~\ref{sharp_construction} (see
Step~5 in the proof Theorem~\ref{generic_complements}).
Actually in our situation we can construct
such a model too even $S/Z\ni o$ has not necessarily wFt.
(But $B,B_{S'},B_{S',n\_\widetilde{\Phi}}$ are boundaries and
$(S/Z\ni o,B_S),(S',B_{S'}),(S',B_{S',n\_\widetilde{\Phi}})$
have klt $\R$-complements by (1-3); cf. Step~2.)
However, in (4) we suppose that such a model
$(S'^\sharp/Z\ni o,B_{S',n\_\widetilde{\Phi}}{}^\sharp)$ exists
and it has a b-$n$-complement $(S'^\sharp/Z\ni o,B_{S'^\sharp}^+)$.

By construction $S,S',S'^\sharp$ are
birationally isomorphic.
So, (5) is meaningful.

Step~2. {\em Reduction to a plt wlF model
$(X^\sharp/Z\ni,B_{n\_\Phi}{}^\sharp)$.\/}
By~Construction~\ref{sharp_construction}
a required maximal model exists:
$$
\begin{array}{ccccc}
(X,B_{n\_\Phi})&\stackrel{\psi}{\dashrightarrow}&(X^\sharp,B_{n\_\Phi}{}^\sharp)\\
\searrow&&\swarrow\\
& Z\ni o
\end{array}
$$
where $\psi$ is a birational $1$-contraction and
$X^\sharp/Z\ni o$ has wFt with connected $X_o^\sharp$.
(For simplicity of notation we use $B_{n\_\Phi}{}^\sharp$
instead of $B_{n\_\Phi}{}^\sharp{}_{X^\sharp}$.)
Indeed, by (1) and (3), $X/Z\ni o$ has wFt with connected $X_o$.
Since $B_{n\_\Phi}\le B,\mult_S B_{n\_\Phi}=
\mult_S B=1$ and by (3),
$(X,B_{n\_\Phi})$ has a plt $\R$-complement
with the same reduced divisor $S$ of $B_{n\_\Phi}$.
In particular, $S$ is not the base locus of $-(K+B_{n\_\Phi})$.
Hence we can suppose that $(X^\sharp,B_{n\_\Phi}{}^\sharp)$
is also plt with birationally the same prime reduced divisor
$S^\sharp=\psi(S)\subset X^\sharp$ of $B_{n\_\Phi}{}^\sharp$.
By~Construction~\ref{sharp_construction} and the assumption (3) respectively,
$-(K_{X^\sharp}+B_{n\_\Phi}{}^\sharp)$ is $\R$-free and big over $Z\ni o$.
Thus $(X^\sharp/Z\ni,B_{n\_\Phi}{}^\sharp)$ is wlF
(a weak log Fano variety or space).

Notice that
we suppose that $\psi$ does nor blow up divisors.
Hence $B_{n\_\Phi}{}^\sharp\in\Gamma(n,\Phi)$.

Step~3. {\em Construction of a b-$n$-complement
$(X/Z\ni o,B^+)$.\/}
By birational nature of b-$n$-complements
(see Definition~\ref{b_n_comp} and Remark~\ref{rem_def_b_n_compl}, (1)),
a required b-$n$-complement $(X/Z\ni o,B^+)$ of
$(X^\sharp/Z\ni o,{B_{n\_\Phi}}^\sharp)$
can be induced from a b-$n$-complement $(X^\sharp/Z\ni o,B_{X^\sharp}^+)$
of $(X^\sharp/Z\ni o,B_{n\_\Phi}{}^\sharp)$, that is,
$B^+=\B_{X^{\sharp}}^+{}_{,X}$.
To construct the latter complement we apply
Theorem~\ref{extension_n_complement} to
$(X^\sharp/Z\ni,B_{n\_\Phi}{}^\sharp)$.
In Step~2 we verified assumptions (1-3) of the theorem.
The conclusion (5) of the theorem follows from
(5) of the corollary.
So, we need only to verify (4) of the theorem.

Step~4. {\em Construction of a b-$n$-complement of
$(S^\sharp/Z\ni o,B_{n\_\Phi}{}^\sharp{}_{S^\sharp})$.\/}
Actually the b-$n$-complement
$(S'^\sharp/Z\ni o,B_{S'{}^\sharp}^+)$
induces a b-$n$-complement $(S^\sharp,B_{S^\sharp}^+)$
of $(S^\sharp/Z\ni o,B_{n\_\Phi}{}^\sharp{}_{S^\sharp})$,
that is, $B_{S^\sharp}^+=B_{S'{}^\sharp}^+{}_{,S^\sharp}$.
By construction varieties $S'^\sharp$ and $S^\sharp$
are birational isomorphic over $Z\ni o$.
The (2-3) of Definition~\ref{n_comp} for
$(S^\sharp,B_{S^\sharp}^+)$ follows from that of
for $(S'^\sharp/Z\ni o,B_{S'{}^\sharp}^+)$.
Thus
it is enough to very (1-b) of~Definition~\ref{b_n_comp}.
This follows from Proposition~\ref{D_D'_complement},
applied to crepant models of
$(S^\sharp/Z\ni o,B_{n\_\Phi}{}^\sharp{}_{S^\sharp}),
(S'^\sharp/Z\ni o,B_{S',n\_\widetilde{\Phi}}{}^\sharp)$
on a common model $S''/Z\ni o$ of $S/Z\ni o$.
The required inequality for the statement is
$$
\B_{n\_\Phi}{}^\sharp{}_{S^\sharp,S''}\le
\B_{S',n\_\widetilde{\Phi}}{}^\sharp{}_{S''}.
$$
Its b-version
$$
\B_{n\_\Phi}{}^\sharp{}_{S^\sharp}\le
\B_{S',n\_\widetilde{\Phi}}{}^\sharp.
$$

By Corollary~\ref{monotonicity_III}, the b-version follows from
its divisorial version on $S^\sharp$:
\begin{equation}\label{sharp_tilde_Phi}
B_{n\_\Phi}{}^\sharp{}_{S^\sharp}\le
\B_{S',n\_\widetilde{\Phi}}{}^\sharp{}_{S^\sharp}.
\end{equation}
Indeed, by construction
$-(\K_{S^\sharp}+\B_{S',n\_\widetilde{\Phi}}{}^\sharp)=
-(\K_{S'^\sharp}+\B_{S',n\_\widetilde{\Phi}}{}^\sharp)$
is nef.
(For any rational differential form $\omega\not =0$ on $S/Z\ni o$,
$\K_{S^\sharp}=\K_{S'^\sharp}=(\omega)$, where the last canonical
divisor is treated as a b-divisor.)
This implies (1) of Corollary~\ref{monotonicity_III}.
The inequality~(\ref{sharp_tilde_Phi}) is (2)
of Corollary~\ref{monotonicity_III}.
By Step~2 $(S^\sharp/Z\ni o,B_{n\_\Phi}{}^\sharp{}_{S^\sharp})$
is a log pair with
$\D(S^\sharp,B_{n\_\Phi}{}^\sharp{}_{S^\sharp})=
\B_{n\_\Phi}{}^\sharp{}_{S^\sharp}$.

Again Corollary~\ref{monotonicity_III}, but
now applied to $(X^\sharp/Z\ni,B_{n\_\Phi}{}^\sharp)$
of Step~2,
implies
\begin{equation}\label{sharp_Phi}
\B_{n\_\Phi}{}^\sharp\le
\B.
\end{equation}

Indeed,$-(\K+\B)$ is nef by (3).
This gives (1) of Corollary~\ref{monotonicity_III}.
Construction in Step~2 gives (2) of Corollary~\ref{monotonicity_III}:
$\B_{X^\sharp}=B_{X^\sharp}\ge B_{n\_\Phi}{}^\sharp$.

The semiadditivity \cite[3.2.1]{Sh92} and~(\ref{sharp_Phi})
implies
\begin{equation}\label{sharp_Phi_S}
\B_{n\_\Phi}{}^\sharp{}_{S^\sharp}\le
\B_S.
\end{equation}

Finally, we verify~(\ref{sharp_tilde_Phi})
in every prime divisor $P$ on $S^\sharp$.
Note $P$ is also a divisor on $S'$
by the maximal property of the crepant model
$(S'/Z\ni o,B_{S'})$.
Indeed, if $P$ is a divisor of $S^\sharp$ then by~(\ref{sharp_Phi_S})
$$
0\le b_{n\_\Phi}{}^\sharp{}_{S^\sharp,P}=
\mult_P \B_{n\_\Phi}{}^\sharp{}_{S^\sharp}\le
\mult_P\B_S=b_{S,P}.
$$
Hence $P$ is a divisor on $S'$.
On the other hand,
$b_{n\_\Phi}{}^\sharp{}_{S^\sharp,P}\in\Gamma(n,\widetilde{\Phi})$
by~(\ref{Phi-to-wtPhi}).
Hence by definition and Proposition~\ref{monotonicity_I}

$$
b_{n\_\Phi}{}^\sharp{}_{S^\sharp,P}\le
b_{S,P,n\_\widetilde{\Phi}}=
\mult_P\B_{S',n\_\widetilde{\Phi}}
\le \mult_P\B_{S',n\_\widetilde{\Phi}}{}^\sharp,
$$
that is, (\ref{sharp_tilde_Phi}) in $P$.

bd-Pairs can be treated similarly.
In this case, a maximal crepant model with boundary
means that of only for the divisorial part $B_S$.
So, we take $(S'/Z\ni o,B_{S'}+\sP\brest{S'})$.
After that we take
$(S'/Z\ni o,B_{S',n\_\widetilde{\Phi}}+\sP\brest{S'})$ and,
finally,
$(S'/Z\ni o,B_{S',n\_\widetilde{\Phi}}{}^\sharp+\sP\brest{S'^\sharp})$.
By definition $\sP\brest{S}=\sP\brest{S'}=\sP\brest{S'^\sharp}$.

\end{proof}

\section{Semiexceptional complements:
lifting exceptional} \label{semiexcep_compl_:}

\paragraph{Semiexceptional pairs.}
They can be defined in terms of $\R$-complements.
A pair $(X,D)$ is called {\em semiexceptional\/}
if it has an $\R$-complement and
$(X,D^+)$ is an $\R$-complement of $(X,D)$ for every
$D^+$ on $X$ such that $D^+\ge D$ and $K+D^+\sim_\R 0$.
(In this situation $D^+$ is always a subboundary and
a boundary if so does $D$.)

The same definition works
for pairs $(X,D+\sP)$ with an arbitrary b-divisor $\sP$, in particular,
for bd-pairs $(X,D+\sP)$ (of index $m$).

\begin{const}[b-Contraction associated to
lc singularities] \label{b_fib_lc}
Let $(X,B)$ be a pair with a nonklt $\R$-complement $(X,B^+)$.
Then there exists a prime b-divisor $P$ such that $\ld(X,B^+;P)=0$,
equivalently, $\mult_P \B^+=1$, where $\B^+=\D(X,B)$ denotes
the codiscrepancy b-divisor of $(X,B)$.
Suppose also that $X$ has wFt and $B$ is a boundary.
Then there exists a b-{\em contraction associated\/} to $(X,B^+)$.
Such a b-contraction is a diagram
$$
(X,B^+)\stackrel{\varphi}{\dashleftarrow} (Y,B_Y^+)
\stackrel{\psi}{\to} Z,
$$
where $\varphi$ is a crepant birational modification and
$\psi$ is a contraction such that
\begin{description}

\item[\rm (1)\/]
$(Y,B_Y^+)$ is lc but nonklt;

\item[\rm (2)\/]
$\varphi$ blows up only prime b-divisors $P$ with $\ld(X,B^+;P)=0$,
equivalently, for every exceptional prime divisor $P$ of $\varphi$,
$\mult_{P}B_Y^+=1$;
and

\item[\rm (3)\/]
$$
B^+=B+\psi^*H \text{ and }
B_Y^+=B^{+,}\logb=B\logb+\psi^* H,
$$
where
$B,B\logb,B^+,B^{+,}\logb$ are respectively
birational and log birational transforms of $B,B^+$ from $X$ to $Y$ and
$H$ is an effective ample $\R$-divisor on $Z$.

\end{description}
For $\varphi$ we can take a dlt crepant blowup of $(X,B^+)$ \cite[Theorem~3.1]{KK}.
This gives (1-2) by construction.
We can suppose that $Y$ is $\Q$-factorial and has Ft,
in particular, projective.
($Z$ is also Ft by \cite[Lemma~2.8,(i)]{PSh08}.)
To satisfy (3) consider the effective $\R$-divisor $E=B^+-B=B^{+,\log}-B^{\log}$ on $Y$.
By construction $E$ is supported outside $\LCS(Y,B_Y^+)$, in particular,
does not contain prime divisors $P$ on $Y$ with $\mult_P B_Y^+=1$.
Moreover, the dlt property implies that $(Y,B_Y^++\ep E)$
is dlt for a sufficiently small positive real number $\ep$.
If $E$ is nef then $E$ is semiample \cite[Corollary~4.5]{ShCh}
and gives the required contraction
$\psi$ because $K_Y+B_Y^+\sim_\R 0$ and
$\ep E\sim_\R K_Y+B_Y^++\ep E$.

If $E$ is not nef we apply the LMMP to $(Y,B_Y^++\ep E)$ or
$E$-MMP to $Y$.
A resulting model is not a Mori fibration or
a fibration negative with respect to $E$, and
$E$ is nef on this model.
The model gives a required pair $(Y,B_Y^+)$ and
$\varphi$ is a birational modification of $(Y,B_Y^+)$ to $(X,B^+)$.
Note that the LMMP contracts only divisors supported on $\Supp E$
and is crepant with respect to $B_Y^+$.
Thus (1-2) hold for constructed $\varphi$.
By construction
\begin{description}

\item[\rm (4)\/]
$(Y/Z,B\logb)$ is a $0$-pair, in particular,
$(Y,B\logb)$ is a log pair.

\end{description}
Indeed, $\psi^*H$ is vertical and $\sim_\R 0$ over $Z$.

We can use also Construction~\ref{sharp_construction} to construct $(Y,B\logb)$
from $(Y,B\logb)$ of the dlt blowup.
Thus the model $(Y/Z,B\logb)$ up to a crepant modification is
defined by the dlt blowup and even by its exceptional divisors.

Since $K_Y+B_Y^+\sim_\R 0$,
we can associate to $Z$ an adjoint boundary
$B_Z^+=B\dv^++B\md^+$ on $Z$  \cite[Constructions~7.2 and 7.5, Remark~7.7]{PSh08}.
If we treat $B_Z^+$ as a usual boundary then we suppose that
$B\md^+$ is the trace on $Z$ of a sufficiently general effective b-divisor $\sB\md^+$ of $Z$
as in Conjecture~\ref{mod_part_b-semiample}.
In this case the pair $(Z,B_Z^+)$ is a usual lc pair with a boundary
(cf. Corollary~\ref{Alexeev_index_m} below).
However, to avoid Conjecture~\ref{mod_part_b-semiample}
we usually treat $B\md^+=\sB\md^+{}_{,Z}$ as
the trace of a nef b-divisor $\sB\md^+$ on $Z$.
In this case the pair $(Z,B\dv^++\sB\md^+)$ is an adjoint bd-pair
with b-nef $\sP=\sB\md^+$ (see bd-Pairs in Section~\ref{technical} and
cf. Warning below).
The pair is a log bd-pair and actually a $0$-bd-pair.
Note also that by (4) we can also associate
to the pair $(Y,B^{\log})$ an adjoint boundary $B_Z^{\log}=B\dv^{\log}+B\md^{\log}$.
The corresponding adjoint bd-pair is $(Z,B\dv^{\log}+\sB\md^{\log})$
with the same moduli b-part $\sB\md^{\log}=\sB\md^+$.
Indeed, by~\cite[Lemma~7.4(ii), Construction~7.5]{PSh08}
or by Proposition~\ref{adjunction_same_mod_etc} and \ref{adjunction_div}
$$
B\dv^+=B\dv^{\log}+H, B\md^+=B\md^{\log}
\text{ and } B_Z^+=B_Z^{\log}+H.
$$
So, $(Z,B_Z^+)$ is an $\R$-complement of $(Z,B_Z^{\log})$.
Respectively, for bd-pairs, $(Z,B\dv^++\sB\md^+)$ is
an $\R$-complement of $(Z,B\dv^{\log}+\sB\md^{\log})$.
By construction $(Z,B_Z^{\log})$ is a log Fano pair with
a log Fano bd-pair $(Z,B\dv^{\log}+\sB\md^{\log})$.

Warning: bd-Pairs $(Z,B\dv^++\sB\md^+),(Z,B\dv^{\log}+\sB\md^{\log})$
have lc singularities by definition.
However, corresponding pairs $(Z,B_Z^+),(B,B_Z^{\log})$ have lc singularities only
for an appropriate choice of the moduli part $\sB\md^+=\sB\md^{\log}$
(cf. \cite[Corollary~7.18(ii)]{PSh08}).

The bd-pair $(Z,B\dv^{\log}+\sB\md^{\log})$ actually depend on the complement $(X,B^+)$ and
on its dlt blowup $(Y,B_Y^+)$; actually on
the exceptional divisors of the blowup.
If we replace in our construction $(Y,B_Y^+)$ by any crepant model $(V,B_V^+)$ of $(X,B)$,
possibly with a subboundary $B_V^+$, with a contraction to $Y$ and to $Z$.
Then the difference $B^+-B=B^{+,\log}-B^{\log}$
(of birational and log birational transforms) usually is not semiample and
even with fixed components.
Moreover, the horizontal difference over $Z$ has only those fixed components and
is exceptional on $Y$, in particular, on the generic fiber of $\psi$.

The same construction works for a bd-pair $(X,B+\sP)$
with a nonklt $\R$-complement $(X,B^++\sP)$ and
such that
\begin{description}

\item[]
$X$ has wFt;

\item[]
$\sP$ is pseudoeffective modulo $\sim_\R$;

\item[]
$B$ is a boundary and

\item[]
$(X,B+\sP)$ has a nonklt $\R$-complement $(X,B^++\sP)$.

\end{description}
In this case and in the case for usual pairs it
is better to apply $E$-MMP.
To construct a dlt Ft blowup we need $\sP$ to be pseudoeffective modulo $\sim_\R$.
For adjoint boundaries see~\ref{adjunction_0_contr} and
(12-13) in \ref{adjunction_div}.

By construction and Proposition~\ref{small_transform_compl}, the b-contraction is
invariant of small birational
modifications of pairs $(X,B)$ or bd-pairs $(X,B+\sP)$.

\end{const}

\begin{prop} \label{horizontal_lc}
Let $(X,B)$ be a semiexceptional but not exceptional pair,
with wFt $X$ and a boundary $B$.
Then Construction~\ref{b_fib_lc} is applicable to $(X,B)$.

Every b-contraction of $(X,B)$ is not birational, that is,
$\dim Z<\dim Y=\dim X$.
Every lc centers and
prime b-divisors $P$ with $\ld(P;X,B^+)=0$ are horizontal with
respect to the b-contraction, that is, for every b-divisor $P$ such that
$\mult_P\B^+=1$, $\psi\cent P=Z$.

Pairs $(Z,B\dv^++\sB\md^+),(Z,B\dv^{\log}+\sB\md^{\log})$ are a klt $0$-bd-pair and
a klt log Fano bd-pair respectively.
The bd-pair $(Z,B\dv^++\sB\md^+)$ is exceptional and
$(Z,B\dv^{\log}+\sB\md^{\log})$ is semiexceptional.

\end{prop}

\begin{add}
The same holds for semiexceptional but not exceptional
bd-pairs $(X,B+\sP)$
with wFt $X$, a boundary $B$ and
a pseudoeffective modulo $\sim_\R$ b-divisor $\sP$.
\end{add}

\begin{proof}
The construction is applicable because $(X,B)$ has
an $\R$-complement.

According to Construction~\ref{b_fib_lc}, if $\psi$ is birational then
$B^+-B$ is big on any dlt blowup $\varphi\colon (Y,B_Y^+)\to (X,B^+)$.
So, there exists an effective divisor $E\sim_\R B^+-B$ on $Y$
such that $(Y,B_Y')$ is nonlc, where $B_Y'=B_Y^+-B^++B+E$.
Since $B_Y^+\sim_\R B_Y^+-B^++B+E$ and $\varphi (B_Y')=B'=B+\varphi(E)$ on $X$,
$(X,B')$ is a noncl $\R$-complement of $(X,B)$.
This contradicts to the semiexceptional property of $(X,B)$.

Similarly, if $P$ is an lc prime b-divisor of $(X,B^+)$ such that
$\psi\cent P$ is proper in $Z$, then there exists an effective divisor
$E\sim_\R B^+-B$ passing through $P$.
Moreover, by construction and our assumptions $Z$ is complete and
proper over $S=\pt$
This leads again to a contradiction by~(6)
and~(8) of~\ref{adjunction_div}.
 Hence by \cite[Lemma~7.4(iii)]{PSh08} and construction
$(Z,B\dv^++\sB\md^+))$ is a klt $0$-bd-pair.
Since $H$ is effective and ample,
$(Z,B\dv^{\log}+\sB\md^{\log})$ is a klt log Fano bd-pair.
Since $(Z,B\dv^++\sB\md^+))$ is a klt $0$-bd-pair,
it is exceptional.
The bd-pair $(Z,B\dv^{\log}+\sB\md^{\log})$ is
semiexceptional again by~(6) and~(8)
of~\ref{adjunction_div}
(cf. the proof of Corollary~\ref{excep_base}).

Similarly we treat bd-pairs.
In this case $(B^{\log}+\sP)\dv^+{}_{,Z}$ is a usual $\R$-divisor on $Z$
and $(B^{\log}+\sP)\md^+$ is a nef b-divisor of $Z$
(see~\ref{adjunction_0_contr}).

\end{proof}

Since $\dim Z<\dim X$ and
$(Z,B\dv+\sB\md^{\log})$ is semiexceptional
we can use the dimensional induction
to construct $n$-complements of
semiexceptional pairs $(X,B)$ and bd-pairs $(X,B+\sP)$.
Finally, this reduces construction of
semiexceptional $n$-complements to exceptional ones.
Any $0$-dimensional pair is exceptional.
We prefer a direct reduction to the exceptional case.

\paragraph{Semiexceptional type.}
Let $(X,B)$ be a semiexceptional pair under
the assumptions of Construction~\ref{b_fib_lc} and
$(X,B^+)\dashrightarrow Z$ be a b-contraction
associated a nonklt complement $(X,B^+)$.
The b-contraction has
a pair $(r,f)$ of invariants, where
$r=\reg(X,B^+)=\dim\Reg(X,B)$ \cite[Proposition-Definition~7.9]{Sh95} is
the {\em regularity\/} of $(X,B^+)$, characterising
topological depth [difficulty] of lc singularities of $(X,B^+)$, and
$f$ is the dimension of $Z$, of the base of b-contraction.
We order these pairs lexicographically:
$$
(r_1,f_1)\ge(r_2,f_2) \text{ if }
\begin{cases}
r_1> r_2;\text{ or}\\
r_1=r_2 \text{ and } f_1\ge f_2.
\end{cases}
$$

A {\em maximal} b-contraction is a largest one
with respect to $(r,f)$, that is, of the largest regularity and
the largest dimension of the base for
such a regularity.
The {\em semiexceptional\/} type of $(X,B)$ is
such a largest pair $(r,f)$.
[
If $(X,B)$ is exceptional then $r$ is not defined [or $=-\infty$] and
$f=\dim X$ is possible (see Semiexceptional filtration below).
]

For pairs of dimension $d$ not all invariants
$0\le r\le d-1,0\le f\le d$ are possible.
E.g., the top type in this situation is $(d-1,0)$ according
to Addendum~\ref{reg_dim_fiber} below.

Similar notion can be applied to a semiexceptional bd-pair $(X,B+\sP)$
under the assumptions of Construction~\ref{b_fib_lc}.

\begin{cor} \label{reg_semiexceptional type}
Under the assumptions and notation of Construction~\ref{b_fib_lc},
$\reg(X,B^+)$ is the same topological invariant as
of a sufficiently general or generic fiber of $(Y,B_Y^+)$
over $Z$: for the generic point $\eta$ of $Z$,
$$
\reg(X,B)\ge \reg(X,B^+)=\reg(Y_\eta,B_\eta^+).
$$
\end{cor}

\begin{add} \label{reg_dim_fiber}
$\reg(X,B^+)\le\dim Y_\eta-1=\dim X-\dim Z -1$, or
$\dim Z\le \dim X-\reg(X,B^+)-1$.
\end{add}

\begin{add}
The same holds for bd-pairs.
\end{add}

\begin{proof}
Immediate by Proposition~\ref{horizontal_lc} and
definition because all lc centers are horizontal.

Addendum~\ref{reg_dim_fiber} follows from
the general fact that, for every lc pair $(X,B)$,
$\reg(X,B)\le\dim X -1$.

For bd-pairs notice only that
$\reg(X,B^++\sP)$ can be defined as for usual pairs.
\end{proof}

\begin{cor} \label{excep_base}
Let $(X,B)$ be a semiexceptional pair under
the assumptions of Construction~\ref{b_fib_lc} and
$(X,B^+)\dashrightarrow Z$ be a maximal b-contraction
with respect to $r$ associated to a nonklt complement $(X,B^+)$.
Then $(Z,B_Z^{\log}+\sB\md^{\log})$ is exceptional.

The same hold for semiexceptional bd-pairs.
\end{cor}

\begin{proof}
Otherwise there exists an effective $\R$-divisor $E\sim_\R H$ on $Z$
such that $(Z,B\dv^{\log}+E+\sB\md^{\log})$ is nonklt.
Hence $(Y,B_Y^{\log}+\psi^*E)$ has a vertical lc center \cite[Lemma~7.4(iii)]{PSh08}
(cf. the proof of Proposition~\ref{horizontal_lc}).
Since $\psi^*E\sim_\R \psi^*H$ and is effective this
contradicts to the maximal property of $\psi$ with respect to $r$.

\end{proof}

\begin{const} \label{semiexceptional_induction}
Let $d$ be a nonnegative integer and
$\Phi=\Phi(\fR)$ be a hyperstandard set associated with
a finite set of rational numbers $\fR$ in $[0,1]$.
By Theorem~\ref{adjunction_index}
there exists a positive integer $J$ such that
every contraction $\psi\colon (Y,B^{\log})\to Z$
of Construction~\ref{b_fib_lc}, (4) has the adjunction index $J$ if
\begin{description}

\item[\rm (1)\/]
$\dim X=d$ and

\item[\rm (2)\/]
$B\hor\in\Phi$.

\end{description}
Note that if $1\not\in\Phi$ then
we can add $0$ to $\fR$ and so $1$ to $\Phi$.
Moreover, there exists
a finite set of rational numbers $\fR'$ in $[0,1]$
such that $\Phi'=\Phi(\fR')$ satisfies Addendum~\ref{adjunction_index_div}.
More precisely, $\fR'$ is defined by~(\ref{const_adj}),
where
$$
\fR''=[0,1]\cap \frac \Z J.
$$

By Addendum~\ref{adjunction_index_bd},
the same adjunction index $J$ has every contraction
$\psi\colon (Y,B^{\log}+\sP)\to Z$
of Construction~\ref{b_fib_lc} if we
apply the construction to a bd-pair $(X,B+\sP)$ and suppose
additionally to (1-2) that $(X,B+\sP)$ is a bd-pair of index $m$,
or equivalently, $(Y,B^{\log}+\sP)$ is
a log bd-pair of index $m$.

Let $I,\ep,v,e$ be the data
as in Restrictions on complementary indices in Section~\ref{intro} and
$f$ be a nonnegative integer such that $f\le d-1$.
By Theorem~\ref{excep_comp} or
by dimensional induction there exists a finite set of
positive integers
$\sN=\sN(f,I,\ep,v,e,\Phi',J)$  such that
\begin{description}

\item[\rm Restrictions:\/]
every $n\in\sN$ satisfies
Restrictions on complementary indices with the given data,
in particular, $J|n$ (see Remark~\ref{m_div_n});

\item[\rm Existence of $n$-complement:\/]
if $(Z,B_Z+\sQ)$ a bd-pair of dimension $f$ and
of index $J$ with wFt $Z$, with a boundary $B_Z$,
with an $\R$-complement and
with exceptional $(Z,B_{Z,\Phi'}+\sQ)$,
then $(Z,B_Z+\sQ)$ has an $n$-complement $(Z,B_Z^++\sQ)$
for some $n\in \sN$.
We can apply dimensional induction by
Addendum~\ref{bd_exceptional_compl}
[or by the assumption~(1) Theorem~\ref{bd_bndc}, bd-version of Theorem~\ref{bndc}]
because $(Z,B_Z+\sQ)$ is also exceptional and
has a klt $\R$-complement.
Moreover,

\item[\rm (3)\/]
$B_Z^+\ge B_{Z,n\underline{\ }\Phi'}$.

\end{description}

\end{const}

The following result applies to the construction but not only.

\begin{cor} \label{disjoint}
For any given finite set $\sN'$ of positive integers,
we can suppose that $\sN(I)=\sN(\dots,I,\dots)$
under Restrictions on complementary indices with the given data
is disjoint from $\sN'$.
\end{cor}

\begin{proof}
Use $\sN(I)=\sN(I')$ with the same data except for $I$
replaced by sufficiently divisible $I'$.
\end{proof}

\begin{thm} \label{semiexcep_compl}
Let $d,\fR,\Phi,J,I,\fR',\Phi',\ep,v,e,f,\sN$ be
the data of Construction~\ref{semiexceptional_induction} and
$r$ be a nonnegative integer.
Let $(X,B)$ be a pair with a boundary $B$ such that
\begin{description}

\item[\rm (1)\/]
$X$ has wFt;

\item[\rm (2)\/]
$\dim X=d$;
and

\item[\rm (3)\/]
both pairs
$$
(X,B_\Phi),\ (X,B_{\sN\underline{\ }\Phi})
$$
are semiexceptional of the same type $(r,f)$.

\end{description}
Then there exists $n\in\sN$ such that $(X,B)$ has
an $n$-complement $(X,B^+)$ with
\begin{description}

\item[\rm (4)\/]
$B^+\ge B_{n\underline{\ }\Phi}$.

\end{description}

\end{thm}

Notice that we do not assume that $(X,B)$ has an $\R$-complement.

\begin{add} \label{B_B_n_Phi_sharp_B_n_Phi_semiexep}
$B^+\ge B_{n\_\Phi}{}^\sharp\ge B_{n\_\Phi}$.

\end{add}

\begin{add} \label{standard_excep_compl_semiexep}
$(X,B^+)$ is a monotonic $n$-complement of itself and
of $(X,B_{n\_\Phi}), (X,B_{n\_\Phi}{}^\sharp)$,
and is a monotonic b-$n$-complement of itself and
of $(X,B_{n\_\Phi}),(X,B_{n\_\Phi}{}^\sharp),(X^\sharp,B_{n\_\Phi}{}^\sharp{}_{X^\sharp})$,
if $(X,B_{n\_\Phi}),(X,B_{n\_\Phi}{}^\sharp)$ are log pairs respectively.

\end{add}

\begin{add}
The same holds for
the bd-pairs $(X,B+\sP)$ of index $m$ and with $\sN$
as in Construction~\ref{semiexceptional_induction}.
That is,
\begin{description}

\item[\rm Existence of $n$-complement:\/]
if $(X,B+\sP)$ is a bd-pair of index $m$ with a boundary $B$, under (1-2) and
such that
\item[]
both bd-pairs
$$
(X,B_\Phi+\sP),\ (X,B_{\sN\underline{\ }\Phi}+\sP)
$$
are semiexceptional of type $(r,f)$,

then $(X,B+\sP)$ has an $n$-complement $(X,B^++\sP)$ for some $n\in\sN$.

\end{description}
Addenda~\ref{B_B_n_Phi_sharp_B_n_Phi_semiexep}
holds literally.
In Addenda~\ref{standard_excep_compl_semiexep}
$(X,B^++\sP)$ is
a monotonic $n$-complement of itself and
of $(X,B_{n\_\Phi}+\sP), (X,B_{n\_\Phi}{}^\sharp+\sP)$,
and is a monotonic b-$n$-complement of itself and
of $(X,B_{n\_\Phi}+\sP),(X,B_{n\_\Phi}{}^\sharp+\sP),
(X^\sharp,B_{n\_\Phi}{}^\sharp{}_{X^\sharp}+\sP)$,
if $(X,B_{n\_\Phi}+\sP_X),(X,B_{n\_\Phi}{}^\sharp+\sP_X)$ are log bd-pairs
respectively.

\end{add}

\begin{lemma} \label{B_D_r_f}
Let $(X,B)$ be a semiexceptional pair with
a boundary $B$ of type $(r,f)$, with wFt $X$ and
$D\ge B$ be an effective $\R$-divisor of $X$ such
that $-(K+D)$ is (pseudo)effective modulo $\sim_\R$.
Then the pair $(X,D)$ is also semiexceptional
of type $\le (r,f)$ with a boundary $D$.
If $(X,D)$ has also type $(r,f)$ then
every maximal b-contraction $(X,D^+)\dashrightarrow Z'$ is
birationally isomorphic to
such a b-contraction for $(X,B^+)\dashrightarrow Z$.
Moreover, the isomorphism is crepant generically over $Z'$,
equivalently, over $Z$.

The same holds for bd-pairs $(X,B+\sP),(X,D+\sP)$
with $-(K+D+\sP_X),\sP$, pseudoeffective modulo $\sim_\R$.

\end{lemma}

\begin{proof}
By our assumptions there exists $D^+\ge D$ such that
$K+D^+\sim_\R 0$.
Since $D^+\ge D\ge B\ge 0$ and $(X,B)$ is semiexceptional,
$(X,D)$ is semiexceptional and $D$ is a boundary too.
In other words, $(X,D^+)$ is lc and
$(X,D^+)$ is an $\R$-complement of $(X,D)$.
Hence by Addendum~\ref{R_D_D'_complement} $(X,D^+)$ is
also an $\R$-complement of $(X,B)$.

Let $(r',f')$ be the semiexceptional type of $(X,D)$
and
$$
(X,D^+)\dashleftarrow
(Y',D_{Y'}^+)\stackrel{\psi'}{\to} Z',
$$
be a maximal b-contraction associated to
a nonklt $\R$-complement $(X,D^+)$, that is,
$\dim Z'=f'$ and $r'=\reg(X,D^+)$.

The complement $(X,D^+)$ is also an $\R$-complement of $(X,B)$.
Hence $r\ge r'$.
If $r>r'$ then $(r,f)>(r',f')$.
Otherwise $r=r'$.
However to verify that $f'\le f$ we need
to return back to a dlt crepant blowup of $(X,D^+)$
of Construction~\ref{b_fib_lc}.
The same blowup can be used for $(X,B)$ with
$B^+=D^+$.
But a maximal b-contraction for $(X,B)$ should be constructed
by E-MMP with
$$
E=B^+-B=D^+-B=D^+-D+D-B=E'+E'',
$$
where $E'=D^+-D,E''=D-B\ge 0$.
Hence $f\ge f'$ and $(r,f)\ge (r',f')$.

If $(r',f)=(r,f)$ then in the last construction $f=f'$.
In this case rational contractions $(X,D^+)\dashrightarrow Z',
(X,B^+)\dashrightarrow Z$ are birationally isomorphic, or
equivalently, $\psi',\psi$ are birationally isomorphic.
Generically over $Z$, $E''$ is exceptional or $0$.
Otherwise
we can construct a b-contraction for $(X,B)$ with
the same $r$ and $f>f'$.

By construction of the proof, Construction~\ref{b_fib_lc} and
in its notation
$\D^+=\B^+=\B^{\log}$ over the generic point of $Z$, where $\B^{\log}=\B(Y,B^{\log})$.
Similarly, $\D^+=\D^{\log}$ over the generic point of $Z'$ with $\D^{\log}=\D(Y',D^{\log})$.
By the above birational isomorphism, $\D^{\log}=\B^{\log}$ over
the generic point of $Z'$ which is the same as over
the generic point of $Z$.
Thus the birational isomorphism of
$(Y',D_{Y'}^+)=(Y',D^{\log}),(Y,B_Y^+)=(Y,B^{\log})$
is crepant over the generic point of $Z'$ which is the same as over
the generic point of $Z$.

Similarly we can treat bd-pairs.

\end{proof}

\begin{proof}[Proof of Theorem~\ref{semiexcep_compl}]
Let $(X,B)$ be a pair satisfying (1-3) of the theorem.
For simplicity of notation suppose that
$B=B_{\sN\_\Phi}$, in particular,
$B\in \Gamma(\sN,\Phi)$
(see Construction~\ref{low_approximation}; but this is not important for the following).
So, instead of (3) we have
\begin{description}

\item[\rm (3$'$)\/]
both pairs
$$
(X,B_\Phi),\ (X,B)
$$
are semiexceptional of the same type $(r,f)$.

\end{description}
Indeed, by definition and Proposition~\ref{Phi_<Phi'}
$$
B_{\sN\_\Phi}=
B_{\sN\_\Phi,\sN\_\Phi},
B_\Phi=
B_{\sN\_\Phi,\Phi}
\text{ and }
B_{\sN\_\Phi,n\_\Phi}=B_{n\_\Phi}
$$
for every $n\in\sN$.
By Corollary~\ref{B_B_+B_sN_Phi}
an $n$-complement of $(X,B_{\sN\_\Phi})$
is also an $n$-complement of $(X,B)$.

Step~1. Choice of a b-contraction.
Let $(X,B^{\R\_+})$ be an $\R$-complement of $(X,B)$
such that its associated b-contraction
\begin{equation} \label{b_contration_sNPhi}
(X,B^{\R\_+})\stackrel{\varphi}{\dashleftarrow}
(Y,B_Y^{\R\_+})\stackrel{\psi}{\to} Z
\end{equation}
has (maximal) type $(r,f)$.
Such an $\R$-complement and a b-contraction exist by (3$'$).
By construction
$(X,B^{\R\_+})$ is lc, $B^{\R\_+}$
is a boundary,
$\dim Y=\dim X=d,\dim Z=f$ and
$r=\reg(X,B^{\R\_+})$.

Notice that $\psi$ is actually a fibration, that is, $f<d$, because
$(X,B)$ is semiexceptional,
$r\ge 0$ by our assumptions and
the complement $(X,B^{\R\_+})$ is nonklt.
By Construction~\ref{b_fib_lc} we get an adjoint bd-pair
$(Z,B\dv^{\log}+\sB\md^{\log})$ of $(Y,B^{\log})\to Z$.
The bd-pair $(Z,B\dv^{\log}+\sB\md^{\log})$ is exceptional by Corollary~\ref{excep_base}.
We need a stronger fact.

Step~2. bd-Pair $(Z,B\dv^{\log}{}_{,\Phi'}+\sB\md^{\log})$ is also exceptional.
To verify this consider a crepant model $(Y^\sharp,B^{\log,\sharp}{}_{Y^\sharp})$
of $(Y,B^{\log})$ which
blows up exactly prime b-divisors of $Y$ which are divisors on $X$ or,
equivalently, on the dlt blowup of $(X,B^+)$ of Construction~\ref{b_fib_lc}.
Thus if we use the same complement $(X,B^+)$ and the same dlt blowup of $(X,B^+)$
to construct a b-contraction for $(X,B_\Phi)$ then we can apply
$E$-MMP to $(Y^\sharp,B_{Y^\sharp}^{\log}{}_{,\Phi})$ with
$E=B_{Y^\sharp}^+-B_{Y^\sharp}^{\log}{}_{,\Phi}$ or, equivalently,
we construct a maximal model using antiflips of
$(Y^\sharp,B_{Y^\sharp}^{\log}{}_{,\Phi})$.
This time we apply $E$-MMP only relatively over $Z$,
that is, Construction~\ref{sharp_construction}
to $(Y^\sharp/Z,B_{Y^\sharp}^{\log}{}_{,\Phi})$.
Denote by $(Y^\sharp/Z,B_{Y^\sharp}^{\log}{}_{,\Phi}{}^\sharp)$ the result.
We can use the same notation $Y^\sharp$ even some divisors can be
contracted because by construction $(Y^\sharp/Z,B^{\log,\sharp}{}_{Y^\sharp})$ is a $0$-pair and
all birational transformations over $Z$ are flops of this pair.
By~\ref{adjunction_0_contr}, the adjoint bd-pair of $(Y^\sharp/Z,B^{\log,\sharp}{}_{Y^\sharp})$ is
$(Z,B\dv^{\log}+\sB\md^{\log})$, the same as of $(Y,B^{\log})\to Z$.
Equivalently,
$$
B^{\log,\sharp}{}_{Y^\sharp,}{}\dv=
(B^{\log,\sharp}{}_{Y^\sharp})\dv{}_{,Z}=B\dv^{\log} \text{ and }
(B^{\log,\sharp}{}_{Y^\sharp})\md=\sB\md^{\log}.
$$

On the other hand, by our assumption (3) and Lemma~\ref{B_D_r_f},
the mobile part of $E$ is vertical.
Thus, generically over $Z$,
the resulting model $(Y^\sharp/Z,B_{Y^\sharp}^{\log}{}_{,\Phi}{}^\sharp)$ is
a $0$-pair.
It is $0$-pair over $Z$ everywhere because it is maximal.
During the construction of this model
we contract all prime divisor $P$ of $Y^\sharp$ for which
$$
\mult_P B_{Y^\sharp}^{\log}{}_{,\Phi}{}^\sharp>
\mult_P B_{Y^\sharp}^{\log}{}_{,\Phi}.
$$
Thus on the resulting model $(Y^\sharp/Z,B_{Y^\sharp}^{\log}{}_{,\Phi}{}^\sharp)$,
for every prime divisor $P$ on $Y^\sharp$,
$$
\mult_P B_{Y^\sharp}^{\log}{}_{,\Phi}{}^\sharp=
\mult_P B_{Y^\sharp}^{\log}{}_{,\Phi} \text{ and }
(Y^\sharp/Z,B_{Y^\sharp}^{\log}{}_{,\Phi}{}^\sharp)=
(Y^\sharp/Z,B_{Y^\sharp}^{\log}{}_{,\Phi}).
$$
We can take adjunction for the last pair too.
Denote its adjoint bd-pair by
$(Z,B_{Y^\sharp}^{\log}{}_{,\Phi,}{}\dv+\sB_{Y^\sharp}^{\log}{}_{,\Phi,}{}\md)$,
where $B_{Y^\sharp}^{\log}{}_{,\Phi,}{}\dv=(\B_{Y^\sharp}^{\log}{}_{,\Phi})\dv{}_{,Z}$ and
$\sB_{Y^\sharp}^{\log}{}_{,\Phi,}{}\md=(\sB_{Y^\sharp}^{\log}{}_{,\Phi}{})\md=
\sB\md^{\log}$.
The last equality follows again from Lemma~\ref{B_D_r_f} and
Proposition~\ref{adjunction_same_mod_etc}, (1).
Indeed, by the lemma b-contractions $(X,B^{\R\_+})\dashrightarrow Z,(X,B^{\R\_+})\dashrightarrow Z_\Phi$
for $(X,B)$ and $(X,B_\Phi)$ are crepant generically over $Z$.
Generically over $Z$ these b-contractions are crepant respectively to
$(Y,B^{\log})\to Z$ and $(Y^\sharp,B_{Y^\sharp}^{\log}{}_{,\Phi})\to Z$.
But by construction $(Y,B^{\log})\to Z$ and $(Y^\sharp,B^{\log,\sharp}{}_{Y^\sharp})\to Z$
are also crepant over $Z$.
Thus $(Y^\sharp,B^{\log,\sharp}{}_{Y^\sharp})\to Z$ and
$(Y^\sharp,B_{Y^\sharp}^{\log}{}_{,\Phi})\to Z$ are
crepant generically over $Z$,
moreover, they are equal generically over $Z$ (i.e., horizontally):
$$
B^{\log,\sharp}{}_{Y^\sharp}{}\hor=B_{Y^\sharp}^{\log}{}_{,\Phi}{}\hor=
B_{Y^\sharp}^{\R\_+,}{}\hor.
$$

Since crepant modifications preserve the semiexceptional property and
by Proposition~\ref{monotonicity_I}, the pair $(Y^\sharp,B_{Y^\sharp}^{\log}{}_{,\Phi})$
is semiexceptional.
Otherwise, by the statement and Construction~\ref{sharp_construction},
$(Y^{\sharp,\sharp},B_{Y^\sharp}^{\log}{}_{,\Phi}{}^\sharp{}_{Y^{\sharp,\sharp}})$ and
$(X,B_\Phi)$ are not semiexceptional, a contradiction.

The adjoint pair
$(Z,B_{Y^\sharp}^{\log}{}_{,\Phi,}{}\dv+\sB_{Y^\sharp}^{\log}{}_{,\Phi,}{}\md)$
is exceptional according to the proof of Corollary~\ref{excep_base}.
Otherwise there exists an $\R$-complement $(Y^\sharp,(B_{Y^\sharp}^{\log}{}_{,\Phi})^+)$
of $(Y^\sharp/Z,B_{Y^\sharp}^{\log}{}_{,\Phi})$ with
the induced complement $(X,(B_{Y^\sharp}^{\log}{}_{,\Phi})_X^+)$ of $(X,B_\Phi)$
(see Remark~\ref{rem_def_b_n_compl}, (1-2)) and
with
$$
\reg(Y^\sharp,(B_{Y^\sharp}^{\log}{}_{,\Phi})^+)=\reg(X,(B_{Y^\sharp}^{\log}{}_{,\Phi})_X^+)
>r=\reg(Y^\sharp{},B_{Y^\sharp}^{\log}{}_{,\Phi})=\reg(X,B^{\R\_+}).
$$

So, $(Z,B\dv^{\log}{}_{,\Phi'}+\sB\md^{\log})$ is exceptional because
$$
B_{Y^\sharp}^{\log}{}_{,\Phi,}{}\dv\le B\dv^{\log}{}_{,\Phi'}
\text{ and }
\sB_{Y^\sharp}^{\log}{}_{,\Phi,}{}\md=\sB\md^{\log}.
$$
The last equality we already know.
By definition the inequality follows from two facts (cf.~\ref{n_Phi_inequality}):
$$
B_{Y^\sharp}^{\log}{}_{,\Phi,}{}\dv\in\Phi'
\text{ and }
B_{Y^\sharp}^{\log}{}_{,\Phi,}{}\dv\le B\dv^{\log}.
$$
The inclusion follows from Construction~\ref{semiexceptional_induction} and
Addendum~\ref{adjunction_index_div}.
Indeed,
by construction $B_{Y^\sharp}^{\log}{}_{,\Phi}\in\Phi\cup\{1\}$, in particular,
$B_{Y^\sharp}^{\log}{}_{,\Phi}\hor\in\Phi\cup\{1\}$.
(Or we can suppose that $1\in\Phi$ already.)
The required inequality follows from~(4) of~\ref{adjunction_div}
because $B_{Y^\sharp}^{\log}{}_{,\Phi}\le B^{\log,\sharp}{}_{Y^\sharp}$ and
$B^{\log,\sharp}{}_{Y^\sharp,}{}\dv=B\dv^{\log}$ by construction
(cf. again~\ref{n_Phi_inequality}).
By Addendum~\ref{R_D_D'_complement} the last inequality also shows that
$(Z,B_{Y^\sharp}^{\log}{}_{,\Phi,}{}\dv+\sB_{Y^\sharp}^{\log}{}_{,\Phi,}{}\md)$
actually has an $\R$-complement.

Notice that the exceptional property of this step can be established
by more sophisticated but quite formal methods of Section~\ref{adjunction_cor_multiplicities}.

Step~3. Choice of $n\in\sN$ and
construction on an $n$-complement $(Z,B\dv^{\log,+}+\sB\md^{\log})$
of the bd-pair $(Z,B\dv^{\log}+\sB\md^{\log})$.
The bd-pair is a bd-pair of index $J$ in notation of Construction~\ref{semiexceptional_induction}.
Indeed, according to Step~2, the adjunctions
$(Y^\sharp,B^{\log,\sharp}{}_{Y^\sharp})\to Z,
(Y^\sharp,B_{Y^\sharp}^{\log}{}_{,\Phi})\to Z$
are crepant generically over $Z$ one to another.
So, by Proposition~\ref{adjunction_same_mod_etc}, (1) and (5)
they have the same moduli part and the same adjunction index
if such one exists for either of them.
The second adjunction
$(Y^\sharp,B_{Y^\sharp}^{\log}{}_{,\Phi})\to Z$ satisfies (1-2)
of Construction~\ref{semiexceptional_induction} and has the adjunction index $J$.
(We can suppose that $1\in\Phi$ or to consider $\Phi\cup\{1\}$ in
(2) instead of $\Phi$.)
Hence $(Y^\sharp,B^{\log,\sharp}{}_{Y^\sharp})\to Z$ has also the adjunction index $J$.
The adjoint bd-pair $(Z,B^{\log,\sharp}{}_{Y^\sharp,}{}\dv+
(B^{\log,\sharp}{}_{Y^\sharp})\md)=(Z,B\dv^{\log}+\sB\md^{\log})$
has also the index $J$ by (3) of \ref{adjunction_index_of} and
Theorem~\ref{b_nef} (cf. Addendum~\ref{adjunction_index_adjoint}).

By construction $\dim Z=f$ and by Step~2 the bd-pair
$(Z,B\dv^{\log}{}_{,\Phi'}+\sB\md^{\log})$ is exceptional.
Notice that $Z$ has Ft by \cite[Lemma~2.8,(i)]{PSh08} because
$Y$ has Ft by construction.
Thus by Existence of $n$-complement in Construction~\ref{semiexceptional_induction},
there exists $n\in\sN$ and a required $n$-complement
$(Z,B\dv^{\log,+}+\sB\md^{\log})$ of $(Z,B\dv^{\log}+\sB\md^{\log})$.
About this complement we suppose (3) of Construction~\ref{semiexceptional_induction}.
However, we need slightly more:
\begin{description}

\item[\rm (3$''$)\/]
$b_Q^+\ge b_{Q,n\_\Phi'}$ for every prime b-divisor $Q$ of $Z$ which
is the image of a prime divisor $P$ on $X$, where
$b_Q^+=\mult_Q \D(Z,B\dv^{\log,+}+\sB\md^{\log})\dv$ and
$b_Q=\mult_Q \D(Z,B\dv^{\log}+\sB\md^{\log})\dv=\mult_Q\B\dv^{\log}$.

\end{description}
To satisfy this inequality we need to take an $n$-complement on
an appropriate crepant model $(Z',B\dv^{\log}{}_{,Z'}+\sB\md^{\log})$  of
$(Z,B\dv^{\log}+\sB\md^{\log})$,
in particular, both are log bd-pairs.
The model should blow up those $Q$ which are exceptional on $Z$.
There are only finitely many of those b-divisors $Q$ and $b_Q\ge 0$,
actually, $b_Q\in[0,1)$ by~(12) of~\ref{adjunction_div}
and $b_{Q,n\_\Phi'}$ is well-defined.
Indeed, every $b_P=\mult_P B^{\log}\in [0,1]$ by Construction~\ref{b_fib_lc}.
Thus a required crepant model exists and its divisorial part is
the  boundary $B\dv^{\log}{}_{,Z'}$. The model
is an exceptional bd-pair $(Z',B\dv^{\log}{}_{,Z'}+\sB\md^{\log})$
of dimension $f$ and of index $J$.
In particular, it has an $\R$-complement.
The log bd-pair $(Z',B\dv^{\log}{}_{,Z',\Phi'}+\sB\md^{\log})$ also
has an $\R$-complement and is exceptional.
The $\R$-complement exists by Addendum~\ref{R_D_D'_complement}.
The exceptional property follows from the fact the image
of a nonlc $\R$-complement of $(Z',B\dv^{\log}{}_{,Z',\Phi'}+\sB\md^{\log})$
gives a nonlc $\R$-complement of $(Z,B\dv^{\log}{}_{,\Phi}+\sB\md^{\log})$
that contradicts to Step~2.

Step~4.
Construction on an $n$-complement $(Y,B^{\log,+})$ of $(Y,B^{\log})$.
Take the induced $n$-complement $(Y,B^{\log,+})$ of $(Z,B\dv^{\log,+}+\sB\md^{\log})$,
that is, $B^{\log,+}$ adjunction corresponds to $B\dv^{\log,+}$
as in~\ref{adjunction_div}.
Actually we need to verify that it is an $n$-complement
with the log version of (4):
\begin{description}

\item[\rm (4$'$)\/]
$b_P^+\ge b_{P,n\_\Phi}^{\log}$ for every prime divisor on $X$ or on $Y$, where
$b_P^+=\mult_P \B^{\log,+},b_{P,n\_\Phi}^{\log}=(b_P^{\log})_{n\_\Phi}$ and
$$
b_P^{\log}=
\begin{cases}
b_P=\mult_P B, &\text{ if }P \text{ is nonexceptional on } X;\\
1 &\text{ otherwise}.
\end{cases}
$$

\end{description}
For this we use Addendum~\ref{invers_b_n_comp_monotonic}.

We apply the addendum to the $0$-contraction
$(Y,B^{\log})\to Z$ over $S=\pt$
The contraction has the adjunction index $J$.
Indeed,
by construction in Step~2, $(Y^\sharp,B^{\log,\sharp}{}_{Y^\sharp})\to Z$ and
$(Y/Z,B^{\log})\to Z$ are crepant generically over $Z$ and
$(Y^\sharp,B^{\log,\sharp}{}_{Y^\sharp})\to Z$ has
the adjunction index $J$.
Thus by~Proposition~\ref{adjunction_same_mod_etc}, (5)
the (adjunction for) $0$-contraction $(Y,B^{\log})\to Z$ has
also the adjunction index $J$.

By Constructions~\ref{semiexceptional_induction} and~\ref{b_fib_lc},
$(Y/Z,B^{\log})$ is lc with a boundary $B^{\log}$ and
$Z$ is projective over $S=\pt$

We can suppose also that the hyperstandard sets $\Phi,\Phi'$
and the set of rational numbers $\fR''$ of
Construction~\ref{semiexceptional_induction} are the same
as in~\ref{direct_hyperst} and in~Theorem~\ref{invers_b_n_comp}.

Again by Construction~\ref{semiexceptional_induction} and Step~3,
$J|n$.

Finally, by Step~3 there exists a required model $Z'/\pt$
of $Z/\pt$ with a boundary $B\dv^{\log,+}{}_{Z'}$,
that is we consider the $n$-complement of Step~3 on $Z'/\pt$
By (3$''$), (4) of Addendum~\ref{invers_b_n_comp_monotonic} holds.
Thus (1-3) of~Theorem~\ref{invers_b_n_comp} also hold by the addendum.
Hence $(Y,B^{\log,+})$ is an $n$-complement $(Y,B^{\log})$
which satisfies~(4$'$) by~(5) of the addendum.

Step~5. Construction on an $n$-complement $(X,B^+)$ of $(X,B)$.
The complement is induced by the $n$-complement $(Y,B^{\log,+})$,
that is, $(X,B^+)$ is crepant to $(Y,B^{\log,+})$.
The properties (2-3) of Definition~\ref{n_comp} holds automatically.
The property (4$'$) implies (4) of the theorem.
In its turn, (4) implies (1) of the definition
as (4$'$) in Step~4 (cf. \cite[Lemma~5.4]{Sh92}).

Step~6. The proof of the addenda and, in particular,
for bd-pairs,
is similar to the above proof for usual pairs and/or
to the proof of Theorem~\ref{excep_comp}.

\end{proof}

\paragraph{Semiexceptional filtration.}
Let $d$ be a nonnegative integer,
$\Phi=\Phi(\fR)$ be a hyperstandard set associated with
a finite set of rational numbers $\fR$ in $[0,1]$ and
$\sN\supseteq\sN'$ be sets of positive integers.
Let
\begin{equation} \label{filtration_se_type}
\sN\supseteq \sN^{(0,0)}
\supseteq\dots\supseteq \sN^{(r,f)}
\supseteq\dots\supseteq \sN^{(d-1,0)}
\supseteq\sN',\
0\le r\le d-f-1, 0\le f\le d-1,
\end{equation}
be its (decreasing) {\em filtration with respect to
semiexceptional types in the dimension\/} $d$.
Its {\em associated\/} (decreasing) filtration of
hyperstandard sets is
$$
\Gamma(\sN,\Phi)\supseteq \Gamma(\sN^{(0,0)},\Phi)
\supseteq\dots\supseteq \Gamma(\sN^{(r,f)},\Phi)
\supseteq\dots\supseteq \Gamma(\sN^{(d-1,0)},\Phi)
\supseteq \Gamma(\sN',\Phi).
$$

{\em Notation\/}: For the filtration~(\ref{filtration_se_type})
and type $(r,f)$, put
$$
\sN_{(r,f)}=\sN^{(r,f)}\setminus
\sN^{(r',f')},
$$
where $(r',f')$ is the next type if such a type exists.
Otherwise $(r,f)=(d-1,0)$ and
$\sN_{(d-1,0)}=\sN^{(d-1,0)}\setminus\sN'$.
That is, the next set in the last case is $\sN'$
but without any type (cf. Generic type filtration with respect
to dimension in Section~\ref{gen_nonsemiexcep_comp}).
Respectively, it is useful to suppose that $\Gamma(\sN',\Phi)$ is
the next set for $\Gamma(\sN^{(d-1,0)},\Phi)$.

The whole set $\sN$ is not necessarily coincide
with $\sN^{(0,0)}$.
The discrepancy
$$
\sN_{(-1,-)}=\sN\setminus \sN^{(0,0)}
$$
corresponds to exceptional types $(-1,f)$
with $f=d$ possible.
So, the low script is relevant.
Respectively, we use $\sN^{(-1,-)}=\sN$.

We do not filter exceptional types here.

\begin{exa}

(1) If $\fR=\{1\}$ is minimal
then $\Gamma(\sN',\Phi(\{1\}))=\Gamma(\sN')$.
The filtration~(\ref{filtration_se_type})
gives the associated filtration
$$
\Gamma(\sN)\supseteq \Gamma(\sN^{(0,0)})
\supseteq\dots\supseteq \Gamma(\sN^{(r,f)})
\supseteq\dots\supseteq \Gamma(\sN^{(d-1,0)})
\supseteq \Gamma(\sN').
$$

(2) If $\sN'=\emptyset$
then $\frak{G}(\emptyset,\fR)=\Phi(\fR)=\Phi$.

\end{exa}

\begin{defn} \label{semiexcep_type}
Let~(\ref{filtration_se_type}) be a filtration
in Semiexceptional filtration.
Consider a certain class of pairs $(X,B)$
of dimension $d$ with boundaries $B$ which have
an $n$-complement with $n\in \sN$.
Such a {\em pair\/} $(X,B)$ and its $n$-{\em complement
have semiexceptional type\/} $(r,f)$
{\em with respect to the filtration}~(\ref{filtration_se_type}) if
both pairs
$$
(X,B_{\sN^{(r',f')}\_\Phi}),\ (X,B_{\sN^{(r,f)}\_\Phi})
$$
are semiexceptional of the same type $(r,f)$, where
$(r',f')$ is its next type.
If the next type $(r',f')$ does not exists, we
take $B_{\sN'\_\Phi}$ instead of $B_{\sN^{(r',f')}\_\Phi}$.
Note also that a {\em pair\/} $(X,B)$ and
its $n$-{\em complement have exceptional type \/}
$(-1,-)$ {\em with respect to the filtration}~(\ref{filtration_se_type}) if the pair
$$
(X,B_{\sN^{(0,0)}\_\Phi})
$$
is exceptional.

We say that the {\em existence\/} $n$-complements {\em agrees
with the filtration\/} if
every pair $(X,B)$ in the class has an $n$-{\em complement
of filtration semiexceptional type\/} $(r,f)$ with $n\in \sN_{(r,f)}$ and
$$
B^+\ge B_{n\_\sN^{(r',f')}\_\Phi}.
$$
If the next type $(r',f')$, does not exists, we
take $B_{n\_\sN'\_\Phi}$ instead of $B_{n\_\sN_{(r',f')}\_\Phi}$.
Notice again that $\sN'$ does not have any semiexceptional or
exceptional type.
Respectively, exceptional $n$-complements
have type $(-1,-)$, $n\in \sN_{(-1,-)}$, and
$$B^+\ge B_{n\_\sN^{(0,0)}\_\Phi}.
$$

The same applies to certain classes of bd-pairs $(X,B+\sP)$
of dimension $d$ and index $m|n$ with boundaries $B$,
which have an $n$-complement with $n\in \sN$.

\end{defn}

\begin{war}
It is possible other $n$-complements which are
not agree with the filtration.
E.g., we can have a triple or a longer chain of
subsequent pairs of the same type $(r,f)$
$$
(X,B_{\sN^{(r',f')}\_\Phi}),\dots,\ (X,B_{\sN^{(r,f)}\_\Phi}).
$$
So, they have many types and many complementary
indices which are agree to some type but not
agree to other. Additionally, we can have $n$-complements
without any type but with $n\in\sN$.
\end{war}

\begin{thm}[Semiexceptional $n$-complements] \label{semiexcep_compl_with_filtration}
Let $d$ be a nonnegative integer,
$I,\ep,v,e$ be the data as
in Restrictions on complementary indices, and
$\Phi=\Phi(\fR)$ be a hyperstandard set associated with
a finite set of rational numbers $\fR$ in $[0,1]$.
Then there exists a finite set
$\sN=\sN(d,I,\ep,v,e,\Phi)$
of positive integers such that
\begin{description}

\item[\rm Restrictions:\/]
every $n\in\sN$ satisfies
Restrictions on complementary indices with the given data.

\item[\rm Existence of $n$-complement:\/]
if $(X,B)$ is a pair with wFt $X$,
$\dim X=d$, with a boundary $B$,
with an $\R$-complement and
with semiexceptional $(X,B_\Phi)$
then $(X,B)$ has an $n$-complement $(X,B^+)$ for some $n\in\sN$.

\end{description}

\end{thm}

\begin{add} \label{B_B_n_Phi_sharp_B_n_Phi_semiexep_fil}
$B^+\ge B_{n\_\Phi}{}^\sharp\ge B_{n\_\Phi}$.

\end{add}

\begin{add} \label{standard_excep_compl_semiexep_fil}
$(X,B^+)$ is a monotonic $n$-complement of itself and
of $(X,B_{n\_\Phi}), (X,B_{n\_\Phi}{}^\sharp)$,
and is a monotonic b-$n$-complement of itself and
of $(X,B_{n\_\Phi}),(X,B_{n\_\Phi}{}^\sharp),(X^\sharp,B_{n\_\Phi}{}^\sharp{}_{X^\sharp})$,
if $(X,B_{n\_\Phi}),(X,B_{n\_\Phi}{}^\sharp)$ are log pairs respectively.

\end{add}

\begin{add} \label{semiexep_fil}
$\sN$ has a semiexceptional filtration~(\ref{filtration_se_type})
with $\sN'=\emptyset$
which agrees with the existence of $n$-complements
for the class of pairs in Existence of $n$-complements.

\end{add}

\begin{add} \label{bd_exceptional_compl_fil}
The same holds for bd-pairs $(X,B+\sP)$ of index $m$
with $\sN=\sN(d,I,\ep,v,e,\Phi,m)$.
That is,
\begin{description}

\item[\rm Restrictions:\/]
every $n\in\sN$ satisfies
Restrictions on complementary indices with the given data and
$m|n$.

\item[\rm Existence of $n$-complement:\/]
if $(X,B+\sP)$ is a bd-pair of index $m$ with wFt $X$,
$\dim X=d$, with a boundary $B$,
with an $\R$-complement and
with semiexceptional $(X,B_\Phi+\sP)$
then $(X,B+\sP)$ has an $n$-complement $(X,B^++\sP)$ for some $n\in\sN$.

\end{description}
Addenda~\ref{B_B_n_Phi_sharp_B_n_Phi_semiexep_fil} and \ref{semiexep_fil}
hold literally.
In Addenda~\ref{standard_excep_compl_semiexep_fil}
$(X,B^++\sP)$ is
a monotonic $n$-complement of itself and
of $(X,B_{n\_\Phi}+\sP), (X,B_{n\_\Phi}{}^\sharp+\sP)$,
and is a monotonic b-$n$-complement of itself and
of $(X,B_{n\_\Phi}+\sP),(X,B_{n\_\Phi}{}^\sharp+\sP),
(X^\sharp,B_{n\_\Phi}{}^\sharp{}_{X^\sharp}+\sP)$,
if $(X,B_{n\_\Phi}+\sP_X),(X,B_{n\_\Phi}{}^\sharp+\sP_X)$ are log bd-pairs
respectively.

\end{add}

\begin{war}
Sets $\sN(d,I,\varepsilon,v,e,\Phi),\sN(d,I,\varepsilon,v,e,\Phi,m)$ are
not unique and may be not suitable in other situations of the paper
(cf. Remark~\ref{N_with_parameters}, (1)).
Notation shows the parameters on which the sets depend.
We can take such minimal sets under inclusion but
they can be not unique.

\end{war}

It would be interesting to find a more canonical way
to construct such set $\sN(d,I,\varepsilon,v,e,\Phi),\sN(d,I,\varepsilon,v,e,\Phi,m)$
and/or to find an estimation on their largest element.

\begin{proof}

Construction of semiexceptional complements
starts from the top type $(d-1,0)$ and
goes to the bottom $(-1,-)$, the exceptional one.
We use induction on types $(r,f)$
for the class of pairs in Existence of $n$-complements
of the theorem.

Step~1. Type $(d-1,0)$.
Consider a set of positive integers
$\sN_{(d-1,0)}=
\sN(0,I,\ep,v,e,\Phi',J)$ as
in Construction~\ref{semiexceptional_induction} with $f=0$.
Put $\sN^{(d-1,0)}=\sN_{(d-1,0)}$.
By Theorem~\ref{semiexcep_compl} this gives
the lowest set in the filtration~(\ref{filtration_se_type})
above $\sN'=\emptyset$.

Step~2. Type $(r,f)$.
We can suppose that it has the next type $(r',f')$ and
the filtration~(\ref{filtration_se_type})
is already constructed for types $\ge (r',f')$.
Consider a set of positive integers
$\sN_{(r,f)}=
\sN(f,I,\ep,v,e,\Phi',J)$ as
in Construction~\ref{semiexceptional_induction}
with a hyperstandard set $\Phi=\frak{G}(\sN^{(r',f')},\fR)$
and with given $f$.
Note that $\Phi,\Phi',m$ in the step can be different
from that of in Step~1 or in the statement of theorem.
By Corollary~\ref{disjoint}
we can suppose that $\sN_{(r,f)}$ is
disjoint from $\sN^{(r',f')}$.
Put  $\sN^{(r,f)}=\sN_{(r,f)}\cup \sN^{(r',f')}$.
Again by Theorem~\ref{semiexcep_compl} this gives
the filtration~(\ref{filtration_se_type})
for types $\ge (r,f)$.

This concludes construction of complements
in the semiexceptional case.

Step~3. Exceptional type $(-1,-)$.
In this case we need to construct $n$-complements
for $(X,B)$ with exceptional $(X,B_{\sN^{(0,0)}\_\Phi})$.
Consider a set of positive integers
$\sN_{(-1,-)}=
\sN(d,I,\ep,v,e,\Phi)$ as
in Theorem~\ref{excep_comp}
with the hyperstandard set $\Phi=\frak{G}(\sN^{(0,0)},\fR)$.
By Corollary~\ref{disjoint}
we can suppose that $\sN_{(-1,-)}$ is
disjoint from $\sN^{(0,0)}$.
Put  $\sN=\sN^{(-1,-)}=\sN_{(-1,-)}\cup \sN^{(0,0)}$.
Now by Theorem~\ref{excep_comp} this completes
the filtration~(\ref{filtration_se_type}).

Step~4.
Addenda follows from the above arguments and from
Theorems~\ref{semiexcep_compl}, \ref{excep_comp}, in particular, for  bd-pairs.
However, we take $\sN_{(-1,-)}=
\sN(d,I,\ep,v,e,\Phi,m)$ in Step~3 with
$m$ as the index of bd-pair $(X,B)$.

Notice also that by Addendum~\ref{semiexep_fil},
Propositions~\ref{Phi_<Phi'} and~\ref{monotonicity_I},
in Addenda~\ref{B_B_n_Phi_sharp_B_n_Phi_semiexep_fil},
\ref{standard_excep_compl_semiexep_fil} and \ref{bd_exceptional_compl_fil},
we have stronger results with
$n\_\sN^{(r',f')}\_\Phi$ instead of $n\_\Phi$
for $n$-complements of type $(r,f)$
(see Definition~\ref{semiexcep_type}).

\end{proof}

\begin{cor} \label{N_N'}
In Theorem~\ref{semiexcep_compl_with_filtration}
we can add addition data $\sN'$, a finite
set of positive integers such that every $n\in\sN'$
satisfies Restrictions on complementary indices with the given data.
Then in Addendum~\ref{semiexep_fil} we can find
an extended filtration~(\ref{filtration_se_type}) starting from $\sN'$.

\end{cor}

\begin{proof}
Take $\Phi:=\frak{G}(\sN',\fR)$.
It is also hyperstandard by Proposition~\ref{every_sN_Phi_hyperstandard}.

\end{proof}

[Remark:]
E.g., in the induction of the next section
we use a set $\sN'=\sN^{d-1}$ of [non[semi]exceptional] complementary indices coming from
low dimensions $\le d-1$ (see Construction~\ref{generic_induction} below).

\section{Generic nonsemiexceptional complements:
extension of lower dimensional complements} \label{gen_nonsemiexcep_comp}

\begin{defn-prop} \label{generic}
Let $(X/Z,B)$ be a pair with proper $X/Z$ and
with a boundary $B$ such that $(X/Z,B)$ has
a klt $\R$-complement.
We say that $(X/Z,B)$ is {\em generic
with respect to $\R$-complements\/} if
the pair additionally satisfies one of the following equivalent
properties.
\begin{description}

\item[\rm (1)\/]
There exists a klt $\R$-complement $(X/Z,B^+)$ of
$(X/Z,B)$ such that $B^+-B$ is big over $Z$.

\item[\rm (2)\/]
There exists an $\R$-complement $(X/Z,B^+)$ of
$(X/Z,B)$ and an effective big $\R$-mobile over $Z$
divisor $A$ on $X$ such that $\Supp A\subseteq \Supp(B^+-B)$.
Moreover, we can suppose that
$A$ is $\R$-ample over $Z$ when $X/Z$ is projective.

\item[\rm (3)\/] For every finite set of divisors
$D_1,\dots,D_n$ on $X$,
there exists an $\R$-complement $(X/Z,B^+)$ of
$(X/Z,B)$ such that
$\Supp D_1,\dots,\Supp D_n\subseteq \Supp(B^+-B)$.
\end{description}

The same works for a bd-pair $(X/Z,B+\sP)$
with a boundary $B$.

\end{defn-prop}

In this section we use the word generic only in this sense.

\begin{rem} For generic $(X/Z,B)$,
every klt $\R$-complement satisfies (1) and
(2) with an effective $\R$-mobile divisor,
but (2) with $\R$-ample and (3) not always.

\end{rem}

\begin{lemma} \label{perturbation}
Let $(X/Z,D)$ be a pair with proper $X/Z$,
$(X/Z,D^+)$ be its klt $\R$-complement and
$(X/Z,D')$ be a pair such that $K+D'\sim_{\R,Z} 0$ and
$D'\ge D$.
Then for every sufficiently small real positive number
$\ep$,
$$
(X/Z,(1-\ep)D^++\ep D')
$$
is a klt $\R$-complement of $(X/Z,D)$
with
$$
\Supp((1-\ep)D^++\ep D')=\Supp D^+\cup\Supp D'.
$$

If $D$ is a boundary, then $D^+, (1-\ep)D^++\ep D'$
are boundaries too.

The same works for bd-pairs.
\end{lemma}

\begin{proof}
Immediate by definition and \cite[(1,3.2)]{Sh92}.
Recall (see Remark~\ref{remark_def_complements}, (5))
that (3) in Definition~\ref{r_comp} for
nonlocal $X/Z$ means that
$$
K+D^+\sim_{\R,Z} 0.
$$

Similarly we can treat bd-pairs.

\end{proof}

\begin{cor} \label{generic_wFt}
Every generic $(X/Z,B)$ has wFt $X/Z$.
Moreover, $X/Z$ has Ft when $X/Z$ is projective.

The same holds for a generic bd-pair $(X/Z,B+\sP)$
with pseudoeffective $\sP$ over $Z$.

\end{cor}

\begin{proof}
Immediate by definition and Definition-Proposition~\ref{generic}, (1).
Respectively, the projective case
by~\cite[Lemma-Definition~2.6, (iii)]{PSh08}.
By Lemma~\ref{perturbation} we can suppose that the complement in
Definition-Proposition~\ref{generic}, (3) is klt.

Similarly we can treat bd-pairs.

\end{proof}

\begin{proof}[Proof of Definition-Proposition~\ref{generic}]

(2)$\Rightarrow$(1) Immediate by definition.
For projective $X/Z$, the ample property implies the big one.
Note also that by Lemma~\ref{perturbation}
we can suppose that the complement in (2) is klt.

(3)$\Rightarrow$(2) Apply (3) for $n=1$ and $D_1=A$.

(1)$\Rightarrow$(3)
We can suppose that $D_i$ are effective and even prime.
Let $(X/Z,B^+)$ be a complement of (1).
We construct another klt $\R$-complement $(X/Z,B')$
of $(X/Z,B)$ such that
$$
\Supp B^+,\Supp D_1,\dots,\Supp D_n\subseteq \Supp B'.
$$
Since $B^+-B$ is big, there exists
an effective divisor $E\sim_{\R,Z} B^+-B$
such that
$$
\Supp D_1,\dots,\Supp D_n\subseteq \Supp E.
$$
Consider $B'=E+B$.
By construction $B'\ge B$ and
$$
K+B'=K+B+E\sim_{\R,Z} K+B+B^+-B=K+B^+\sim_{\R,Z} 0.
$$
Now Lemma~\ref{perturbation}
implies the existence of a required $\R$-complement $(X/Z,B')$.

Similarly we can treat bd-pairs.

\end{proof}

\paragraph{Basic properties of generic pairs.}
(1)
{\em Decreasing of boundary preserves
the generic property\/}: if $(X/Z,B)$
is generic and $B'$ is boundary on $X$
such that $B'\le B$ then
$(X/Z,B')$ is also generic.

\begin{proof}
Immediate by definition.
\end{proof}

(2) {\em Small modifications preserves the generic property\/}:
if $(X/Z,B)$ is generic and $(X'/Z,B)$ is its small modification
over $Z$ then $(X'/Z,B)$ is also generic.

\begin{proof}
Immediate by definition.
The big property over $Z$ is preserved under
small modifications.
\end{proof}

(3) {\em Crepant models with a boundary preserve the generic property\/}:
$(X/Z,B)$ is generic log pair and
$(Y/Z,B_Y)$ its crepant model with a boundary $B_Y$
then $(Y/Z,B_Y)$ is also generic.

\begin{proof}
Immediate by definition.
Indeed, $\B^+-\B=\B_Y^+-\B_Y$ is b-big
with a big trace on $Y$ and
$B_Y=\B_Y$ is a boundary
by definition and our assumptions.

\end{proof}

(4) {\em The same properties holds for generic bd-pairs.\/}

\begin{proof}
Immediate by definition.
\end{proof}

\begin{lemma} \label{exist_plt_model}
Let $(X/Z\ni o,B)$ be a local generic pair
which is not semiexceptional.
Then there exists an $\R$-complement $(X/Z\ni o',B^+)$
with a crepant plt model
\begin{equation} \label{plt_model}
\begin{array}{ccc}
(Y,B_Y^+)&\stackrel{\varphi}{\dashrightarrow}&(X,B^+)\\
&\searrow& \downarrow\\
&& Z\ni o'
\end{array},
\end{equation}
where $\varphi$ is a crepant birational $1$-contraction $\varphi$ over $Z\ni o'$ and
$o'$ is a specialization of $o$, such that
\begin{description}

\item[]
$Y$ is projective ($\Q$-factorial);

\item[]
$(Y,B_Y^+)$ is plt with a single lc center $S$;

\item[]
the only possible exceptional divisor of $\varphi$ is $S$;
and

\item[]
there exists an effective ample over $Z\ni o'$ divisor $A$ on $Y$
with
$$
S\not\in\Supp A\subseteq\Supp(B_Y^+-B\logb_Y).
$$

\end{description}

If $(X/Z\ni o,B)$ is not global,
then the nonsemiexceptional assumption is redundant and
$S$ is over $o'$, a sufficiently general closed
point of the closure of $o$.

The same holds for a generic bd-pair $(X/Z,B+\sP)$
with pseudoeffective $\sP$ over $Z$.

\end{lemma}

Warning: Every global plt model $(Y/Z\ni o',B\logb_Y)$ is
not generic because it is not klt.
However, the klt property holds in the nonglobal case for
the {\em generalization\/} $(Y/Z\ni o,B\logb_Y)$ when $o'\not= o$.

\begin{proof}
After a small birational modification of $X$ as in Lemma~\ref{wTt_vs_Ft}
and $\Q$-factorialization over $Z\ni o$,
we can suppose that $X$ is projective over $Z\ni o$ and $\Q$-factorial.
Every small modification of $X/Z\ni o$ over $Z\ni o$ is
again generic and nonsemiexceptional.
Thus by Corollary~\ref{generic_wFt} we can suppose that
$X/Z\ni o$ has Ft.
We can suppose also that $\dim X\ge 2$.

Suppose that $(X/Z\ni o,B^+)$ is a klt $\R$-compliment
as in (3) with distinct very ample prime divisors $A_1,\dots,A_r$
generating numerically $\Pic(X/Z\ni o)$.
(Actually, the generation is modulo $\sim_\R/Z\ni o$
by \cite[Corollary~4.5]{ShCh}.)
We suppose also that $\cap\Supp A_i=\emptyset$.
So, every component $a_iA_i$ in $B^+=E'+\sum a_iA_i,a_i=\mult_{A_i}B^+>0,E'\ge 0$,
can be replace by any general $\R$-linearly equivalent over $Z$
divisor $a_iA_i'$, where $A_i\sim A_i'/Z\ni o$.
(It is better to think about $A_i$ as
a divisor in the Alexeev sense.)

Since $(X/Z\ni o,B)$ is nonsemiexceptional,
there exists an effective divisor
$D'\ge B$ such that $(X/Z\ni o,D')$ is a log pair with $K+D'\sim_\R 0/Z\ni o$
(a nonlc $\R$-complement).
Taking a weighted combinations $D':=aB^++(1-a)D',a\in (0,1)$,
and similarly for $B^+$,
we can suppose additionally that $\Supp B^+=\Supp D'$.
Note that if $(X/Z\ni o)$ is not global, that is,
$o$ is not a closed point or is not the image of $X$ on $Z$, then
such a complement always exists with only vertical nonklt centers
of $(X,D')$ over a sufficiently general closed
point $o'$ of the closure of $o$.
In this case we replace $(X/Z\ni o,D')$ by the specialization $(X/Z\ni o',D')$.
Recall that by definition $o,o'$ belong to
the image of $X\to Z\ni o$ and in the lemma
we can replace $o$ by such $o'$.
In particular, we can suppose below that $o'=o$ is closed.

Now consider a crepant projective over $Z\ni o$
log resolution $(V,D_V')$ of $(X,D')$.
The same works for $B^+$ because $\Supp B^+=\Supp D'$.
Let $M_1,\dots,M_n$ be a finite set of
(very) ample effective generators of $\Pic(V/Z\ni o)$
which are in general position to the log birational transform
$D_V'\logb$, in particular, they do not pass through the lc centers
of $(V,D_V')$ on $V$,
and, moreover,
$$
\Supp (D_V'\logb+\sum M_i)=
\Supp(B^+\logb+\sum M_i)
$$
is also with simple normal crossings.
(The finite generation modulo $\sim/Z\ni o$
holds because $X/Z\ni 0$ has Ft.)
Divisors $M_i$ are possibly not in $\Supp D_V'\logb$ but
we can add them preserving our assumptions.
Indeed, we can add $\psi(M_i)$ to $B^+$ with
certain positive multiplicities, where
$\psi\colon V\to X$ is the log resolution.
By construction every $\psi(M_j)$ is
$\R$-linear equivalent over $Z\ni o$ to $\sum a_i^jA_i,a_i^j\in\R$, over $Z\ni o$.
Thus for sufficiently small positive real number $\ep$,
$$
(X/Z\ni o,B^+-\ep(\sum a_i^jA_i)+\ep \psi(M_j))
$$
is a klt $\R$-complement of $(X/Z\ni o,B)$ with
$\psi(M_j)$ in the support of $B^+:=B^+-\ep(\sum a_iA_i)+\ep \psi(M_j)$.
Adding each $\psi(M_j)$ we got a complement with
required properties.
Again as above we can suppose that $\Supp B^+=\Supp D'$.
By construction $\psi$ is also a log resolution
for new perturbed $B^+,D'$.
If $X/Z\ni o$ is not global, the nonklt centers of $(X,D')$ are
vertical over $o$.
But in the global case we do not have such a control.

Taking a weighted combination $B'=aB^++(1-a)D',a\in (0,1)$,
we can suppose that $(X/Z\ni o,B')$ is
an lc but nonklt $\R$-complement of $(X/Z\ni o,B)$.
Again $\psi\colon (V,B_V')\to (X,B')$ is a log resolution of
$(X,B')$.
In particular, $(V,B_V')$ is dlt.
If $X/Z\ni o$ is not global, the nonklt centers of $(X,B')$ are
vertical over $o$.

We would like to find another lc $\R$-complement
$(X/Z\ni o,B'')$ with a single lc center.
Let
$$
E=\sum E_i
$$
be the sum of prime components of $B_V'$ with
multiplicity $1$ except for one $E_0$.
(The latter component exists because $(X,B_V')$
is dlt but not klt.)
By construction
$$
E\sim \sum m_iM_i/Z\ni o, m_i\in \Z.
$$
So, for every sufficiently small positive real number $\ep$,
$$
(V,B_V'') \text{ with }
B_V''=B_V'-\ep E+\ep(\sum m_iM_i)
$$
is plt with the single nonklt center $E_0$
and $\sim_\R 0/Z\ni o$.
If $X/Z\ni o$ is not global $E_0$ is vertical over $o$.

Since $X$ is $\Q$-factorial,
the pair
$$
(X/Z\ni o,B'') \text{ with }
B''=\psi(B_V'-\ep E+\ep(\sum m_i,M_i))=
B'-\ep\psi(E)+\ep (\sum m_i \psi(M_i))
$$
with a single lc prime b-divisor $S=E_0$, that is,
with $\mult_S\B''=1$.
Actually, $B''$ is effective and a boundary for
sufficiently small $\ep$ because
$\psi(E),\psi(M_i)$ are supported in $\Supp B'$.
By construction
$$
E-\sum m_iM_i\sim 0/X
$$
too.
Hence
$$
\psi^*(\psi(E-\sum m_i M_i))=
E-\sum m_i M_i\sim 0/X
$$
because $\psi$ is a birational contraction
(e.g., by~\cite[Proposition~3]{Sh19}).
For every curve $C$ on $X$ over $Z$,
there exists a curve $C'$ on $V$ and
by the  last relation
$$
(C.\psi(E-\sum m_i M_i))=
(\psi(C').\psi(E-\sum m_i M_i))=
(C'.\psi^*(\psi(E-\sum m_i M_i)))=
(C'.0)=0
$$
by the projection formula.
The pair $(X/Z\ni o,B'')$ is a $0$-pair too:
by the last vanishing
\begin{align*}
(C.K+B'')&=(C.K+B'-\ep\psi(E)+\ep (\sum m_i \psi(M_i))) \\
&=(C.K+B')-\ep(C.\psi(E-\sum m_iM_i))=
0.
\end{align*}
So, at least the complement is numerical.
But since $X/Z\ni o$ has Ft it is actually
an $\R$-complement with a single
lc prime b-divisor $S$.

The divisor $B^+=B''$ gives a required plt complement.
Take a dlt resolution $\varphi\colon (Y,B_Y^+)\to (X,B^+)$.
It has an exceptional divisor if $S$ is exceptional.
Moreover, in this case we can suppose that $\varphi$ is extremal:
the last contraction in the LMMP for $(Y,B_Y)$ over $X$.
Otherwise it is identical (or, more precisely, the above $\Q$-factorialization).
In the last case the property with an ample divisor $A$
holds by construction.
Otherwise $S$ is exceptional on $X$ and
we need to construct an ample divisor $A$ on $Y$.
By construction, for every divisor $M_j$, its
birational image $M_{j,Y}$ on $Y$ is big on $S$
over $Z$ and does not contain $S$.
So, $M_{j,Y}$ is ample on $Y$ over $X$.
Recall for this that $\varphi$ is extremal.
A required $A$ we can construct as
linear combination $\varphi^*(A')+\ep M_{j,Y}$
for some $0<\ep\ll 1$, where $A'$ is
an ample over $Z$ divisor on $X$ supported
in $B^+$ and not passing through $\varphi(S)$.
Take $A'=A_i$ for such $A_i$ that $\varphi(S)\not\subset \Supp A_i$.
Such $A_i$ exists because $\cap\Supp A_i=\emptyset$.

If $X/Z\ni o$ is not global, $S$ is
vertical over $o$.

Similarly we can treat bd-pairs.

\end{proof}

Remark: it looks that similarly we can construct two
disjoint lc center $S_1,S_2$ if
we have at least two summands in $E$.
However, this is impossible because in
our construction every two components of $E$
intersects each other.

\begin{const}[Adjoint pair] \label{adjoint_pair}
Consider a plt model~(\ref{plt_model})
of Lemma~\ref{exist_plt_model}
and assume that $Y$is $\Q$-factorial.
The model gives an {\em adjoint\/} projective log pair
$(S,B_S)$ for $(Y,B\logb_Y)$ or $(Y,B_Y+(1-\mult_S B)S)$,
if $S$ is not exceptional on $X$, with respectively
$$
B_S=\Diff(B\logb_Y-S)\text { or }
B_S=\Diff(B-(\mult_S B) S).
$$
The adjoint pair is always global by Lemma~\ref{exist_plt_model}.
By divisorial adjunction and its monotonicity,
$(S,B_S^+)$ with
$$
B_S^+=\Diff(B_Y^+-S)
$$
is an {\em induced\/} adjoint $\R$-complement of
$(S,B_S)$.

The same applies to~(\ref{plt_model})
with a bd-pair $(X/Z\ni o,B+\sP)$ having pseudoeffective $\sP$.
The adjoint bd-pair $(S,B_S+\sP_S)$ is a log one with
$\sP_S=\sP\brest{S}$ (see the birational restriction $\brest{}$
in \cite[Mixed restriction~7.3]{Sh03}.
So, $\sP_S$ is b-nef (and pseudoeffective) if
$\sP$ is b-nef over $Z\ni o$.
Hence $(S,B_S+\sP_S)$ is a bd-pair of index $m$ if
$(X/Z\ni o,B+\sP)$ is a bd-pair of index $m$.

\end{const}

\begin{cor} \label{induced_pair}
The adjoint pair $(S,B_S)$ is (global) generic
with the induced klt $\R$-complement.
In particular, $S$ is irreducible of wFt.

\end{cor}

\begin{add}
The same works for the adjoint bd-pair $(S,B_S+\sP_S)$
of index $m$.

\end{add}

\begin{proof}
By construction
$(Y/Z\ni o,B_Y^+),(S,B_S^+) $ are log pairs, where
$S$ is normal irreducible \cite[Lemma~3.6]{Sh92}.
Moreover, $S$ is projective.
It is true if $Y/Z\ni o$ is global
because $Y$ is projective.
Otherwise by Lemma~\ref{exist_plt_model}
$S$ is projective over a closed point.

If $C\subseteq S$ is a curve then $C$ is over the closed point $o'\in Z$ and
$$
(C.K_S+B_S^+)=(C.K_Y+B_Y^+)=(\varphi(C).K+B^+)=0.
$$
Note that even in the local case $C$ is over
a closed point in $Z$ and so is over $Z$.
By construction and monotonicity of divisorial adjunction,
$$
B_Y^+\ge B_Y \text{ and }
B_S^+\ge B_S.
$$
So, $(S,B_S^+)$ is an $\R$-complement of
$(S,B_S)$.

The complement is generic because
there exists an effective ample divisor $A$
supported in $\Supp(B_Y^+-B_Y\logb)$.
It is in general position with $S$.
So, the restriction $A\rest{S}$ is well-defined and
is an effective ample divisor on $S$
supported in $\Supp(B_S^+-B_S)$.

For the bd-pair $(S,B_S+\sP)$ note that
$\sP_S$ is b-nef if $\sP$ is b-nef over $Z\ni o$.

\end{proof}

\begin{const} \label{generic_induction}
Let $d$ be a nonnegative integer,
$I,\ep,v,e$ be the data as
in Restrictions on complementary indices, and
$\Phi=\Phi(\fR)$ be a hyperstandard set associated with
a finite set of rational numbers $\fR$ in $[0,1]$.
Denote by $\widetilde{\Phi}=\Phi(\overline{\fR}\cup\{0\})$
the hyperstandard set constructed in~\ref{direct_hyperst_on_div}.

By dimensional induction there exists a finite set of
positive integers
$\sN=\sN(d-1,I,\ep,v,e,\widetilde{\Phi})$ such that
\begin{description}

\item[\rm Restrictions:\/]
every $n\in\sN$ satisfies
Restrictions on complementary indices with the given data;

\item[\rm Existence of b-$n$-complement:\/]
if $(S,B_S)$ is a generic pair of dimension $d-1$
with Ft $S$ and a boundary $B_S$
then $(S^\sharp, B_{S,n\_\widetilde{\Phi}}{}^\sharp{}_{S^\sharp})$
has a b-$n$-complement $(S,B_S^+)$ for some $n\in\sN$.

\end{description}

For bd-pairs we add a positive integer $m$.
So, $\sN=\sN(d-1,I,\ep,v,e,\widetilde{\Phi},m)$ and
replace $(S,B_S)$ by a bd-pair $(S,B_S+\sP_S)$ of
dimension $d-1$ and of index $m|n$.

Dimensional induction gives a more precise choice of $n$
that has important geometrical implications
(see Corollary~\ref{reg_generic type} below).
For this introduce the following (decreasing) filtration of $\sN$.

\paragraph{Generic type filtration with respect to
dimension $i$}
\begin{equation} \label{filtration_generic}
\sN=\sN^d
\supseteq\dots\supseteq \sN^i
\supseteq\dots\supseteq \sN^0
\supseteq\sN',\
0\le i\le d,
\end{equation}
where $\sN'$ is a subset of $\sN$.
Its {\em associated\/}  filtration of
hyperstandard sets is
$$
\Gamma(\sN^d,\Phi)
\supseteq\dots\supseteq \Gamma(\sN^i,\Phi)
\supseteq\dots\supseteq \Gamma(\sN^0,\Phi)
\supseteq \Gamma(\sN,\Phi).
$$
For dimension $i>0$, put
$$
\sN_i=\sN^i\setminus
\sN^{i-1};
$$
for $i=0$, $\sN_0=\sN^0\setminus \sN'$.
That is, the next set in the last case is $\sN^{-1}=\sN'$
but without dimension.
According to Theorem~\ref{generic_complements} below,
b-$n$-complements of generic type are coming (extended)
from dimension $i$ and $n\in\sN_i$ in this case.

For every $i$, between $\sN^i\supseteq \sN^{i-1}$,
there exists an additional filtration
with respect to semiexceptional types
in the dimension $i$ (see Semiexceptional filtration
in~Section~\ref{semiexcep_compl_:}).

\end{const}

\begin{defn} \label{filtration_generic_df}
Let~(\ref{filtration_generic}) be a filtration
of Construction~\ref{generic_induction}.
Consider a certain class of pairs $(X/Z\ni o,B)$
of dimension $d$ with boundaries which have
an (b-)$n$-complement with $n\in \sN^d$.
Such a {\em pair\/} $(X/Z\ni o,B)$ and its (b-)$n$-{\em complement
have generic type of dimension\/} $0\le i\le d$ if
the complement is extended from
a semiexceptional b-$n$-complement in dimension $i$.
So, additionally, we can associate to the pair its
(filtration) semiexceptional type $(r,f),0\le r\le i-f-1,0\le f\le i-1$
or the exceptional type $(-1,-)$.

We say that the {\em existence\/} $n$-complements {\em agrees
with the generic type filtration\/} if
every pair $(X/Z\ni o,B)$ in the class has an $n$-{\em complement\/}
$(X/Z\ni o,B^+)$
{\em of filtration generic type\/} $i$ with $n\in \sN_i$.
Additionally,
$(X/Z\ni o,B^+)$ {\em of the type\/} $i$ is
an $n$-complement of itself,
of $(X/Z\ni o,B_{n\_\sN^{i-1}\_\Phi}), (X/Z\ni o,B_{n\_\sN^{i-1}\_\Phi}{}^\sharp)$
and
a b-$n$-complement of itself,
of $(X/Z\ni o,B_{n\_\sN^{i-1}\_\Phi}),
(X/Z\ni o,B_{n\_\sN^{i-1}\_\Phi}{}^\sharp),
(X^\sharp/Z\ni o,B_{n\_\sN^{i-1}\_\Phi}{}^\sharp{}_{X^\sharp})$,
if $(X,B_{n\_\sN^{i-1}\_\Phi}),
(X,B_{n\_\sN^{i-1}\_\Phi}{}^\sharp)$ are log pairs respectively.

The boundary $B_{n\_\sN^{i-1}\_\Phi}{}^\sharp{}_{X^\sharp}$ can be
slightly increased if we take into consideration
the (filtration) semiexceptional type $(r,f)$
of $(X/Z\ni o,B)$.

Possibly, the generic type is not unique.
For the uniqueness we can take minimal (better than maximal) type.
The same concerns the semiexceptional type.

The same applies to bd-pairs of dimension $d$.

\end{defn}

\begin{rem} \label{filtration_generic_type}

(1) In particular, $(X/Z\ni o,B)$ has generic type $d$ if
the pair is generic and semiexceptional itself.
If the pair is generic and not semiexceptional then by
Lemma~\ref{exist_plt_model} and Construction~\ref{adjoint_pair}
there exists an adjoint generic pair $(S,B_S)$.
By definition $(S,B_S)$ has generic type $i\le d-1$ and
by dimensional induction this is a generic type of $(X/Z\ni o,B)$
by Step~6 in the proof of Theorem~\ref{generic_complements}.
By the induction it has also [some] semiexceptional type $(r,f)$.

(2) Generic type $d$ only possible for global pairs.

(3) In the proof of Theorem~\ref{generic_complements}
we apply Construction~\ref{adjoint_pair}
to $(X/Z\ni o,B_{\sN\_\Phi})$, where $\sN=\sN^{d-1}$.
However, we can apply the same construction directly to $(X/Z\ni o,B)$ if
the latter pair is generic (cf. Addendum~\ref{generic_complements_B}).

(4) [Warning:] Possibly, there are other $n$-complements which are
not agree with the filtration (cf. Example~\ref{n_comp_dim_1}).
Additionally, we can have $n$-complements
without any type but with $n\in\sN$.

\end{rem}

\begin{thm}[Generic $n$-complements] \label{generic_complements}
Let $d$ be a nonnegative integer,
$I,\ep,v,e$ be the data as
in Restrictions on complementary indices, and
$\Phi=\Phi(\fR)$ be a hyperstandard set associated with
a finite set of rational numbers $\fR$ in $[0,1]$.
Then there exists a finite set
$\sN=\sN(d,I,\ep,v,e,\Phi)$
of positive integers such that
\begin{description}

\item[\rm Restrictions:\/]
every $n\in\sN$ satisfies
Restrictions on complementary indices with the given data.

\item[\rm Existence of $n$-complement:\/]
if $(X/Z\ni o,B)$ is a pair with
$\dim X=d$, a boundary $B$, connected $X_o$ and
with generic $(X/Z\ni o,B_{\sN\_\Phi})$
then $(X/Z\ni o,B)$ has an $n$-complement $(X/Z\ni o,B^+)$ for some $n\in\sN$.

\end{description}

\end{thm}

\begin{add} \label{standard_generic_complements}
$(X/Z\ni o,B^+)$ is an $n$-complement of itself and
of $(X/Z\ni o,B_{n\_\Phi}), (X/Z\ni o,B_{n\_\Phi}{}^\sharp)$,
and is a b-$n$-complement of itself and
of $(X/Z\ni o,B_{n\_\Phi}),
(X/Z\ni o,B_{n\_\Phi}{}^\sharp),(X^\sharp/Z\ni o,B_{n\_\Phi}{}^\sharp{}_{X^\sharp})$,
if $(X,B_{n\_\Phi}),(X,B_{n\_\Phi}{}^\sharp)$ are log pairs respectively.

\end{add}

\begin{add} \label{generic_complements_fil}
$\sN$ has a generic type filtration~(\ref{filtration_generic})
with $\sN^d=\sN$, with any finite set
of positive integers  $\sN'$, satisfying
Restrictions on complementary indices with the given data,
and
the existence $n$-complements agrees the filtration
for the class of pairs under assumptions of
Existence of $n$-complements in the theorem.

\end{add}

\begin{add} \label{generic_complements_B}
In particular, the theorem and addenda applies
to generic pairs $(X/Z\ni o,B)$ instead of
with generic $(X/Z\ni o,B_{\sN\_\Phi})$.
\end{add}

\begin{add} \label{bd_generic_compl_fil}
The same holds for bd-pairs $(X/Z\ni o,B+\sP)$ of index $m$
with $\sN=\sN(d,I,\ep,v,e,\Phi,m)$.
That is,
\begin{description}

\item[\rm Restrictions:\/]
every $n\in\sN$ satisfies
Restrictions on complementary indices with the given data and
$m|n$.

\item[\rm Existence of $n$-complement:\/]
if $(X/Z\ni o,B+\sP)$ is a bd-pair of index $m$ with
$\dim X=d$, a boundary $B$, connected $X_o$ and
with generic $(X/Z\ni o,B_{\sN\_\Phi}+\sP)$
then $(X/Z\ni o,B+\sP)$ has an $n$-complement $(X/Z\ni o,B^++\sP)$ for some $n\in\sN$.

\end{description}
Addendum~\ref{generic_complements_fil}
holds literally.
In Addendum~\ref{standard_generic_complements}
$(X/Z\ni o,B^++\sP)$ is an $n$-complement of itself and
of $(X/Z\ni o,B_{n\_\Phi}+\sP), (X/Z\ni o,B_{n\_\Phi}{}^\sharp+\sP)$,
and is a b-$n$-complement of itself and
of $(X/Z\ni o,B_{n\_\Phi}+\sP),
(X/Z\ni o,B_{n\_\Phi}{}^\sharp+\sP),(X^\sharp/Z\ni o,B_{n\_\Phi}{}^\sharp{}_{X^\sharp}+\sP)$,
if $(X,B_{n\_\Phi}+\sP_X),(X,B_{n\_\Phi}{}^\sharp+\sP_X)$ are log bd-pairs respectively.
In Addendum~\ref{generic_complements_B}
$(X/Z\ni o,B+\sP),(X/Z\ni o,B_{\sN\_\Phi}+\sP)$
should be instead of
$(X/Z\ni o,B),(X/Z\ni o,B_{\sN\_\Phi})$ respectively.

\end{add}

\begin{war}
Sets $\sN(d,I,\varepsilon,v,e,\Phi),\sN(d,I,\varepsilon,v,e,\Phi,m)$ are
not unique and may be not suitable in other situations of the paper
(cf. Remark~\ref{N_with_parameters}, (1)).
We can take such minimal sets under inclusion but
they can be not unique.

\end{war}

It would be interesting to find a more canonical way
to construct such set $\sN(d,I,\varepsilon,v,e,\Phi),\sN(d,I,\varepsilon,v,e,\Phi,m)$
and/or to find an estimation on their largest element.

\begin{proof}
Take $\sN=\sN_d\cup\sN^{d-1}$,
where $\sN_d=\sN(d,I,\ep,v,e,\Gamma(\sN^{d-1},\Phi))$
is the finite set from Theorem~\ref{semiexcep_compl_with_filtration}
and $\sN^{d-1}=\sN(d-1,I,\ep,v,e,\widetilde{\Phi})$
is from  Construction~\ref{generic_induction}.
By Corollary~\ref{N_N'}
we can suppose that $\sN_d\cap \sN^{d-1}=\emptyset$.
By construction $\sN$ is a finite set of positive integers
and satisfies Restrictions.

Let $(X/Z\ni o,B)$ be a pair satisfying
the assumptions in Existence of $n$-complements.
Since $B$ is a boundary, $B_{\sN\_\Phi}$ is well-defined.

Step~1. {\em It is enough to construct
a b-$n$-complement $(X/Z\ni o,B^+)$ of
$(X^\sharp/Z\ni o,B_{n\_\Phi}{}^\sharp{}_{X^\sharp})$
for some $n\in \sN$.\/}
Indeed, then we have all required complements in the theorem and
in Addendum~\ref{standard_generic_complements}
by Propositions~\ref{D_D'_complement}, \ref{monotonicity_I}
and Corollary~\ref{B_B_+B_sN_Phi}.
In particular, the corollary and fact that
$(X/Z\ni o,B^+)$ is a b-$n$-complement of
$(X^\sharp/Z\ni o,B_{n\_\Phi}{}^\sharp{}_{X^\sharp})$
imply that
$(X/Z\ni o,B^+)$ is an $n$-complement of
$(X/Z\ni o,B)$ (cf. the proof of Addendum~\ref{standard_excep_compl}).

Step~2. {\em We can suppose that $(X/Z\ni o,B_{\sN^{d-1}\_\Phi})$
is generic and not semiexceptional\/}.
Indeed by Basic properties of generic pairs (1) and
Proposition~\ref{Phi_<Phi'},
$B_{\sN^{d-1}\_\Phi}\le B_{\sN\_\Phi}$ and the pair $(X/Z\ni o,B_{\sN^{d-1}\_\Phi})$
is also generic.
If $(X/Z\ni o,B_{\sN^{d-1}\_\Phi})$ is semiexceptional
then by Addendum~\ref{standard_excep_compl_semiexep_fil}
$(X^\sharp/Z\ni o,B_{n\_\sN^{d-1}\_\Phi}{}^\sharp{}_{X^\sharp})$
has a b-$n$-complement $(X/Z\ni o,B^+)$ for some $n\in\sN_d$.
We apply Theorem~\ref{semiexcep_compl_with_filtration} and
its Addendum~\ref{standard_excep_compl_semiexep_fil}
to the pair $(X,B_{\sN\_\Phi})$ with $\dim X=d$,
the boundary $B_{\sN\_\Phi}$ and
to the hyperstandard set $\Gamma(\sN^{d-1},\Phi)$
(see Proposition~\ref{every_sN_Phi_hyperstandard});
the other data is the same.
Notice that $X/Z$ is global in this case (cf. Remark~\ref{filtration_generic_type}, (2)),
that is , we can suppose that $Z=\pt$ and $X/Z$ is just $X$.
Since $(X,B_{\sN^{d-1}\_\Phi}),(X,B_{\sN\_\Phi})$ are generic,
$X$ has wFt and $(X,B_{\sN\_\Phi})$ has an $\R$-complement
by Corollary~\ref{generic_wFt} and definition respectively.
By definition and construction $\sN^{d-1}\subseteq \sN,
B_{\sN\_\Phi,\sN^{d-1}\_\Phi}=B_{\sN^{d-1}\_\Phi}$
and $(X,B_{\sN\_\Phi,\sN^{d-1}\_\Phi})$ is  semiexceptional.
By our assumptions $(X,B_{\sN\_\Phi})$ is also semiexceptional and
has generic type $d$.
Indeed, $B_{\sN^{d-1}\_\Phi}\le B_{\sN\_\Phi}$
and $(X,B_{\sN\_\Phi})$ has an $\R$-complement.
If $(X,B_{\sN\_\Phi})$ is not semiexceptional then
there exists $B'\ge B_{\sN\_\Phi}\ge B_{\sN^{d-1}\_\Phi}$ such that
$K+B'\sim_\R 0$ and $(X,B')$ is not lc.
Hence $(X,B_{\sN^{d-1}\_\Phi})$ is also not semiexceptional,
a contradiction.
Thus by Addendum~\ref{standard_excep_compl_semiexep_fil}
$(X^\sharp,B_{n\_\sN^{d-1}\_\Phi}{}^\sharp{}_{X^\sharp})$
has a b-$n$-complements $(X,B^+)$ with $n\in \sN_d$.
The constructed complement exactly
agrees with the filtration.
Hence Addendum~\ref{generic_complements_fil}
holds in this case too.

But we need a slightly weaker complement.
Again by Proposition~\ref{Phi_<Phi'} $B_{n\_\Phi}\le B_{n\_\sN^{d-1}\_\Phi}$.
Hence by Corollary~\ref{monotonicity_II} or arguments as in Step~4 below
$\B_{n\_\Phi}{}^\sharp\le \B_{n\_\sN^{d-1}\_\Phi}{}^\sharp$.
Thus by Proposition~\ref{D_D'_complement}
$(X,B^+)$ is also a b-$n$-complement of
$(X^\sharp,B_{n\_\Phi}{}^\sharp{}_{X^\sharp})$
($X^\sharp$ here is not necessary the same as above and below).

This concludes the semiexceptional case.
Below we suppose that $(X/Z\ni o,B_{\sN^{d-1}\_\Phi})$
is generic and not semiexceptional.
In this case we construct a b-$n$-complement
$(X/Z\ni o,B^+)$ of
$(X^\sharp/Z\ni o,B_{n\_\Phi}{}^\sharp{}_{X^\sharp})$
with $n\in\sN^{d-1}$.
We extend also the filtration of $\sN$
starting from $\sN^{d-1}$.

For simplicity of notation suppose now that
$B=B_{\sN^{d-1}\_\Phi}$, in particular,
$B\in \Gamma(\sN^{d-1},\Phi)$
(but this is not important for the following).
Indeed, by definition and Proposition~\ref{Phi_<Phi'}
$B_{\sN^{d-1}\_\Phi,n\_\Phi}=B_{n\_\Phi}$ and
$B_{\sN^{d-1}\_\Phi,n\_\sN^{i-1}\_\Phi}=B_{n\_\sN^{i-1}\_\Phi}$
for every $n\in\sN^{d-1},0\le i\le d$.
By Step~2 $(X/Z\ni o,B)$ is generic and not semiexceptional itself.
(Actually, if $(X/Z\ni o,B)$ is generic and not semiexceptional then
the same holds for $(X/Z\ni o,B_{\sN\_\Phi}),(X/Z\ni o,B_{\sN^{d-1}\_\Phi})$.)

Step~3. {\em Construction of $(Y/Z\ni o',B_Y^+)$
as in Lemma~\ref{exist_plt_model} and
$(Y/Z\ni o',D)$ as in Corollary~\ref{plt_b_n_complement}.\/}
Indeed, we can apply the lemma to $(X/Z\ni o,B)$.
So, there exists an $\R$-complement with
a crepant model $(Y/Z\ni o',B_Y^+)$ as in(\ref{plt_model}) and
such that

\begin{description}

\item[\rm (1)\/]
$Y\dashrightarrow X$ is a birational $1$-contraction,
in particular, every prime divisor of $X$ is a divisor on $Y$;

\item[\rm (2)\/]
$o'$ is a sufficiently general point of
the closure of $o$;

\item[\rm (3)\/]
the central fiber $Y_{o'}$ is connected;

\item[\rm (4)\/]
$(Y,B_Y^+)$ is plt with a complete single lc center $S$;
and

\item[\rm (5)\/]
there exists an effective ample over $Z\ni o'$ divisor $A$ on $Y$
with
$$
S\not\in\Supp A\subseteq\Supp(B_Y^+-B\logb_Y).
$$

\end{description}
(1-2) and (4-5) hold by Lemma~\ref{exist_plt_model}.
The connectedness in (3) and (4)
holds by the connectedness of $X_o$ and
(5) with lc connectedness respectively.

Now we take a boundary $D=B_Y^+-aA$ on $Y$,
where $a$ is a sufficiently small real number.
The pair $(Y/Z\ni o',D)$ satisfies the assumptions
of Corollary~\ref{plt_b_n_complement}.
(1) of the corollary holds by construction and (2-3).
(2) of the corollary follows from (4-5).
(3) of the corollary follows from (5):
$$
-(K_Y+D)=-(K_Y+B_Y^+)+aA\equiv aA/Z\ni o'.
$$
By (1), construction and since $B$ is a boundary,
$D$ is actually a boundary if
$a$ is sufficiently small.
Additionally, we can suppose that
\begin{description}
  \item[\rm (6)\/]
for every prime divisor $P$ of $X$,
$$
\mult_P D=\mult_P \D\ge \mult_P B.
$$
\end{description}
Indeed, by (1) $P$ is a prime divisor on $Y$
and
$$
\mult_P \D=\mult_P D=(\mult_P B_Y^+)-a\mult_P A.
$$
Thus by (5) $\mult_P D=\mult_PB_Y^+\ge \mult_P B$ if
$P\not\in\Supp A\subseteq\Supp(B_Y^+-B\logb_Y)$.
Otherwise $P$ belongs to a finite set for which
$\mult_PB_Y^+>\mult_P B$.
Notice also that $S\not\in \Supp A$ by (5)
and $\mult_S D=\mult_SB_Y^+=1$.

Step~4. {\em It is enough to construct a b-$n$-complement
of $(Y^\sharp/Z\ni o',D_{n\_\Phi}{}^\sharp{}_{Y^\sharp})$
for some $n\in\sN^{d-1}$.\/}
Indeed, this  complement will be also
a b-$n$-complement of $(X^\sharp/Z\ni o',B_{n\_\Phi}{}^\sharp{}_{X^\sharp})$
and of $(X^\sharp/Z\ni o,B_{n\_\Phi}{}^\sharp{}_{X^\sharp})$
by Proposition~\ref{D_D'_complement} because $o'$ is
a specialization of $o$ and
$\D_{n\_\Phi}{}^\sharp\ge \B_{n\_\Phi}{}^\sharp{}_{X^\sharp}$.
The last inequality follows from (6) by Corollary~\ref{monotonicity_III},
actually, for any $n$ and $\Phi$.
We apply the corollary to
$(X^\sharp/Z\ni o,B_{n\_\Phi}{}^\sharp{}_{X^\sharp})$
with a birational $1$-contraction $X\dashrightarrow X^\sharp/Z\ni o$.
Such a contraction exists by Construction~\ref{sharp_construction}
because $B_{n\_\Phi}\le B$ and $(X/Z\ni o',B_{n\_\Phi})$
has an $\R$-complement again by Proposition~\ref{D_D'_complement}.
The same arguments imply that Construction~\ref{sharp_construction}
is applicable to $(Y/Z\ni o',D_{n\_\Phi})$.
The assumption~(1) of Corollary~\ref{monotonicity_III}
holds by the construction.
On the other hand, by construction $Y\dashrightarrow X^\sharp/Z\ni o'$
is a birational $1$-contraction and every prime divisor of $X^\sharp$
is a divisor on $X$ and $Y$.
Hence the assumption~(2) of Corollary~\ref{monotonicity_III}
follows from (6), the definition of $(-)_{n\_\Phi}$
and Proposition~\ref{monotonicity_I}: for every
prime divisor $P$ of $X^\sharp$,
$$
\mult_P B_{n\_\Phi}{}^\sharp{}_{X^\sharp}=
\mult_P B_{n\_\Phi}\le \mult_P D_{n\_\Phi}\le
\mult_P\D_{n\_\Phi}{}^\sharp.
$$

Step~5. {\em Construction of a b-$n$-complement
of $(Y^\sharp/Z\ni o',D_{n\_\Phi}{}^\sharp{}_{Y^\sharp})$
for some $n\in\sN^{d-1}$.\/}
We apply Corollary~\ref{plt_b_n_complement} to
$(Y/Z\ni o',D)$.
The assumptions~(1-3) of the corollary we verified in Step~3.
So, we need to verify (4) of Corollary~\ref{plt_b_n_complement}.
As in Corollary~\ref{induced_pair} but now by~(5),
the adjoint pair $(S,D_S)$ in (4) of Corollary~\ref{plt_b_n_complement}
is generic global of
dimension $\dim X -1=d-1$.
Notice that $S$ is irreducible by (3) and
\cite[Lemma~3.6]{Sh92}.
By (4) and construction, $(S,D_S)$ is klt with a boundary $D_S$.
It has a highest crepant model $(S',D_{S'})$
with a boundary $D_{S'}$.
By Basic properties of generic pairs~(3)
$(S',D_{S'})$ is also generic.
By Corollary~\ref{generic_wFt} $S'$ has also wFt.
Thus by Construction~\ref{generic_induction}
$(S'^\sharp, D_{S,n\_\widetilde{\Phi}}{}^\sharp{}_{S'^\sharp})$
has a b-$n$-complement $(S',D_{S'}^+)$ for some $n\in\sN^{d-1}$.
This is exactly (4) of Corollary~\ref{plt_b_n_complement}.
More accurately, we need to use
the induced b-$n$-complement $(S'^\sharp,D_{S'^\sharp}^+)$.

So, the theorem and Addendum~\ref{standard_generic_complements} are established.

Step~6. {\em Addendum~\ref{generic_complements_fil}.\/}
We use dimensional induction to construct $\sN^{d-1}$
with a required filtration.
By~\ref{direct_hyperst_on_div} $\widetilde{\Phi}$ is already closed:
$\widetilde{\widetilde{\Phi}}=\widetilde{\Phi}$.
In the induction we work with $\widetilde{\Phi}$ instead of $\Phi$.
By Steps~1-5 above and especially by Lemma~\ref{exist_plt_model},
we can suppose that $X$ is global and every pair $(X,B)$
in the induction has $\dim X\le d-1$.
By~(4-5) of Step~3 $(X,B)$ is generic.
Essentially, we verify the induction in Construction~\ref{generic_induction}.

{\em Induction step\/} $d=0$. Take $\sN_0=
\sN(0,I,\ep,v,e,\widetilde{\Phi})$, a finite set
from Theorem~\ref{semiexcep_compl_with_filtration}
in dimension $0$.
By construction $\sN_0$ is a finite set of positive integers
and satisfies Restrictions on complementary indices with the given data.
The existence of complements in dimension $0$ is trivial.
Moreover, we can add any finite set of positive integers $\sN'=\sN\1$
satisfying Restrictions:
$\sN^0=\sN_0\cup\sN'$ and $\sN_0\cap\sN'=\emptyset$.
Actually \cite[Corollary~1.3]{BSh} is enough for this step.

{\em General step of induction:
construction of $\sN^{i+1}$.\/}
Suppose that a filtration is constructed up to $\sN^i, 0\le i\le d-2$,
and the filtration satisfies Addendum~\ref{generic_complements_fil}
in dimension $i$ with $\widetilde{\Phi}$ instead of $\Phi$.
The class of pairs consists of global generic pairs $(X,B)$
with a boundary $B$ and $\dim X=i$.
Additionally we can assume that $(X,B)$ is klt and highest
with a boundary.
Take $\sN_{i+1}=\sN(i+1,I,\ep,v,e,\Gamma(\sN^i,\widetilde{\Phi}))$
of Theorem~\ref{semiexcep_compl_with_filtration}.
By Corollary~\ref{N_N'}
we can suppose that $\sN_{i+1}$ is disjoint from $\sN^i$.
This gives a required filtration up to $\sN^{i+1}=\sN_{i+1}\cup\sN^i$ and
by induction up to $\sN^{d-1}=\sN_{d-1}\cup\sN^{d-2}$.
By construction $\sN^{i+1}$ is a finite set of positive integers
and satisfies Restrictions.
We need to verify Addendum~~\ref{generic_complements_fil}
in dimension $i+1$ with $\widetilde{\Phi}$ instead of $\Phi$.
The class of pairs consists of global generic pairs $(X,B)$
with a boundary $B$ and $\dim X=i+1$.
Additionally we assume that $(X,B)$ is klt and highest
with a boundary.
In particular, we assign generic type
$0\le t\le i+1$ for every such pair.
For this we apply Steps~1-5 above with $d=i+1$ and $\Phi=\widetilde{\Phi}$.

By Step~1 it is enough to construct
a b-$n$-complement $(X,B^+)$ of
$(X^\sharp,B_{n\_\sN^{t-1}\_\widetilde{\Phi}}{}^\sharp{}_{X^\sharp})$
with $n\in \sN_t$ for $(X,B)$ of type $t$.

By our assumptions $(X,B)$ is generic.
Thus $(X,B_{\sN^i\_\widetilde{\Phi}})$
is also generic by Basic properties of generic pairs (1).
If $(X,B_{\sN^i\_\widetilde{\Phi}})$ is
additionally semiexceptional then its type is $i+1$ and
the required b-$n$-complement of type $i+1$ with $n\in\sN_{i+1}$
exists by Addendum~\ref{standard_excep_compl_semiexep_fil}.

So, we assume that $(X,B_{\sN^i\_\widetilde{\Phi}})$ is
not semiexceptional as in Step~2.
In this case $t\le i$.
By Steps~3-5 a required b-$n$-complement is extended
from a b-$n$-complement of a generic pair $(S,D_S)$
of dimension $i$.
By induction the pair has some type $0\le t\le i$.
This is a type of $(X,B)$ and of $(X,B_{\sN^i\_\widetilde{\Phi}})$.
Indeed, by induction $(S^\sharp,D_{S,n\_\sN^{t-1}\_\widetilde{\Phi}}{}^\sharp{}_{X^\sharp})$
has a b-$n$-complement of type $t$ with $n\in \sN_t$.
By Steps~4-5, the complement can be extended
to a b-$n$-complement $(X,B^+)$ of
$(X^\sharp,B_{n\_\sN^{t-1}\_\widetilde{\Phi}}{}^\sharp{}_{X^\sharp})$
with the same $n\in \sN_t$ for $(X,B)$ and same type $t$.
Notice only that $B_{\sN^i\_\widetilde{\Phi},n\_\sN^{t-1}\_\widetilde{\Phi}}=
B_{n\_\sN^{t-1}\_\widetilde{\Phi}}$ because
$t\le i$ and $n\in\sN^i,\sN^{t-1}\subseteq \sN^i$.

This completes the induction in Construction~\ref{generic_induction}.

Finally, to complete the proof of Addendum~~\ref{generic_complements_fil}
we assign type $t\le d-1$ to every  pair
$(Y/Z\ni o',D)$ of Step~3.
This is the type of $(S',D_{S'})$ in Step~5.
The same type we assign to $(X/Z\ni o,B)$.
This is immediate for $(Y/Z\ni o',D)$ by
Corollary~\ref{plt_b_n_complement} as in Step~5.
By arguments of Step~4 this applies also to $(X/Z\ni o,B)$.

Warning:
In the proof of the theorem we
can't replace $(X/Z\ni o, B)$ by its
crepant model $(X'/Z\ni o, B_{X'})$ because
$'$ and $\sN\_\Phi$ do not commute.
Moreover,
assumptions in Existence of $n$-complements
do not imply that $(X'/Z\ni o,B_{X',\sN\_\Phi})$ is generic.
(This is true but in opposite direction.)
However, we do not need this in the induction because we take $'$ on
an induced model with the generic property
(cf. Corollary~\ref{plt_b_n_complement}
and Step~5).

Step~7. {\em Other addenda.\/}
Addendum~\ref{generic_complements_B} follows from
Basic properties of generic pairs~(1).

Similarly we can treat bd-pairs.

\end{proof}

\begin{cor} \label{reg_generic type}
Under assumptions and in notation of Theorem~\ref{generic_complements} if
$(X/Z\ni o,B_{\sN\_\Phi})$ has generic type $i$ then
$$
\reg(X)\ge \reg(X/Z\ni o,B^+)\ge d-i-1.
$$
More precisely, with the additional
(filtration) semiexceptional type $(r,f)$ with:
$$
\reg(X/Z\ni o)\ge \reg(X/Z\ni o,B^+)\ge d+r-i.
$$
\end{cor}

In general, we can't replace $\reg(X/Z\ni o)=\reg(X/Z\ni o,0)$
by $\reg(X/Z\ni o,B)$.
However, this works if $B^+\ge B$.

\begin{proof}
Indeed, the first right inequality hold for type $d$
because $\reg(X/Z\ni o,B^+)\ge -1$ ($-1$ for the empty $\mathrm{R}(X/Z\ni o,B^+)$).
Otherwise by induction, Theorem~\ref{extension_n_complement} and
Corollary~\ref{plt_b_n_complement}
the extension increase $\reg$ by $1$:
$$
\reg(X/Z\ni o,B^+)=\reg(S',B_{S'}^+)+1=
\reg(S,B_S^+)+1\ge d-1-i -1 +1=d-i-1.
$$
The second right inequality follows by the same arguments.
However the induction use definition:
if (filtration) semiexceptional generic $(X,B)$ of dimension $d$ has
the (filtration) semiexceptional type $(r,f)$
then $\reg(X,B^+)= r$ (cf. Corollary~\ref{reg_semiexceptional type}).

Both left inequalities with $\reg(X/Z\ni o)$ hold
by \cite[Proposition-Definition~7.11]{Sh95}.
\end{proof}

\section{Klt nongeneric complements: lifting of generic} \label{klt_nongeneric}

In this section we establish Theorem~\ref{bndc}
under the assumption (1) of the theorem.

\paragraph{Klt type.}
A pair $(X/Z\ni,B)$ with a boundary $B$ has
{\em klt\/} type if
there exists a klt $\R$-complement
$(X/Z\ni o,B^+)$ of $(X/Z\ni o,B)$.
In this situation a log pair $(X/Z\ni o,B)$ is klt itself.

The same definition works
for bd-pairs $(X/Z\ni o,B+\sP)$ with a boundary$B$.

\begin{const} \label{const_b_contraction}
Let $(X/Z\ni o,B)$ be a pair of klt type with wFt $X/Z\ni o$.
Then by Construction~\ref{sharp_construction}
there exists an {\em associated\/} b-$0$-contraction:
a birational $1$-contraction $\varphi$
to a $0$-contraction $\psi$ over $Z\ni o$
$$ \label{klt_model}
\begin{array}{ccc}
(X,B)&\stackrel{\varphi}{\dashrightarrow}&(X^\sharp,B_{X^\sharp})\\
\downarrow&&\psi\downarrow\\
Z\ni o&\leftarrow& (Y,B^\sharp{}\dv+B^\sharp{}\md)
\end{array},
$$
where
\begin{description}

  \item[]

$(X^\sharp/Z\ni o,B_{X^\sharp})$ is a maximal model of $(X/Z\ni o,B)$
with contracted fixed over $Z\ni o$ components of $-(K+B)$,
in particular, $B_{X^\sharp}^\sharp=B_{X^\sharp}$;

  \item[]
$\psi$ is a $0$-contraction given over $Z\ni o$ by
the nef over $Z\ni o$ divisor
$-(K_{X^\sharp}+B_{X^\sharp})$;
$\psi$ satisfies the assumptions of~\ref{adjunction_0_contr};

  \item[]
$B^\sharp{}\dv,B^\sharp{}\md$ are respectively the divisorial and
moduli part of adjunction for $\psi$;

  \item[]
$(Y/Z\ni o,B^\sharp{}\dv+\sB^\sharp{}\md)$
is the {\em adjoint\/} generic klt bd-pair.

\end{description}
Indeed, since $(X/Z\ni o,B)$ has a klt $\R$-complement,
$(X^\sharp,B_{X^\sharp})$ also has a klt $\R$-complement.
Hence by~\ref{adjunction_div}, (6)
$(Y/Z\ni o,B^\sharp{}\dv+\sB^\sharp{}\md)$
is klt too.
The bd-pair is generic
by construction and~\ref{adjunction_div}, (8) and~(12).

Note that if $X_o$ is connected then
$X_o^\sharp$ and $Y_o$ are connected too.

The same construction works
for bd-pairs $(X/Z\ni o,B+\sP)$
of klt type with wFt $X/Z\ni o$.

\end{const}

A pair $(X/Z\ni o,B)$ of klt type is
generic itself if and only if $\psi$ is birational.
Otherwise $(X/Z\ni o,B)$ is fibered and
we can apply Theorem~\ref{generic_complements},
the existence of generic $n$-complements.

\begin{defn}
A klt type pair $(X/Z\ni o, B)$  with wFt $X/Z\ni o$
has {\em klt type\/} $f$ if
in Construction~\ref{const_b_contraction}
$\dim Y=f$.
So, $f\in \Z$ and $0\le f\le \dim X$.

The same definition works
for bd-pairs $(X/Z\ni o,B+\sP)$
of klt type with wFt $X/Z\ni o$.

\end{defn}

In particular, the generic type is exactly a klt one with $f=\dim X$.

\begin{const}[Cf. Construction~\ref{semiexceptional_induction}] \label{klt_induction}
Let $d$ be a nonnegative integer and
$\Phi=\Phi(\fR)$ be a hyperstandard set associated with
a finite set of rational numbers $\fR$ in $[0,1]$.
By Theorem~\ref{adjunction_index}
there exists a positive integer $J$ such that
every contraction $\psi\colon (X^\sharp,B_{X^\sharp})\to Y/Z\ni o$
of Construction~\ref{const_b_contraction} has
the adjunction index $J$ if
\begin{description}

\item[\rm (1)\/]
$\dim X=d$, and

\item[\rm (2)\/]
$B\hor\in\Phi$.

\end{description}
Moreover, there exists
a finite set of rational numbers $\fR'$ in $[0,1]$
such that $\Phi'=\Phi(\fR')$ satisfies Addendum~\ref{adjunction_index_div}.
More precisely, $\fR'$ is defined by~(\ref{const_adj}),
where
$$
\fR''=[0,1]\cap \frac \Z J.
$$

By Addendum~\ref{adjunction_index_bd},
the same adjunction index $J$ has every contraction
$\psi\colon (X^\sharp,B_{X^\sharp}+\sP)\to Y/Z\ni o$
of Construction~\ref{const_b_contraction} if we
apply the construction to a bd-pair $(X/Z\ni o,B+\sP)$ and suppose
additionally to (1-2) that $(X/Z\ni o,B+\sP)$ is a bd-pair of index $m$,
or equivalently, $(X^\sharp/Z\ni o,B_{X^\sharp}+\sP)$ is
a log bd-pair of index $m$.

Let $I,\ep,v,e$ be the data
as in Restrictions on complementary indices in Section~\ref{intro} and
$f$ be a nonnegative integer such that $f\le d-1$.
By Addenda~\ref{generic_complements_B}-\ref{bd_generic_compl_fil} of
Theorem~\ref{generic_complements} or
by dimensional induction there exists a finite set of
positive integers
$\sN=\sN(f,I,\ep,v,e,\Phi',J)$  such that
\begin{description}

\item[\rm Restrictions:\/]
every $n\in\sN$ satisfies
Restrictions on complementary indices with the given data[,
in particular, $J|n$];

\item[\rm Existence of $n$-complement:\/]
if $(Y/Z\ni o,B_Y+\sQ)$ is a generic bd-pair of dimension $f$ and
of index $J$ with wFt $Y/Z\ni o$, with a boundary $B_Y$ and
connected $Y_o$
then $(Y^\sharp/Z\ni o,B_{Y,n\_\Phi'}{}^\sharp{}_{Y^\sharp}+\sQ)$
has a b-$n$-complement
$(Y/Z\ni o,B_Y^++\sQ)$ for some $n\in \sN$.

\end{description}

\end{const}

\begin{thm} \label{klt_compl}
Let $d,\fR,\Phi,J,I,\fR',\Phi',\ep,v,e,f,\sN$ be
the data of Construction~\ref{klt_induction}.
Let $(X/Z\ni o,B)$ be a pair with a boundary $B$,
connected $X_o$ such that
\begin{description}

\item[\rm (1)\/]
$X/Z\ni o$ has wFt;

\item[\rm (2)\/]
$\dim X=d$;
and

\item[\rm (3)]
both pairs
$$
(X/Z\ni o,B_\Phi),\ (X/Z\ni o,B_{\sN\underline{\ }\Phi})
$$
have the same klt type $f$.

\end{description}
Then there exists $n\in\sN$ such that $(X/Z\ni o,B)$ has
an $n$-complement $(X/Z\ni o,B^+)$.

\end{thm}

Notice that we do not assume that $(X/Z\ni o,B)$ has an $\R$-complement.

\begin{add} \label{standard_klt_compl_semiexep}
$(X/Z\ni o,B^+)$ is an $n$-complement of itself and
of $(X/Z\ni o,B_{n\_\Phi}), (X/Z\ni o,B_{n\_\Phi}{}^\sharp)$,
and is a b-$n$-complement of itself and
of $(X/Z\ni o,B_{n\_\Phi}),(X/Z\ni o,B_{n\_\Phi}{}^\sharp),(X^\sharp/Z\ni o,B_{n\_\Phi}{}^\sharp{}_{X^\sharp})$,
if $(X/Z\ni o,B_{n\_\Phi}),(X/Z\ni o,B_{n\_\Phi}{}^\sharp)$ are log pairs respectively.

\end{add}

\begin{add} \label{klt_compl_N_N'}
$\sN$ disjoint from any finite set of positive integers $\sN'$.
\end{add}

\begin{add}
The same holds for
the bd-pairs $(X/Z\ni o,B+\sP)$ of index $m$ and with $\sN$
as in Construction~\ref{klt_induction}.
That is,
\begin{description}

\item[\rm Existence of $n$-complement:\/]
if $(X/Z\ni o,B+\sP)$ is a bd-pair of index $m$ with a boundary $B$,
$X_o$ connected, under (1-2) and
such that
\item[]
both bd-pairs
$$
(X/Z\ni o,B_\Phi+\sP),\ (X/Z\ni o,B_{\sN\underline{\ }\Phi}+\sP)
$$
have the same klt type $f$,

then $(X/Z\ni o,B+\sP)$ has an $n$-complement $(X/Z\ni o,B^++\sP)$ for some $n\in\sN$.

\end{description}
In Addenda~\ref{standard_klt_compl_semiexep}
$(X/Z\ni o,B^++\sP)$ is
an $n$-complement of itself and
of $(X/Z\ni o,B_{n\_\Phi}+\sP), (X/Z\ni o,B_{n\_\Phi}{}^\sharp+\sP)$,
and is a b-$n$-complement of itself and
of $(X/Z\ni o,B_{n\_\Phi}+\sP),(X/Z\ni o,B_{n\_\Phi}{}^\sharp+\sP),
(X^\sharp/Z\ni o,B_{n\_\Phi}{}^\sharp{}_{X^\sharp}+\sP)$,
if $(X/Z\ni o,B_{n\_\Phi}+\sP_X),(X/Z\ni o,B_{n\_\Phi}{}^\sharp+\sP_X)$ are log bd-pairs
respectively.

\end{add}

The proof of the theorem is very similar to
proofs of Theorems~\ref{semiexcep_compl} and~\ref{generic_complements}.
So, we will be sketchy of it.

\begin{proof}
By Construction~\ref{klt_induction}
$\sN$ is a finite set of positive integers
and satisfies Restrictions.
So, it is enough to construct required $n$-complements.

Let $(X/Z\ni o,B)$ be a pair satisfying
the assumptions in Existence of $n$-complements.
Since $B$ is a boundary, $B_{\sN\_\Phi}$ is well-defined.
It is enough to construct
a b-$n$-complement $(X/Z\ni o,B^+)$ of
$(X^\sharp/Z\ni o,B_{n\_\Phi}{}^\sharp{}_{X^\sharp})$
for some $n\in \sN$
(cf. Step~1 in the proof of Theorem~\ref{generic_complements}).

As in the proof of Theorem~\ref{semiexcep_compl}
we can suppose that $B=B_{\sN\_\Phi}$,
that is, $(X/Z\ni o,B)$ has also klt type $f$.

According to Construction~\ref{const_b_contraction}
there exists $0$-contraction
$\psi\colon (X^\sharp,B_{X^\sharp})\to (Y,B^\sharp{}\dv+B^\sharp{}\md)$
over $Z\ni o$ with $\dim Y=f$.

Step~1. {\em We can suppose that $\psi$ is defined on $(X/Z\ni o,B)$.\/}
Equivalently, $\varphi$ is identical.
For this we use a small modification $\varphi$.
Thus we use $(X^\sharp/Z\ni o,B_{X^\sharp}^\sharp{}_{,X^\sharp})$
crepant to $(X^\sharp/Z\ni o,B_{X^\sharp})$ of Construction~\ref{const_b_contraction}.
By Proposition~\ref{monotonicity_I}
$B_{X^\sharp}^\sharp{}_{,X^\sharp}\ge B$.
Hence by definition and Corollary~\ref{monotonicity_II},
$\B_{X^\sharp}^\sharp{}_{,X^\sharp,n\_\Phi}{}^\sharp\ge
\B_{n\_\Phi}{}^\sharp$.
By Proposition~\ref{D_D'_complement} a b-$n$-complement of
$(X^{\sharp,\sharp}/Z\ni o,
B_{X^\sharp}^\sharp{}_{,X^\sharp,n\_\Phi}{}^\sharp{}_{X^{\sharp,\sharp}})$
is a b-$n$-complement of
$(X^\sharp/Z\ni o,B_{n\_\Phi}{}^\sharp{}_{X^\sharp})$.
On the other hand, $(X^\sharp/Z\ni o,B_{X^\sharp}^\sharp{}_{,X^\sharp})$
also has klt type $f$ because Construction~\ref{sharp_construction}
preserves the property by Proposition~\ref{monotonicity_I} and the invariance of
klt singularities for crepant $0$-pairs.

Additionally,
$(X^\sharp/Z\ni o,B_{X^\sharp}^\sharp{}_{,X^\sharp,\Phi})$
has the same klt type $f$.
Indeed, by Proposition~\ref{monotonicity_I},
construction and definition
$B_{X^\sharp,\Phi}\le B_{X^\sharp}^\sharp{}_{,X^\sharp,\Phi}\le
B_{X^\sharp}^\sharp{}_{,X^\sharp}$, where
$B_{X^\sharp}$ is the birational transform of $B$ on $X^\sharp$.
Since $X^\sharp$ is a small modification of $X$ over $Z\ni o$,
$(X^\sharp/Z\ni o, B_{X^\sharp,\Phi})$ has the same klt type $f$ as
$(X/Z\ni o,B_\Phi)$ (cf. Basic properties of generic pairs (2)).
Hence $(X^\sharp/Z\ni o,B_{X^\sharp}^\sharp{}_{,X^\sharp,\Phi})$
also has klt type $f$
because so does $(X^\sharp/Z\ni o,B_{X^\sharp}^\sharp{}_{,X^\sharp})$
(cf. Lemma~\ref{B_D_r_f}).

By Lemma~\ref{wTt_vs_Ft} $X^\sharp/Z\ni o$ has
wFt, $\dim X^\sharp=\dim X=d$.
(Actually we need the wFt property only for
Construction~\ref{const_b_contraction} and do not use it further.)

So, we can denote $(X^\sharp/Z\ni o,B_{X^\sharp}^\sharp{}_{,X^\sharp})$
by $(X/Z\ni o,B)$.
We have a $0$-contraction $\psi$ on $X$:
$$
(X,B)\stackrel{\psi}{\to} (Y,B\dv+B\md)/ Z\ni o.
$$
To construct a required complement we use Theorem~\ref{invers_b_n_comp}
with the same sets $\Phi,\Phi'$, the index $J$ and
some $n\in\sN$.
The sets $\Phi,\Phi'$  of Construction~\ref{klt_induction} agree
with~\ref{direct_hyperst} for $\fR''=[0,1]\cap(\Z/J)$ and
$J|n$ for every $n\in\sN$.
We apply the theorem to the $0$-contraction
$\psi\colon (X,B)\to Y/Z\ni o$ with lc $(X,B)$.
By construction the contraction $\psi$ satisfies~\ref{adjunction_0_contr},
$B$ is a boundary and $Y\to Z\ni o$ is proper (projective).
So, we need only to verify that $\psi$ has the adjunction index $J$,
to choose $n\in \sN$ and
to construct an appropriate model $Y'/Z\ni o$ of $Y/Z\ni o$
with an $\R$-divisor $D'\dv$ satisfying required properties.

Step~2. $(X,B)\to Y$ {\em has the adjunction index\/} $J$.
By~(3) $(X/Z\ni o,B_\Phi)$ has klt type $f$.
By Construction~\ref{const_b_contraction},
applied to $(X/Z\ni o,B_\Phi)$,
there exists another $0$-contraction
$\psi''\colon (X''^\sharp,B_{\Phi,X''^\sharp})\to Y''$ over $Z\ni o$
which is generically crepant to
$(X,B)\to Y$ over $Y$.
Indeed, $B_\Phi\le B$ and pairs
$(X/Z\ni o, B_\Phi),(X/Z\ni o,B)$ have
the same klt type $f$.
Thus $\B_\Phi{}^\sharp=\B^\sharp$
generically over $Y$.
Thus by Proposition~\ref{adjunction_same_mod_etc}, (5)
$\psi,\psi''$ have the same adjunction index.

On the hand, by definition and Construction~\ref{const_b_contraction}
$B_\Phi$ and $B_{\Phi,X''^\sharp}\in\Phi$.
Thus by Construction~\ref{klt_induction} and
Theorem~\ref{adjunction_index},
$(X''^\sharp,B_{\Phi,X''^\sharp})\to Y''$ has
the adjunction index $J$.
Hence $(X\ni o,B)\to Y$ also has the adjunction index $J$.

Step~3. {\em Choice of $n\in\sN$ and construction
of $(Y'/Z\ni o,D'\dv+\sB\md)$.\/}
By Construction~\ref{const_b_contraction}
the adjoint bd-pair $(Y/Z\ni o,B\dv+\sB\md)$
is generic klt.
By construction $\dim Y=f$,
$B\dv$ is a boundary and $Y_o$ is connected.
Additionally, $(Y/Z\ni o,B\dv+\sB\md)$ is
a bd-pair of index $J$ by \ref{adjunction_index_of}, (3)
(cf. Addendum~\ref{adjunction_index_adjoint}).
Thus by Existence of complement in Construction~\ref{klt_induction}
$(Y^\sharp/Z\ni o,B\dv{}_{,n\_\Phi'}{}^\sharp{}_{Y^\sharp}+\sB\md)$
has a b-$n$-complements for some $n\in\sN$.
However we need a slightly stronger complement
to apply Theorem~\ref{invers_b_n_comp}.

For this take a crepant blowup
$(\widetilde{Y},B\dv{}_{,\widetilde{Y}}+\sB\md)\to
(Y,B\dv+\sB\md)$ such that,
for every vertical over $Y$ prime divisor $P$ on
$X$, its image $Q=\psi(P)$ on $\widetilde{Y}$ is
a divisor.
By construction $(\widetilde{Y}/Z\ni o,B\dv{}_{,\widetilde{Y}}+\sB\md)$
is a bd-pair of index $J$ with $\dim \widetilde{Y}=f$ and
connected $\widetilde{Y}_o$.
By (12) in~\ref{adjunction_div}
$B\dv{}_{,\widetilde{Y}}$ is a boundary
because $(X,B)$ is lc.
Hence by Basic properties of generic pairs (3-4) and
Corollary~\ref{generic_wFt},
$(\widetilde{Y}/Z\ni o,B\dv{}_{,\widetilde{Y}}+\sB\md)$
is generic with wFt $\widetilde{Y}/Z\ni o$.
By definition of bd-pairs of index $J$,
$\sB\md$ is b-nef and $B\md{}_{,\widetilde{Y}}$
is pseudoeffective.
Thus again by
Existence of complement in Construction~\ref{klt_induction}
$(\widetilde{Y}^\sharp/Z\ni o,
B\dv{}_{,\widetilde{Y},n\_\Phi'}{}^\sharp{}_{\widetilde{Y}^\sharp}+\sB\md)$
has a b-$n$-complement for some $n\in \sN$.
This is our choice of $n$ and a required model
is
$$
(Y'/Z\ni o,D'\dv+\sB\md)=
(\widetilde{Y}^\sharp/Z\ni o,
B\dv{}_{,\widetilde{Y},n\_\Phi'}{}^\sharp{}_{\widetilde{Y}^\sharp}+\sB\md)
$$
with a small modification $\widetilde{Y}\dashrightarrow\widetilde{Y}^\sharp=Y'$.
Such a model with a small modification exists by
Construction~\ref{sharp_construction}.
Notice that notation $D'\dv$ means that $\D'\dv$
adjunction corresponds to some
b-divisor $\D'$ on $X$ \ref{adjunction_div}, (3).
By construction and definition
$(Y'/Z\ni o,D'\dv+\sB\md)$ satisfies (2-3)
of Theorem~\ref{invers_b_n_comp}.
By construction, for every vertical over $Y$ prime divisor $P$ on
$X$, its image $Q=\psi(P)$ on $\widetilde{Y}^\sharp$ is
a divisor and inequality~(1) of Theorem~\ref{invers_b_n_comp}
holds by Proposition~\ref{monotonicity_I}.
Thus by Theorem~\ref{invers_b_n_comp}
$(X^\sharp/Z\ni o,B_{n\_\Phi}{}^\sharp{}_{X^\sharp})$
has a b-$n$-complement.

Step~4. {\em Addenda.\/}
Addendum~\ref{klt_compl_N_N'} can
be proved by the same arguments as Corollary~\ref{N_N'}.

Similarly we can treat bd-pairs.

\end{proof}

\paragraph{Klt type filtration.\/}
Let $d, \Phi,\Phi'$ be the same data as in Construction~\ref{klt_induction} and
$\sN\supseteq\sN'$ be two sets of positive integers.
A {\em klt type filtration with respect to
dimension\/} $i$ is a (decreasing) filtration
\begin{equation} \label{filtration_klt}
\sN=\sN^0
\supseteq\dots\supseteq \sN^i
\supseteq\dots\supseteq \sN^d
\supseteq\sN',\
0\le i\le d.
\end{equation}
Its {\em associated\/} (decreasing) filtration of
hyperstandard sets is
$$
\Gamma(\sN^0,\Phi')
\supseteq\dots\supseteq \Gamma(\sN^i,\Phi')
\supseteq\dots\supseteq \Gamma(\sN^d,\Phi')
\supseteq \Gamma(\sN',\Phi').
$$
For dimension $0\le i<d$, put
$$
\sN_i=\sN^i\setminus
\sN^{i+1};
$$
for $i=d$, $\sN_d=\sN^d\setminus \sN'$.
That is, the next set in the last case is $\sN^{d+1}=\sN'$
but without dimension.

For every $i$, between $\sN^i\supseteq \sN^{i+1}$,
there exists an additional filtration
with respect to generic types in dimension $i$.

\begin{defn} \label{klt_filtration}
Let~(\ref{filtration_klt}) be a klt type filtration.
Consider a certain class of pairs $(X/Z\ni o,B)$
of dimension $d$ with boundaries $B$ which have
an (b-)$n$-complement with $n\in \sN^0$.
Such a {\em pair\/} $(X/Z\ni o,B)$ and its (b-)$n$-{\em complement
have klt type of dimension\/} $0\le i\le d$
{\em with respect to filtration\/}~(\ref{filtration_klt}) if
both pairs
$$
(X/Z\ni o,B_{\sN^{i+1}\_\Phi}),(X/Z\ni o,B_{\sN^i\_\Phi})
$$
are of klt type $i$.
(We suppose that $\sN^{d+1}=\sN'$.)

We say that the {\em existence\/} $n$-complements {\em agrees
with the klt type filtration\/} if
every pair $(X/Z\ni ,B)$ of klt type has an $n$-{\em complement\/}
$(X/Z\ni ,B^+)$
{\em of filtration klt type\/} $i$ with $n\in \sN_i$.
Additionally,
$(X/Z\ni o,B^+)$ {\em of the type\/} $i$ is
an $n$-complement of itself,
of $(X/Z\ni o,B_{n\_\sN^{i+1}\_\Phi}), (X/Z\ni o,B_{n\_\sN^{i+1}\_\Phi}{}^\sharp)$
and
a b-$n$-complement of itself,
of $(X/Z\ni o,B_{n\_\sN^{i+1}\_\Phi}),
(X/Z\ni o,B_{n\_\sN^{i+1}\_\Phi}{}^\sharp),
(X^\sharp/Z\ni o,B_{n\_\sN^{i+1}\_\Phi}{}^\sharp{}_{X^\sharp})$,
if $(X,B_{n\_\sN^{i+1}\_\Phi}),
(X,B_{n\_\sN^{i+1}\_\Phi}{}^\sharp)$ are log pairs respectively.

The same applies to bd-pairs of dimension $d$.

\end{defn}

Remark:
(1)
In particular, $(X/Z\ni o,B)$ has (filtration) klt type $d$ if
the pair is of generic type.
The converse holds for the plain klt type $d$.
But this is not true for filtration klt type (cf. (3) below).
If the pair has nongeneric filtration klt type $i$, that is, $i<d$
then by Theorems~\ref{klt_compl} above and~\ref{b_n_comp_klt} below
a required (b-)$n$-complement can be lifted
from a generic one in dimension $i$.

Note also that in the proof of Theorem~\ref{klt_compl}
we apply Construction~\ref{const_b_contraction}
to $(X/Z\ni o,B_{\sN\_\Phi})$, where $\sN=\sN^i$
($f=i$ in the theorem).
However, we can use the same construction
directly for $(X/Z\ni o,B)$ if
the latter pair has the same klt type $i$.
Otherwise we use the construction for the former pair.

(2) By the proof of Theorem~\ref{klt_compl}, for a klt type pair $(X/Z\ni o,B)$,
a complement can be lifted from
a generic complement in dimension $i$ ($f=i$ in the theorem).
So, additionally, we can associate to the pair its
generic type $j\le i$ and
(filtration) semiexceptional type $(r,f)$
(see Definition~\ref{filtration_generic_df}).

Boundary $B_{n\_\sN^{i+1}\_\Phi}{}^\sharp{}_{X^\sharp}$
in Definition~\ref{klt_filtration} can be
slightly increased if we take into consideration the generic type $j$
or better a (filtration) semiexceptional type of $(X/Z\ni o,B)$.

(3) By definition filtration klt type is bigger of equal than klt type.
Filtration klt type gives better understanding of geometry than
plain klt type (cf. Corollary~\ref{reg_generic type}).
The same concerns filtration semiexceptional type vs
semiexceptional type.

(4) Warning: It is possible other $n$-complements which are
not agree with the filtration.
Additionally, we can have $n$-complements
without any type but with $n\in\sN$.

\begin{thm} [Klt type $n$-complements] \label{b_n_comp_klt}
Let $d$ be a nonnegative integer,
$I,\ep,v,e$ be the data as
in Restrictions on complementary indices, and
$\Phi=\Phi(\fR)$ be a hyperstandard set associated with
a finite set of rational numbers $\fR$ in $[0,1]$.
Then there exists a finite set
$\sN=\sN(d,I,\ep,v,e,\Phi)$
of positive integers such that
\begin{description}

\item[\rm Restrictions:\/]
every $n\in\sN$ satisfies
Restrictions on complementary indices with the given data.

\item[\rm Existence of $n$-complement:\/]
if $(X/Z\ni o,B)$ is a pair with
$\dim X=d$, a boundary $B$, wFt $X/Z\ni o$, connected $X_o$ and
with klt type $(X/Z\ni o,B_{\sN\_\Phi})$
then $(X/Z\ni o,B)$ has an $n$-complement $(X/Z\ni o,B^+)$ for some $n\in\sN$.

\end{description}

\end{thm}

\begin{add} \label{standard_klt_complements}
$(X/Z\ni o,B^+)$ is an $n$-complement of itself and
of $(X/Z\ni o,B_{n\_\Phi}), (X/Z\ni o,B_{n\_\Phi}{}^\sharp)$,
and is a b-$n$-complement of itself and
of $(X/Z\ni o,B_{n\_\Phi}),
(X/Z\ni o,B_{n\_\Phi}{}^\sharp),(X^\sharp/Z\ni o,B_{n\_\Phi}{}^\sharp{}_{X^\sharp})$,
if $(X,B_{n\_\Phi}),(X,B_{n\_\Phi}{}^\sharp)$ are log pairs respectively.

\end{add}

\begin{add} \label{klt_complements_fil}
$\sN$ has a klt type filtration~(\ref{filtration_klt})
with $\sN^0=\sN$, with any finite set
of positive integers  $\sN'$, satisfying
Restrictions on complementary indices with the given data,
and
the existence $n$-complements agrees the filtration
for the class of pairs under assumptions of
Existence of $n$-complements in the theorem.

\end{add}

\begin{add} \label{klt_complements_B}
In particular, the theorem and addenda applies
to klt type pairs $(X/Z\ni o,B)$ instead of
with klt type $(X/Z\ni o,B_{\sN\_\Phi})$.
\end{add}

\begin{add} \label{bd_klt_compl}
The same holds for bd-pairs $(X/Z\ni o,B+\sP)$ of index $m$
with $\sN=\sN(d,I,\ep,v,e,\Phi,m)$.
That is,
\begin{description}

\item[\rm Restrictions:\/]
every $n\in\sN$ satisfies
Restrictions on complementary indices with the given data and
$m|n$.

\item[\rm Existence of $n$-complement:\/]
if $(X/Z\ni o,B+\sP)$ is a bd-pair of index $m$ with
$\dim X=d$, a boundary $B$, wFt $X/Z\ni o$, connected $X_o$ and
with klt type $(X/Z\ni o,B_{\sN\_\Phi}+\sP)$
then $(X/Z\ni o,B+\sP)$ has an $n$-complement $(X/Z\ni o,B^++\sP)$ for some $n\in\sN$.

\end{description}
Addendum~\ref{klt_complements_fil}
holds literally.
In Addendum~\ref{standard_klt_complements}
$(X/Z\ni o,B^++\sP)$ is an $n$-complement of itself and
of $(X/Z\ni o,B_{n\_\Phi}+\sP), (X/Z\ni o,B_{n\_\Phi}{}^\sharp+\sP)$,
and is a b-$n$-complement of itself and
of $(X/Z\ni o,B_{n\_\Phi}+\sP),
(X/Z\ni o,B_{n\_\Phi}{}^\sharp+\sP),(X^\sharp/Z\ni o,B_{n\_\Phi}{}^\sharp{}_{X^\sharp}+\sP)$,
if $(X,B_{n\_\Phi}+\sP_X),(X,B_{n\_\Phi}{}^\sharp+\sP_X)$ are log bd-pairs respectively.
In Addendum~\ref{klt_complements_B}
$(X/Z\ni o,B+\sP),(X/Z\ni o,B_{\sN\_\Phi}+\sP)$
should be instead of
$(X/Z\ni o,B),(X/Z\ni o,B_{\sN\_\Phi})$ respectively.

\end{add}

\begin{proof}
Construction of klt type complements
starts from the top type $d$ and
descends to the bottom $0$,
where we use Corollary~\ref{bounded_lc_index}, the boundedness of lc index.
We use [decreasing] induction on klt type $i$ in
dimension $d$.
In other words, we prove the theorem and Addendum~\ref{klt_complements_fil}
simultaneously.
In the induction we suppose that
$\sN=\sN^i$ and, in Existence of $n$-complements,
$(X/Z\ni o,B_{\sN\_\Phi})$ has klt type $\ge i$
instead of just klt type, that is, $\ge 0$.
We take the same class of pairs in Addendum~\ref{klt_complements_fil}.

By Proposition~\ref{D_D'_complement} and Corollary~\ref{monotonicity_II},
more subtle complements of Addendum~\ref{klt_complements_fil}
with $B_{n\_\sN^{i+1}\_\Phi}$
give the same complements of Addendum~\ref{standard_klt_complements}
with $B_{n\_\Phi}$ instead
because $B_{n\_\Phi}\le B_{n\_\sN^{i+1}\_\Phi}$.

Step~1. {\em Klt type filtration for $i=d$.\/}
Consider a set of positive integers
$\sN_d=
\sN(d,I,\ep,v,e,\Phi)$ of Theorem~\ref{generic_complements}.
Then we are done if $\sN'=\emptyset$: $\sN=\sN^d=\sN_d$.
In general we put $\sN=\sN^d=\sN_d\cup \sN'$
assuming that $\sN_d\cap\sN'=\emptyset$.
The last condition holds for $\sN_d=\sN\setminus\sN'$
of Addendum~\ref{generic_complements_fil}
(cf. Corollary~\ref{N_N'}).
Note that in this case generic or filtration klt type $d$ pairs
$(X/Z\ni o,B_{\sN\_\Phi})$
are only possible under inductive version of
Existence of $n$-complements.
Indeed, $B_{\Phi}\le B_{\sN\_\Phi}$ and
$(X/Z\ni o,B_\Phi)$ is also generic and of klt type $d$
by Basic properties of generic pairs, (1).

Step~2. {\em Klt type filtration for $i<d$.\/}
We can suppose that the theorem is established for
klt types $\ge i+1$ and $\sN^{i+1},i+1\le d$,
the next filtration klt type set
is already constructed with
a filtration~(\ref{filtration_klt})
for the class of pairs in Existence of $n$-complements
with klt types $\ge i+1$ $(X/Z\ni o,B_{\sN^{i+1}\_\Phi})$.
Consider a set of positive integers
$\sN_i=\sN(i,I,\ep,v,e,\Phi',J)$ as
in Construction~\ref{klt_induction}
with a hyperstandard set $\Phi=\Gamma(\sN^{i+1},\fR)$
and $f=i$.
(Note that $\Phi,\Phi',J$ in the step depends also on $i$.)
By Addendum~\ref{klt_compl_N_N'}
we can suppose that $\sN_i$ is
disjoint from $\sN^{i+1}$.
Put  $\sN=\sN^i=\sN_i\cup \sN^{i+1}$.
Now by Theorem~\ref{klt_compl} this proves Theorem~\ref{b_n_comp_klt}
and gives
a required filtration~(\ref{filtration_klt})
for klt types $\ge i$.
Note that the complements agree with the filtration
still in the class of pairs
$(X/Z\ni o,B_{\sN\_\Phi})$ of klt type $\ge i$.
If $(X/Z\ni o,B_{\sN^{i+1}\_\Phi})$ has klt type $i$ then
$(X/Z\ni o,B_{\sN\_\Phi})$ has also klt type $i$ and
we use Theorem~\ref{klt_compl} to construct
a required $n$-complement with $n\in\sN_i$.
Otherwise,  $(X/Z\ni o,B_{\sN^{i+1}\_\Phi})$ has klt type $\ge i+1$
and we use induction.

This concludes construction of complements
for the klt type case when $i=0$.

Step~3. {\em Other Addenda.\/}
Addendum~\ref{klt_complements_B} follows from the fact
that if $(X/Z\ni o,B)$ has klt type then so does
$(X/Z\ni o,B_{\sN\_\Phi})$ because
$B_{\sN\_\Phi}\le B$.

Similarly we can treat bd-pairs.

\end{proof}

\section{Lc type complements} \label{lc_type_compl}

In this section we establish Theorem~\ref{bndc}
under assumption (2-3) of the theorem.
For this we consider complements for
the pairs $(X/Z\ni o,B)$ without a klt $\R$-complement, equivalently,
for the pairs $(X/Z\ni o,B)$ with only lc but not klt $\R$-complements.
However, we suppose that an (lc) $\R$-complement exists.
An {\em lc type\/} pair is a local morphism
$(X/Z\ni o,B)$ of this kind.
Note that, if this is the case, $(X/Z\ni o,B)$ is lc
(possibly klt) itself when
it is a log pair.

We start with two examples which show that
$n$-complements of lc type are not bounded in
any dimension $d\ge 3$.
We construct such examples in two extremal situations
when $X/Z\ni o$ is identical (local) or with $Z=o=\pt$ (global).

\begin{exa} \label{unbounded_lc_compl}
(1)
Let $(Z\ni o, B_Z)$ be a pair such that
\begin{description}

  \item[]
$Z$ is a normal variety of the dimension $d\ge 3$ and
$o$ is its closed point;

  \item[]
$B_Z=\sum_{i=1}^{l+d} D_i$ be a reduced divisor near $o$
with $l+d$ prime Weil divisorial components; integers $d,l\ge 0$;
and

  \item[]
$(Z,B_Z)$ is {\em maximally\/} lc (of index $1$), that is, lc (of index $1$) and
with an lc center $o$.

\end{description}
We suppose also that there exists a $\Q$-factorialization
$\varphi\colon X\to Z\ni o$ near $o$ such that each divisor $D_i$,
more precisely, its birational transform,
with $i\ge d+1$ intersects $D_1$, respectively, its birational transform,  on $X$, and
the intersection $D_1\cap D_i$ is exceptional on $D_1$.
Such a factorialization is typical for toric morphisms
associated with a triangulation of a polyhedral cone with
the same $l+d$ edges and one common edge for all simplicial subcones
of dimension $d$.

Now we perturb $B_X=B=\sum_{i=1}^{l+d} D_i$ on $X$ as follows.
Every divisor $D_i$ is mobile and
moreover its linear system is big.
Replace each $D_i$ with $i\ge d+1$ by
$$
\sum_{j=1}^{i-d}\frac1{i-d} D_{i,j},
$$
where $D_{i,j}$ are $i-d$ sufficiently general divisors in
$\linsys{D_i}$.
They are also prime Weil.
So, after the perturbation
$$
B=\sum_{i=1}^d D_i+\sum_{i=1}^l\sum_{j=1}^i \frac1i D_{i+d,j}.
$$
By construction $(X/Z\ni o,B)$ is an lc $0$-pair.
It is an $\R$-complement of itself.
Also by construction the intersection $D_1\cap D_{i+d},l\ge i\ge 1$, is
exceptional on $D_1$ on $X$.
Thus every $D_{i+d,j}$ passes through the intersection and
$(X,B)$ is maximally lc near the intersection.

We contend that $(X/Z\ni o,B)$ does not have
$n$-complements for all $1\le n\le l-1$.
Suppose such an $n$-complement $(X/Z\ni o,B^+)$ exists.
Then
$$
\rddown{(n+1)\frac 1{n+1}}/n=1/n>\frac 1{n+1}.
$$
Hence, by Definition~\ref{n_comp} (1), $B^+>B$ near the intersection $D_1\cap D_{n+d+1}$,
a contradiction by Definition~\ref{n_comp} (2).
Indeed, it was noticed above that $(X,B)$ is maximally lc near the intersection.

Finally note that $l=0$ in the dimension $d\le 2$.
However, for $d\ge 3$, $l$ can be any natural number.

We can construct a required $\varphi$  as a toric morphism.
We can push the example to $Z\ni o$.

(2)
We construct a locally trivial $\PP^1$-bundle $\varphi\colon X\to Y$
with projective $X$, $\dim X=d$ and
with a reduced divisor $D=\sum_{i=1}^l D_i$ on $X$ with only simple normal crossings and
such that
\begin{description}

\item[\rm (1)\/]
divisors $D_1$ and $D_2$ are disjoint sections of $\varphi$;

\item[\rm (2)\/]
$(X,D)$ is a projective $0$-pair;

\item[\rm (3)\/]
all other divisors $D_i,i\ge 3$, are vertical with respect to $\varphi$;

\item[\rm (4)\/]
every irreducible closed curve $C_j,j\in J$, of $X$ with the generic point in
the $1$-dimensional strata of $D$ is rational and has a closed point in
the $0$-dimensional strata;

\item[\rm (5)\/]
those curves generate the cone of effective curves of $X$, that is,
every $1$-dimensional effective algebraic cycle of $X$
is numerically equivalent to $\sum r_j C_j,r_i\in \Q$, and $r_i\ge 0$;

\item[\rm (6)\/]
$D_1$ is base point free and big on $X$, ample on itself;
and

\item[\rm (7)\/]
the restriction of $\linsys{D_1}$ on $D^1=\sum_{i=2}^l D_i$
is surjective.

\end{description}

We start our construction from a product $X=Y\times \PP^1$
of projective toric nonsingular varieties
with the natural projection $\varphi\colon X\to Y$, $\dim X=d$ and
with the invariant divisor $D$ on $X$.
We can suppose that $D_1,D_2$ are horizontal.
The product satisfies the properties (1-5).

To satisfy (6-7) we transform $\varphi$ as follows.
Take a general hyperplane section $H$ on $Y$
for some projective embedding of $Y$.
So, $\varphi\1 H+D$ has only simple normal crossings too.
Let $\psi\colon X'\to X$ be the blowup of $D_2\cap \varphi\1 H$ in $X$.
Denote its exceptional divisor by $E$ and by
$E'$ the proper inverse transform of the vertical divisor
$\varphi\1 H$.
Now we can blow down $E'$ and construct a required
locally trivial $\PP^1$-bundle $X'/Y$ with the birational transform
$D'$ of $D$.
It also satisfies the properties (1-5).
To prove (5) it is better to know that $X'\to Y$ is toric again
(e.g. by the criterion in \cite{BMSZ};
this can be done also directly by a toric construction).

For sufficiently ample $H$, $D_1'$ is very ample on $D_1'$.
We can suppose that, for every  closed curve $C_j$ generically in
the $1$-dimensional strata of $D$ in $D_1$,
$$
(C_j,D_1')\ge 2.
$$
By construction and (5), $D_1'$ is nef and big on $X$.
Since $X$ is projective toric it is Ft and $D_1'$ is semiample.
This concludes (6).

Denote the constructed pair $(X'/Y,D')$ by $(X/Y,D)$.
The restriction of linear systems
$$
\linsys{D_1}\dashrightarrow \linsys{D_1\rest{D^1}}
$$
is surjective due to the vanishing
$$
H^1(X,D_1-D^1)=H^1(X,D_1-(D-D_1))=
H^1(X,2D_1-D)=H^1(X,K+2D_1)=0
$$
by the Grauert-Riemenschneider vanishing.
This gives (7).

Actually, we can prove now that $D_1$ is base point free.
Moreover, let $P_j\in C_j$ be arbitrary closed points
on the vertical curves $C_j$ of the $1$-dimensional strata of $D$.
Then the points $P_j$ are the only base points
of the linear system $\linsys{D_1-\sum P_j}$.
In particular, the linear system is nonempty.
To construct a divisor $M$ in the linear system
we can use the dimensional induction.
If the dimension of $Y=1$, then $Y=\PP^1$ and
we have only two vertical (disjoint) curves $C_1=D_3,C_2=D_4$,
$D^1=D_2+D_3+D_4$ and
$P_1+P_2\in \linsys{D_1\rest{D^1}}$.
By the surjectivity in (7)
there exists such an effective divisor $M$ on $X$ that
$M\rest{D^1}=P_1+P_2$ and $M\in \linsys{D_1-P_1-P_2}$.
Moreover, $M\cap D_2=\emptyset$ for general $M$.
Suppose by induction that on every vertical divisor $D_i,i\ge 3$,
generically in
the $(d-1)$-dimensional (divisorial) strata  of $D$ there exists
an effective divisor
$$
M_i\in\linsys{D_1\rest{D_i}-\sum_{P_j\in D_i} P_j}.
$$
We can also suppose by induction that
the divisors $M_i$ are agree on the vertical
varieties $V$ generically in the smaller dimensional
strata, that is, if $D_l, l\ge 3$, is
another vertical divisor generically in the $(d-1)$-dimensional strata of $D$
with the effective divisor $M_l$ and $V\subseteq D_i,D_l$
then
$$
M_i\rest{V}=M_l\rest{V}.
$$
So we can glue the divisors $M_i$ into a single Cartier divisor $M_{D^1}$
on $D^1$ such that
$$
M_{D^1}\rest{D_i}=M_i.
$$
Note also for this that the general divisors $M_i$ as $D_1$ do not
interest $D_2$ by construction: $D_1\cap D_2=\emptyset$.
By construction $M_{D^1}\in \linsys{D_1\rest{D^1}-\sum P_j}$.
By (7) there exists $M\in\linsys{D_1-\sum P_j}$.
For general $M$, $M\cap D_2=\emptyset$.

Actually, general $M$ behaves as $D_1$, that is, if we replace $D_1$ by
$M$ we have the same properties (1-7).
So, to verify that $P_j$ are the only base points of $\linsys{D_1-\sum P_j}$,
it is enough to consider the case when $P_j$ are points of
the $0$-dimensional strata on $D_1$.
If $\dim Y=1$, then there exists $M$ such that
$P_1,P_2$ are the only base points of $\linsys{D_1-P_1-P_2}$
because $D_1$ is a $1$-dimensional strata and $D_1^2\ge 2$.
So, the base locus does not contain the $1$-dimensional strata.
In higher dimensions the base locus should be
a closed torus invariant subset in $D_1$.
Since it contains the points $P_j$ but does not contain
the $1$-dimensional strata,
the base locus have only the points $P_j$.

According to the above construction for any choice of
$M_V\in \linsys{D_1\rest{V}-P_{1,V}-P_{2,V}}$
there exists $M\in \linsys{D_1-\sum P_j}$ with $M\rest{V}=M_V$,
where $V$ are the vertical irreducible surfaces generically in
the $2$-dimensional strata and
$P_{1,V},P_{2,V}$ are the only points on $V$ from the points $P_j$ on $D_1$.
Suppose that $V_1,\dots,V_m$ are the vertical irreducible surfaces
generically in the $2$-dimensional strata.
Then we can construct $m!$ distinct divisors $M_i, i=1,\dots,m!$,
as follows.
Take $j$ divisors (curves) $C_{j,h},h=1,\dots,j$, on $V_j$ in
$\linsys{D_1\rest{V_j}-P_{1,V_j}-P_{2,V_j}}$.
Then take $M_i$ with
$$
M_i\rest{V_j}=C_{j,h}
$$
for some $h=1,\dots, j$.
But each curve $C_{j,h}$ on $V_j$ is repeated $m!/j$ times
in this construction.
So, we use $m!$ curves $C_{j,h}$ with the multiplicity $m!/j$ on $V_j$.

Now we perturb $D_1$:
$$
B=\sum_{i=1}^{m!}\frac1{m!}M_i+\sum_{i=2}^l D_i.
$$
Finally, we blow up every curve $C_{j,h}$:
$$
(X',B')\to (X,B)
$$
with the crepant boundary $B'$ and with exceptional divisors $E_{j,h}$.
(The order of blowups is unimportant.)
Note that $(X',B')$ is again $0$-pair, in particular, lc.
Since, every divisor $M_i$ intersects every surface $V_j$
transversally in $C_{j,h}$ and $m!/j$ such divisors passing
through $C_{j,h}$,
$$
\mult_{E_{j,h}}B'=\frac{m!}j\frac1{m!}=\frac1j
$$
and every $E_{j,h}$ intersects $V_j$ transversally along
the generic point of $C_{j,h}$, the proper birational transform of
$C_{j,h}$ on $V_j$, the proper birational transform of $V_j$ on $X'$.

The pair $(X',B')$ is an $\R$-complement of itself.
But it has only $n$-complements for $n\ge m$.
Indeed, if $n\le m-1$ and $(X',B^+)$ is an $n$-complement
of $(X',B')$ then near $V_{n+1}$ the pair $K_{X'}+B^+$
is not $\equiv 0$, a contradiction.
For this note that
$$
\rddown{(n+1)\frac 1{n+1}}/n=1/n>1/(n+1).
$$
On the other hand, the reduced part of $D^1=\sum_{i=2}^l D_i$
already gives the lc singularity along $V_{n+1}$.
So,
$$
B^+\ge \sum_{i=2}^l D_i+\sum_{i=1}^{n+1} \frac 1n E_{n+1,i}
$$
and $B^+$ does not contain other divisors than $D_i, i=2,\dots,l$,
passing through $V_{n+1}$ (maximally lc).
Take $C$, the birational transform of a generic fiber
(vertical curve) of $V_{n+1}$ over $Y$.
Then by construction and adjunction,
\begin{align*}
0=(C.K_{X'}+B^+)\ge
(C.K_{X'}+\sum_{i=2}^l D_i+\sum_{i=1}^{n+1} \frac 1n E_{n+1,i})=\\
(C.K_{V_{n+1}}+D_2\rest{V_{n+1}}+\sum_{i=1}^{n+1} \frac 1n E_{n+1,i}\rest{V_{n+1}})\ge \\
(C.K_{V_{n+1}}+D_2\rest{V_{n+1}}+\sum_{i=1}^{n+1} \frac 1n C_{n+1,i})\ge
-2+1+(n+1)\frac 1n=\frac 1n>0.
\end{align*}
Note that $(C.D_2\rest{V_{n+1}})=(C.D_2)=1$.

Finally, if $\dim X=d\ge 3$ or, equivalently, $\dim Y\ge 2$,
we can find such a pair $(X,B)$ with $m\gg 0$.
Thus in dimensions $d\ge 3$ the global $n$-complements
of lc type are not bounded.

(3) Let $d$ be a positive integer $\ge 3$ and
$\sN$ be a finite set of positive integers.
Then, for every integer $r\gg 0$, there exists
a pair $(X/Z\ni o,B)$ with a $\Q$-boundary
$B=\sum_{i=1}^r b_i D_i$, $\dim X=d$ and connected $X_o$ such that
$(X/Z\ni o,B)$ has an $\R$-complement but
does not have an $n$-complement for every $n\in\sN$,
where $D_1,\dots,D_r$ are effective Weil divisors
(cf. (2-3) of Theorem~\ref{bndc}).
Replace in Example~(2) $m!$ by $\sN!=\prod_{n\in\sN}(n+1)$
and $1/j$ by $1/(n+1),n\in\sN$.
The divisor $\sum_{i=1}^{\sN!}M_i$ can be replaced by
a single prime one.
This gives a global example.
A local example can redone from Example~(1).

Actually, for appropriate examples,
the low bound for $r$
depend only on $d$ and the number of elements in $\sN$.
For this we should aggregate curves $C_{j,h}$
into at most $4$ curves $C_{h,j},h=1,2,3,4$, with $4$ multiplicities
$a_h\sN!/j$ such that $a_h$ is a positive integer $<j$
with $\sum_{h=1}^4 a_h/j=2$.
However, the multiplicities $b_i$ in those examples with
all possible bounded $\sN$ can't belong to a dcc set,
in particular, to a finite set by Theorem~\ref{bndc}
under the assumption~(3) \cite[Theorem~1.6]{HLSh}.
Thus the existence of $\R$-complements for every $D$ in
(3) of the theorem is essential.

\end{exa}

This section provides a construction of bounded $n$-complements
under the additional assumption (finiteness).

\paragraph{Maximal lc $0$-pairs.}
Let $(X/Z\ni o,D)$ be a $0$-pair, that is,
$(X,D)$ is lc and $K+D\sim_\R 0/Z\ni o$.
We say that $(X/Z\ni o,D)$ is a {\em maximal lc\/}
$0$-pair if additionally $(X/Z\ni o,D)$ is
the only possible $\R$-complement of $(X/Z\ni o,D)$.
By definition $(X/Z\ni o,D)$ is an $\R$-complement of $(X/Z\ni o,D)$.
So, the maximal lc property means that
if $(X/Z\ni o,D^+)$ is
another $\R$-complement of $(X/Z\ni o,D)$ then
$D^+=D$.
By (1) of Definition~\ref{r_comp}
every global $0$-pair is maximally lc
(even it is klt).
However, every local maximal lc $0$-pair $(X/Z\ni o,D)$
should have an lc center over $o$ as a point
in every connected component of $X_o$
(cf. Example~\ref{1stexe}, (3)).
Notice that there are no such nonglobal
and nonlocal $0$-pairs $(X/Z,D)$.

The same applies to $0$-bd-pairs $(X/Z\ni o,D+\sP)$.

Let $d$ be a nonnegative integer and
$\Gamma$ be a set of boundary multiplicities:
$\Gamma\subseteq [0,1]$, including $0$.
The set $\Gamma$ has {\em bounded rational maximal lc
multiplicities in dimension\/} $d$ if the following
set of rational numbers in $\Gamma$ is finite:
\begin{align*}
\Gamma_{\max}=\{b\in \Gamma\cap\Q\mid
(X/Z\ni o,B) \text{ is a maximal lc }
0\text{-pair},\\ \dim X=d, B\in \Gamma\cap\Q,
\text{ and } b \text{ is a multiplicity of } B\}.
\end{align*}
E.g., $0\in\Gamma_{\max}$.
Note that this assumption allows any irrational
numbers of $[0,1]$ in $\Gamma$.

The same applies to $0$-bd-pairs $(X/Z\ni o,B+\sP)$
of dimension $d$ with $B\in\Gamma\cap\Q$ and
a multiplicity $b$ of $B$.

Moreover, the finiteness and Boundedness of lc index conjecture
(Corollary-Conjecture~\ref{conj_bounded_lc_index}) imply that
there exists a positive integer $I=I(d,\Gamma\cap\Q)$ depending
only on $d$ and $\Gamma$ such that if
$(X/Z\ni o,B)$ is a maximal lc $0$-pair in dimension $d$
with $B\in\Gamma\cap\Q$ then
$I(K+B)\sim 0/Z\ni o$.
For wFt $X/Z\ni o$ this holds without Boundedness of lc index conjecture
by Corollary~\ref{bounded_lc_index} (or \cite[Theorem~1.7]{B}).
We say that $I$ is the {\em  rational maximal lc index of $0$-pairs
in dimension $d$ with respect to\/} $\Gamma$.

The same applies to maximal lc $0$-bd-pairs $(X/Z\ni o,B+\sP)$ of index $m$
but only for wFt $X/Z\ni o$ (cf. Example~\ref{bd_elliptic},
Conjectures~\ref{conj_bounded_lc_index}, \ref{mod_part_b-semiample} and
Corollary~\ref{Alexeev_index_m}).

\begin{exa} \label{max_lc_index}
Every dcc subset $\Gamma\subset [0,1]$ with $0$
has bounded rational maximal lc multiplicities
in dimension $d$ under wFt.
In general this is expected by Boundedness of lc index conjecture.
In particular, every hyperstandard set associated
to a finite set of rational numbers satisfies the property under wFt.
Indeed, we can suppose that the dcc set $\Gamma$ is rational.
Then $\Gamma_{\max}$ is finite by \cite{HX} or Corollary~\ref{bounded_lc_index}.
Under wFt means that $X/Z\ni o$ in the definition
of $\Gamma_{\max}$ is a wFt morphisms.

The same holds for bd-pairs of index $m$ under wFt.

\end{exa}

\begin{const} \label{lc_induction}
Let $d$ be a nonnegative integer,
$I,\ep,v,e$ be the data as
in Restrictions on complementary indices,
$\Phi=\Phi(\fR)$ be a hyperstandard set associated with
a finite set of rational numbers $\fR$ in $[0,1]$ and
$\Gamma$ be a subset of $[0,1]$ with
finite $\Gamma_{\max}$.
Denote by $I(d,\Gamma\cap\Q)$ the corresponding
rational maximal lc index.
In particular, by Example~\ref{max_lc_index}
the finiteness holds and the index exists under wFt
if $\Gamma\cap\Q$ is a dcc set as in Theorem~\ref{bndc}.
We can suppose also that $I=I(d,\Gamma\cap\Q)$, or
equivalently, that $I$ is sufficiently divisible.

By Theorem~\ref{b_n_comp_klt} and Addendum~\ref{standard_klt_complements}
there exists a finite set of positive integers
$\sN=\sN(d,I,\ep,v,e,\Phi)$ such that
\begin{description}

\item[\rm Restrictions:\/]
every $n\in\sN$ satisfies
Restrictions on complementary indices with the given data;

\item[\rm Existence of $n$-complement:\/]
if $(X/Z\ni o,B)$ is a pair with
$\dim X=d$, a boundary $B$, wFt $X/Z\ni o$, connected $X_o$ and
with klt type $(X/Z\ni o,B_{\sN\_\Phi})$
then $(X^\sharp/Z\ni o,B_{n\_\Phi}{}^\sharp{}_{X^\sharp})$
is a b-$n$-complement $(X/Z\ni o,B^+)$ for some $n\in\sN$.

\end{description}

For bd-pairs we add a positive integer $m$.
So, $\sN=\sN(d,I,\ep,v,e,\Phi,m)$ and
replace $(X/Z\ni o,B)$ by a bd-pair $(X/Z\ni o,B+\sP)$ of
dimension $d$ and of index $m|n$.

\end{const}

\begin{thm}[Lc type $n$-complements] \label{lc_compl}
Let $d$,$I,\ep,v,e$,$\Gamma,\sN=\sN(d,I,\ep,v,e)$ be
the data as in Construction~\ref{lc_induction}.
Then every $n\in\sN$ satisfies Restrictions
on complimentary indices with the given data and
\begin{description}

\item[\rm Existence of $n$-complement:\/]
if $(X/Z\ni o,B)$ is a pair with
$\dim X=d$, $B\in\Gamma$, wFt $X/Z\ni o$, connected $X_o$ and
with lc type $(X/Z\ni o,B)$,
or more generally, with an $\R$-complement,
then $(X/Z\ni o,B)$ has an $n$-complement $(X/Z\ni o,B^+)$ for some $n\in\sN$.

\end{description}

\end{thm}

\begin{add}
The same holds for bd-pairs $(X/Z\ni o,B+\sP)$ of index $m$
with $\sN=\sN(d,I,\ep,v,e,m)$.
That is,
\begin{description}

\item[\rm Restrictions:\/]
every $n\in\sN$ satisfies
Restrictions on complementary indices with the given data and
$m|n$.

\item[\rm Existence of $n$-complement:\/]
if $(X/Z\ni o,B+\sP)$ is a bd-pair of index $m$ with
$\dim X=d$, $B\in\Gamma$, wFt $X/Z\ni o$, connected $X_o$ and
with lc type $(X/Z\ni o,B+\sP)$
then $(X/Z\ni o,B+\sP)$ has an $n$-complement $(X/Z\ni o,B^++\sP)$ for some $n\in\sN$.

\end{description}

\end{add}

We do not use $\Phi$ in the statement of the theorem
(cf. Remark~\ref{lc_compl_alternative} below), that is,
we can take any $\Phi$, e.g., $\Phi=\{1\}$.
However, $\Phi$ is hidden in the proof of Theorem~\ref{lc_compl}.
The proof below uses
a reduction to the klt type of Theorem~\ref{b_n_comp_klt}
and Addendum~\ref{klt_complements_B}.
An alternative and more right proof is
sketched in Remark~\ref{lc_compl_alternative}.

\begin{proof}
By construction $\sN$ is a finite set of positive integers
and satisfies Restrictions.

Let $(X/Z\ni o,B)$ be a pair satisfying
the assumptions in Existence of $n$-complements.
Since $B$ is a boundary, $B_{\sN\_\Phi}$ is well-defined.
By Propositions~\ref{D_D'_complement}
and~\ref{monotonicity_I}, $(X/Z\ni o,B_{\sN\_\Phi})$
has an $\R$-complement because
so does $(X/Z\ni o,B)$.
If the pair $(X/Z\ni o,B_{\sN\_\Phi})$ has klt type
then the result holds by Theorem~\ref{b_n_comp_klt}.

Otherwise $(X/Z\ni o,B_{\sN\_\Phi})$ has lc type.
In this case we use approximation and Theorem~\ref{b_n_comp_klt} again.

Step~1. {\em Reduction to the wlF pair
$(X^\sharp/Z\ni o, B_{\sN\_\Phi}{}^\sharp{}_{X^\sharp})$
with a $0$-contraction\/}
$$
\psi\colon (X^\sharp, B_{\sN\_\Phi}{}^\sharp{}_{X^\sharp})\to Y/Z\ni o
$$
{\em such that\/}
\begin{description}

  \item[\rm (1$^\sharp$)]
$\psi^*H\sim_\R -(K_{X^\sharp}+B_{\sN\_\Phi}{}^\sharp{}_{X^\sharp})$
for some
ample over $Z\ni o$ $\R$-divisor $H$ on $Y$.

\end{description}
Use Construction~\ref{sharp_construction}.
We can suppose
that the modification $\varphi\colon X\dashrightarrow X^\sharp$
is small.
So, if $n\in \sN$ gives an $n$-complement
$(X^\sharp/Z\ni o,B^+)$ of
$(X^\sharp/Z\ni o, B_{\sN\_\Phi}{}^\sharp{}_{X^\sharp})$, then
the complement induces an $n$-complement $(X/Z\ni,B^+)$ of
$(X/Z\ni o,B_{\sN\_\Phi}{})$
by Propositions~\ref{small_transform_compl}, \ref{D_D'_complement}
and~\ref{monotonicity_I}.
On its turn,  $(X/Z\ni,B^+)$ is
a required $n$-complement of $(X/Z\ni o,B)$ too by
Corollary~\ref{B_B_+B_sN_Phi}.

By construction $\dim X^\sharp=d$,
$X^\sharp/Z\ni o$ has  wFt and connected $X_o^\sharp$.
Additionally, $(X^\sharp/Z\ni o, B_{\sN\_\Phi}{}^\sharp{}_{X^\sharp})$
has an $\R$-complement.
That is, $(X^\sharp/Z\ni o, B_{\sN\_\Phi}{}^\sharp{}_{X^\sharp})$
satisfies the assumptions of Existence of $n$-complements
except for $B_{\sN\_\Phi}{}^\sharp{}_{X^\sharp}\in\Gamma$.
To compensate this we verify the following property.

Step~2.
$(X^\sharp, B_{\sN\_\Phi}{}^\sharp{}_{X^\sharp}\to Y)$
{\em has index $I$ over $Y\ni P$ if the $0$-contraction
is maximal lc over $Y\ni P$\/},
in particular, over the generic point of $Y$.
Equivalently, by~(6)~\ref{adjunction_div}
$P$ is an lc center of
the adjoint bd-pair
$(Y,B_{\sN\_\Phi}{}^\sharp{}\dv+\sB_{\sN\_\Phi}{}^\sharp{}\md)$
(see~\ref{adjunction_0_contr}).
So, there exits only finitely many of those points $P\in Y/Z\ni o$.
We call them {\em maximal lc\/}.
Let $P\in Y$ be such a maximal lc point.
By Construction~\ref{sharp_construction} the $0$-pair
$(X^\sharp/Y, B_{\sN\_\Phi}{}^\sharp{}_{X^\sharp})$ is
crepant over $Z\ni o$ to a $0$-pair
$(X'/Y, B_{\sN\_\Phi,X'})$
with a boundary $B_{\sN\_\Phi,X'}$
being the birational transform of $B_{\sN\_\Phi}$ on $X'$
by a birational $1$-contraction $X\dashrightarrow X'/Z\ni o$.
The maximal lc property is an invariant
of crepant models, e.g., since
the adjoint bd-pair does so.
Thus $(X'/Y\ni P,B_{\sN\_\Phi,X'})$ is
maximal lc.
On the other hand, $(X/Z\ni o,B)$ has an $\R$-complement
$(X/Z\ni o,B^{+,\R})$.
Hence $(X'/Z\ni o,B^{+,\R}_{X'})$ is also
an $\R$-complement of $(X'/Z\ni o,B_{X'})$, where
$B^{+,\R}_{X'}=\B^{+,\R}_{X'}$ and
$B_{X'}$ is the birational transform of $B$ on $X'$.
By construction and definition $B\ge B_{\sN\_\Phi}$ and
$B^{+,\R}_{X'}\ge B_{X'}\ge B_{\sN\_\Phi,X'}$;
$(X'/Z\ni o,B^{+,\R}_{X'})$ is also
an $\R$-complement of $(X'/Z\ni o,B_{\sN\_\Phi,X'})$.
Since $\sim_\R$ over $Z\ni o$ gives
$\sim_\R$ over $Y\ni P$ for every point $P\in Y$ over $o$,
$(X'/Y\ni P,B^{+,\R}_{X'})$ is
an $\R$-complement of $(X'/Y\ni P,B_{\sN\_\Phi,X'})$ too.
Notice also that by construction $(X'/Y\ni P,B_{\sN\_\Phi,X'})$
is an $\R$-complement of itself.
By the maximal lc property over $Y\ni P$,
$B^{+,\R}_{X'}=B_{X'}=B_{\sN\_\Phi,X'}$ over $Y\ni o$.
Thus $B_{X'}=B_{\sN\_\Phi,X'}$ is rational
and $\in\Gamma$ over $Y\ni o$.
Therefore by our assumptions
$(X'/Y\ni P,B_{\sN\_\Phi,X'})=(X'/Y\ni P,B_{,X'})$
has index $I$.
Now the require index property follows from
the invariance of index under crepant modifications.

For simplicity of notation,
we denote $(X^\sharp/Z\ni o, B_{\sN\_\Phi}{}^\sharp{}_{X^\sharp})$
by $(X/Z\ni o,B)$.
By Step~1 $(X/Z\ni o,B)$ is a wlF pair and has a $0$-contraction
$$
\psi\colon (X,B)\to Y/Z\ni o
$$
such that
\begin{description}

  \item[\rm (1)]
$\psi^*H\sim_\R -(K+B)$
for some
ample over $Z\ni o$ $\R$-divisor $H$ on $Y$.

\end{description}
Additionally, $(X/Z\ni o, B)$
satisfies the assumptions of Existence of $n$-complements
except for $B\in\Gamma$.
By Step~2 the $0$-contraction $\psi$ has
index $I$ locally over the lc centers of $(Y,B\dv+\sB\md)$,
including the generic point of $Y$.
By~(6)~\ref{adjunction_div}
the lc centers of $(X,B)$, including the generic point of $X$,
correspond to that of
$(Y,B\dv+\sB\md)$ under $\psi$.
Hence there exists an open non empty subset $U$ in $Y/Z\ni o$ such that
all lc centers of $(X,B)$ are generically over $U$ and
$(X,B)$ is a $0$-pair over $U$ of index $I$.
In particular,
\begin{description}

\item[\rm (2)]
$(X,B)$ is klt over $Y\setminus U/Z\ni o$; and

\item[\rm (3)]
all multiplicities $b$ of $B$  over $U$ belong to
$$
[0,1]\cap \frac \Z I.
$$

\end{description}
In Step~4 below we construct an $n$-complement of
$(X/Z\ni o,B)$ for some $n\in\sN$.

Step~3. {\em Approximation.\/}
There exists a perturbation $B'$ of the boundary $B$ such that
\begin{description}

\item[\rm (4)]
$(X/Z\ni o, B')$ is a klt wlF;

\item[\rm (5)]
$B'> B$ over $Y\setminus U/Z\ni o$
($>$ in every prime divisor of $X$ over $Y\setminus U/Z\ni o$);
and

\item[\rm (6)]
the multiplicities $b'$ of $B'$ over $U$ are
arbitrary closed to fractions $b=l/I$ with
nonnegative integer $l\le I$.
\end{description}
Indeed, by~(1) we can find and effective $\R$-divisor
$E\sim_\R H/Z\ni o$ on $Y$ over $Z\ni o$ such that
$Y\setminus U\subseteq \Supp E$ locally over $Z\ni o$.
We can suppose also that $E$ does not pass through
the lc centers of $(Y,B\dv+\sB\md)$.

First, we perturb over $E$: take $B_1=B+\ep\psi^*E$
for sufficiently small $\ep>0$.
Then we get (5).
In (6) $b_1=b$ if $b=1$, and in (4) lc instead of klt
and even klt over $Y\setminus U/Z\ni o$ by
(1) and (2).

To make a second approximation we take a boundary
$B_2$ on $X$ such that $(X/Z\ni o, B_2)$ is
a klt $0$-pair.
Such a boundary exists because $X/Z\ni o$ has wFt.
By (1-3) and above versions of (4-6),
a required perturbation is
$$
B'=(1-\delta)B_1+\delta B_2
$$
for some $0<\delta\ll 1$.

Step~4. {\em Construction of an $n$-complement of $(X/Z\ni o,B)$.\/}
By~(4) and Theorem~\ref{b_n_comp_klt} with Addendum~\ref{klt_complements_B},
$(X/Z\ni o,B')$ has an $n$-complement $(X/Z\ni o,B^+)$ for
some $n\in \sN$.
We contend that it is also an $n$-complement of $(X/Z\ni o,B)$.
For this we need to verify only Definition~\ref{n_comp}, (1).
For the prime divisors of $X$ over $E$ it follows from~(5).
For the prime divisors of $X$ over $U$ it follows from (6) and
Lemmas~\ref{approximation_I}, \ref{approximation_II}.
Since $\sN$ is finite, we can find $\delta>0$ such that
$\norm{b'-l/I}<\delta\le 1/I(n+1)$ for all $n\in\sN$.

Step~5. The same arguments works for the bd-pairs.

\end{proof}

\begin{rem} \label{lc_compl_alternative}
However, more conceptual and precise proof
of Theorem~\ref{lc_compl} should work as follows.
This allows to verify that under the assumptions
and notation of Existence of $n$-complements
of the theorem
\begin{description}

  \item[\rm (1)]
$(X/Z\ni o,B^+)$ is an $n$-complement of itself and
of $(X/Z\ni o,B_{n\_\Phi}), (X/Z\ni o,B_{n\_\Phi}{}^\sharp)$,
and is a b-$n$-complement of itself and
of $(X/Z\ni o,B_{n\_\Phi}),
(X/Z\ni o,B_{n\_\Phi}{}^\sharp),(X^\sharp/Z\ni o,B_{n\_\Phi}{}^\sharp{}_{X^\sharp})$,
if $(X,B_{n\_\Phi}),(X,B_{n\_\Phi}{}^\sharp)$ are log pairs respectively;

  \item[\rm (2)]
$\sN$ has a filtration similar to the klt filtration~(\ref{filtration_klt})
of Addendum~\ref{klt_complements_fil} and
the existence $n$-complements agrees the filtration
for the class of pairs under assumptions of
Existence of $n$-complements in the theorem;
and

  \item[\rm (3)]
the bd-pairs can be treated similar to Addendum~\ref{bd_klt_compl}.

  \item[\rm (4)]
However, in Existence of $n$-complements of the theorem
we can not to replace $(X/Z\ni o,B)$ by $(X/Z\ni o,B_{\sN\_\Phi})$
(because lc type is the top of types; cf. Addendum~\ref{klt_complements_B}).

\end{description}
In particular, we use $\Phi$ here.

We sketch a construction of an lc type filtration.
We consider pairs $(X/Z\ni o,B)$ under the assumptions
of Existence of $n$-complements of the theorem.
We start from the top, {\em big lc type\/}:
$(X/Z\ni o,B)$ has an $\R$-complement $(X/Z\ni o,B^+)$
and $B^+-B$ is big over $Z\ni o$, equivalently,
$-(K+B)$ is big over $Z\ni o$ (cf. (1) of
Proposition-Definition~\ref{generic}).
The big lc type is also lc type of dimension $d$
for $d=\dim X$.
Any pair $(X/Z\ni o,B)$ of big lc type has
a maximal model $(X^\sharp/Z\ni o,B_{X^\sharp})$
of Construction~\ref{sharp_construction} with
$\R$-ample $-(K_{X^\sharp}+B_{X^\sharp})$ over $Z\ni o$.
The model is unique up to an isomorphism over $Z\ni o$
and has a unique minimal lc center $S$
by lc connectedness \cite[p.~203]{Sh03} \cite[Theorem~6.3]{A14}.
So, the lc type $d$ has additional parameter $s=\dim S$,
a nonnegative integer $\le d$.
The birational $1$-contraction preserves all multiplicities
of $B$ and we denote by $B_{X^\sharp}$ the birational transform of
$B$ on $X^\sharp$.
In its turn,
the finite set of positive integers $\sN^d$ for lc type $d$
of the lc filtration in (2) has
a (deceasing) subfiltration
$$
\sN^d=\sN^{(d,0)}\supseteq \sN^{(d,1)}\supseteq \dots
\supseteq \sN^{(d,d-1)}\supseteq \sN^{(d,d)}
$$
with respect to $s$.
A pair $(X/Z\ni o,B)$ has lc type $(d,s)$ with respect to
the filtration if both pairs
$$
(X/Z\ni o,B_{\sN^{(d,s+1)}\_\Phi}),(X/Z\ni o,B_{\sN^{(d,s)}\_\Phi})
$$
have lc type $(d,s)$.
The corresponding (b-)$n$-complement of type $(d,s)$
for some $n\in \sN_{(d,s)}=\sN^{(d,s)}\setminus\sN^{(d,s+1)}$ is
extended from an (b-)$n$-complement of an adjoint bd-pair
on $(S^\sharp/Z\ni o,B_{n\_\sN^{(d,s+1)}}{}^\sharp{}\dv+
\sB_{n\_\sN^{(d,s+1)}}{}^\sharp{}\md)$ of dimension $s$.
The adjoint bd-pair has generic klt type and
a finite set of positive numbers
$\sN_{(d,s)}=\sN(d,I,\ep,v,e,\widetilde{\Phi}',I)$
exists by Addendum~\ref{bd_klt_compl}.
So, it looks that we do not need the assumption $B\in\Gamma$.
No, we use the assumption for the log adjunction on $S^\sharp$
for $s\le d-2$ (in codimension $\ge 2$).
Indeed, such an adjunction (after a dlt or better log resolution)
is a sequence of adjunctions on divisors as in~\ref{adjunction_on_divisor}
concluding by an adjunction for a $0$-contraction
as in~\ref{adjunction_0_contr}.
The last adjunction has index $I=I(d,\Gamma\cap\Q)$ that
can be verified as in Step~2 in the proof of Theorem~\ref{lc_compl}.
For this we use lc type property of $(X/Z\ni o,B)$ and
the assumption that $B\in\Gamma$.
Respectively, in general we change $\Phi$ into $\widetilde{\Phi}$
by~(\ref{Phi-to-wtPhi}) for adjunctions on divisors and,
finally, into $\widetilde{\Phi}'$ by Addendum~\ref{adjunction_index_div}
for the adjunction of the $0$-contraction.
In the reverse direction, first
we lift $n$-complements by Theorem~\ref{invers_b_n_comp}.
Then we extend them by induction and by Theorem~\ref{extension_n_complement}.
During the extension we glue these complements on
the reduced divisors of a dlt or log resolution and,
finally, extend on $X$.

Notice that the subfiltration starts from $\sN^{(d,d)}$
and $\sN^{(d,d+1)}=\emptyset$ or $\sN'$ as in Addendum~\ref{klt_complements_fil}.
Type $(d,d)$ is generic klt type with
the minimal lc center
$S=X,S^\sharp=X^\sharp$ and $\Phi$ instead of $\widetilde{\Phi}'$.
However, the previous type $(d,d-1)$
is lc type but it is plt with the divisorial minimal
lc center $S,S^\sharp$, a reduced divisor
of $B_{n\_\sN^{(d,d)}},B_{n\_\sN{(d,d)}}{}^\sharp$
respectively and $\widetilde{\Phi}$ instead of $\widetilde{\Phi}'$.
The starting type $(d,0)$ has $S^\sharp=\pt$, a closed point
and $\sN^d=\sN^{(d,0)}$.

After that we continue with lc type of dimension $i\le d-1$,
where the $0$-contraction
$(X^\sharp,B^\sharp{}_{X^\sharp})\to Y/Z\ni o$ is fibered
and $\dim Y=i$.
By~(8) of~\ref{adjunction_div}
the adjoint bd-pair has lc type but
in general without index because $B$ may have
real horizontal over $Z\ni o$ multiplicities.
Unfortunately, a straightforward reduction
to dimension $i$ does not work because
does not preserve the assumption $B\in\Gamma$
and not preserve $\Gamma$.
However, the index $I$ is preserved for
$0$-contractions
$$
(X^\sharp,B_{n\_\sN^{(i,s+1)}\_\Phi}{}^\sharp{}_{X^\sharp}),
\to Y/Z\ni o,
n\in\sN_{(i,s)}=\sN^{(s,i)}\setminus \sN^{(i,s+1)},
$$
where type $(i,s)$ is the {\em filtration lc type\/}
with $i=\dim Y,s=\dim S$ and $S$
is a minimal lc center of the adjoint bd-pair on $Y$.
The types are ordered lexicographically.
E.g., lc type $(d-1,d-1)$ precede to lc type $(d,0)$;
in this case $\dim Y=d-1$, the adjoint bd-pair
is generic klt and the adjunction index is $I=I(d,\Gamma\cap \Q)$.
In this case we can construct $n$-complements by
Addendum~\ref{bd_klt_compl} and Theorem~\ref{invers_b_n_comp}.
In general construction of $n$-complements is
more involved but extend the construction for types $(d,s)$:
first we construct a (b-)$n$-complement on $Y$ and
then lift it to $X^\sharp$.

The bottom lc type $(0,0)$ has only global pairs
with $S,S^\sharp=\pt$ and
amounts the special global case of Corollary~\ref{bounded_lc_index}.

In particular, Theorem~\ref{lc_compl} can be applied
to any dcc set $\Gamma\subset [0,1]$ by Example~\ref{max_lc_index}.
Moreover, under the dcc assumption on $\Gamma$ in
Theorem~\ref{lc_compl} and in the remark we can suppose that
\begin{description}

  \item[\rm (5)]
$\sN$ has a single element, or equivalently,
there exists a positive integer $n=n(d,I,\ep,v,e)$ such that
Existence of $n$-complements holds for this $n$
\cite[Theorem~1.6]{HLSh}.

\end{description}
Indeed, all $n$-complements are coming from the exceptional
case.
By~\ref{direct_dcc} and a similar fact for adjunction
on a divisor, we can suppose that the exceptional pairs
also have boundaries with dcc multiplicities.
Actually in this situation (and even in general)
increasing multiplicities we can suppose that
the boundary multiplicities form a finite set \cite[Theorem~5.20]{HLSh}.
Then we can find a single complementary index $n$
(cf. \cite[Theorem~1.7]{B}).

The same works for bd-pairs.

Similar results expected for $a$-lc complements
and not only under wFt (cf. Conjecture~\ref{a_n_compl} below).
One of crucial pieces -- Corollary~\ref{bounded_lc_index},
the boundedness of lc index, -
does not hold in general for maximal $a$-lc $0$-pairs
but may hold under slightly stricter assumption
(cf. Addendum~\ref{strict_a_n_compl}).

\end{rem}

\paragraph{Affine maps to divisors.}
Let $\R^r$ be
a finite dimensional $\R$-linear space and
$X$ be an algebraic variety or space.
An affine map $A$ into $\R$-divisors of $X$ is
a map
$$
A\colon \R^r\to \WDiv_\R X
$$
which is $\R$-linear for every multiplicity
of $\R$-divisors.
That is, for every prime divisor $P$ on $X$,
there exist real numbers $a_P,a_{P,1},\dots,a_{P,r}$ such that,
for every point $(x_1,\dots,x_r)\in\R^r$,
\begin{equation}\label{P_component}
\mult_P A(x_1,\dots,x_r)=a_P+\sum_{i=1}^r a_{P,i}x_i.
\end{equation}
The map $A$ is $\Q$-affine if all $a_P,a_{P,i}\in\Q$.
The hight of such a $\Q$-affine map $A$ is
$$
h(A)=\max\{h(a_P),h(a_{P,i})\},
$$
where $h(a)$ is the usual hight of $a\in \Q$.
The boundedness of $h(A)$ does not imply the finiteness
of maps $A$ but the finiteness of their linear components~(\ref{P_component}).
More generally, the finiteness holds if $A\in\fA$,
where $\fA$ is finite set of real numbers and $A\in\fA$
means that every $a_P,a_{P,i}\in\fA$.

\begin{const} \label{Delta_complements}
Let $d$ be a nonnegative integer,
$I,\ep,v,e$ be the data as
in Restrictions on complementary indices,
$\fA$ be a finite set of real numbers and
$\Delta$ be a compact subset, e.g., a compact
polyhedron, in a finite dimensional
$\R$-linear space $\R^r$.

For every $x\in\R^r$, the set of real numbers
$$
\Gamma(x)=\Gamma(x,\Delta,\fA)=
\{a+\sum_{i=1}^r a_ix_i\mid a,a_1,\dots,a_r\in\fA\}\cap [0,1]
$$
is finite.
Hence $\Gamma(x)$ satisfies the assumption of Construction~\ref{lc_induction}:
$\Gamma(x)_{\max}$ is finite too.
Let $I(x)=\Gamma(d,\Gamma\cap\Q)$ be
the corresponding rational maximal lc index.
Additionally, we can suppose that
$I|I(x)$.

By Theorem~\ref{lc_compl}
there exists a finite set of positive integers
$\sN(x)=\sN(d,I(x),\ep,v,e)$ such that
\begin{description}

\item[\rm Restrictions:\/]
every $n\in\sN(x)$ satisfies
Restrictions on complementary indices with the given data;

\item[\rm Existence of $n$-complement:\/]
if $(X/Z\ni o,B)$ is a pair with
$\dim X=d$, $B\in\Gamma(x)$, wFt $X/Z\ni o$, connected $X_o$ and
with lc type $(X/Z\ni o,B)$
then $(X/Z\ni o,B)$
has an $n$-complement $(X/Z\ni o,B^+)$ for some $n\in\sN(x)$.

\end{description}

Additionally, by Lemma~\ref{approximation_lc_complements} below
with $\Gamma=\Gamma(x)$ and $\sN=\sN(x)$
there exists a positive real number $\delta(x)$.

By the finiteness of linear functions
$L(y_1,\dots,y_r)=a+\sum_{i=1}^r a_iy_i,a,a_1,\dots,a_r\in\fA,$
there exists an open neighborhood $U(x)$ of $x$
in $\R^r$ such
that for every those function $L$ and $(y_1,\dots,y_r)$ in $U$
$$
\norm{L(x_1,\dots,x_r)-L(y_1,\dots,y_r)}<\delta(x).
$$

Finally, since $\Delta$ is compact there
exists a finite covering
$$
\Delta\subset \bigcup_j U(x^j), x^j=(x_1^j,\dots,x_r^j)\in\Delta.
$$
Respectively, consider
$$
\sN=\sN(d,I,\ep,v,e)=\bigcup_j \sN(x^j)=
\bigcup_j \sN(d,I(x^j),\ep,v,e).
$$
Thus
\begin{description}

\item[\rm Restrictions:\/]
every $n\in\sN$ satisfies
Restrictions on complementary indices with the given data.

\end{description}

For bd-pairs we add a positive integer $m$.
So, $\sN=\sN(d,I,\ep,v,e,m),\sN(x)=\sN(d,I(x),\ep,v,e,m)$ and
$m|n\in\sN,\sN(x)$.

\end{const}

\begin{thm}[Lc type $n$-complements with $A$] \label{lc_compl_A}
Let $d$,$I,\ep,v,e$,$\Delta,\fA,\sN$ be
the data as in Construction~\ref{Delta_complements}.
Then every $n\in\sN$ satisfies Restrictions
on complimentary indices with the given data and
\begin{description}

\item[\rm Existence of $n$-complement:\/]
if $(X/Z\ni o,B)$ is a pair with
$\dim X=d$, wFt $X/Z\ni o$, connected $X_o$ and
such that there exists an affine map
$A\colon \R^r\to\WDiv_\R X$, where $A\in\fA$,
$$
A(\Delta)\subseteq \Compd_\R\cap \fD_\R^+=
\{D\in\WDiv_\R X\mid (X/Z\ni o,D)
\text{ has an }\R-\text{complement and }
D\ge 0\}
$$
and $B\in A(\Delta)$,
then $(X/Z\ni o,B)$ has an $n$-complement
$(X/Z\ni o,B^+)$ for some $n\in\sN$.

\end{description}

\end{thm}

\begin{add}
The same holds for bd-pairs $(X/Z\ni o,B+\sP)$ of index $m$
with $\sN=\sN(d,I,\ep,v,e,m)$.
That is,
\begin{description}

\item[\rm Restrictions:\/]
every $n\in\sN$ satisfies
Restrictions on complementary indices with the given data and
$m|n$.

\item[\rm Existence of $n$-complement:\/]
if $(X/Z\ni o,B+\sP)$ is a bd-pair of index $m$ with
$\dim X=d$, wFt $X/Z\ni o$, connected $X_o$ and
such that there exists an affine map
$A\colon \R^r\to\WDiv_\R X$, where $A\in\fA$,
$$
A(\Delta)\subseteq \Compd_\R\sP\cap \fD_\R^+=
\{D\in\WDiv_\R X\mid (X/Z\ni o,D+\sP)
\text{ has an }\R-\text{complement and }
D\ge 0\}
$$
and $B\in A(\Delta)$,
then $(X/Z\ni o,B+\sP)$ has an $n$-complement $(X/Z\ni o,B^++\sP)$ for some $n\in\sN$.

\end{description}

\end{add}

The proof uses the following.

\begin{lemma}[Approximation of $n$-complements] \label{approximation_lc_complements}
Let $\Gamma$ be a finite subset in $[0,1]$ and
$\sN$ be a finite set of sufficiently divisible positive integers:
$nb\in\Z$ for every $n\in\sN$ and rational $b\in\Gamma$.
Then there exists a positive real number $\delta$
with the following approximation property.
If $(X/Z,B^+)$ is an $n$-complement of
$(X/Z,B)$ with $B\in\Gamma$ with $n\in\sN,B\in\Gamma$ and
$D$ is a subboundary on $X$ such that $\norm{B-D}<\delta$
then $(X/Z,B^+)$ is also an $n$-complement of
$(X/Z,D)$.

\end{lemma}

\begin{add} \label{approximation_lc_b_n_complements}
Let $\Phi=\Phi(\fR)$ be a hyperstandard set associated with
a finite set of rational numbers $\fR$ in $[0,1]$.
Suppose additionally that $D$ is a boundary.
Then every b-$n$-complement $(X/Z,B^+)$ of
$(X/Z,B_{n\_\Phi})$ is also a b-$n$-complement of
$(X/Z,D_{n\_\Phi})$ if $(X,B_{n\_\Phi}),(X,D_{n\_\Phi})$
are log pairs.

\end{add}

\begin{add}
The same holds for bd-pairs $(X/Z,B+\sP),(X/Z,D+\sP)$
of index $m|n$ with the same b-divisor $\sP$.

\end{add}

\begin{proof}
Immediate by definition and Lemmas~\ref{approximation_I}, \ref{approximation_II}.
Actually, we need to verify only that (1) of Definition~\ref{n_comp}
for $B^+$ with respect to $B$ implies that of with respect to $D$
(cf. Proposition~\ref{1_of def_2 for B}).
For rational $b$, we use Lemmas~\ref{approximation_I}, \ref{approximation_II}.
(The case with $d\le b=0$ is easy to add to Lemma~\ref{approximation_I}.)
For irrational $b$, we can use the continuity of $\rddown{(n+1)d}/n$
in a neighborhood of $b$.

To prove Addendum~\ref{approximation_lc_b_n_complements}
it is sufficient to verify that $D_{n\_\Phi}\le B_{n\_\Phi}$
under our assumptions.
This follows from definition and Corollary~\ref{accumulation_1}.

Similarly we can treat bd-pairs.

\end{proof}

\begin{proof}[Proof of Theorem~\ref{lc_compl_A}]
Restrictions on complementary indices hold by
Construction~\ref{Delta_complements}.

Let $(X/Z\ni o,B)$ be a pair under
the assumptions of Existence of $n$-complements in
the theorem and $A\colon \R^r\to\WDiv_\R X$ be a corresponding affine map.
By our assumptions $B=A(x)$ for some $x\in\Delta$.
On the other hand, by Construction~\ref{Delta_complements},
$x\in U(x^j)$, where $x^j\in\Delta$ and
$$
\norm{A(x^j)-B}<\delta(x^j).
$$
Again by Construction~\ref{Delta_complements}
$(X/Z\ni o, A(x^j))$ has an $n$-complement $(X/Z\ni o,B^+)$
for some $n\in\sN(x^j)$.
Indeed, $(X/Z\ni o, A(x^j))$ satisfies the assumptions of
Existence of $n$-complements in Construction~\ref{Delta_complements}.
In particular, $(X/Z\ni o, A(x^j))$ has an $\R$-complement or
lc type because $A(x^j)\in \Compd_\R$ by our assumptions.
This with $A(x^j)\ge 0$ implies that
$A(x^j)$ is a boundary and $\in\Gamma(x^j)$
(cf. Remark~\ref{remark_def_complements},(1)).

Finally, $(X/Z\ni o,B^+)$ is also an $n$-complement of
$(X/Z\ni o,B)$ by Construction~\ref{Delta_complements}
and Lemma~\ref{approximation_lc_complements}.

Similarly we can treat bd-pairs.

\end{proof}

\begin{rem}
(1) Actually Existence of $n$-complements in Theorem~\ref{lc_compl_A}
holds for a slightly large set of $B$
because the covering of $\Delta$ is larger than $\Delta$ itself.

(2) For an analog of Remark~\ref{lc_compl_alternative}
and Theorem~\ref{lc_compl_A} we can use the remark and
Addendum~\ref{approximation_lc_b_n_complements}.
In this situation under the assumptions of
Existence of $n$-complements of Theorem~\ref{lc_compl_A},
$(X/Z\ni o,B^+)$ is a b-$n$-complement of
$(X/Z\ni o,B_{n\_\Phi}))$ too, if
$(X,B_{n\_\Phi})$ is a log pair.
Actually, in the proof it is better to take
a $\Q$-factorialization at the begging.

The same works for bd-pairs.

\end{rem}

\begin{proof}[Proof of Theorems~\ref{bndc} and~\ref{bd_bndc}]
The theorem under the assumption~(1) is immediate
by Theorem~\ref{b_n_comp_klt} and Addendum~\ref{klt_complements_B}.

For the assumption~(2) we can use Theorem~\ref{lc_compl_A}
with an appropriate set of real numbers $\fA$.
The $P$-component~(\ref{P_component}) of a map $A$ in this case
has $a_P=0,a_{P,i}=\mult_P D_i, i=1,\dots,r$, nonnegative integers.
By definition, for $(d_1,\dots,d_r)\in \R^r$,
$$
A(d_1,\dots,d_r)=\sum_{i=1}^r d_iD_i
\text{ and }
\mult_P A(d_1,\dots,d_r)=\sum_{i=1}^r a_{P,i}d_i.
$$
On the other hand, for $(d_1,\dots,d_r)\in\Delta$,
all $d_i\ge 0$ and $\sum_{i=1}^{r}a_{P,i}d_i\le 1$ hold
by our assumptions.
Hence we can suppose that all numbers $a_{P,i}$ belong to a finite set
$\fA$, including $0$.
Indeed, if $a_{P,i}\gg 0$ then every $d_i=0$ and
we can take $a_{P,i}=0$.

Theorem~~\ref{bndc} under the assumption~(3) is
immediate by Theorem~\ref{lc_compl} and
Example~\ref{max_lc_index}.

To prove Addendum~\ref{bounded_compon} we
can consider in our dimensional induction
simultaneously a bounded number of lc centers, e.g.,
in Lemma~\ref{exist_plt_model} with nonconnected $X_o$
a required plt model for every connected component
of $X_o$.
Equivalently, we can prove Theorem~\ref{bndc}
relaxing the assumption that $X$ is irreducible but
assuming that $X$ has a bounded number of
irreducible (connected) components $X_i$
and $\dim X=\max \dim X_i$.
All (b-)$n$-complements are coming from the exceptional one
of bounded dimension.
So, we add also the boundedness of components of
the exceptional pairs or consider exceptional pairs with
bounded number of irreducible (connected) components.
Notice that in the proof of Theorem~\ref{excep_comp} we
use only boundedness of exceptional pairs but
not their irreducibility.
Since we are working with algebraic spaces $X$ too,
it is possible to use an appropriate \'etale neighborhood
(a branch) for every connected component of $X_o$.
We can take the same $n$-complements
for isomorphic neighborhoods.
Thus it is to count only nonisomorphic neighborhoods.

Similarly we can treat bd-pairs.

To prove Addendum~\ref{nonclosed},
for a global pair $(X,B)$, we use the invariance of $H^0$ with respect
to the algebraic closure.
The connectedness of $X_0$ depends on
the algebraic closure but independent modulo conjugation
of the closure.
Thus after taking the algebraic closure we can use Addendum~\ref{bounded_compon}.
Note also that existence of an $n$-complement means existence of
an element $\overline B$ in a linear system
$$
\linsys{-nK-nS-\rddown{(n+1)(B-S)}}
$$
such that $(X,B^+)$ is lc, where
$S=\rddown{B}$, the reduced part of $B$, and
$$
B^+=B+\frac 1n (\rddown{(n+1)(B-S)}+\overline B)
$$
\cite[after Definition~5.1]{Sh92}.
Thus $\overline B$ should be sufficiently general in the Zariski topology.
If such element exists over the algebraic closure $\overline k$,
it exists also over $k$ because $k$ is infinite
(cf. Example~\ref{1stexe}, (5)).

The same arguments work for global $G$-pairs assuming that $n$ is
sufficiently divisible.
Indeed, we can suppose that a canonical divisor $K$ of $X$ is
$G$-semicanonical: $|G|K$ is $G$-invariant.

Similarly we can treat nonglobal pairs and $G$-pairs.

\end{proof}

In dimension $1$ such subtleties are not need.

\begin{exa}[$n$-complements in dimension $1$] \label{n_comp_dim_1}
Theorem~\ref{bndc} for $\dim X=1$ holds for every pair $(X/Z,B)$,
without the local (with unbounded number of leaves for $X/Z$),
wFt and connectedness of $X_o$ assumptions.
Any such a local pair with a boundary $B$ has an $\R$-complement and
$n$-complement for any positive integer $n$ satisfying
Restrictions on complementary indices.

In the global case we suppose existence of an $\R$-complement
instead of~(1) in Theorem~\ref{bndc} (cf. Example~\ref{R_comp_dim 1}).

The global case with $X=E$, a curve of genus $1$, is also
trivial: $B^+=B=0$ and $n$ is any positive integer
satisfying Restrictions as above.

So, the main interesting case as in Example~\ref{R_comp_dim 1}
concerns global pairs $(\PP^1,B)$ with $B=\sum_{i\ge 1} b_i P_i,
1\ge b_1\ge b_2\ge \dots \ge b_i \ge \dots\ge 0$ and
$\sum_{i\ge 1} b_i\le 2$.
However this time to exclude points $P_i$ with
small $b_i$ we use our general approach with
low approximations but, for simplicity,
without Restrictions on complementary indices.

{\em Lc type\/}: $b_1=1$ and $B$ has at most two points $P_i$
with $b_i=1$.
In this case, $(\PP^1,B)$ has $1$-complement, except for,
$$
(\PP^1,P_1+\frac 12 b_2+\frac 12 b_3).
$$
The pair is a $2$-complement of itself.

{\em Generic pairs\/}: every $b_i<1$ and
$\sum_{i\ge 2} b_i<1$.
In this case we have again a $1$-complement
by an elementary computation (cf. Theorem~\ref{extension_n_complement}).
Moreover, we can suppose that
$(\PP^1,B_1)$ has generic type where
$B_1$ is the approximation of $B$ with
$\Gamma(1,\emptyset)=\{0,1/2,1\}$.
(For $\Phi=\emptyset$, we take an abridged set of multiplicities,
in Hyperstandard sets of Section~\ref{technical}, with
only $l=1$.)
So, $(\PP^1,B_1)$ is generic only for
$$
B_1=\begin{cases}
0, & \mbox {if } 1/2> b_1\\
\frac 12 P_1 , & \mbox{if } b_1\ge 1/2> b_2\\
\frac 12 P_1+\frac 12 P_2, & \mbox{ if } b_2\ge 1/2 >b_3.
    \end{cases}
$$
In all these cases $(\PP^1,B)$ has a $1$-complement.

{\em Semiexceptional pairs\/}: $(\PP^1,B_1)$ with
$$
B_1=\begin{cases}
\frac 12 P_1 +\frac 12 P_2+\frac 12 P_3, & \mbox{if } b_3\ge 1/2>b_4 \\
\frac 12 P_1+\frac 12 P_2+\frac 12 P_3+\frac 12 P_4, & \mbox{ if }
b_1=b_2=b_3=b_4=1/2 \text{ and }b_5=0.
    \end{cases}
$$
In the last case
$$
(\PP^1,\frac 12 P_1+\frac 12 P_2+\frac 12 P_3+\frac 12 P_4)
$$
is a $2$-complement of itself.

Thus we need to construct $n$-complements only
in the case
$$
B_1=\frac 12 P_1+\frac 12 P_2+\frac 12 P_3.
$$
The pair $(\PP^1,B_1)$ has the semiexceptional type $(0,0)$
and this is the top type with the next generic type.
The previous type $(-1,-)$ is exceptional.
According to our general approach we need to
extend the set $\sN'=\{1\}$ to another set of positive
integers $\sN^{(0,0)}$ (see Semiexceptional filtration in
Section~\ref{semiexcep_compl_:}).
For simplicity consider $\sN_{(0,0)}=\{2\},\sN^{(0,0)}=\{1,2\}$
and $\Phi=\emptyset$.
Then $\Gamma(\{1,2\},\emptyset)=\{0,1/3,1/2,2/3,5/6,1\}$.
For the last set of boundary multiplicities $(\PP^1,B_{\{1,2\}})$ is
again semiexceptional but not exceptional only for
$$
B_{\{1,2\}}=\begin{cases}
\frac 12 P_1 +\frac 12 P_2+\frac 12 P_3, & \mbox{if }
2/3>b_1\ge b_3\ge 1/2 \text { and }1/3> b_4\\
\frac 23 P_1+\frac 12 P_2+\frac 12 P_3, & \mbox{ if }
5/6>b_1\ge 2/3> b_2\ge b_3\ge 1/2 \text{ and } 1/3>b_4\\
\frac 56 P_1+\frac 12 P_2+\frac 12 P_3, & \mbox{ if }
b_1\ge 5/6, 2/3> b_2\ge b_3\ge 1/2 \text{ and } 1/3>b_4.
    \end{cases}
$$
In all three cases $(\PP^1,B)$ has a $2$-complement
by a direct computation (cf. Theorem~\ref{semiexcep_compl}).

{\em Exceptional pairs\/}: $(\PP^1,B_{\{1,2\}})$
with $B_{\{1,2\}}\in \{0,1/3,1/2,2/3,5/6,1\}$ form a bounded
family.
As in Example~\ref{R_comp_dim 1} we can join all
small multiplicities $b_i<1/3$ to $b_1$
(or to any other $b_i\ge 1/3$; cf. Step~7 of the proof of Theorem~\ref{excep_comp}).
The new $(X,B)$ will have also exceptional
$(\PP^1,B_{\{1,2\}})$ (possibly with different $B_{\{1,2\}}$).
Moreover, $B=b_1P_1+b_2P_3+b_3P_4$, except for,
$B$ with
$$
B_{\{1,2\}}=\frac 23 P_1 +\frac 12 P_2+\frac 12 P_3+\frac 13 P_4.
$$
In this case $(\PP^1,B)$ is a $6$-complement of itself
(and has also a $4$-complement).

In all other cases $B=b_1P_1+b_2B_2+b_3P_3$ and
$(\PP^1,B_{\{1,2\}})$ is exceptional.
In these cases we can find a finite set of small complementary
indices, e.g., $\{3,4,6\}$ as in \cite[Example~5.2]{Sh92}.
Or we can find a finite set of complementary indices satisfying Restrictions using
Theorem~\ref{excep_comp} or \cite[Example~1.11]{Sh95}.

\end{exa}

\section{Applications} \label{applic}

\paragraph{Inverse stability for $\R$-complements.}
Let $(X/Z,D)$ be a pair and $\ep$ be a positive integer.
We say that the {\em inverse stability for\/}
$\R$-{\em complements\/} holds for $(X/Z,D)$ if,
for every boundary $B$ on $X$ such that $\norm{B-D}<\ep$
and under certain additional assumptions,
the existence of an $\R$-complement for $(X/Z,B)$ implies
that of for $(X/Z,D)$.

The same definition works for bd-pairs $(X/Z,B+\sP),(X/Z,D+\sP)$,
that is, we compare only the divisorial parts $B,D$.

\begin{thm} \label{invers_stability_R_complements}
Let $d,h,l$ be nonnegative integers and
$v\in\R^l$ be a vector.
Then there exists a positive real number $\ep$ such that,
for every
\begin{description}

  \item[]
$\Q$-affine map $A\colon\R^l\to \WDiv_\R X$ of
the height not exceeding $h$; and

  \item[]
pair $(X/Z\ni o,B)$ with wFt $X/Z\ni o$, $\dim X=d$,
a boundary $B$ with $\norm{B-A(v)}<\ep$ and
under either of the (additional) assumptions
(1-3) of Theorem~\ref{bndc},
\end{description}
the inverse stability for $\R$-complements holds for $(X/Z\ni o,A(v))$.

\end{thm}

Notice that we do not suppose that $X_o$ is connected.

\begin{add} \label{invers_stability_R_complements_local}
We can assume that $B$ is given
locally over $Z\ni o$
(even in the \'etale topology)
near every connected component of $X_o$.

\end{add}

\begin{add} \label{invers_stability_R_complements_boundary}
$A(v)$ is a boundary.

\end{add}

\begin{add} \label{invers_stability_R_complements_U}
There exists a neighborhood $U$ of $v$ in $\spn{v}$
such that, for every vector $u\in U$,
$A(u)$ is a boundary and $(X/Z\ni o,A(u))$ has
an $\R$-complement.

\end{add}

\begin{add} The same holds for bd-pairs
$(X/Z\ni o,A(v)+\sP),(X/Z\ni o,B+\sP)$ of index $m$
with $\ep$ depending also on $m$.

\end{add}

\begin{cor}[Direct stability for $\R$-complements;
cf. {\cite[Theorem~1.6]{N}
\cite[Theorems~5.6 and 5.16]{HLSh}}] \label{direct_stability}
Under the assumptions and notation of Theorem~\ref{invers_stability_R_complements}
there exists a neighborhood $U$ of $v$ in $\spn{v}$
such that if $A(v)$ is a boundary and $(X/Z\ni o,A(v))$ has an $\R$-complement
and additionally either of the assumptions
(1-3) of Theorem~\ref{bndc} holds for $(X/Z\ni o,A(v))$
then, for every vector $u\in U$,
$A(u)$ is a boundary and $(X/Z\ni o,A(u))$ has
an $\R$-complement.

The same holds for bd-pairs of index $m$ $(X/Z\ni o,A(v)+\sP)$.
\end{cor}

\begin{proof}
Take $U$ as in Addendum~\ref{invers_stability_R_complements_U}
and apply Theorem~\ref{invers_stability_R_complements}
with the addendum to $B=A(v)$.

Similarly we can treat bd-pairs.

\end{proof}

\begin{cor}[$n$-complements vs $\R$-complements]
Let $(X/Z\ni o,B)$ be a pair with  wFt $X/Z\ni o$ and a boundary $B$.
Then  $(X/Z\ni o,B)$ has an $\R$-complement under
either of the following assumptions:
\begin{description}

  \item[\rm Weak version:]
$(X/Z\ni o,B)$ has $n$-complements for
infinitely many positive integers $n$;
or

  \item[\rm Strong version:]
$(X/Z\ni o,B)$ has an $n$-complements for
one but sufficiently large positive integer $n$.

\end{description}

The same holds for bd-pairs
$(X/Z\ni o,B+\sP)$ of index $m$.

\end{cor}

\begin{proof}
{\em Weak version.} Immediate by Theorem~\ref{R-vs-n-complements}.

{\em Strong version.} Immediate
by the closed rational polyhedral property of Theorem~\ref{R_compl_polyhedral}
and by the proof of Theorem~\ref{R-vs-n-complements}.

Alternatively, we can use Theorem~\ref{invers_stability_R_complements}
instead of Theorem~\ref{R_compl_polyhedral}.
Indeed, according to the proof of Theorem~\ref{R-vs-n-complements} there exists
a sequence  of positive integers $n_i$ and
of rational boundaries $B_i=\rdn{B}{n_i},i=1,2,\dots$, on $X$ such that
$B=\lim_{i\to\infty}B_i,\lim_{i\to\infty}n_i=+\infty$ and
every pair $(X/Z\ni o,B_i)$
has an $\R$- and monotonic $n_i$-complement.
For $\ep$ of Theorem~\ref{invers_stability_R_complements}
and $i$ such that $\norm{B_i-B}<\ep$,
$n=n_i$ works.
We apply the inverse stability for $\R$-complements
to $A(v)=B=\sum b_jD_j,v=(b_j)$, and $B=B_i$.

Similarly we can treat bd-pairs.

\end{proof}

\begin{thm}[Inverse stability for $n$-complements] \label{invers_stability_n_complements}
Let $n$ be a positive integer
and $\mu$ be a positive real number.
There exists a positive real number $\delta$
such that if
\begin{description}

  \item[]
$(X/Z,B)$ is a pair with a boundary $B$, having
an $n$-complement, and
  \item[]
$D$ is a $\Q$-divisor on $X$ such that
$nD$ is an integral Weil divisor,
$\norm{D-B}< \delta/n$ and
every positive multiplicity $d$ of $D$ is $\ge \mu$

\end{description}
then
$(X/Z,D)$ has an $n$-complement.

\end{thm}

\begin{add} \label{invers_stability_n_complements_boundary}
$D$ is a boundary.

\end{add}

\begin{add} \label{invers_stability_n_complements_B+}
If $(X/Z,B^+)$ is an $n$-complement of $(X/Z,B)$
then $(X/Z,B^+)$ is also an $n$-complement of $(X,D)$.

\end{add}

\begin{add} \label{invers_stability_n_complements_monotonic}
Any $n$-complement of $(X/Z,D)$ is monotonic.

\end{add}

\begin{add} \label{invers_stability_n_complements_delta}
We can take any positive real number $\delta\le \mu n/(n+1)]$ and $1$.

\end{add}

\begin{add} \label{invers_stability_n_complements_0}
For $\mu>1$, $D=0$ and the theorem is trivial.
For $\mu \le 1$, we can take any positive real number $\delta\le \mu n/(n+1)$.
in particular, $0<\delta\le \mu/2$ independent of $n$.

\end{add}

\begin{add} The same holds for bd-pairs
$(X/Z\ni o,B+\sP),(X/Z\ni o,D+\sP),(X/Z,B^++\sP)$ of index $m|n$.

\end{add}

\begin{proof}
It is enough to verify Addendum~\ref{invers_stability_n_complements_B+}.
By Definition~\ref{n_comp},
it is enough to verify (1)~of the definition.

Let $P$ be a prime divisor on $X$.
Put $d=\mult_P D$ and $b^+=\mult_P B^+$.
We need to verify that
$$
b^+\ge
\begin{cases}
1, \text{ if } &d=1;\\
\rddown{(n+1)d}/n &\text{ otherwise.}
\end{cases}
$$

Take $\delta$ of Addendum~\ref{invers_stability_n_complements_delta}.

Step~1. {\em Addendum~\ref{invers_stability_n_complements_boundary}.\/}
$D$ is a boundary, that is,
$d\in [0,1]$.
Moreover, $d=0$ or $\ge \mu$.
By our assumptions $d=m/n$, where $m\in\Z$, and
$$
d< b+\frac\delta n\le 1+\frac1n,
$$
where $b=\mult_P B$.
Hence $m<n+1$ and, moreover, $\le n$, that is,
$d\le 1$.
Similarly, $m\ge 0$ and $d\ge 0$ because $b\ge 0$.

Thus if $d\not= 0$ then $d$ is positive and $\ge \mu$
by our assumptions.
This implies Addendum~\ref{invers_stability_n_complements_0}.
Indeed, for $\mu>1$, $D=0\le B$ and the theorem follows from
Proposition~\ref{D_D'_complement}.

Step~2. Addendum~\ref{invers_stability_n_complements_monotonic} follows from
Example~\ref{rddown_(n+1)_m_m}, (2).

Step~3. {\em Case $d\le b$.\/}
For $b<1$, follows from~(1) of Definition~\ref{n_comp} for $B$ and
the monotonicity of $\rddown{\ }$: $d<1$ and
$$
\rddown{(n+1)d}/n\le\rddown{(n+1)b}/n\le b^+ .
$$
Otherwise, $b=b^+=1,d\le 1=b^+$ and
the required inequality holds by definition.

Step~4.
{\em Case $d>b$.\/}
In this case $d>0$ and $\ge\mu$ because $b\ge 0$.
Since $\norm{d-b}<\delta/n$ holds,
$$
b> d-\frac\delta n\ge d-\frac{\mu n}{(n+1)n}=
d-\frac\mu{n+1}.
$$

By Step~1, $d\le 1$ and $\mu\le 1$ (cf. Addendum~\ref{invers_stability_n_complements_0}).
Moreover, if $\mu=1$ then $d=1,1-1/(n+1)=n/(n+1)<b<1$ and
$$
b^+\ge\rddown{(n+1)b}/n=1=d.
$$

Otherwise, $\mu<1$ and by our assumptions
$d=m/n\ge \mu$, where $m$ is integer and $1\le m\le n$.

If $d<1$ then again Example~\ref{rddown_(n+1)_m_m}, (2) gives the required inequality
\begin{align*}
b^+\ge \rddown{(n+1)b}/n\ge & \rddown{(n+1)(\frac mn-\mu/(n+1))}/n=\\
\rddown{m+\frac mn-\mu}/n=&
m/n+\rddown{d-\mu}/n=m/n=d=\rddown{(n+1)d}/n.
\end{align*}

If $d=1$ then
$$
b^+\ge\rddown{(n+1)b}/n=1=d
$$
because $n/(n+1)<b<1$ as above.

Similarly we can treat bd-pairs.

\end{proof}

\begin{proof}[Proof of Theorem~\ref{invers_stability_R_complements}]
We will chose $\ep$ below.

Step~1. {\em Renormalization of $v$ and $I$.\/}
We can suppose that every $A$ has integral ($\Z$-matrix) linear part.
For this we change the standard basis
$(1,0,\dots,0),\dots, (0,0,\dots,1)$ of $\R^l$ by
$(N,0,\dots,0),\dots,(0,0,\dots,N)$
for sufficiently divisible positive integer $N$.
Then we need to increase the height $h$ and replace
the vector $v$ by $v/N$.

In this step we can chose a positive integer $I$ such
that $I A$ is integral, equivalently, for every constant $c$ of
$A$, $Ic$ is an integer.
Hence $nA(w)$ is integral if $w\in\Q^l$, $nw\in\Z^l$
and $I|n$.

Step~2. {\em Choice of $\mu$.\/}
Since the height of $A$ is bounded by $h$,
every multiplicity $a=\mult_P A(v)$ in prime $P$ has the form
$$
a=c+\sum_{i=1}^l a_i v_i,
$$
where $a_i\in\Z,\norm{a_i}\le h,c\in\Q, Ic\in\Z$,
[remark: we need bounded hight of $c$ only to bound
the denominators of $c$] and
$v=(v_1,\dots,v_l)$.
Hence the set of those multiplicities $a$ is finite and
there exists $\mu>0$ such that $a\ge \mu$ if $a>0$.
Indeed, the constant terms $c$ belong to a finite set, e.g.,
because by our assumptions
$$
\norm{b-a}=\norm{b-c-\sum_{i=1}^la_iv_i}< \ep
$$
and $b\in[0,1]$.
This implies also Addendum~\ref{invers_stability_R_complements_boundary},
if $\ep$ is sufficiently small.
We suppose that $\mu$ is also sufficiently small, e.g., $\mu\le 1$.

Note also that for any two vectors $w_1,w_2\in\R^l$,
$$
\norm{A(w_1)-A(w_2)}\le hl\norm{w_1-w_2}.
$$

Step~3. {\em Rational case.\/}
For rational $v$, $\spn{v}=v$.
In this case we can suppose that $A$ is constant, e.g., all $a_i=0$.
Then Theorem~\ref{invers_stability_n_complements} implies the required stability for $n$ sufficiently
divisible and sufficiently small $\ep$.
The $n$-complement is monotonic for $(X/Z\ni o,A(v))$ by
Addenda~\ref{invers_stability_n_complements_monotonic}.

So, we suppose below that $v$ is irrational.
Thus $r=\dim\spn{v}\ge 1$ and $l\ge r\ge 1$.
Notice that by definition $h\ge 1$ because
otherwise there are no $A$.

Step~4. {\em Choice of directions $e_0,\dots,e_r$ in $\spn{v}$.\/}
Take $r+1,r=\dim\spn{v}$,
directions $e_i,i=0,\dots,r$, in $\spn{v}$ such that
$0$ is inside of the simplex $\simp{e_0,e_1,\dots,e_r}$.
For rather small $\mu$, if every $\norm{e_i'-e_i}\le \mu$,
for any other directions $e_0',e_1',\dots,e_r'$
in $\spn{v}$,
then $\simp{e_0',e_1',\dots,e_r'}$ is also a simplex
with $0$ inside of it.

Step~5. {\em Choice of complementary indices.\/}
For $i=0,1,\dots,r$, there exists a finite set
of positive integers
$$
\sN_i=\sN(d,I,\mu/8hl,v,e_i).
$$
By Theorem~\ref{bndc}
under either of assumptions (1-3) of the theorem, there exists such a finite set.

Step~6. {\em Choice of $\ep$.\/} We can take
$\ep\le\mu/8n$ for every $i=0,1,\dots,r$ and
$n\in \sN_i$.

Indeed, let $B$ be a boundary on $X$ such that
\begin{description}

  \item[\rm (4)]
 $\norm{B-A(v)}< \ep$.

\end{description}

Then under other assumptions of Theorem~~\ref{invers_stability_R_complements}
and either of assumptions~(1-3) of Theorem~\ref{bndc},
for every $i=0,1,\dots,r$, by Theorem~\ref{bndc}
$(X/Z\ni o,B)$ has an $n_i$-complement
$(X/Z\ni o,B_i^+)$ with $n_i\in\sN_i$.
By Restrictions of Theorem~\ref{bndc} we have also approximations
$w_i\in \spn{v}$ such that
\begin{description}

  \item[\rm (5)]
$n_iw_i\in\Z^l$;

  \item[\rm (6)]
$\norm{w_i-v}<\mu/8hln_i$;
and

  \item[\rm (7)]
$$
\norm{\frac{w_i-v}{\norm{w_i-v}}-e_i}<\mu/8hl\le \mu.
$$
\end{description}

By  Theorem~\ref{invers_stability_n_complements}
there exist $\R$-complements $(X/Z\ni o,B_i^+)$
of $(X/Z\ni o,D_i)$, where $D_i=A(w_i)$.

Indeed, by (5) and Step~1 $n_iD_i$ is integral because
$I | n_i$.

By (6) and Step~2,
\begin{description}

  \item[\rm (8)]
$\norm{A(v)-D_i}<hl\mu/8hln_i=\mu/8n_i\le \mu/8$.

\end{description}
In particular, this implies that
if $\mult_P A(v)=0$ for a prime divisor $P$ on $X$,
then $\mult_P D_i=0$ too because $n_i\mult_P D_i$ is integral
but $\mu/8n_i\le 1/8n_i<1$ for $\mu\le 1$.
Otherwise, $\mult_P A(v)\ge \mu$ and by (8)
$\mult_P D_i\ge \mu-\mu/8=7\mu/8\ge \mu/2$.
Again by (8) and (4),
$$
\norm{D_i-B}\le \norm{B-A(v)}+\norm{A(v)-D_i}<
\mu/8n_i+\mu/8n_i=\mu/4n_i.
$$
Thus the inversion of Theorem~\ref{invers_stability_n_complements} holds
for $(X/Z\ni o,D_i)$ with $\mu/2$ instead of $\mu$
and $\delta=\mu/4$ by Addendum~\ref{invers_stability_n_complements_0}.
So, $(X/Z\ni o,D_i)$ has an $n_i$-complement.
It is monotonic by Addendum~\ref{invers_stability_n_complements_monotonic}.
Hence  every $(X/Z\ni o,D_i)$ has an $\R$-complement too.

Step~7. {\em Conclusion.\/} By the convexity of Theorem~\ref{R_compl_polyhedral}
we need to verify that
$A(v)$ belongs to $\simp{D_0,D_1,\dots,D_r}$.
In its turn, this follows from the inclusion
$v\in\simp{w_0,w_1,\dots,w_r}$.

By Step~4 and (7)
$$
0\in\simp{\frac{w_0-v}{\norm{w_0-v}},
\frac{w_1-v}{\norm{w_1-v}},\dots,\frac{w_r-v}{\norm{w_r-v}}}.
$$
Since the denominators $\norm{w_i-v}$ are positive,
there exist positive real numbers $\nu_0,\nu_1,\dots,\nu_r$ such that
$$
\nu_0(w_0-v)+\nu_1(w_1-v)+\dots +\nu_r(w_r-v)=0
\text{ and }
\sum_{i=0}^r \nu_i=1.
$$
Hence $v=\nu_0w_o+\nu_1w_1+\dots+\nu_rw_r$ which gives
the required inclusion.

We were cheating a little bit.
The previous proof works only if $X_o$ is connected.
Now we make

Step~8. {\em General case and Addendum~\ref{invers_stability_R_complements_local}.\/}
(Cf. Proposition~\ref{local_compl})
Suppose that $(X/Z\ni o,A(v))$ does not have $\R$-complement.
We can assume that $X$ is $\Q$-factorial.
Otherwise we replace $X$ by its $\Q$-factorialization.
Again $(X/Z\ni o,A(v))$ does not have $\R$-complement.
We can suppose that $X/Z\ni o$ has Ft, in particular,
is projective.
Otherwise we use a small modification of Lemma~\ref{wTt_vs_Ft}.
Still $(X/Z\ni o,A(v))$ does not have $\R$-complement.

The pair $(X/Z\ni o,A(v))$ is lc.
Otherwise $(X/Z\ni o,A(v))$ is not lc near some
connected component of $X_0$ and
does not have an $\R$-complement near this component.
This is impossible by the connected case.

The divisor $-(K+A(v))$ is not nef.
Otherwise $(X/Z\ni o,A(v))$ has $\R$-complement.
So, we can apply $-(K+A(v))$-MMP as in Construction~\ref{sharp_construction}.
This preserves the $\R$-complements by Proposition~\ref{monotonicity_I}.
Thus after finitely may steps \cite[Corollary~5.5]{ShCh}
we have an extremal contraction $X/Y/Z\ni o$ which
is positive with respect to $K+A(v)$.
The transformations preserve connected components of $X_o$.
So, the initial model $(X/Z\ni o, A(v))$ does not
have an $\R$-complement near the connected component
of $X_o$ corresponding to the connected component
with the nearby fibration.
This is impossible by the connected case.

Step~9. {\em Addendum~\ref{invers_stability_R_complements_U}.\/}
Immediate by Step~7 and the convexity of Theorem~\ref{R_compl_polyhedral}:
$U$ is a neighborhood of $v$ in
$[w_0,w_1,\dots,w_r]\subset \spn{v}$.

Similarly we can treat bd-pairs
with the assumptions (1-3-bd) of Theorem~\ref{bd_bndc}.

\end{proof}

\paragraph{$\R$-complement thresholds.}
Let $(X/Z,D)$ be a pair such that
it has an $\R$-complement and
$F>0$ be an effective $\R$-divisor on $X$.
Then the following threshold
$$
\Rct(X/Z,D;F)=
\sup\{t\in \R\mid (X/Z,D+tF)
\text{ has an }\R-\text{complement}\}
$$
is a nonnegative real number and well-defined.
If $X/Z$ has wFt then
we can use the maximum instead of the supremum
by the closed property in Theorem~\ref{R_compl_polyhedral}.
The threshold will be called the $\R$-{\em complement thereshold\/}
of $F$ for $(X/Z,D)$.

The same definition works for a bd-pair $(X/Z,D+\sP)$ and
gives the threshold $\Rct(X/Z,D+\sP;F)$.

\begin{thm} \label{acc_R-complt}
Let $d$ be a nonnegative integer and
$\Gamma_b,\Gamma_f$ be two dcc sets of nonnegative real numbers.
Then the set of thresholds
$$
\{\Rct(X/Z,B;F)\mid \dim X=d,X/Z
\text{ has wFt}, B\in\Gamma_b \text{ and } F\in \Gamma_f\}
$$
satisfies the acc.

\end{thm}

\begin{add} \label{acc_R-complt_add}
Instead of $B\in\Gamma_b,F\in\Gamma_f$, we can suppose that
$B=\sum b_iE_i, F_i=\sum f_iE_i$,
where $E_i$ are effective Weil $\Z$-divisors and
$b_i\in\Gamma_b$ and $f_i\in\Gamma_f$.

\end{add}

\begin{add}
The same holds for
bd-pairs $(X/Z,B+\sP)$ of index $m$
and the set of corresponding thresholds depends also on $m$.

\end{add}

Actually, the set of $\R$-complement thresholds is a union of
two well-known sets of thresholds.

\begin{lemma} \label{lct_act}
Let $(X/Z,D)$ be a pair with wFt $X/Z$ and
$F$ be an effective and $\not=0$ divisor on $X$.
Then
$$
\Rct(X/Z,D;F)=\begin{cases}
\lct(X'\ni o',D';F')=\Rct(X'\ni o',D';F')\text{ or}\\
\act(X''/o'',D'';F'')=\Rct(X''/o'',D'';F'')
                   \end{cases},
$$
where
$\lct,\act$ are respectively {\em log canonical\/}
and {\em anticanonical\/} thresholds;
$X',X''$ are $\Q$-factorial of dimension $\le \dim X$,
$\rho(X'')=1$, $X'\ni o',X''$ are Ft and $o',o''$ are closed points;
the multiplicities of $D',D''$ and
of $F',F''$ are respectively multiplicities of $D$ and of $F$.
In particular, if $D\in \Gamma_d,F\in\Gamma_f$ then
$D',D''\in\Gamma_d,F',F''\in \Gamma_f$ respectively.

Conversely,
$\lct(X'\ni o',D';F')=\Rct(X'\ni o',D';F')$ and
$\act(X''/o'',D'';F'')=\Rct(X''/o'',D'';F'')=a$
if $(X'',D''+aF'')$ is lc.

The same holds for bd-pairs:
$$
\Rct(X/Z,D;F+\sP)=\begin{cases}
\lct(X'\ni o',D';F'+\sP')=\Rct(X'\ni o',D';F'+\sP')\text{ or}\\
\act(X''/o'',D'';F''+\sP'')=\Rct(X''/o'',D'';F''+\sP'').
                   \end{cases}
$$
Additionally, the bd-pairs
$(X'\ni o',D';F'+\sP'),(X''/o'',D'';F''+\sP'')$ have index $m$ if
$(X/Z,D+\sP)$ has index $m$.

\end{lemma}

\begin{proof}
Put $t=\Rct(X/Z,D;F)$.
We assume that $t\ge 0$ and
is well-defined.

Step~1. {\em We can suppose that $X$ is $\Q$-factorial
and Ft.\/}
Taking a $\Q$-factorialization $Y\to X$,
we reduce the proof to the $\Q$-factorial case
by Proposition~\ref{small_transform_compl}.

Below we assume that $X$ is $\Q$-factorial.
By Lemma~\ref{wTt_vs_Ft}
we can suppose also that $X/Z$ is
Ft, in particular, projective over $Z$.

Step~2. {\em We can suppose that $(X,D+tF)$ is klt over $\Supp F$,\/}
that is, the lc centers of $(X,B+tF)$ are not in $\Supp F$.
If $p\in X$ is a point such that
$(X,D+tF)$ is lc but not klt in $p$, and
$F$ passes through the point (i.e., $F>0$ near $p$),
then we have the {\em lc threshold\/} in $p$:
$\lct(X\ni p,D;F)=t=\Rct(X/Z,D;F)$.
Taking hyperplane sections we can suppose also
that $p$ is closed (cf. Step~6).

Warning: $p$ is not necessarily over $o$
but contains a point over $o$.

So, we can assume that $F$ only passes
klt points of $(X,D+tF)$ or
$(X,D+tF)$ is klt near $\Supp F$.

Step~3. {\em We can suppose that
$-(K+D+tF)$ is nef over $Z$.\/}
Otherwise, there exists an extremal contraction $X\to Y/Z$
which is positive with respect to $K+D+tF$.
The contraction is birational because
otherwise by Definition~\ref{r_comp}, (1) and (3)
$$
K+B^+-(K+D+tF)=B^+-D-tF\ge 0
$$
is numerically negative over $Y$,
a contradiction,
where $(X/Z,B^+)$ is an $\R$-complement of $(X/Z,D+tF)$.

If the contraction is small then
we make an antiflip.
The antiflip preserves the threshold $t$ by
Proposition~\ref{small_transform_compl}.
If after that $F$ passes a nonklt point
then we go to Step~2.

If the contraction is divisorial then
we contract a prime divisor $P$.
The contraction preserves the $\R$-complements and
$t=\Rct(X/Z,D;F)=\Rct(Y/Z,D_Y;F_Y)$, where
$D_Y,F_Y$ are images of respectively $D,F$ on $Y$.
Indeed, by definition $(Y/Z,D_Y+tF_Y)$ has
the $\R$-complement $(Y/Z,B_Y^+)$ with
the image $B_Y^+$ of $B^+$ on $Y$.
Thus $\Rct(Y/Z,D_Y;F_Y)\ge t$.
Actually, $=t$ holds by Proposition~\ref{monotonicity_I}
applied to $(X/Y,D+tF)$.
Indeed, for sufficiently small real number $\ep>0$,
$K+D+(t+\ep)F$ is negative over $Y$ and
$(X^\sharp/Y,(D+(t+\ep)F)^\sharp_{X^\sharp})=(Y/Y,D_Y+(t+\ep)F_Y)$.
Moreover, if $(Y/Z,D_Y+(t+\ep)F_Y)$ has
an $\R$-complement $(Y/Z,B')$.
Then it induces an $\R$-complement $(X/Z,B_X')$,
with crepant $B_X'$, of $(X/Z,D+(t+\ep)F)$, a contradiction.
By definition it is enough to verify that
$$
\mult_P B_X'\ge \mult_P(D+(t+\ep)F).
$$
This is equivalent to (1) of Definition~\ref{r_comp}
over $Y$ locally near the center (image) of $P$.
This holds by Proposition~\ref{monotonicity_I} or
Addendum~\ref{R_complement_criterion}.
In particular we prove that $\Supp F\not=P$ and
$F_Y\not= 0$.
If after the divisorial contraction $F$ passes a nonklt point
then we go again to Step~2.

So, by the termination \cite[Corollary~5.5]{ShCh},
after finitely many steps we
get nef $-(K+D+tF)/Z$.
The last divisor $-(K+D+tF)$ gives a contraction
$X\to Y/Z$.

Step~4. {\em $-F$ is not nef with respect to
the contraction $X/Y$.\/}
Since $X/Z$ has Ft, the cone of curves $\NE(X/Z)$
in $\Nc(X/Z)$
is closed convex rational polyhedral with finitely many extremal rays \cite[Corollary~4.5]{ShCh}.
Denote by $V$ its face generated by the extremal rays $R$
which are contracted on $Y$, that is, they have a curve $C/o$
with $(C.K+D+tF)=0$.
Actually this is a face by Step~3.
Moreover, the cone $V$ is also closed convex rational polyhedral.
If $-F$ is nef on the face $V$ then,
for every sufficiently small real number $\ep>0$,
$-(K+D+(t+\ep)F)$ is nef by Step~3 and
semiample over $Z$ \cite[Corollary~4.5]{ShCh}.
By Step~2 we can suppose also that $(X,D+(t+\ep)F)$ has
lc singularities.
Thus $t$ is not $\Rct(X/Z,D;F)$ by Addendum~\ref{R_complement_criterion},
a contradiction.

Step~5.
Moreover, {\em we can suppose that $X/Y=Z$ itself is
a fibered (extremal) contraction with $o$, the generic point of $Z$,
and $F$ is positive over $Z$.\/}
To establish this we apply $(-F)$-MMP to $X/Z$.
However, we consider only extremal rays $R$ of
the cone of curves $\NE(X/Z)$ which are
numerically trivial with respect to $K+D+tF$, that is, $R\subseteq V$.
By Step~4 there exists such an extremal ray $R$ negative for $-F$,
equivalently, $F$ is positive for $F$.

If $R$ gives a small birational contraction $X\to X'/Z$,
actually, over $Y$ then
we make a flip in $R$ and this preserves the threshold $t$.
If $R$ gives a divisorial contraction $X\to X'/Z$, again also over $Y$, then after the
contraction we have the same threshold
$$
\Rct(X/Z,D;F)=\Rct(X'/Z,D';F'),
$$
where $D',F'$ are birational transforms of $D,F$ respectively on $X'$.
In particular, $F$ is not exceptional for $X/X'$ and $F'\not =0$.
We can argue here as in Step~3.
However, $F'$ can pass lc singularities of $(X',D'+tF')$.
In this case we go again to Step~2.

By Step~4 and termination we get an extremal fibered contraction $X\to X'/Z$
for which $K+D+tF\equiv 0/X'$ and $F$ is numerically positive over $X'$.
By construction and since any contraction of Ft is Ft,
$X/X'\ni \eta$ has Ft, where $\eta$ is the generic point of $X'$ and
$\Rct(X/Z,D;F)=\Rct(X/\eta,D;F)=\act(X/\eta,D;F)$.
Finally, denote $X/\eta$ by $X/o$.

Step~6. {\em We can suppose that $Z=o$ is a closed point and $\rho(X/o)=1$.\/}
In terms of Italian understanding of the generic point
we can replace $X/o$ by $X_p/p$, where $p$ is
a sufficiently general closed point of $Z$ and
$X_p$ is the fiber of $X/Z$, and
$$
\Rct(X/Z,D;F)=\Rct(X/o,D;F)=
\Rct(X_p/p,D_p;F_p)=\act(X_p/p,D_p;F_p),
$$
where $D_p=D\rest{X_p},F_p=F\rest{X_p}$.
We loose only the extremal property of contraction $X\to X'/o$
but only when $o\not=p$.
By Proposition~\ref{bounded_rank} and Addendum~\ref{const_Cl}
and since $X/o$ has Ft,
$X_p$ is $\Q$-factorial for sufficiently general $p\in Z$
(specialization).
For $o\not= p$, $\dim X_p<\dim X$.
So, we use dimensional induction in this case and
go to Step~2.
Otherwise, $o=p$ is a closed point and we are done.

Notice that our algorithm preserves multiplicities of $D$ and of $F$.

The converse holds by definition.
Moreover, the converse holds for constructed pairs.

Similarly we can treat bd-pairs.

\end{proof}

\begin{proof}[Proof of Theorem~\ref{acc_R-complt}]
We reduce the acc to two acc's:
for lc and ac thresholds, local over $Z_i\ni o$.

Consider a sequence of pairs
\begin{equation} \label{sequence_pairs}
(X_i/Z_i\ni o,B_i), i=1,2,\dots,i,\dots
\end{equation}
with effective $\R$-divisors $F_i$ on $X_i$,
such that
\begin{description}

  \item[\rm (1)]
every $X_i/Z_i\ni o$ has wFt with $\dim X_i=d$;

  \item[\rm (2)]
every $B_i\in \Gamma_b,F_i\in \Gamma_f$ and $F_i>0$;

  \item[\rm (3)]
every $(X_i/Z_i\ni o,B_i)$ has an $\R$-complement;
and

  \item[\rm (4)]
the sequence of thresholds is nonnegative and monotonic
\begin{align*}
0&\le \Rct(X_1/Z_1\ni o,B_1;F_1)\le
\Rct(X_2/Z_2\ni o,B_2;F_2)\\
&\le\dots\le
\Rct(X_i/Z_i\ni o,B_i;F_i)\le\dots.
\end{align*}

\end{description}
For simplicity of notation we use same $o$
everywhere instead of $o_i$ (as same $0$
for every field).
We need to verify the stabilization: for every $i\gg 0$,
$$
\Rct(X_i/Z_i\ni o,B_i;F_i)=
\Rct(X_{i+1}/Z_{i+1}\ni o,B_{i+1};F_{i+1}).
$$

Step~1. {\em The case when every $\Supp B_i,\Supp F_i$ have
at most $l$ prime divisors.\/}
Equivalently, $\Supp B_i\cup\Supp F_i$ has at most $l$
prime divisors,
possibly for a different natural number $l$.
In other words, for every $i=1,2,\dots,i,\dots$,
there is a $\Z$-linear map
$A_i\colon \R^l\to\WDiv_\R X_i$  of height $1$ and
vectors $x_i,y_i\in\R^l$ such that
$A_i(x_i)=B_i,A_i(y_i)=F_i$.

Indeed, there exist distinct prime divisors $D_{i,j},j=1,\dots,l$, on $X_i$
such that
$$
B_i=\sum_{j=1}^l b_{i,j}D_{i,j}
\text{ and }
F_i=\sum_{j=1}^l f_{i,j}D_{i,j}.
$$
We put
$$
A_i(1,0,\dots,0)=D_{i,1},
A_i(0,1,\dots,0)=D_{i,2},\dots,
A_i(0,0,\dots,1)=D_{i,l}.
$$
Hence $x_i=(b_{i,1},b_{i,2},\dots,b_{i,l})$ and
$y_i=(f_{i,1},f_{i,2},\dots,f_{i,l})$,
where by~(2) every $b_{i,j}\in\Gamma_b$ and $f_{i,j}\in\Gamma_f$.

Since $\Gamma_g,\Gamma_f$ satisfy the dcc,
taking a subsequence of (\ref{sequence_pairs}) we can suppose that
$0\le x_1\le x_2\le\dots\le x_i\le \dots$ and
$0<y_1\le y_2\le\dots\le y_i\le \dots$,
where
$\le$ for vectors as for divisors:
$(v_1,\dots,v_l)\le (w_1,\dots,w_l)\in\R^l$
if every $v_i\le w_i$.

Put $x=\lim_{i\to\infty}x_i,y=\lim_{i\to \infty} y_i$ and
$$
t=\lim_{i\to\infty}t_i,\ \ t_i=\Rct(X_i/Z_i\ni o,B_i;F_i).
$$
The limits $x,x+ty$ and $t$ are proper ($<+\infty$): $x,x+ty\in\R^l$ and $t\in \R$.
(Thus the limit $y$ is proper if $x+ty$ is proper and $t\not=0$.)
Indeed, for $x$, this follows from the subboundary
property: every $x_i\le (1,1,\dots,1)\in\R^l$ or
$B_i$ is a subboundary by (3).
The same works for $x+ty$ because
every $x_i+t_iy_i\le (1,1,\dots,1)$ or
$B_i+t_iF_i$ is a subboundary
by the definition of $\R$-complements and~(3) again.
By (2-3) and construction every $B_i$ and $B_i+t_iF_i$ are boundaries.
On the other hand, since $F_i\not=0$,
$$
t_i\le 1/\min\{\gamma\mid\gamma\in \Gamma_f
\text{ and }
\gamma>0\}.
$$
The minimum exists and is positive by the dcc of $\Gamma_f$.
Thus $t\in\R$.
Note that $y$ and every $y_i$ are $>(0,0,\dots,0)\in\R^l$ by (2),
in particular, $\not=0\in\R^l$.

Now it is enough to verify that $(X_i/Z_i\ni o,A_i(x+ty))$
has an $\R$-complement for every $i\gg 0$.
Indeed, by construction,
$A_i(x+ty)=A_i(x)+tA_i(y)\ge
A_i(x_i)+t_iA_i(y_i)=B_i+t_iF_i$ by the monotonic property of every $A_i$.
So, by definition
$A_i(x+ty)=A_i(x)+tA_i(y)=A_i(x_i)+t_iA_i(y_i)=A_i(x_i+t_iy_i)$
and $x+ty=x_i+t_iy_i$ for those $i$
because every $A_i$ is injective.
Thus $t_i=t$ for those $i$ too because every $y\ge y_i>0$.
This is the required stabilization.

The $\R$-complement property follows from~(1) and
Theorem~\ref{invers_stability_R_complements}
under the assumption~(3) of Theorem~\ref{n_comp}.
Indeed, $\Gamma=\{b_{i,j}+t_if_{i,j}\}$ satisfies the dcc
with the only limiting points $x_1,\dots,x_l$.
On the other hand, by construction, definition and (3),
every pair $(X_i/Z_i\ni o,B_i+t_iF_i)$ has an $\R$-complement.
For every positive real number $\ep$, the estimation
$$
\norm{B_i+t_iF_i-A_i(x+ty)}\le
l\norm{x_i+t_iy_i-x-ty}< \ep
$$
for every $i\gg 0$ concludes the proof in Step~1.

Below we reduce the general case to the situation of Step~1.
By Lemma~\ref{lct_act} we need to consider
two local over $X_i'\ni o$ or $X_i''/o$ cases.
Indeed, by the lemma every
$$
t_i=
\begin{cases}
\lct(X_i'\ni o,B_i';F_i')=\Rct(X_i'\ni o,B_i';F_i'),\text{ or}\\
\act(X_i''/o,B_i'';F_i'')=\Rct(X_i''/o,B_i'';F_i''),
\end{cases}
$$
where
$X_i',X_i''$ are $\Q$-factorial of dimension $\le \dim X_i=d$,
$\rho(X_i'')=1$, $X_i'\ni o,X_i''$ are Ft and all points $o$ are closed;
$B_i',B_i''\in\Gamma_b,F_i',F_i''\in \Gamma_f$, and
$F_i',F_i''\not=0$.

By dimensional induction and the converse
we can suppose that every $\dim X_i'=X_i''=d$.

Step~2. ({\em lct\/})
Suppose that there exists infinitely many lct cases with
$t_i=\lct(X_i'\ni o,B_i';F_i')=\Rct(X_i'\ni o,B_i';F_i')$.
Taking a subsequence of (\ref{sequence_pairs})
and changing notation we can suppose that
every $t_i=\Rct(X_i'\ni o,B_i';F_i')=\lct(X_i'\ni o,B_i';F_i')$,
where $X_i'$ is $\Q$-factorial.
By \cite[Theorem~18.22]{K}
$\Supp B_i,\Supp F_i$ have not more than
$d/\mu_b,d/c\mu_f$ prime divisors respectively, where
$$
\mu_b=\min\{\gamma\in\Gamma_b\mid \gamma>0\},
\mu_f=\min\{\gamma\in\Gamma_f\mid \gamma>0\}
\text{ and } c=\min\{t_i\},
$$
assuming that all $t_i>0$ (see Warning below).
Minima exist and are positive by the dcc of $\Gamma_b,\Gamma_f$ and
by the monotonic property of $t_i$ (assuming that
all $t_i>0$).
So, by Step~1 the required stabilization of $t_i$ holds.

Warning: our estimations depend on $c$.
Either we can consider a fix sequence of thresholds $t_i$
or consider the (truncated) thresholds $\ge c>0$.
This is sufficient to prove the require acc.

Step~3. ({\em act\/})
Now we suppose that there exists infinitely many act cases with
$t_i=\act(X_i''\ni o,B_i'';F_i'')=\Rct(X_i''\ni o,B_i'';F_i'')$.
As in Step~2, we can suppose that
every $t_i=\Rct(X_i''\ni o,B_i'';F_i'')=\act(X_i''/o,B_i'';F_i'')$,
where $X_i''$ is $\Q$-factorial Ft over $o$ and $\rho(X_i)=1$.
By \cite[Corollary~1.3]{BMSZ}
$\Supp B_i,\Supp F_i$ have not more than
$(d+1)/\mu_b,(d+1)/c\mu_f$ prime divisors respectively, where
$\mu_b,\mu_f,c$ are same as in Step~2.
Again by Step~1 the required stabilization of $t_i$ holds.

Step~4. {\em Addenda.\/}
Addendum~\ref{acc_R-complt_add} is
immediate by the case with only the prime divisors $E_i$ and
the dcc property of $\{\sum \gamma_i\mid \gamma_i\in\Gamma\}\subset [0,+\infty)$
for every dcc $\Gamma\subset [0,+\infty)$.

Similarly we can treat bd-pairs.

\end{proof}

\begin{cor}[Lc thresholds {\cite[Theorem~1.1]{HMX}}] \label{acc_lc_thresholds}
Let $d$ be a nonnegative integer and
$\Gamma_b,\Gamma_f$ be two dcc sets of nonnegative real numbers.
Then the set of lc thresholds
$$
\{\lct(X/X,B;F)\mid \dim X=d, B\in\Gamma_b \text{ and } F\in \Gamma_f\}
$$
satisfies the acc.

\end{cor}

\begin{add}
Instead of $B\in\Gamma_b,F\in\Gamma_f$, we can suppose that
$B=\sum b_iE_i, F_i=\sum f_iE_i$,
where $E_i$ are effective Weil $\Z$-divisors and
$b_i\in\Gamma_b,f_i\in\Gamma_f$.

\end{add}

\begin{add}
The same holds for
bd-pairs $(X/X,B+\sP)$ of index $m$
and the set of corresponding thresholds depends also on $m$.

\end{add}

\begin{proof}
Immediate by Theorem~\ref{acc_R-complt}
for klt $(X,B)$.
Indeed, in this case $X/X$ or locally
$X$ has wFt.
By definition of the lc threshold we suppose that
$(X,B)$ is lc and $F$ is $>0$, $\R$-Cartier.
Since every lc $(X/X,B+tF)$ is a $0$-pair,
$\lct(X/X,B;F)=\Rct(X/Z,B;F)$ by Example~\ref{1stexe}, (3).

If $(X,B)$ is lc and $F$ passes an lc center of $(X,B)$ then
$\lct(X/X,B;F)=\Rct(X/X,B;F)=0$.
Otherwise, $(X,B)$ is lc and $F$ does not pass the lc centers of $(X,B)$
and we can replace $(X,B)$ by its dlt resolution and
$F$ by its pull-back on the resolution.
The construction preserves the lc threshold.
The pull-back of $F$ does not blow up divisors,
is the birational transform of $F$ with the same
multiplicities as $F$.
The new $X$ is klt and we can apply Theorem~\ref{acc_R-complt} again.

Similarly we can treat bd-pairs.

\end{proof}

\begin{cor}[Ac thresholds] \label{acc_ac_free}
Let $d$ be a nonnegative integer and
$\Gamma_b,\Gamma_f$ be two dcc sets of nonnegative real numbers.
Then the set of ac thresholds
\begin{align*}
\{\act(X/Z,B;F)\mid& \dim X=d,X/Z \text{ has wFt},
(X,B)\text{ is lc},
B\in\Gamma_b \text{ and }  \\
&F=\sum f_i F_i,\text{ where }
F_i\text{ are locally free over }Z \text{ and }f_i\in\Gamma_f\}
\end{align*}
satisfies the acc.

\end{cor}

\begin{add}
The same holds for
bd-pairs $(X/Z,B+\sP)$ of index $m$
and the set of corresponding thresholds depends also on $m$.

\end{add}

\begin{rem} \label{act_rem}
In general, the {\em anticanonical\/} threshold of
a log pair $(X/Z,D)$ with respect to a nef over $Z$ divisor $H$ on $X$ is
$$
\act(X/Z,D;H)=\inf\{t\in\R\mid K+D+tH\text{ is nef over }Z\}.
$$
In \cite[p.~47]{ISh}
it was used for Fano varieties as one of invariants
in the Sarkisov program.
In this situation Fano varieties $X$ are $\Q$-factorial,
have the Picard number $1$, terminal singularities and
$H$ is an effective nonzero (Weil) divisor.
Thus $\act(X,0;H)$ is $t$ of~(\ref{equiv_tF}) in the proof below
with $D=0,F=H$.
Moreover, such thresholds $\act(X,0;H)$ form an acc set
with a single accumulation point $0$ if
$\dim X=d$ is fixed (cf. Corollary~\ref{acc_Fano_index}).
This follows from boundedness of those Fano varieties \cite[Theorem~1.1]{B16}.
In general it is not true (cf. Remark~\ref{acc_Fano_index_rem}, (1) below).
\end{rem}

\begin{exa} \label{act_exa}
Let $S$ be a cone over a rational normal curve of degree $n$ and
$L$ be its generator.
Then
$$
\act(S,0;L)=n+2.
$$
In this situation $L$ is not free for $n\ge 2$ and
pairs $(S,L)$ form an unbounded family.
Notice also that $\Rct(S,0;L)=1$ and $\not=\act(S,0;L)$
for $n\ge 2$.
This is why we consider only $\act=\Rct$.
But $\lct=\Rct$ holds automatically (cf. Lemma~\ref{lct_act}).

\end{exa}

\begin{proof}[Proof of Corollary~\ref{acc_ac_free}]
Recall that by definition of
$t=\act(X/Z,D;F)$ we suppose that
$(X,D)$ is a log pair,
$F\not\equiv 0$ generically over $Z$ and
\begin{equation}\label{equiv_tF}
-(K+D)\equiv tF/Z.
\end{equation}
The proof uses Theorem~\ref{acc_R-complt}.

Step~1. {\em It is enough to consider
a local case $\act(X/Z\ni o,B;F)$
with a closed point $o\in Z$.\/}
Indeed, for any closed point $o\in Z$,
$$
\act(X/Z,B;F)=\act(X/Z\ni o,B;F)
$$
because $F\not\equiv 0$ generically over $Z$.
We omit locally the divisors $F_i$ with $f_i=0$ and suppose that
every $f_i\not=0$.
Hence every $f_i\ge \mu_f>0$, where
$$
\mu_f=\min\{\gamma\in\Gamma_f\mid \gamma>0\}.
$$
We can consider truncated ac thresholds: $t\ge  c$
for some positive real number $c$.

Step~2. {\em We can suppose that $t$ and every $f_i\le 1$.\/}
Notice for this that $t$ is bounded: $t\le (d+1)/\mu_f$.
Indeed, since $-(K+B)\equiv tF/Z\ni o$ and $t\ge c>0$,
$X$ is covered by curves $C/Z\ni o$ with
$-(C.K)\le d+1$ and $-(C.K+B)\le d+1$ for
such a sufficiently general curve $C/o$.
Additionally, we can assume that $-(C.K+B)>0$.
On the other hand, every $(C.F_i)\ge 0$ and
by~(\ref{equiv_tF}) some $(C.F_i)\ge 1$ because
$F_i$ are free over $Z\ni o$.
Hence, for $(C.F_i)\ge 1$,
$$
t\mu_f\le tf_i\le (C.tf_iF_i)\le (C.tF)=-(C.K+B)\le d+1.
$$
This gives the required bound.

Similarly, we can verify that multiplicities $f_i$
are bounded: every $f_i\le (d+1)/c$.

We replace every $f_iF_i$
by
$$
D_i=\{f_i\}F_{i,0}+\sum_{j=1}^{\rddown{f_i}}F_{i,j},
$$
where $\{f_i\}$ is the fractional part of $f_i$,
belongs to $[0,1)$ and every $F_{i,j}\sim F_i$ over $Z\ni o$.
By definition
$$
t=\act(X/Z\ni o,B;F)=\act(X/Z\ni o,B;D),
$$
where $D=\sum D_i$.
We denote below $D$ by $F$.
Now $F$ has additional multiplicities $1,\{f_i\}$
which satisfy the dcc because $f_i$ are bounded.
Thus we can add them to $\Gamma_f$ and
suppose that every $f_i\le 1$.

If $t\ge 1$ we replace $tF$ by $\{t\}F$ and
$B$ by $B'=B+\rddown{t}F$, where $\{t\}$ is
the fractional part of $t$.
The log pair $(X,B')$ is lc with a boundary $B'$ by Bertini
if every copy of ($\rddown{t}$ copies of) $F_i$ in
$\rddown{t}F$ is sufficiently general effective
divisor in its linear system over $Z\ni o$.
We need to add $1$ to $\Gamma_b$ if it is necessary.
Here we loos the assumption that $t\ge c$.
By construction
$$
\{t\}=\act(X/Z\ni o,B';F).
$$

It is enough to verify the acc for $\{t\}$
because $t$ is bounded.
Below we denote $B'$ by $B$ and $\{t\}$ by $t$.

Step~3. {\em The acc for $t\le 1$ holds.\/}
Since every $f_i\le 1$ and $t\le 1$,
$(X,B+tF)$ is lc again by Bertini if
$F_i$ in are sufficiently general in their
linear system over $Z\ni o$.
Hence by definition and since $F>0$,
$$
\act(X/Z\ni o,B;F)=\Rct(X/Z\ni o,B;F)
$$
holds and
the corollary holds by Theorem~\ref{acc_R-complt}.

Similarly we can treat bd-pairs.

\end{proof}

The next result (with $B=0$) was conjectured by the author in late 80's in
relation to the acc of mld's.
It was suggested as a problem to V.~Alexeev who
was a graduate student of V.~Iskovskikh in that time
(cf. Remark~\ref{acc_Fano_index_rem}, (1) below).

\begin{cor}[Acc for Fano indices] \label{acc_Fano_index}
Let $d$ be a nonnegative integer and $\Gamma_b$ be
a dcc set of nonnegative real numbers.
Then the set of Fano indices
\begin{align*}
\{0\le h\in\R\mid& -(K+B)\equiv hH\text{, where }
\dim X=d, X\text{ has (w)Ft, } \\
&(X,B)\text{ is lc, }
B\in\Gamma_b \text{ and }
H\text{ is a primitive ample divisor on } X\}
\end{align*}
satisfies the acc.

\end{cor}

Actually, $X$ has Ft in the theorem,
$(X,B)$ in the definition of $h$ is an {\em lc log Fano\/} variety,
if $h>0$,
and $h$ in this case is its {\em Fano index\/}.

\begin{add} \label{acc_Fano_index_H_i}
Let $\Gamma_h$ be a dcc set of nonnegative real numbers.
Instead of $B\in\Gamma_b,H$, we can take
$B=\sum b_iE_i,H=\sum h_i H_i\not\equiv 0$, where
$E_i$ are effective Weil $\Z$-divisors,
$H_i$ are nef Cartier divisors and
$b_i\in\Gamma_b,h_i\in\Gamma_h$.

\end{add}

\begin{add}
The same holds for the lc Fano
bd-pairs $(X,B+\sP)$ of index $m$
with $-(K+B+\sP_X)\equiv hH$
and the set of corresponding Fano indices depends also on $m$.

\end{add}

\begin{proof}
Primitive (ample) means that if $H\equiv hH'$,
where $H'$ is also (ample) divisor, then $h\ge 1$.

There exists a positive integer $N$ such that
$F=NH$ is free for every (w)Ft $X$ of dimension $d$
\cite[Theorem~1.1]{K93} (cf. Corollary~\ref{free} above).
Hence
$$
h/N=\act(X,B;F).
$$
Then Corollary~\ref{acc_ac_free} implies the acc for $h$.

In Addendum~\ref{acc_Fano_index_H_i}
we can use freeness of $NH_i$ with $N$ depending only on
the dimension $d$ by Corollary~\ref{free}.

Similarly we can treat bd-pairs.

\end{proof}

\begin{rem} \label{acc_Fano_index_rem}
(1) For every positive real number $\ep$,
the Fano indices for $\ep$-lc log Fano
pairs $(X,B)$ of the corollary with a finite set $\Gamma_b$ form
a finite set by BBAB \cite[Theorem~1.1]{B16}, that is,
there are only finitely many of those indices.
This was conjectured by Alexeev for $B=0$.

However, the union of this finite sets
for all $\ep$ gives the acc set of
the theorem.
In general, the finiteness does not imply the acc.

(2) A Fano index is not always defined
even for log Fano varieties $(X,B)$.
However, it is defined for such a pair
if $B$ is a $\Q$-boundary, e.g.,
$\Gamma_b\subset\Q$.

\end{rem}

\paragraph{$\ainv$-Invariant.\/}
Let $(X,B)$ be a log Fano variety.
Then the $\ainv$-invariant of $(X,B)$ is
$$
\ainv=\ainv(X,B)=(1-\glct(X,B))h,
$$
where $\glct(X,B)$ is the global lc threshold or
$\alpha$-invariant of $(X,B)$ and $h$ is
the Fano index of $(X,B)$.
Since $h>0$, $\ainv\ge 0$ if and only if
$\glct(X,B)\le 1$.

Notice that $\ainv$ is not always defined
but defined for $\Q$-boundaries $B$
(cf. Remark~\ref{acc_Fano_index_rem}, (2) and Addendum~\ref{acc_a_inv_Q} below).

The same definition works for a bd-pair $(X/Z,D+\sP)$
with
$$
\glct(X,B+\sP)=\sup\{t\in\R\mid
(X,B+E+\sP) \text{ is lc for all }
 0\le E\equiv -t(K+B+\sP_X)\}
$$
and $-(K+B+\sP_X)=hH$, where
$H$ is a primitive ample divisor on $X$.

Below we discuss some properties of $\ainv$-invariant and
of $\glct$.

\begin{cor}[Acc for $\ainv$-invariant] \label{acc_a_inv}
Let $d$ be a nonnegative integer and $\Gamma_b$ be
a dcc set of nonnegative real numbers.
Then the set of $\ainv$-invariants
$$
\{0\le \ainv(X,B)\mid (X,B) \text{ is an lc log Fano variety, }
\dim X=d, X \text{ has Ft and }
B\in\Gamma_b\}
$$
satisfies the acc.

\end{cor}

\begin{add} \label{acc_a_inv_H}
Let $\Gamma_h$ be a dcc set of nonnegative real numbers.
Instead of $B\in\Gamma_b$ and ample $H$ in the definition of
$\ainv$-invariant, we can take
$B=\sum b_iE_i$ and $H=\sum h_i H_i$, where
$E_i$ are effective Weil $\Z$-divisors,
$H_i$ are nef Cartier divisors and
$b_i\in\Gamma_b,h_i\in\Gamma_h$.

\end{add}

\begin{add} \label{acc_a_inv_Q}
If additionally $\Gamma_b,\Gamma_h\subset \Q$ then
$H$ always exists and $\ainv(X,B), \glct(X,B)\in\Q$.
Moreover, if $\Gamma_b,\Gamma_h$ are closed $\Q$,
then the sets of $\ainv$-invariants is also closed in $\Q$ limits.

\end{add}

(Cf. other statements about  limits of thresholds in
Corollary~\ref{constrain} and its addenda.)

\begin{add} \label{acc_a_inv_bd}
The same holds for the lc Fano
bd-pairs $(X,B+\sP)$ of index $m$
with $-(K+B+\sP_X)\equiv hH$
and the set of corresponding $\ainv$-invariants $\ge 0$ depends also on $m$.

\end{add}

\begin{proof}
We can replace $H$ by free $F=NH$, where
$N$ is a positive integer
depending only on $d$.
It is enough to verify that $t=\ainv/N$ satisfies the acc.
Since $\ainv$ is bounded ($\le d+1$), we
can suppose that $t<1$ for appropriate $N$.

The acc is enough to verify for $\ainv>0$, equivalently, $t>0$.
In this situation $t$ is also a threshold:
\begin{align*}
t=&\sup\{0\le r\in\R\mid
K+B+rF+E\equiv 0
\text{ and }\\
&(X,B+E)
\text{ is lc but nonklt pair for some
effective $\R$-divisor } E \text{ on } X\}.
\end{align*}
This threshold can be converted into an $\Rct$ one.
In particular, it is attained, that is, the supremum can
be replaced by the maximum (cf. Corollary~\ref{a_attained} below).
For sufficiently general effective divisor $M\sim F$,
$M$ does not pass the prime components of $\Supp B,\Supp E$,
the lc centers of $(X,B+E)$ and
$(X,B+rM+E)$ is lc.
Take a prime b-divisor $P$ of $X$ such that
$\ld(P;X,B+rM+E)=\ld(P;X,B+E)= 0$,
log discrepancies at $P$.
Let $(Y,(B+rM+E)_Y)$ be a crepant
blowup of $P$.
The blowup is an isomorphism $X=Y$ if
$P$ is not exceptional on $X$.
For exceptional $P$, the crepant transform
$(B+rM+E)_P$ of
$B+rM+E$ on $Y$ is equal to
$B+rM+E+P$, where
$B+rM+E$ denotes also
its birational transform on $Y$.
Then by construction in the exceptional case
$$
r\le
\Rct(Y,B+P;M)
\le t.
$$
In $\Rct(Y,B+P;M)$, the divisor $M$ is the birational transform
of $M$ on $Y$ and is free again.
Note also that every $Y$ in the construction has Ft
\cite[Lemma-Definition~2.6, (iii)]{PSh08}.
Replacing $r$ by $\Rct(Y,B+P;M)$,
we can suppose that
$$
r=\Rct(Y,B+P;M).
$$
Now we can use Theorem~\ref{invers_stability_R_complements}
with assumption~(3) of Theorem~\ref{bndc} or
Theorem~\ref{acc_R-complt} with Addendum~\ref{acc_R-complt_add} and
get $t$ in terms of
an $\R$-complement threshold: for $r$ sufficiently close to $t$,
$$
t=r=\Rct(Y,B+P;M).
$$
Indeed, $M=1M$ and
$$
B+P,1\in\Gamma_b\cup\{1\}
$$
and it is a dcc set.
So, Theorem~\ref{acc_R-complt} with Addendum~\ref{acc_R-complt_add}
again implies the acc for
$t$ and $\ainv$.

Similarly, if $P$ is not exceptional then $Y=X$ and
$(B+rM+E)_Y=B+rM+E$.
By construction
$\mult_P(B+rM+E)=\mult_P(B+E)=1$.
In this case
$$
r\le
\Rct(X,B+pP;M)
\le t,
$$
where $p=\mult_PE$.
(So, $\mult_P(B+pE)=1,E'=E-pP\ge 0$ and
$B+pP+E'=B+E$.)
Replacing $r$ by $\Rct(X,B+pP;M)$
we can suppose that
$$
r=\Rct(X,B+pP;M).
$$
Again we can use Theorem~\ref{invers_stability_R_complements} or
Theorem~\ref{acc_R-complt} with Addendum~\ref{acc_R-complt_add} and
get $t$ in terms of
an $\R$-complement threshold: for $r$ sufficiently close to $t$,
$$
t=r=\Rct(X,B+pP;M).
$$
Indeed, $M=1M$,
$$
B+pP,1\in\Gamma_b\cup\{1\}
$$
and it is a dcc set.
As above Theorem~\ref{acc_R-complt} with Addendum~\ref{acc_R-complt_add}
implies the acc for $t$ and $a$.

We prove more: every $\ainv>0,\glct<1$ invariants are
attained. See explanations in Corollary~\ref{a_attained} below.
We also established that $\ainv\in\Q$ for $B\in\Q$ because
$\ainv=\Rct(Y,B+P;M)$ or $=\Rct(X,B+pP;P)$ is rational
by the Theorem~\ref{R_compl_polyhedral} and $B+pP\in\Q$ too.
This proves rationality in Addendum~\ref{acc_a_inv_Q}.
The closed rational property follows from
the similar result for $\Rct$ thresholds
(see Corollary~\ref{constrain} below).

Similarly we can treat $\ainv$-invariants with $H$ as in Addendum~\ref{acc_a_inv_H}
(cf. the proof of Corollary~\ref{a_attained})
and bd-pairs.

\end{proof}

\begin{cor} \label{a_attained}
Let $(X,B)$ be an lc log Fano variety with a boundary $B$.
Then every threshold $\glct(X,B)\le 1$ is
attained, that is, there exists an effective $\R$-divisor
$E$ such that $(X,B+E)$ is lc but not klt and
$E\equiv-\glct(X,B)(K+B)$.

The same holds for $\glct(X,B+\sP)\le 1$
of lc Fano bd-pairs $(X,B+\sP)$ of index $m$.

\end{cor}

Notice that $\ainv(X,B),\ainv(X,B+\sP)\ge 0$ are also attained as it was established
in the proof of Corollary~\ref{acc_a_inv}.

\begin{proof}
We can suppose that $(X,B)$ is a klt log Fano variety and
has Ft.
Otherwise, $(X,B)$ is not klt and $\glct(X,B)=0$ and
$E=0$.
The cone of semiample divisors on $X$ is rational polyhedral.
In particular,
$$
-(K+B)\equiv \sum_{i=1}^l r_i H_i,
$$
where $r_i$ are positive real numbers and
$H_i$ are very ample (Cartier) divisors.
We can suppose also that every $r_i\le 1$.

We start from the case $t=\glct(X,B)<1$.
Then by definition there exists
a real number $r>0$ and effective $\R$-divisor $E$
such that $(X,B+E)$ is lc but not klt and
$$
K+B+E+r\sum_{i=1}^lr_iH_i\equiv 0.
$$
Those $r\le a<1$ and have a tendency to $a$,
where $a=1-t>0$ and
$$
E\equiv (1-r)(-K-B)\equiv (1-r)\sum_{i=1}^lr_iH_i.
$$
(Remark that this $a$ is not the $\ainv$-invariant
but it is its nonintegral and possibly irrational
but more anti log canonical version.)
For given $r,E$,
take sufficiently general effective divisors
$F_i$ on $X$ such that every $F_i\sim H_i$
does not pass
the prime component of $\Supp B,\Supp E,F_j,j\not=i$,
the lc centers of $(X,B+E)$ and
$(X,B+E+\sum_{i=1}^l F_l)$ is lc.
By construction $r<1$ and
$(X,B+E+r\sum_{i=1}^lr_iF_i)$ is an lc but not klt $0$-pair.
As in the proof of Corollary~\ref{acc_a_inv}
take a prime b-divisor $P$ such that
($P\not=F_1,\dots,F_l$ and)
$\ld(P;X,B+E+r\sum_{i=1}^lr_iF_i)=0$.
If $P$ is exceptional then
$$
r\le\Rct(Y,B+P;\sum_{i=1}^lr_iF_i)\le a
$$
where $Y\to X$ is blowup of $P$ and
$B,F_i$ are respectively birational transforms of $B,F_i$ on $Y$.
Actually, we can suppose that (for given $Y$ and $P$ but
possibly different $r,E$)
$$
r=\Rct(Y,B+P;\sum_{i=1}^lr_iF_i).
$$
We can take $r$ arbitrary close to $a$.
Hence by Theorem~\ref{invers_stability_R_complements}
with assumption~(3) of Theorem~\ref{bndc} or
Theorem~\ref{acc_R-complt} with Addendum~\ref{acc_R-complt_add}
we get $a$ in terms of
an $\R$-complement threshold: for $r$ sufficiently close to $a$,
$$
a=r=\Rct(Y,B+P;\sum_{i=1}^lr_iF_i).
$$
Indeed, there exists a finite subset $\Gamma$ in $[0,1]$ such that
$B\in\Gamma$ and
$$
B+P\text{ and every }r_i\in\Gamma\cup\{1,r_1,\dots,r_l\}
$$
and it is a dcc set.
So, $t$ and $a$ are attained because
the $\Rct$ thresholds are attained on Ft $X$
by the closed property in Theorem~\ref{R_compl_polyhedral}.

Similarly, if $P$ is not exceptional then $Y=X$ and
$\mult_P(B+E+r\sum_{i=1}^lr_iF_i)=1$, actually,
$\mult_P(B+E)=1.$
In this case
$$
r\le
\Rct(X,B+pP;\sum_{i=1}^lr_iF_i)
\le a,
$$
where $p=\mult_PE$.
(So, $\mult_P(B+pE)=1,E'=E-pP\ge 0$ and
$B+pP+E'=B+E$.)
Replacing $r$ by $\Rct(X,B+pP;\sum_{i=1}^lr_iF_i)$
we can suppose that
$$
r=\Rct(X,B+pP;\sum_{i=1}^lr_iF_i).
$$
Again we can use Theorem~\ref{invers_stability_R_complements}
with assumption~(3) of Theorem~\ref{bndc} or
Theorem~\ref{acc_R-complt} with Addendum~\ref{acc_R-complt_add}
and get $a$ in terms of
an $\R$-complement threshold: for $r$ sufficiently close to $a$,
$$
a=r=\Rct(X,B+pP;\sum_{i=1}^lr_iF_i).
$$
Indeed,
$$
B+pP\text{ and every }r_i\in\Gamma\cup\{1,r_1,\dots,r_l\}
$$
and it is a dcc set.
As above Theorem~\ref{acc_R-complt} with Addendum~\ref{acc_R-complt_add}
implies the attainment of
$t$ and $a$.

The case with $t=\glct(X,B)=1$ and $a(X,B)=0$ is
more delicate because in this case $r<a=0$ are
negative.
In this case, we use $r=0$,
effective $E'\equiv-(K+B)$ and a prime b-divisor $P$ with
the log discrepancy
$$
\ld(P;X,B+E')=\ep>0
$$
very close to $0$.
Such $E'$ can be constructed by normalization of $E$
in the definition: put
$$
E'=\frac1s E,
$$
where $(X,B+E)$ is lc but not klt and
$E\equiv -s(K+B)$.
By construction $s\ge 1$.
By \cite[Theorem~1.1]{B16}
if $s$ goes to $1$ then $\ep$ goes to $0$.
(In other words, if $(X,B)$ is $\ep$-lc
then $\glct(X,B)\ge\delta>0$.
This case is easier than general BBAB because here
$X$ is fixed! But $B$ is not fixed!)

Then we apply to a sequence of $\ep_i,E_i$ with
$\lim_{i\to\infty}\ep_i=0$
Theorem~\ref{acc_R-complt} with Addendum~\ref{acc_R-complt_add} or
Theorem~\ref{invers_stability_R_complements}
with assumption~(1) of Theorem~\ref{bndc} , that is,
for $\ep$ sufficiently close to $0$.
For exceptional prime b-divisors $P_i,P$,
in both cases we replace $P_i,P$ by $(1-\ep_i)P,(1-\ep)P$
respectively.
In the nonexceptional case we replace $B+p_iP_i,B+pP$ by
$B+(p_i-\ep_i)P,B+(p-\ep)P$ respectively.

Similarly we can treat bd-pairs.

\end{proof}

Remark: In general $\glct(X,B)$ behaves badly, e.g.,
does not satisfies the acc or dcc even if $\dim X=d$
is fixed and $B\in \Gamma_b$, a dcc set \cite{Sh06}.

However, $\glct(X,B)$ satisfies certain interesting
properties. Some of them were conjectured by G.~Tian.

\begin{cor}[glct gap]
Let $d$ be a nonnegative integer and $\Gamma_b$ be
a dcc set of nonnegative real numbers.
There exists a positive real number $g$
such that
if $(X,B)$ is an lc log Fano variety with $\dim X=d$,
$B\in\Gamma_b$ and $\glct(X,B)>1$
then $\glct(X,B)\ge 1+g$.

The same holds for $\glct(X,B+\sP)>1$
of lc Fano bd-pairs $(X,B+\sP)$ of index $m$
with $g$ also depending on $m$.

\end{cor}

\begin{proof}
Similar to the proof of Corollary~\ref{a_attained}
in the case $t=\glct(X,B)=1$.
Actually, we need to prove that if $t\ge 1$ and
sufficiently close to $1$ then $t=1$.
However, in this situation we need BBAB of
the full strength.

\end{proof}

\begin{cor}
Let $(X,B)$ be an lc log Fano variety with
a rational boundary $B$.
Then $\ainv(X,B)\ge 0,\glct(X,B)\le 1$ are also rational.

The same holds for $\ainv(X,B+\sP)\ge 0,\glct(X,B+\sP)\le 1$
of lc Fano bd-pairs $(X,B+\sP)$ of index $m$.

\end{cor}

Examples of nonrational $\ainv(X,B)< 0,\glct(X,B)>1$
with a rational boundary $B$ are unknown.

\begin{proof}
Immediate by Addendum~\ref{acc_a_inv_Q}
and Addendum~\ref{acc_a_inv_bd}.

\end{proof}

But we have a more effective statement.

\begin{cor}[Effective attainment]
Let $\ainv\ge 0$ be a rational number,
$d$ be a nonnegative integer and $\Gamma_b$ be
a rational closed dcc set in $[0,1]$.
Then there exists a positive integer $n=n(d,\Gamma_b,\ainv)$
such that every $\ainv(X,B)=\ainv$, equivalently,
$\glct(X,B)$, with $\dim X=d$, is
attained by a divisor $E=D/n\equiv-\glct(X,B)(K+B)$, where
$$
D\in \linsys{-nK-nB-naH}
$$
and $H$ is a (primitive) ample divisor on $X$ such
that $-(K+B)\equiv hH$.

The same holds for $\ainv(X,B+\sP)=\ainv,\glct(X,B+\sP)$
of lc Fano bd-pairs $(X,B+\sP)$ of index $m$
with $E=D/n\equiv-\glct(X,B)(K+B+\sP_X)$ and
$-(K+B+\sP)\equiv hH$, where
$n$ depends also on $m$ and
$$
D\in \linsys{-nK-n\sP_X-nB-naH}.
$$.

\end{cor}

\begin{proof}
By Corollary~\ref{a_attained} $a$ is attained.
Actually, $t\le 1$ is attained and $a=(1-t)h$.
Thus there exists an effective $\R$-divisor $E$ such that
$(X,B+E)$ is lc but not klt and
$$
K+B+E+aH\equiv 0.
$$
There exists $N$ depending only on $d$ such that
$NH\equiv F$, where $F$ is a free divisor on $X$.
In particular, we can suppose that
$(X,B+E+(a/N)F)$ is lc too, that is,
$F$ does not pass the prime components of $\Supp B,\Supp E$
and the center of any prime b-divisor $P$
with $\ld(P;X,B+E)=0$.
As in the proof of Corollary~\ref{a_attained},
in the exceptional case,
$$
(Y,B+\frac aNF+P)
$$
has an $\R$-complement.
More precisely, in this case we suppose that
every $P$ is exceptional.
The set of multiplicities $\Gamma_b\cup\{1,a/N\}$
is a rational closed dcc subset of $[0,1]$.
(We can suppose that $a/N\le 1$.)
Hence there exists a monotonic $n$-complement
$(Y,B^+)$ of $(Y,B+(a/N)F+P)$.
We suppose also that $na/N$ is integer.
Then by monotonicity the divisor
(in the linear system of $\Q$-divisors)
$$
D'=nB^+-nB-\frac {na}N F -nP
\in
\linsys{-nK_Y-nB-\frac {na}N F-nP}
$$
is effective.
Taking the image of $D'$ on $X$ we get
$$
D\in
\linsys{-nK-nB-naH}
$$
and $E=D/n$ is required effective.

Similarly, if there are nonexceptional $P$ then
$B+E=B'+E'+S$, where $S$ is reduced part of $B+E$,
that is, the sum of those $P$ (assuming that $(X,B)$ is klt) and $B'=B,E'=E\ge 0$
outside of $\Supp S$.
In this situation,
we have a monotonic $n$-complement $(X,B^+)$
of $(X,B'+(a/N)F+S)$.
Thus there exists an effective divisor
$$
D'=nB^+-nB'-\frac {na}N F -nS
\in
\linsys{-nK-nB'-naH-nS}.
$$
By construction $B'=B$ outside of $S$ and
$B'=0$ on $S$.
Hence
$$
nB'+nS\ge nB,
$$
because $B$ is a (sub)boundary and $S$ is reduced.
Hence again
$$
D=D'+nB'+nS-nB=
nB^+-nB-\frac {na}N F
\in
\linsys{-nK-nB-naH}
$$
is also effective and  $E=D/n$ is required effective.

Similarly we can treat bd-pairs.

\end{proof}

\paragraph{Bounded affine span and index of divisor.}
Let $V\subseteq\R^n$ be a class of $\R$-linear spaces $\R^l$ with a standard basis
and with their $\Q$-affine subspace $V$.
Such a subspace $V$ can be given by linear equations
(possibly nonhomogeneous) with integral coefficients
in the standard basis.
We say that $V$ in this class is {\em bounded\/} if,
for all $V$ of the class,
the integer coefficients of equations are bounded.
Equivalently, $V$ is bounded if, in every $V$, there exists
a finite set of rational generators $v_i,i=1,\dots,l$,
with coordinates in a finite set of rational numbers.
Vectors $v_i$ generate $V$ if
$\spn{v_1,\dots,v_l}=V$.

We apply the boundedness to affine $\Q$-spans $\spn{D}$
of certain $\R$-divisors $D$ on $X$.
Every such span is in the space $\WDiv_\R X$ of $\R$-divisors,
actually, in the space of divisors supported
on $\Supp D$.
The standard basis of $\WDiv_\R X$
consists the prime components of $\Supp D$ or
prime divisors of $X$.

\begin{exa}

(1)
If $D$ is rational divisor then $\spn{D}$ is
a rational divisor itself.
Those spans or divisors
are bounded if their multiplicities belong to
a finite set of rational numbers.
In general, $\spn{D}$ is bounded if there exists
a finite set of $\Q$-divisors $D_i,i=1,\dots,l$, in $\spn{D}$
with multiplicities in a finite set of rational numbers,
which generate $\spn{D}$.

(2) The spaces $x_1=x_2,\dots,x_{2l-1}=x_{2l}$ are bounded.
Every of those spaces has generators
$(1,1,0,\dots,0,0),(0,0,1,1,\dots,0,0),\dots,
(0,0,0,0,\dots,1,1)$
with $2l$ coordinates $0$ or $1$.

\end{exa}

Due to the rationality of intersection theory and
since the lc property is rational,
if $(X/Z\ni o,D)$ is a $0$-pair then
$(X/Z\ni o,D')$ is a possibly nonlc $0$-pair
for some $D'\in \spn{D}$ but a $0$-pair in
some neighborhood of $D$ in $\spn{D}$.
The last neighborhood depends on $(X/Z\ni o,D)$.
Notice also that the maximal lc property also
holds in some neighborhood of $D$ in $\spn{D}$
if it holds for $(X/Z\ni o,D)$.
Moreover, $a=\ld(P;X,D)=\ld(P;X,D')$ holds for
all $D'\in \spn{D}$ if $a$ is rational.

Let $I$ be a positive integer.
We say that $I$ is a {\em lc index\/} of a $0$-pair $(X/Z\ni o,D)$
if there are rational generators $D_i,i=1,\dots,l$, of $\spn{D}$ such that
every $I(K+D_i)\sim 0/Z\ni o$.
Note that if $(X\ni o,D)$ is a log pair with $D\in\Q$ then
$(X\ni o,D)$ is a $0$-pair over $X\ni o$ and
an lc index of $(X\ni o,D)$ is a Cartier index of
$K+D$.

The same applies to $0$-bd-pairs $(X/Z\ni o,B+\sP)$
(of index $m$)
with $I(K+D_i+\sP_X)\sim 0/Z\ni o$.

\begin{cor}[Boundedness of lc index] \label{bounded_lc_index}
Let $d$ be a nonnegative integer and
$\Gamma$ be a dcc subset in $[0,1]$.
Then there exists a finite subset
$\Gamma(d)\subseteq \Gamma$ and
positive integer $I=I(d,\Gamma)$ such that,
for every maximal lc $0$-pair $(X/Z\ni o,B)$
with wFt $X/Z\ni o$, $\dim X=d$ and $B\in \Gamma$,
\begin{description}

\item[\rm (1)\/]
$B\in\Gamma(d)$;

\item[\rm (2)\/]
$\spn{B}$ is bounded; and

\item[\rm (3)\/]
$I$ is an lc index of $(X/Z\ni o,B)$.

\end{description}

The same holds for maximal lc $0$-bd-pairs $(X/Z\ni o,B+\sP)$ of index $m$
with $\Gamma(d,m)$,$I(d,\Gamma,m)$ depending also on $m$.

\end{cor}

In the proof below, we use $n$-complements for a finite set
of rational boundary multiplicities.
On the other hand, in construction of complements
we use the corollary in very special case $\Gamma=\Phi(\fR)$,
a hyperstandard set, and $(X,B)$ is a klt $0$-pair with wFt
(cf. Corollary~\ref{invar_adj} and
What do we use in the proof? in Introduction).

\begin{proof}
We suppose that $d\ge 1$.
(The case $d=0$ is trivial.)

(1) $\Gamma(d)=\Gamma\cap \Rct(d,\Gamma)$,
where $\Rct(d,\Gamma)$ denotes the set
of $\Rct(X/Z\ni o,B)$ for $\dim X=d,
B\in\Gamma$.
By Theorem~\ref{acc_R-complt}
$\Rct(d,\Gamma)$ is an acc set.
Thus $\Gamma(d)$ satisfies acc and dcc, and
is finite.

By definition, for every prime divisor $P$ on $X$ over $Z\ni o$,
$$
\Rct(X/Z\ni o,B-bP;P)=b,
$$
where $b=\mult_PB$.
By construction $B-bP,b\in\Gamma$ and
$b$ belongs $\Rct(d,\Gamma)$.
Hence $b,B\in\Gamma(d)$.

(2) Put $l$ to be the number of elements in $\Gamma(d)$.
Then for every $0$-pair $(X/Z\ni o,B)$ under
the assumptions of the corollary there exist
distinct reduced Weil divisors $D_1,\dots,D_l$ such that
$B=\sum_{i=1}^l b_iD_i$, where $b_i\in\Gamma(d)$.
(Some of multiplicities $b_i$ are $0$.)
This gives a $\Q$-linear map $A\colon \R^l\to\WDiv_\R X$,
which transforms the standard basis
$e_1=(1,0,\dots,0),\dots,e_l=(0,0,\dots,1)$
into the divisors $D_1,\dots,D_l$ respectively.
The hight of $A$ is $1$.
There exists a unique vector $v\in\R^l$ with $A(v)=B$.

By Corollary~\ref{direct_stability} under the assumption~(3) of Theorem~\ref{bndc},
there exist vectors $w_0,w_1,\dots,w_r\in\spn{v},r=\dim\spn{v}\le l$,
which generate $\spn{v}$ and such that
$B_0=A(w_0),B_1=A(w_1),\dots,B_r=A(w_r)$ are
rational boundaries.

Thus $\spn{B}=A(\spn{v})$ is generated by $B_0,B_1,\dots,B_r$ and
is bounded.

(3) By construction all multiplicities of boundaries $B_i$
belong to a finite set of rational numbers.
On the other hand, every pair $(X/Z\ni o,B_i)$ is a $0$-pair
and has a monotonic $n$-complement $(X/Z\ni 0,B^+)$ for some $n$, depending
only on the multiplicities of boundaries $B_i$ and on $d$.
We can take $I=n$.
Indeed, $n(K+B_i)\sim 0$ over $Z\ni o$ because $B^+=B_i$.

Similarly we can treat bd-pairs.

\end{proof}

\begin{cor}[Invariants of adjunction] \label{invar_adj}
Let $d$ be a nonnegative integer and
$\Gamma$ be a dcc set of rational numbers in $[0,1]$ .
Then there exists a finite subset
$\Gamma(d)\subseteq \Gamma$ and
positive integer $I=I(d,\Gamma)$ such that
every $0$-contraction $f\colon(X,D)\to Z$ as in Theorem~\ref{adjunction_index}
has the adjunction index $I$.
Moreover, $D\hor\in \Gamma(d)$ and
$Ir_P\in\Z$ for every adjunction constant $r_P$ as in~\ref{adjunction_mult}.

The same holds for every $0$-contraction $(X,D+\sP)\to Z$ as
in Addendum~\ref{adjunction_index_bd}.
In this situation $I=I(d,\Gamma,m),\Gamma(d,m)$ depend also on
the index $m$ of the bd-pair $(X,D+\sP)$.

\end{cor}

However, first we establish the following.

\begin{proof}[Proof of Theorem~\ref{adjunction_index}]
(General case.)
We suppose that $1\in\Gamma$ and
take $\Gamma(d)$ and $I=I(d,\Gamma)$ as in Corollary~\ref{bounded_lc_index}.
We use the same proof as in the hyperstandard case
with one improvement.
According to Corollary~\ref{bounded_lc_index}
the maximal lc $0$-pair $(X/Z,D)$ locally over $\Supp D\dv$,
in Step~2 of the proof of Theorem~\ref{adjunction_index},
has index $I=I(d,\Gamma)$ and $D\hor=D\in\Gamma(d)$.
Recall, that the vertical multiplicities of $D$
are $0$ or $1$ locally over $\Supp D\dv$.

\end{proof}

\begin{proof}[Proof of Corollary~\ref{invar_adj}]
Immediate by Theorem~\ref{adjunction_index}.
\end{proof}

\paragraph{Accumulations of $\Rct$ thresholds.}
Denote by $\Rct(d,\Gamma_b,\Gamma_f)$
the thresholds of Theorem~\ref{acc_R-complt}.
Denote by $\act(d,\Gamma_b,\Gamma_f),\lct(d,\Gamma_b,\Gamma_f)$
corresponding ac and lc thresholds
(see Lemma~\ref{lct_act} and Corollaries~\ref{acc_lc_thresholds}, \ref{acc_ac_free};
cf. Remark~\ref{act_rem} and Example~\ref{act_exa}).
The thresholds satisfies the acc but not the dcc.
However, the accumulations have
rational constrains it terms of the closures
$\overline{\Gamma_b},\overline{\Gamma_f}$.
Notice that both closures
$\overline{\Gamma_b},\overline{\Gamma_f}$ are
also dcc (nonnegative) sets if so do $\Gamma_b,\Gamma_f$.
Thus for simplicity we can suppose that
$\Gamma_b,\Gamma_f$ are already closed.

The same applies to bd-pairs of index $m$.
In particular, we can consider
$\Rct(d,\Gamma_b,\Gamma_f,m),\act(d,\Gamma_b,\Gamma_f,m),\lct(d,\Gamma_b,\Gamma_f,m)$
and their accumulation points.

\begin{cor} \label{constrain}
Let $d$ be a nonnegative integer and
$\Gamma_b,\Gamma_f$ be two closed dcc sets of nonnegative real numbers.
Then every accumulation threshold $t$ in every set
$$
\Rct(d,\Gamma_b,\Gamma_f),
\act(d,\Gamma_b,\Gamma_f),
\lct(d,\Gamma_b,\Gamma_f)
$$
has a rational constrain between $B=(b_1,\dots,b_l),
0<tF=t(f_1,\dots,f_l)$, where all $b_i\in\Gamma_b,f_i\in\Gamma_f$,
that is,
$B+\R F\not\subseteq\spn{B+tF}$.

\end{cor}

\begin{add}[cf. {\cite[Corollary$^+$]{Sh94}}] \label{rational_accum}
If additionally $\Gamma_b,\Gamma_f\subset \Q$,
then
$$\overline{\Rct(d,\Gamma_b,\Gamma_f)},
\overline{\lct(d,\Gamma_b,\Gamma_f)},
\overline{\act(d,\Gamma_b,\Gamma_f)}\subset\Q.
$$
\end{add}

\begin{add}
The same holds for thresholds
$$
\Rct(d,\Gamma_b,\Gamma_f,m),
\act(d,\Gamma_b,\Gamma_f,m),
\lct(d,\Gamma_b,\Gamma_f,m)
$$
of bd-pairs of index $m$.

\end{add}

\begin{proof}
By Lemma~\ref{lct_act} and dimensional induction
it is enough to prove the result for
the accumulation points of ac and lc thresholds in the dimension $d$.
We consider only ac thresholds.
Lc thresholds can be treated similarly.

Consider now an accumulation point $t$ for
$\act(X_i,B_i;F_i)=\Rct(X_i,B_i;F_i),i=1,2,\dots,i,\dots$, where
\begin{description}

  \item[\rm (1)]
$X$ is $\Q$-factorial, Ft of dimension $d$ with $\rho(X)=1$;

  \item[\rm (2)]
$B_i\in\Gamma_b,F_i\in\Gamma_f$ and $F_i>0$;
and
  \item[\rm (3)]
$(X_i,B_i+t_iF_i)$ is a $0$-pair, in particular maximal lc,
where $t_i=\act(X_i,B_i;F_i)$.

\end{description}
Put $t=\lim_{i\to\infty}t_i$.
However, this time
$$
0<t< \dots <t_i<\dots<t_1
$$
because $\act(d,\Gamma_b,\Gamma_f)$ satisfies the acc
and $t$ is an accumulation point.
($t=0$ is rational and always an accumulation point in
dimensions $\ge 1$. Cf.
with our assumption $0\in\Phi$ in Hyperstandard sets is Section~\ref{technical}.)

By (1-2) and our assumptions,
every $\Supp B_i\cup\Supp F_i$ have at most $l$
prime divisors for some positive integer $l$.
We will use the same notation as in Step~1 of
the proof of Theorem~\ref{acc_R-complt}.
Since $\Gamma_b,\Gamma_f$ are closed,
$A(x)\in \Gamma_b,A(y)\in\Gamma_f$ hold.
As in Step~1 ibid $(X_i,A(x+ty))$ has an $\R$-complement
and $A(x+ty)=A(x)+tA(y)$ is a boundary for all $i\gg 0$.

By Corollary~\ref{direct_stability}
there exists a neighborhood $U$ of $x+ty$ in
$\spn{x+ty}$ such that, for every $v\in U$,
$(X_i,A(v))$ has an $\R$-complement.
Thus if $x+\R y\subseteq\spn{x+ty}$ then
there exists a real number $\ep>0$ such that
$(X_i,A(x+(t+\ep y))$ has an $\R$-complement,
or equivalently, $\Rct(X_i,A(x);A(y))\ge t+\ep$
for all $i\gg 0$.
On the other hand, by construction
every $x_i\le x,B_i=A(x_i)\le A(x)$ and $y_i\le y,F_i=A(y_i)\le A(y)>0$;
by~(1) $\rho(X)=1$.
Hence $\Rct(X_i,A(x);B(x))\le t_i$: for every curve $C$ on $X$ and $t'>t_i$,
$$
(C.A(x)+t'A(y))>(C.A(x)+t_iA(y))\ge
(C.B_i+t_iF_i)=0
$$
by (3).
So, $t_i\ge t+\ep$, a contradiction.
This proves that $B+\R F\not\subseteq \spn{B+tF}$,
where $B=x\in \Gamma_b,F=y\in\Gamma_f$.

For Addendum~\ref{rational_accum} notice that
if $B,0\not=F\in\Q$ and $B+\R F\not\subseteq \spn{B+tF}$
then $t$ is rational: $(B+\R F)\cap\spn{B+tF}=\{B+tF\}$.

Similarly we can treat bd-pairs.

\end{proof}

\section{Open problems}

It is expected that in most of our results and, in particular,
in most of our applications we can omit the assumption
to have wFt but bd-pairs should be of Alexeev type.

\begin{exa} \label{bd_elliptic}
Let $E$ be a complete nonsingular curve of genus $1$
and $P,Q\in E$ be two closed points on $E$.
Then $\sP=p-q$ is a numerically trivial Cartier (b-)divisor, in particular,
nef.
Thus $(E,\sP)$ is a bd-pair of index $1$ in
the Birkar-Zhang sense but in the Alexeev sense only
if $p=q$ and $\sP=0$.
The pair $(E,\sP)$ has an $\R$-complement if and only if
$\sP\sim_\Q 0$, that is, a torsion.
Since the torsions are not bounded,
$n$-complements of those pairs are not bounded too.

Notice also that $(E,\sP)$ is a maximal lc pair.
But the lc indices of those pairs are also not bounded.

\end{exa}

Thus general bd-pairs are a very good instrument
for wFt morphisms but for general morphisms
it is expected that a generalization of
our results works only for Alexeev bd-pairs $(X/Z,D+\sP)$ of index $m$.
Such a pair or, for short, {\em Alexeev pair of index\/} $m$ is
an Alexeev pair $(X/Z,D+\sP)$ with $\sP=\sum r_iL_i$ and $0\le mr_i\in\Z$
(cf. Conjecture~\ref{mod_part_b-semiample} and Corollary~\ref{Alexeev_index_m}).
Notice also that we collect all mobile linear systems in $\sP$ and
with one fixed prime divisor in $D$.

\begin{conj}[Existence and boundedness of $n$-complements] \label{conjecture_bndc}
It is expected that Theorems~\ref{R-vs-n-complements}, \ref{exist_n_compl},
\ref{bndc}, \ref{lc_compl} and \ref{lc_compl_A}
hold without the assumption that $X/Z\ni o$ has wFt.
Additionally we can relax the connectedness assumption on $X_o$ and
suppose instead that the number of connected components of $X_o$
is bounded (cf. Addendum~\ref{bounded_compon}).

The same expected for Alexeev pairs of index $m$
(cf. Conjecture~\ref{mod_part_b-semiample} below).

\end{conj}

The conjecture holds in dimension $d\le 2$.
Methods of the paper allows to prove Theorem~\ref{bndc}
in dimension $d\le 2$ with the existence of an $\R$-complement
instead of a klt $\R$-complement in (1) of the theorem
(cf. Examples~\ref{unbounded_lc_compl}, (1-3) and
\cite[Inductive Theorem~2.3]{Sh95}).
In this situation assumptions (2-3) of Theorem~\ref{bndc} and
Theorems~\ref{lc_compl} and \ref{lc_compl_A} are redundant.
See $n$-complements for $d=1$, essentially, for $\PP^1$
in Examples~\ref{R_comp_dim 1}, \ref{n_comp_dim_1}.

Most of our applications follows directly from
results about $n$-complements.
The following result also follows
from the boundedness of $n$-complements.

\begin{cor-conj}[Boundedness of lc index] \label{conj_bounded_lc_index}
Let $d$ be a nonnegative integer and
$\Gamma\subset [0,1]$ be a dcc subset.
Then there exists a finite subset
$\Gamma(d)\subseteq \Gamma$ and
positive integer $I=I(d,\Gamma)$,
depending only on $d$ and $\Gamma$, such that,
for every maximal lc $0$-pair $(X/Z\ni o,B)$
with $\dim X=d$ and $B\in \Gamma$,
\begin{description}

\item[\rm (1)\/]
$B\in\Gamma(d)$;

\item[\rm (2)\/]
$\spn{B}$ is bounded; and

\item[\rm (3)\/]
$I$ is an lc index of $(X/Z\ni o,B)$.

\end{description}

The same is expected for maximal lc Alexeev $0$-pairs $(X/Z\ni o,B+\sP)$ of index $m$
with $\Gamma(d,m)$,$I(d,\Gamma,m)$ depending also on $m$.

\end{cor-conj}

\begin{proof}[Proof-derivation]
Indeed, we can use the proof of Corollary~\ref{bounded_lc_index}.
However, before we need to derive nonwFt versions of
Theorem~\ref{acc_R-complt} and of Corollary~\ref{direct_stability}.
\end{proof}

In particular, the last conjecture includes the global case
with a rational dcc subset $\Gamma\subset [0,1]$:
every $0$-pair $(X,B)$ with $B\in \Gamma$
of a {\em bounded\/} dimension has a {\em bounded\/} (global) lc index $I$,
that is,
$$
I(K+B)\sim 0.
$$
Moreover, for a proof of Conjecture~\ref{conjecture_bndc},
two extreme cases are important: Ft pairs and
global $0$-pairs with $B=0$ and canonical singularities, known as Calabi-Yau varieties.
So, the following very famous conjecture is indispensable
for the theory of complements and
possibly related to topology (cf. \cite[Theorema~6.1]{FM}).

\begin{conj}[Index conjecture]
Let $d$ be a nonnegative integer.
Then there exists
a positive integer $I=I(d)$ such that,
every complete variety or space $X$
with $\dim X=d$, only canonical singularities and $K\equiv 0$ has
index $I$:
$$
IK\sim 0.
$$

\end{conj}

E.g, if we assume additionally that $X$ has only canonical singularities,
then $I(1)=1,I(2)=12$ (classical) and
$$
I(3)=2^5\cdot 3^3\cdot 5^2\cdot 7\cdot 11\cdot 13\cdot 17\cdot 19,
$$
the Bauville number.
The next indices $I(d),d\ge 4$, and their existence is unknown.
We know that, for every variety or space in Index conjecture, $K\sim_\Q 0$,
that is, there exists an index \cite[Theorem~0.1,(1)]{A05}.

A reduction to Ft varieties and varieties of Index conjecture
uses the LMMP and the following semiampleness.

\begin{conj}[Effective b-semiampleness; cf. {\cite[(7.13.3)]{PSh08}}]
\label{mod_part_b-semiample}
Let $d$ be a nonnegative integer and
$\Gamma\subset [0,1]$ be a rational dcc subset.
Then there exists a positive integer $I=I(d,\Gamma)$ such that,
for every $0$-contraction $(X,D)\to Z$
under assumptions~\ref{adjunction_0_contr} and
with $\dim X=d,D\hor\in\Gamma$,
$I\sD\md$ is b-free, that is,
$I\sD\md=\overline{M}$, where $M$ is a base point free
divisor on some model $Z'$ of $Z$.
In other words, $(Z,D\dv+\sD\md)$ is an Alexeev log pair
(not necessarily complete) of index $I$.

The same is expected for $0$-contractions $(X,B+\sP)\to Z$
of Alexeev pairs $(X,B+\sP)$ of index $m$
with $I(d,\Gamma,m)$ depending also on $m$.

\end{conj}

Corollary-Conjecture~\ref{conj_bounded_lc_index} and
construction of $n$-complements for fibrations
can be reduced to the global case as for Ft fibrations.
In its turn, if a global $0$-pair $(X,B)$ has a nonzero
boundary then the lc index conjecture can be reduced to
a Ft variety $X$ or to a nontrivial Ft fibration $(X,B)\to Z$.
In the latter case we can use dimensional induction for
the base $Z$.
If $\dim Z\ge 1$, we get a bd-pair $(Z,B\dv+\sB\md)$ with
possibly nonFt $Z$.
As we already noticed the induction will not
work for bd-pairs of Birkar-Zhang pairs but
is expected to work for Alexeev pairs.
As for Ft varieties in the most critical for us
situation $B$ has hyperstandard multiplicities.
Then according to Conjecture~\ref{mod_part_b-semiample}
$(X,B\dv+\sB\md)$ is actually an Alexeev pair
of bounded index.

It is well-known \cite[Corollary~7.18]{PSh08} that
the Alexeev log pair $(X,B\dv+\sB\md)$ can be easily
converted into a log pair $(X,B\dv+B\md)$ for
a suitable boundary $B\md$ with a finite set of
rational multiplicities (depending only on the index of
the Alexeev pair).

\begin{cor} \label{Alexeev_index_m}
Let $(X,D+\sP)$ be an Alexeev (log, lc, klt) pair of index $m$.
Then there exists an effective divisor $P$ such
that $(X,D+P)$ is  also a (respectively, log, lc klt) pair,
$mP\in m\sP_X$ and $\Supp D,\Supp P$ are disjoint.
In particular, $P\in [0,1]\cap(\Z/m)$.

\end{cor}

\begin{add} \label{Alexeev_index_m_discrip}
$\D(X,D+P)\ge\D\dv=\D(X,D+\sP_X)-\sP$.
\end{add}

\begin{add}
Let $\Gamma$ be a set of real numbers.
If additionally $D\in\Gamma$ then
$$
B+P\in \Gamma\cup([0,1]\cap \frac \Z m).
$$
\end{add}

\begin{add}
If additionally $\Gamma$ satisfies the dcc (in $[0,1]$, in a hyperstandard set)
then
$$
\Gamma\cup([0,1]\cap \frac \Z m).
$$
satisfies the dcc (respectively, in $[0,1]$, in a hyperstandard set).

\end{add}

\begin{proof}(Cf. \cite[Corollay~7.18]{PSh08}.)
Take a crepant model $f\colon (Y, D_Y+\sP)\to (X,D+\sP)$
such that $\sP$ is {\em free over\/} $Y$, that is, $\sP=\sum r_i L_i$ and
every $L_i$ is a base point free linear system on $Y$.
Since $(X,D+\sP)$ is an Alexeev pair of index$m$,
$mr_i=m_i$ is a nonnegative integer.
By the Bertini theorem we can pick up $m_i$ rather general effective divisors
$E_{i,j},j=1,\dots,m_i$ in every linear system $L_i$ such
that
$$
P=\sum f(E_{i,j})/m
$$
satisfies  the required properties.
The inclusion $mP\in m\sP_X$ means  that
every summand $f(E_{i,j})$ of $mP$ belongs
to the corresponding linear system $L_i$ on $X$.
The klt property of $(X,D+P)$ follows from that of
$(X,D+\sP)$ if $m\ge 2$.
Otherwise we replace $m=1$ in our construction by
any positive integer $\ge 2$.

Notice that log discrepancies of $(X,D+P)$ are
strictly larger than that of $(X,D+\sP)$ only
over $\Supp \cup E_{i,j}$.
The other log discrepancies are the same.
This proves Addendum~\ref{Alexeev_index_m_discrip}.

Two other addenda are immediate by definition.

\end{proof}

The most important case of Index conjecture is
the case with a variety $X$ having only canonical and
even terminal singularities.
Otherwise, we can blow up a prime b-divisor
$P$ of $X$ with a log discrepancy $a=\ld(P;X,0)<1$.
A result is a crepant $0$-pair $(Y,(1-a)P)$ with $0<1-a\le 1$.
Every such case can be reduced as above to
Ft varieties or varieties of Index conjecture of
smaller dimension.
However, this works for all lc or klt $0$-pairs $X$ of given
dimension $d$ if $a$ belongs to an acc set.
This follows from the acc of mld's \cite[Problem~5]{Sh88}
for mld's $a<1$ because then $1-a$ form a dcc set in
the dimension $d$.
Actually, we can use much weaker and already known
results about $n$-complements when $a$ is sufficiently
close to $0$.
In this situation $a=0,1-a=1$ and Index conjecture again
can be done by dimensional induction.
The same works for $a$ which are not close to $1$.
But $a<1$ are actually not close to $1$ by the following
very special case of the acc for mld's conjecture.
If singularities are canonical we can blow up
all $P$ with $a(P;X,0)=1$; in this case $1-a=0$.

\begin{conj}[Gap conjecture] \label{gap}
Let $d$ be a nonnegative integer.
Then there exists a positive integer $a<1$
such that if $X\ni P$ with $\dim X=d$ is $a$-klt, that is,
the mld of $X$ at $P$ is $>a$, then $X\ni p$ is
canonical, that is the mld of $X$ at $P$ is $\ge 1$.

\end{conj}

Index and Gap conjectures are established up to
dimension $d=3$ \cite[Theorem~1.3 and Corollary~1.7]{J} \cite[Theorem~1.4]{LX}.

Finally, we state the following $\ep$-lc strengthening
of $n$-complements.
However, we use $a$ instead of $\ep$.

\begin{conj}[$a$-$n$-complements; cf. {\cite[Conjecture]{Sh04b}}] \label{a_n_compl}
Let $a$ be a nonnegative real number.
Assume additionally in Theorems~\ref{R-vs-n-complements}, \ref{exist_n_compl},
\ref{bndc}, \ref{lc_compl} and \ref{lc_compl_A} that
$(X/Z\ni o,B)$ has an $a$-lc over $o$
$\R$-complement then in all of these theorems it expected
the existence an $a$-$n$-complement $(X/Z,B^+)$ with $n\in\sN$, that is,
additionally $(X,B^+)$ is $a$-lc over $o$.
The set of complementary indices $\sN$ is this case depends also on $a$.
However, Restrictions on complementary indices are expected to hold
only for $I$ but not for approximations if $a$ is not rational.

The same expected for Alexeev pairs of index $m$.

\end{conj}

To remove the wFt assumption in the conjecture we
need the following concept.

\begin{defn} \label{strict_delta_lc}
Let $(X/Z,D)$ be a log pair and $\delta$ be a nonnegative real number.
The pair is {\em strictly\/} $\delta$-lc if
for every b-$0$-contraction
$$
\begin{array}{ccc}
  (X,D) &\stackrel{\varphi}{\dashrightarrow} & (Y,D_Y)\\
  & &f\downarrow \\
  & &Z
\end{array}
$$
such that
\begin{description}
  \item[\rm (1)]
$\varphi$ is a crepant birational $1$-contraction;
  \item[\rm (2)]
$f$ is a $0$-contraction as in~\ref{adjunction_0_contr},

\end{description}
$(Z,\D_{Y,}{}\dv)$ is $\delta$-lc, that is,
every multiplicity of the b-$\R$-divisor $\D_{Y,}{}\dv$
is $\le 1-\delta$, or, equivalently, the bd-pair
$(Z,D_{Y,}{}\dv+\sD_{Y,}{}\md)$ is $\delta$-lc.

\end{defn}

Notice that the strict $\delta$-lc property implies the usual
$\delta$-lc property of $(X,D)$.
The converse holds for $\delta=0$ by (6) in~\ref{adjunction_div}.
The converse also is expected for wFt $X/Z$ but with
a different $\delta'\le \delta$ but not in general,
e.g., for fibrations of genus $1$ curves.
Of course, the strict $\delta$-lc property means exactly the
boundedness of adjunction constants (multiplicities) $l=l_P$ in~\ref{adjunction_mult}
for every vertical prime b-divisors $P$ of $Y$ with $r=r_P=1$
(cf. \cite[Theorem~1.6, (3)]{CH}).

\begin{add} \label{strict_a_n_compl}
In Conjecture~\ref{a_n_compl}
we can replace the
wFt assumption by the {\em strict\/} $a$-lc property over $o$:
the exists $\delta\ge 0$ such that
\begin{description}
  \item[]
$\delta >0$ if $a>0$; and

  \item[]
there exists a strict $\delta$-lc over $o$ $\R$-complement
of $(X/Z\ni o,B)$.

\end{description}
But still we keep the assumption that $(X/Z\ni o,B)$
has an $a$-lc over $o$ $\R$-complement.

The boundedness of lc indices expected for indices of maximal $a$-lc
$0$-pairs under additionally the strict $\delta$-lc property over $o$.
\end{add}

We already established the conjecture for $a=0$.
So, corresponding $n$-complements are better to
call lc $n$-complements.

Notice that the $a$-lc property over a point $o$
means that $a(P;X,B)\ge a$ for every b-divisor of $X$
over $o$.
In particular, if $(X/o,B)$ is global and $a>0$
then by BBAB projective Ft varieties $X$ are bounded in any fixed dimension $d$
and the conjecture can be easily established by approximations as
for lc $n$-complements in the exceptional case of Section~\ref{excep_n_comp}
if positive multiplicities of $B$ are bounded from below.
However, the case with small positive multiplicities of $B$
is already difficult.
This case is interesting because then $a$ can be very close to $1$
but not above $1$ in the global case.
The nonglobal case is much more difficult and related
to other nontrivial conjectures \cite[Conjectures~1.1 and 1.2]{BC}.

Notice also that if $a$ is not rational then
Restrictions on complementary indices related
to approximations can collide with $a$ because
for $a$-$n$-complements,
the $a$-lc property implies $m/n$-lc property
with upper approximation $m/n$ of $a$.

Conjecture~\ref{a_n_compl} does not hold for
nonFt morphisms $X/Z$.
For instance, minimal nonsingular fibrations
of genus $1$ curves over a curve have fibers with unbounded
multiplicities.
Thus lc indices of a canonical divisor near those fibers
are unbounded too.
This contradicts to related Index conjecture
for $1$-lc maximal $0$-pairs (for $a=1$).
However this does not give a contradiction if additionally
the pair is strictly $\delta$-lc
(see Addendum~\ref{strict_a_n_compl}).

Possibly, it is better to start from Addendum~\ref{strict_a_n_compl}
and \cite[Conjecture~1.2]{BC} because they
imply Conjecture~\ref{a_n_compl}.

Department of Mathematics, Johns Hopkins University, Baltimore, MD 21218, USA

shokurov@math.jhu.edu

Institute of Mathematics, Russian Academy of Sciences, Moscow, Russia

shokurov@mi-ras.ru

\end{document}